\documentclass[11pt,english]{smfbook}

\def\longformule#1#2{
\displaylines{ \qquad{#1} \hfill\cr \hfill {#2} \qquad\cr } }
\def\build#1_#2^#3{\mathrel{
\mathop{\kern 0pt#1}\limits_{#2}^{#3}}}

\def\aref#1{(\ref{#1})}
\def\e{\rm e}
\def\R{\mathop{\mathbb R\kern 0pt}\nolimits}
\def\H{\mathop{\mathbb H\kern 0pt}\nolimits}
\def\C{\mathop{\mathbb C\kern 0pt}\nolimits}
\def\N{\mathop{\mathbb N\kern 0pt}\nolimits}
\def\ZZZ{\mathop{\mathbb Z\kern 0pt}\nolimits}
\def\Q{\mathop{\mathbb Q\kern 0pt}\nolimits}

\renewcommand{\theequation}{\thesection.\arabic{equation}}

\newtheorem{theo}{Theorem}
\newtheorem{prop}{Proposition}[chapter]
\newtheorem{lemme}[prop]{Lemma}
\newtheorem{defin}[prop]{Definition}
\newtheorem{cor}[prop]{Corollary}

\newtheorem{rem}[prop]{Remark}

\begin{document}


\makeatletter
\def\sommaire{\@restonecolfalse\if@twocolumn\@restonecoltrue\onecolumn
\fi\chapter*{Sommaire\@mkboth{SOMMAIRE}{SOMMAIRE}}
  \@starttoc{toc}\if@restonecol\twocolumn\fi}
\makeatother

\makeatletter
\def\thebibliographie#1{\chapter*{Bibliographie\@mkboth
  {BIBLIOGRAPHIE}{BIBLIOGRAPHIE}}\list
  {[\arabic{enumi}]}{\settowidth\labelwidth{[#1]}\leftmargin\labelwidth
  \advance\leftmargin\labelsep
  \usecounter{enumi}}
  \def\newblock{\hskip .11em plus .33em minus .07em}
  \sloppy\clubpenalty4000\widowpenalty4000
  \sfcode`\.=1000\relax}
\let\endthebibliography=\endlist
\makeatother

\makeatletter
\def\references#1{\section*{R\'ef\'erences\@mkboth
  {R\'EF\'ERENCES}{R\'EF\'ERENCES}}\list
  {[\arabic{enumi}]}{\settowidth\labelwidth{[#1]}\leftmargin\labelwidth
  \advance\leftmargin\labelsep
  \usecounter{enumi}}
  \def\newblock{\hskip .11em plus .33em minus .07em}
  \sloppy\clubpenalty4000\widowpenalty4000
  \sfcode`\.=1000\relax}
\let\endthebibliography=\endlist
\makeatother

\def\lunloc#1#2{L^1_{loc}(#1 ; #2)}
\def\bornva#1#2{L^\infty(#1 ; #2)}
\def\bornlocva#1#2{L_{loc}^\infty(#1 ;\penalty-100{#2})}
\def\integ#1#2#3#4{\int_{#1}^{#2}#3d#4}
\def\reel#1{\R^#1}
\def\norm#1#2{\|#1\|_{#2}}
\def\normsup#1{\|#1\|_{L^\infty}}
\def\normld#1{\|#1\|_{L^2}}
\def\nsob#1#2{|#1|_{#2}}
\def\normbornva#1#2#3{\|#1\|_{L^\infty({#2};{#3})}}
\def\refer#1{~\ref{#1}}
\def\refeq#1{~(\ref{#1})}
\def\ccite#1{~\cite{#1}}
\def\pagerefer#1{page~\pageref{#1}}
\def\referloin#1{~\ref{#1} page~\pageref{#1}}
\def\refeqloin#1{~(\ref{#1}) page~\pageref{#1}}
\def\suite#1#2#3{(#1_{#2})_{#2\in {#3}}}
\def\ssuite#1#2#3{\hbox{suite}\ (#1_{#2})_{#2\in {#3}}}
\def\longformule#1#2{
\displaylines{
\qquad{#1}
\hfill\cr
\hfill {#2}
\qquad\cr
}
}
\def\inte#1{
\displaystyle\mathop{#1\kern0pt}^\circ
}
\def\sumetage#1#2{
\sum_{\scriptstyle {#1}\atop\scriptstyle {#2}}
}
\def\limetage#1#2{
\lim_{\scriptstyle {#1}\atop\scriptstyle {#2}}
}
\def\infetage#1#2{
\inf_{\scriptstyle {#1}\atop\scriptstyle {#2}}
}
\def\maxetage#1#2{
\max_{\scriptstyle {#1}\atop\scriptstyle {#2}}
}
\def\supetage#1#2{
\sup_{\scriptstyle {#1}\atop\scriptstyle {#2}}
}
\def\prodetage#1#2{
\prod_{\scriptstyle {#1}\atop\scriptstyle {#2}}
}
\def\convm#1{\mathop{\star}\limits_{#1}
}
\def\vect#1{
\overrightarrow{#1}
}
\def\Hd#1{{\cal H}^{d/2+1}_{1,{#1}}}

\def\derconv#1{\partial_t#1 + v\cdot\nabla #1}
\def\esptourb{\sigma + L^2(\R^2;\R^2)}
\def\tourb{tour\bil\-lon}


\newcommand{\beq}{\begin{eqnarray}}
\newcommand{\eeq}{\end{eqnarray}}
\newcommand{\bq}{\begin{equation}}
\newcommand{\eq}{\end{equation}}
\newcommand{\beqn}{\begin{eqnarray*}}
\newcommand{\eeqn}{\end{eqnarray*}}

\let\al=\alpha
\let\b=\beta
\let\g=\gamma
\let\d=\delta
\let\e=\varepsilon
\let\z=\zeta
\let\lam=\lambda
\let\r=\rho
\let\s=\sigma
\let\f=\phi
\let\vf=\varphi
\let\p=\psi
\let\om=\omega
\let\G= \Gamma
\let\D=\Delta
\let\Lam=\Lambda
\let\S=\Sigma
\let\Om=\Omega
\let\wt=\widetilde
\let\wh=\widehat
\let\convf=\leftharpoonup
\let\tri\triangle

\def\cA{{\cal A}}
\def\cB{{\cal B}}
\def\cC{{\cal C}}
\def\cD{{\cal D}}
\def\cE{{\cal E}}
\def\cF{{\cal F}}
\def\cG{{\cal G}}
\def\cH{{\cal H}}
\def\cI{{\cal I}}
\def\cJ{{\cal J}}
\def\cK{{\cal K}}
\def\cL{{\cal L}}
\def\cM{{\cal M}}
\def\cN{{\cal N}}
\def\cO{{\cal O}}
\def\cP{{\cal P}}
\def\cQ{{\cal Q}}
\def\cR{{\cal R}}
\def\cS{{\cal S}}
\def\cT{{\cal T}}
\def\cU{{\cal U}}
\def\cV{{\cal V}}
\def\cW{{\cal W}}
\def\cX{{\cal X}}
\def\cY{{\cal Y}}
\def\cZ{{\cal Z}}

\def\virgp{\raise 2pt\hbox{,}}
\def\cdotpv{\raise 2pt\hbox{;}}
\def\eqdef{\buildrel\hbox{\footnotesize d\'ef}\over =}
\def\eqdefa{\buildrel\hbox{\footnotesize def}\over =}
\def\Id{\mathop{\rm Id}\nolimits}
\def\limf{\mathop{\rm limf}\limits}
\def\limfst{\mathop{\rm limf\star}\limits}
\def\RE{\mathop{\Re e}\nolimits}
\def\IM{\mathop{\Im m}\nolimits}
\def\im {\mathop{\rm Im}\nolimits}
\def\Sp{\mathop{\rm Sp}\nolimits}
\def\DD{\mathop{\bf D\kern 0pt}\nolimits}
\def\K{\mathop{\bf K\kern 0pt}\nolimits}
\def\SS{\mathop{\bf S\kern 0pt}\nolimits}
\def\ZZ{\mathop{\bf Z\kern 0pt}\nolimits}
\def\TT{\mathop{\bf T\kern 0pt}\nolimits}
\newcommand{\ds}{\displaystyle}
\newcommand{\la}{\lambda}
\newcommand{\hn}{{\bf H}^n}
\newcommand{\hnn}{{\mathbf H}^{n'}}
\newcommand{\ulzs}{u^\lam_{z,s}}
\def\bes#1#2#3{{B^{#1}_{#2,#3}}}
\def\pbes#1#2#3{{\dot B^{#1}_{#2,#3}}}
\newcommand{\ppd}{\dot{\Delta}}
\def\psob#1{{\dot H^{#1}}}
\def\pc#1{{\dot C^{#1}}}
\newcommand{\Hl}{{{\cal  H}_\lam}}
\newcommand{\fal}{F_{\al, \lam}}
\newcommand{\Dh}{\Delta_{{\mathbf H}^n}}
\newcommand{\car}{{\mathbf 1}}
\newcommand{\X}{{\cal  X}}
\newcommand{\fgl}{F_{\g, \lam}}

\def\demo{d\'e\-mons\-tra\-tion}
\def\dive{\mathop{\rm div}\nolimits}
\def\curl{\mathop{\rm curl}\nolimits}
\def\cdv{champ de vec\-teurs}
\def\cdvs{champs de vec\-teurs}
\def\cdvdivn{champ de vec\-teurs de diver\-gence nul\-le}
\def\cdvdivns{champs de vec\-teurs de diver\-gence
nul\-le}
\def\stp{stric\-te\-ment po\-si\-tif}
\def\stpe{stric\-te\-ment po\-si\-ti\-ve}
\def\reelnonentier{\R\setminus{\bf N}}
\def\qq{pour tout\ }
\def\qqe{pour toute\ }
\def\Supp{\mathop{\rm Supp}\nolimits\ }
\def\coinfty{in\-d\'e\-fi\-ni\-ment
dif\-f\'e\-ren\-tia\-ble \`a sup\-port com\-pact}
\def\coinftys{in\-d\'e\-fi\-ni\-ment
dif\-f\'e\-ren\-tia\-bles \`a sup\-port com\-pact}
\def\cinfty{in\-d\'e\-fi\-ni\-ment
dif\-f\'e\-ren\-tia\-ble}
\def\opd{op\'e\-ra\-teur pseu\-do-dif\-f\'e\-ren\tiel}
\def\opds{op\'e\-ra\-teurs pseu\-do-dif\-f\'e\-ren\-tiels}
\def\edps{\'equa\-tions aux d\'e\-ri\-v\'ees
par\-tiel\-les}
\def\edp{\'equa\-tion aux d\'e\-ri\-v\'ees
par\-tiel\-les}
\def\edpnl{\'equa\-tion aux d\'e\-ri\-v\'ees
par\-tiel\-les non li\-n\'e\-ai\-re}
\def\edpnls{\'equa\-tions aux d\'e\-ri\-v\'ees
par\-tiel\-les non li\-n\'e\-ai\-res}
\def\ets{espace topologique s\'epar\'e}
\def\ssi{si et seulement si}

\def\pde{partial differential equation}
\def\iff{if and only if}
\def\stpa{strictly positive}
\def\ode{ordinary differential equation}
\def\coinftya{compactly supported smooth}
\def\R{\mathop{\mathbb R\kern 0pt}\nolimits}
\def\H{\mathop{\mathbb H\kern 0pt}\nolimits}
\def\C{\mathop{\mathbb C\kern 0pt}\nolimits}
\def\N{\mathop{\mathbb N\kern 0pt}\nolimits}
\def\ZZZ{\mathop{\mathbb Z\kern 0pt}\nolimits}
\def\Q{\mathop{\mathbb Q\kern 0pt}\nolimits}

\title{Phase-space analysis and pseudodifferential calculus on the Heisenberg group}

\author[H. Bahouri]{Hajer Bahouri}
\address[H. Bahouri]%
{ Facult{\'e} des Sciences de Tunis\\ D{\'e}partement de
Math{\'e}matiques\\
 2092 Manar\\TUNISIE }
\email{hajer.bahouri@fst.rnu.tn }
\author[C. Fermanian-Kammerer]{Clotilde Fermanian Kammerer}
\address[C. Fermanian Kammerer]{Universit{\'e} Paris Est, UMR 8050 du CNRS\\
61 avenue du G\'en\'eral de Gaulle\\ 94010 Cr\'eteil cedex\\
France} \email{Clotilde.Fermanian@univ-paris12.fr}

\author[I. Gallagher]{Isabelle Gallagher}\address[I. Gallagher]%
{Institut de Math{\'e}matiques UMR 7586 \\
      Universit{\'e} Paris Diderot (Paris 7) \\
175, rue du Chevaleret\\ 75013 Paris
     \\
      FRANCE}
\email{Isabelle.Gallagher@math.jussieu.fr}

\date{}

\begin{abstract}
A  class of pseudodifferential operators
on the Heisenberg group is defined. As it should be, this class is an algebra
containing the class of differential operators. Furthermore, those
pseudodifferential operators   act continuously on Sobolev spaces
and the loss of derivatives may be controled by the order of the
operator. Although a large number of works have been devoted in the past to the construction and the study of algebras of variable-coefficient operators, including some very interesting works on the Heisenberg group,  our approach is different, and in particular puts into light   microlocal directions and
completes,  with the Littlewood-Paley theory developed
in~\cite{bgx} and~\cite{bg},   a microlocal analysis of the
Heisenberg group.

 $ $

\noindent {\sc R\'esum\'e. } Nous d\'efinissons
une classe d'op\'erateurs pseudo-diff\'erentiels sur le groupe de
Heisenberg. Comme il se doit,  cette classe constitue une
alg\`ebre  contenant les op\'erateurs diff\'erentiels. De plus,
ces op\'erateurs pseudo-diff\'erentiels sont continus sur les
espaces de Sobolev et l'on peut contr\^oler la perte de
d\'eriv\'ee par leur ordre. Si un grand nombre de travaux ont \'et\'e d\'ej\`a consacr\'es \`a la construction et \`a l'\'etude d'alg\`ebres d'op\'erateurs
\`a coefficients variables, y compris des travaux tr\`es int\'eressants  sur le groupe de Heisenberg,
notre approche est diff\'erente et en particulier  elle conduit \`a la notion
de direction microlocale, et  compl\`ete l'\'elaboration d'une
analyse microlocale sur le groupe de Heisenberg   commenc\'ee
dans \cite{bgx} et \cite{bg} par le d\'eveloppement d'une
th\'eorie de Littlewood-Paley.

\end{abstract}


\maketitle
\vfill

{\bf Acknowledgements}: This project originates in a discussion
with G. Lebeau, and we are happy  to acknowledge his influence in
this study. We   also thank   J.-Y.~Chemin and  N.~Lerner for
numerous stimulating  discussions. H. Bahouri gratefully acknowledges
the hospitality of the Fondation Sciences Math\'ematiques de Paris
which supported a stay in the Institut de Math\'ematiques de
Jussieu,  during which part of this project was accomplished.

Finally we extend our thanks to the  anonymous referee for a careful reading of the manuscript and fruitful remarks.

\tableofcontents


\chapter{Introduction and main results}\label{intro}
\setcounter{equation}{0}
\section{Introduction}

\subsection{The Heisenberg group}
The Heisenberg group is obtained by constructing the group of unitary operators on~$L^2(\R^n)$ generated by the~$n$-dimensional group of translations and the~$n$-dimensional group of multiplications (see for instance the book by M. Taylor~\cite{taylor}). It is an unimodular, nilpotent Lie group whose Haar measure coincides with the Lebesgue measure, and its remarkable  feature   is that its representation theory is rich as well as simple in structure. It is actually the first locally compact group whose infinite-dimensional, irreducible representations were classified (see~\cite{corwingreenleaf}).
It can be identified with a subgroup of the group of~$(n+2) \times (n+2) $ real matrices with 1's on the diagonal and 0's below the diagonal.

 It has a dual nature, in the sense that it may be realized as the boundary of the unit ball in several complex variables (thus extending to several complex variables the role played by the upper half plane and the Hilbert transform on its boundary) as well as being closely tied to quantum theory (via the Heisenberg commutators). We refer to the book by E. Stein~\cite{stein2}, Chapter XII,  for a comprehensive presentation of that duality.

Harmonic analysis on the Heisenberg group is a subject of constant interest, due on the one hand to its rich structure
 (though simple compared to other noncommutative Lie groups), and on the other hand to its importance  in various areas of mathematics, from Partial Differential Equations (see among others~\cite{bgx}, \cite{birindelli}, \cite{brockett}  \cite{fmv}, \cite{fv}, \cite{mullerstein}, \cite{nach}, \cite{z}, \cite{zuily}) to Geometry (see~\cite{ambrosiorigot}, \cite{capognadanielli}, \cite{garofalovassilev}, \cite{rigot}) or Number Theory (see for instance~\cite{manin}, \cite{tolimieri}). Many research articles and monographs have been devoted to harmonic analysis on the Heisenberg group, and we shall give plenty of references as we go along.

\subsection{Microlocal analysis on~$\R^n$}
Microlocal analysis in the euclidian space appeared in the early seventies (\cite{skk}-\cite{skk2}), and has at its
 foundation the theory of pseudodifferential operators.
The main idea of microlocal analysis is to study a function simultaneously in the space variables of the physical space and in the Fourier variables. Indeed, some phenomenon need both analysis to be correctly understood. As an example, let us consider  the obstuctions to the convergence to zero in $L^2(\R^d)$ of two sequences, one of the form
$\displaystyle u_n=h_n^{-d/2}\phi\left({x-x_0\over h_n}\right)$ and the other of the form~$\displaystyle v_n = {\rm exp} \left(i{(x\cdot\xi_0) \over h_n} \right) \phi(x)$ where $h_n\rightarrow 0$ and $\phi$ is in the Schwartz class for example. Of course, the point~$x_0$ is a point of  {\it concentration} in the space variables for the sequence~$u_n$ and as such, a point of obstruction  to strong convergence to zero of the sequence. Similarly the {\it oscillations} in the direction~$\xi_0$   correspond to {\it concentration} in Fourier variables for the sequence~$v_n$, and they are also an obstruction to the strong convergence of the sequence.

 With this point of view, it appears crucial to be able to use localization operators in space variables {\it and} in frequencies: the latter are Fourier multipliers. The theory of pseudodifferential operators provides    a framework in which   both points of view are unified: multiplication operators {\it and} Fourier multipliers are indeed pseudodifferential operators. More precisely,  a pseudodifferential operator is defined by its {\it symbol} which is a function on the phase space: the symbol of the operator of multiplication by $\phi(x)$ is the function $(x,\xi)\mapsto \phi(x)$ and the symbol of the Fourier multiplier $\chi(D)$ is the function $(x,\xi)\mapsto\chi(\xi)$.

\medskip

With pseudodifferential operators comes the concept of properties which hold {\it microlocally}. A function $f$ satisfies a property $(P)$ locally if for all cut-off function $\chi$, the function $\chi f$ satisfies $(P)$; similarly, replacing the functions $\chi$ by a pseudodifferential operator  with symbol supported in a given subset $\Omega$ of the phase-space, one gets a property satisfied microlocally in $\Omega$. This notion allows a closer perception of the singularities of a function: in the 70's was developed the notion of {\it wave fronts}, analytic wave front, ${\mathcal C}^\infty$ wave front, etc. The idea is to associate with a given function $f$ a region of the phase space where, microlocally,  $f$ is analytic  or ${\mathcal C}^\infty$ or whatever else: this region is by definition the complement of the wave front.

\medskip

One should notice that the phase space corresponds to the space of positions-impulsions of Quantum Mechanics, and thus enjoys nice geometric properties. It can be understood as the cotangent space to $\R^d$ (or to a submanifold if one works on a manifold) and is a symplectic space once endowed with the adapted symplectic form. This geometric aspect has been used successfully in numerous works and is one of the satisfying aspects  of microlocal analysis (see for example the development of microlocal defect measures, semi-classical measures and Wigner measures as in~\cite{G} and \cite{GMMP} for example).

\medskip

Microlocal analysis allowed for a very general study and classification of linear Partial Differential Equations with variable coefficients,  using for example Littlewood-Paley operators which select
 a range of frequencies; such operators are pseudodifferential operators.
In the case of nonlinear Partial Differential Equations,  the situation is of course much more complicated, but paradifferential calculus (\cite{bony}) turned out to be a very powerful tool, for instance to analyze the propagation of singularities  of solutions to such equations, or to study the associate Cauchy problem (see for instance~\cite{bc}, in the case of quasilinear wave equations).

\medskip

 Pseudodifferential operators on the euclidian space form an algebra, which is a very important fact. This algebra contains   Fourier multipliers such as differentiation operators, microlocalisation operators, Littlewood-Paley operators, paradifferential operators.

\subsection{Microlocal analysis on the Heisenberg group}
The development of microlocal tools adapted to the geometric situation at hand is an important issue: we refer for instance to the work of S. Klainerman and I. Rodnianski~\cite{klainermanrodnianski} in the case of the Einstein equation, where the construction of an adapted Littlewood-Paley theory is a crucial tool to reach   optimal regularity indexes for the initial data.
Microlocal  theory on $\R^n$ easily passes to submanifolds.
Other constructions have been performed on the torus, or more general compact Lie groups
 (see for instance~\cite{ruzhanskyturunen}).

 A number of articles can be found in the literature,
which develop a pseudodifferential calculus on the Heisenberg
group.  For example, in \cite{stein2},
\cite{taylor}, this question is investigated through the angle of
the Weyl correspondence (see also the previous work~\cite{grossmannloupiasstein}): as recalled above, that correspondence is one of the rich features of the Heisenberg group, and is thoroughly developed in those references.  The important work~\cite{geller} consists in  constructing an
analytic calculus enabling one to obtain parametrices for a
class of operators which are analytic hypoelliptic; we
also refer to~\cite{melin} and~\cite{bealsgreiner}
 as well as~\cite{ccx} where a parametrix is constructed for sum-of-squares type operators.
 One also must mention the series of papers by P. Greiner and his coauthors (see for instance~\cite{bealsgaveaugreinervauthier}, \cite{ggv} and~\cite{greinerxedp} and and the references therein) in which in particular
symbols of left-invariant vector fields  are constructed, from the point of view of Laguerre calculus as well as using the Hermite basis and  the recent works~\cite{vanerp1}-\cite{vanerp3}, where a symbolic calculus on the Heisenberg group is developped,   related to   contact manifolds. Finally, we refer to the work \cite{gellercore} where is constructed a pseudodifferential calculus based on H\"ormander calculus, using exclusively the convolution rather than the Fourier transform.

 Our approach   here is not quite of the same nature as in the works refered to above, as we aim at defining an algebra of operators on functions defined on the Heisenberg group, which contains   differential operators and   Fourier multipliers, and which has a structure close to that of pseudodifferential operators in the Euclidian space.  The difficulty in this approach is that  there is  no simple notion of
symbols as functions  on the Heisenberg group~$ \H^d$, since  the Fourier transform is a family of
operators on Hilbert spaces depending on a real-valued parameter~$\lambda$. Those operators are built using the so-called Bargmann representation, or the Schr\"odinger representation (obtained from the previous one by intertwining operators). One can easily check that what may appear as the symbol
associated with a left-invariant vector field  is itself a family
of operators.  This family reads in the Schr\"odinger
representation of~$ \H^d$  as a family of differential operators
belonging to a class of operators of order~$1$ for the
Weyl-H\"ormander calculus  (see~\cite{hormander}) of the  harmonic
oscillator. That  basic observation   is the heart of  the matter
achieved in this paper. Let us point out that in fact symbols on
the Heisenberg group cannot   depend only  on the harmonic
oscillator, and this has to do with the  dependence on the parameter~$\lambda$. This induces a
number of technical problems that are dealt with by introducing also a specific calculus  in the~$\lambda$ direction.

A symbol on the Heisenberg group is thus a function on~$\H^d$
valued in the space of families of symbols of the Weyl-H\"ormander
class associated to the harmonic oscillator, indexed by the parameter~$\lambda$. Then, to this symbol,
one associates a  pseudodifferential operator as   is  usually
done by use of the inverse Fourier transform   as well as the
family of Weyl-quantized operators associated with the symbol.

Once those pseudodifferential operators have been defined, we
first prove that they  are operators on  the Schwartz class,
which results from classical Fourier analysis on the Heisenberg
group. We  then prove that the  adjoint of a pseudodifferential operator and the composition of two  pseudodifferential operators are also
pseudodifferential operators.  Our arguments here  are deeply
inspired by the analysis of the classical case as developped for instance in
the book of  S.~Alinhac and P.~G\'erard~\cite{ag}. We  analyze
first   the link between the kernel of a pseudo\-differential
operator and its symbol, using the Fourier transform and its
inverse. Then, it is possible to compute the
function which could be the symbol of the adjoint of a
 pseudodifferential operator or of the composition of two pseudodifferential operators and to prove that it actually is   a  symbol. This comes from the careful analysis of  oscillatory integrals.
 We also give  asymptotic formula for
the symbol of the adjoint or of the composition.
 These formulas result from a Taylor formula in the   spirit of what is done in the Euclidian space but adapted to the case of the Heisenberg group; in particular, we crucially use functional calculus.
 The specific feature of these asymptotic formula is that there is no gain  on the Heisenberg group:
 the commutator of two  horizontal vector fields is a derivation.

We also   study the action of pseudodifferential operators  on
Sobolev spaces. We prove in par\-ticular that zero order operators are bounded on
$L^2(\H^d)$ and more generally a pseudodifferential operator is continuous from one Sobolev space to another,
the link between the regularity exponents of the Sobolev spaces being controled by the order of the
symbol.
 The arguments of this proof are
inspired by the Euclidian proof of R. Coifman and Y.
Meyer~\cite{cm} whose approach consists mainly in decomposing the
symbol of the pseudodifferential operator on~$\R^n$ (which is a function on
the phase space~$T^*\R^n$) into a convergent series of reduced
symbols  for which the continuity is a consequence of
paradifferential calculus of  J.-M. Bony~\cite{bony}.  The main interest of this approach is that it requires little regularity on the symbol and that it can be carried out when the pseudodifferential calculus has no gain, which is the case in our situation. Roughly
speaking, the proof of R. Coifman and Y. Meyer is done in three steps.
In the first step, a symbol is decomposed   using a dyadic
partition of unity. This reduces the problem to the study of
symbols compactly supported in the frequency variable. Next, using
a Fourier series expansion,  the symbol is expressed as a sum of
reduced symbols   which are much easier to deal with. Finally,
taking advantage of the Littlewood-Paley decomposition on~$\R^n$, the
continuity on Sobolev spaces of the associate  operator
 is established.
To adapt that method  to the setting of the Heisenberg group~$
\H^d$, we  begin by decomposing the    symbol associated with a
given operator (defined as explained above via the
Weyl-H\"ormander calculus of the harmonic oscillator), using a
suitable dyadic partition of unity. Then, we   use   Fourier
series   to write the symbol as a convergent series of reduced
symbols. But, in contrast to the~$\R^n$ setting, the reduced
symbols in that case cannot be treated as a sum of
Littlewood-Paley operators on the Heisenberg group.   To overcome
this difficulty, we use Mehler's formula   to prove that these
operators can be related in some sense to the reduced symbols
obtained in the~$\R^n$ case. This allows us to finish the proof in
more or less the same way as  in the~$\R^n$ case, up to the fact
that an additionnal  microlocalization is needed because the  spectral parameter is made of two different
variables -- as pointed out above,  this is due to the special structure of the
Heisenberg group.

\medskip

This paper  completes,  with the Littlewood-Paley theory developed
in~\cite{bgx} and~\cite{bg},   a microlocal analysis of the
Heisenberg group. It calls for developments :  a significant
application  would be the generalization of the concept of   wave
front set  to the setting of the Heisenberg group, in order to obtain results related to the propagation of singularities as in~\cite{X2} for instance. One  can
also expect a  construction of parametrices,  as well as the development of a notion of microlocal
defect measure (or~$H$-measure).  Such studies are postponed to a future work.

\medskip

Generalizations to other locally compact Lie groups should also be considered. The generalization of the Littlewood-Paley decomposition is in itself a challenge : although it is known (see~\cite{hulanicki})  that a frequency localization process can
be defined in general as a convolution product with a function of the Schwartz class, Bernstein inequalities seem very difficult to obtain in general (and these inequalities are the crucial property that allow to construct a Littlewood-Paley theory). Once that difficulty is overcome, the next step should be the understanding of the phase space in more general contexts.

\subsection{Structure of the paper}
The structure of the paper is the following.   The rest of this
chapter  is devoted to a recollection
 of the main facts on the Heisenberg group which will be useful for us, as well as to the statement of the main results. More precisely, in
 Section~\ref{heisenberggroup}, we introduce
our notation and give the basic definitions and in
Section~\ref{sec:Fourier}, we  recall the definition of the
Fourier transform, using irreducible representations.
 The purpose of the next section of this chapter is to provide the setting for symbols and operators on the Heisenberg group,  and it also contains the statement of the main results; for this some elements of Weyl-H\"ormander calculus are required, and the necessary definitions are recalled.  The main  results stated in this chapter (in Section~\ref{statementresult}) concern the continuity of pseudodifferential operators on Sobolev spaces, along with   the fact that those classes of operators form an algebra.

The second chapter is devoted to the analysis of examples and to the proof of some fundamental
properties of pseudodifferential operators, such as their action
on the Schwartz class, the study of their kernel, their
composition with differentiation operators.

 In the third chapter, we prove that the classes of pseudodifferential
 operators defined in the previous chapter are stable by adjunction and composition and prove asymptotic expansion of their symbol.

In the fourth chapter  we give   an outline of the basic elements of Littlewood-Paley theory on the Heisenberg group developed in~\cite{bgx} and~\cite{bg} recalling in that framework the properties of Besov spaces that we shall need later on. Next, we compare Littlewood-Paley operators with pseudodifferential operators. This is  of crucial importance in the next chapter. More precisely, we   prove that in some sense, a pseudodifferential operator associated to a truncated symbol,  in the Weyl-H\"ormander calculus of the harmonic oscillator, is close to a Littlewood-Paley operator.

  In the fifth chapter,  we prove  the continuity on Sobolev
spaces, by a (non trivial) adaptation of the technique of R.
Coifman and Y. Meyer~\cite{cm} to the case of the Heisenberg group;  in particular   an additional microlocalization is required, compared to the classical case.

Finally this paper comprises two appendixes. Appendix A  is devoted to the proof of some technical lemmas and formulas concerning  the Heisenberg group that are used in the paper. In Appendix~B we prove a number of important results used in the proofs of the main theorems of this paper,
but for which the arguments are too lengthy or too technical to appear in the main text; they are mainly related to Weyl-H\"ormander calculus.


\section{Basic facts on the Heisenberg group~$ \H^d$ }\label{basicfacts}
\setcounter{equation}{0}


\subsection{The Heisenberg group}
\label{heisenberggroup} Before stating the principal results of
this paper, let us collect a few well-known definitions and
results on the Heisenberg group~$ \H^d$. We recall that it is
defined as
  the space~$\R^{2d+1}$ whose elements~$w \in
\R^{2d+1}$ can be written~$w=(x,y,s)$ with~$(x,y) \in \R^d \times
\R^d $, endowed with the following product law:
\begin{equation} \label{lawH}
 w \cdot w' =
(x,y,s)\cdot (x',y',s') = \left(x+x',y+y',s+s'- 2\,x\cdot y' +
2\,y\cdot x'\right),
\end{equation}
where for $x,x'\in \R^d$,~$x\cdot x'$ denotes the Euclidean scalar product
of the vectors $x$ and $x'$.  Equipped with the standard differential structure of the manifold $\R^{2d+1},$
the set~$ \H^d$ is a  non commutative
Lie group with identity~$( 0,0).$ Note also that
$$ \forall \:  \: w = (x,y,s) \in \H^d,  \quad w^{-1} = (-x,-y,-s). $$

 The Lie algebra of left invariant vector fields
 (see Section~\ref{leftinvariant} of the Appendix) is spanned
by the vector fields
\[
X_j \eqdefa \partial_{x_j} +2y_j\partial_s\,,\ Y_j \eqdefa
\partial_{y_j} -2x_j\partial_s\quad\hbox{with}\ j\in
\{1,\dots,d\},  \quad\hbox{and}\quad S \eqdefa \partial_s = \frac
1 4[Y_j,X_j]
\]
for~$j \in \{1,\dots, d\}$. In the following, we will denote
by~${\mathcal X}$ the family of vector fields generated by~$X_j$
and by~$X_{j+d} = Y_j $  for~$j \in \{1,\dots, d\}$. Then for any
multi-index~$\alpha \in \{1,\dots, 2d\}^k $, we write
\begin{equation}
\label{prodreal} {\mathcal X}^{\alpha} \eqdefa X_{\alpha_1}\dots
X_{\alpha_k}.
\end{equation}
Using the complex  coordinate system~$( z, s )$ obtained by
setting~$ z_j = x_j + i y _j $, we note that $$ \forall \left((z,
s), (z', s')\right) \in {\H^d} \times {\H^d},  \quad (z, s)\cdot
(z', s') = (z+z', s+s'+ 2  \mbox{Im}(z \cdot \overline z')), $$
where~$z\cdot \overline z'= \sum_{j=1}^d z_j  \overline z'_j$.
Furthermore, the Lie algebra of left invariant vector fields  on
the Heisenberg group~$ \H^d $ is generated by the vector fields:
\[Z_j =
\partial_{z_j} + i \overline z_j
\partial_{s}, \quad \overline Z_j = \partial_{\overline z_j} -  i
z_j
 \partial_{s},  \quad\hbox{with}\ j\in
\{1,\dots,d\}\quad\hbox{and}\quad S= \partial_s=\frac {1
}{2i}[\overline Z_j,Z_j].
\]
Denoting by~${\mathcal Z}$ the family of vector fields generated
by~$Z_j$   and by~$Z_{j+d} = \overline Z_j $  for~$j \in
\{1,\dots, d\}$, we write for any multi-index~$\alpha \in
\{1,\dots, 2d\}^k $
\begin{equation}
\label{prodreal2} {\mathcal Z}^{\alpha} \eqdefa Z_{\alpha_1}\dots
Z_{\alpha_k}.
\end{equation}
One can easily check that for all $j\in \{1,\dots,d\}$,
\beq\label{XYZ}
X_j=Z_j+\overline Z_j\;\;{\rm and}\;\;Y_j=i(Z_j-\overline
Z_j).
\eeq

The space~$\H^d$ is endowed with a smooth  left invariant measure,
the Haar measure, which in the coordinate system~$(x,y,s)$ is
simply the Lebesgue measure~$dw \eqdefa dx\,dy\,ds$.  It satisfies the
fundamental property:
\begin{equation}\label{leftinvariance}
\forall f \in L^{1}( \H^{d}), \: \forall w' \in {\H}^{d}, \quad
\int_{\H^{d}} f(w) \: dw = \int_{\H^{d}} f(w'\cdot w)\: dw .
\end{equation}
The convolution product of two functions~$ f$ and~$g$ on~$ {\H^d}
$ is defined by
$$f \star g ( w ) \eqdefa \int_{\H^d} f ( w \cdot v^{-1}
) g( v) dv = \int_{\H^d} f ( v ) g(v^{-1} \cdot w) dv.$$ It should
be emphasized that the convolution on the Heisenberg group  is not
commutative. Moreover if~$P$ is a left invariant vector field
on~$\H^d$, then one  has
\begin{equation}
\label{convleftheis} P ( f \star g ) = f \star ( P (g ) ).
\end{equation}
Indeed, thanks to the classical differentiation theorem, we have
\[
P (f \star g)(w) =  \int_{\H^d} f(v)P( g(v^{-1}\cdot w)) dv.
\] Due
to \refeq{leftdef}, one can write
\[ P( g(v^{-1}\cdot w)) = (Pg)(v^{-1}\cdot w),\]
which yields\refeq{convleftheis}.
However in general~$f \star ( P (g ) )\neq ( P (f ) )
\star g.$

Note that the usual Young inequalities
are nevertheless valid on the Heisenberg group, namely
$$
\forall (p,q,r) \in [1,\infty]^3, \quad \|f \star g\|_{L^r(\H^d)} \leq \|f \|_{L^p(\H^d)}  \|  g\|_{L^q(\H^d)} , \quad 1+\frac1r = \frac1p + \frac1q \cdotp
$$
In fact, Young inequalities are more generally available on any  locally compact topological
group endowed with a left invariant Haar
measure $\mu$ which satisfies  in addition
$$
\mu(A^{-1})=\mu(A)\: \hbox{ for all borelian sets }\: A.
$$
Let us also point out that on
the Heisenberg group~$\H^d$, there is a notion of dilation defined
for~$ a > 0 $ by
\begin{equation}\label{def:dilation}
\delta_a ( z, s ) \eqdefa ( a z,  a^2s ).
\end{equation}
Observe that for any real number~$ a > 0 $, the dilation  $\delta_a$
satisfies\[
  \delta_a ( z,  s )\cdot\delta_a ( z',  s' )= \delta_a (( z,  s )\cdot( z',  s'
  ))\]
  and that the vector fields~$Z_j$ change
the homogeneity in the following way: \beq\label{zjdelta} Z_j (f
\circ \delta_a) = a( Z_j f ) \circ \delta_a. \eeq This fact is
crucial in order to obtain Bernstein or Hardy
inequalities~\cite{bcg} (see Chapter~\ref{prelimins}).\\
  Let us also remark that  the Jacobian of the dilation~$\delta_a$ is~$a^N$
where~$N \eqdefa 2d + 2$ is called
 the  homogeneous dimension of~$\H^d$.

Let us now recall how   Sobolev spaces on the Heisenberg group are
associated with the system of vector fields~${\mathcal X}$ for nonnegative integer indexes.
\begin{defin}
\label{definsobbasicheis}   {  Let~$k$ be a nonnegative integer.
We denote by~$H^k(\H^d)$ the inhomogeneous Sobolev space on the
Heisenberg group of order~$k$ which is the space of functions~$u$
in~$L^2(\H^d)$ (for the Haar measure) such that
\[
{\mathcal X}^{\alpha}u\in L^2 \quad\hbox{for any multi-index}\quad
\alpha \in \{1,\dots,2d\}^{\N}\quad\hbox{with}\quad |\alpha|\leq
k.
\]
Moreover, we state
 \beq\label{defnormHk}\norm {u}{H^k(\H^d)} \eqdefa \left(
\sum_{|\alpha|\leq k } \norm {{\mathcal
X}^{\alpha}u}{L^2({\H}^d)}^2\right)^\frac12.
\eeq}
\end{defin}

\medbreak

\begin{rem} Equivalently, powers
of  the Laplacian-Kohn operator  defined by
\begin{equation}
\label{lapKohn} \Delta_{ \H^d  }\eqdefa \sum_{j=1}^{d}  (X_j^2 +
Y_j^2 ) = 2 \sum_{j=1}^{d}(Z_j \overline Z_j  +  \overline Z_j Z_j
)=4\sum_{j=1}^{d}  (Z_j \overline Z_j +i\partial_s ),
\end{equation} can be used to define those Sobolev  spaces,
 which take into account the different role played
by the~$s $-direction. Thus
$$
 \norm {u}{H^k(\H^d)} \sim
  \norm { ({\rm Id} -\Delta_{ \H^d  })^\frac{k}2u}{L^2({\H}^d)}
$$
where~$\sim$ stands for equivalent norms.

Note that homogeneous norms may also be defined, where the summation in~(\ref{defnormHk}) is replaced by a summation over~$|\alpha| = k$, and above~$ ({\rm Id} -\Delta_{ \H^d  })^\frac{k}2$ is replaced by~$(-\Delta_{ \H^d  })^\frac{k}2$.
\end{rem}

 When~$\sigma$ is any nonnegative real number one can,
as in the case of classical Sobolev spaces on~$\R^n$, define the
space~$H^ \sigma(\H^d)$ by complex interpolation (see for
instance~\cite{berghlofs}).
   As in the euclidian case, other equivalent definitions of  Sobolev spaces~$H^\sigma({\H}^d)$ can be used: the definition
using integrals and kernels (see~\cite{rotschildstein}
and~\cite{stein2}), or the definition using   Weyl-H\"ormander
calculus (see~\cite{ccx}).
Finally,  a definition using the
Littlewood-Paley  theory on the Heisenberg group,  in the same
spirit as in the Euclidian case and due to~\cite{bgx},  will be given in
Section~\ref{prelimins}.\ref{besovspaceslp}.

There is a natural Heisenberg distance to the
 origin defined by
$$ \rho(z,s) \eqdefa (|z |^4 +s^2)^{\frac 1 4} , $$
where~$\displaystyle |z|^2 =\sum_{j=1}^d z_j  \overline z_j $.
Similarly, we define the Heisenberg distance by
\begin{equation}
\label{dist2heis} d(w,w') = \rho\left(w^{-1}\cdot w'\right).
\end{equation}
The distance~$d$ incorporates  left translation invariant
properties
\begin{equation}
\label{heisdistleft}
\forall \widetilde w \in \H^d, \quad d\left(\widetilde w \cdot w , \widetilde w
\cdot w' \right)= d(w,w').\end{equation}
To define H\"older spaces on the Heisenberg group, we shall introduce
another distance on~$\H^d$. Denote
by~$P=P(X_1,\dots,X_{2d})$ the set of continuous curves which are
piecewise
 integral curves of one of the
vectors fields~$ \pm X_1,\dots,  \pm X_{2d}$.
To any such curve $\gamma:[0,T]\longrightarrow \H^d,$
we associate its length~$l(\gamma)\eqdefa T.$ It is known (see for instance\ccite{franchigallotwheeden,franchilanconelli}) that, for any couple of points~$w$ and~$w'$ of~$\H^d$, there exists  a curve of~$P$ joining~$w$
to~$w'$ and that the function
 \begin{equation}
\label{defdith} \wt{d}(w, w')= \min \Bigl\{ l(\gamma),  \: \gamma
\in P, \: \! \gamma \  \hbox{joining}\  w \
\mbox{to}\  w'\Bigr\}
\end{equation}
is a distance on the Heisenberg group, which
turns out to be
 equivalent to the
one introduced in \refeq{dist2heis}.

Now, up to the change of the Euclidean distance into~$\wt d,$ the definition of H\"older spaces on the Heisenberg group is similar to the definition of H\"older spaces on~$\R^d$.
\begin{defin}
\label{definholderheis} {  Let~$r=k+ \sigma $, where~$k$ is an
integer and~$\sigma \in ]0,1[$. The H\"older space~$C^{r}({\H}^d)$
on the Heisenberg group
 is the space of
functions~$u$ on~$\H^d$ such that $$ \|u\|_{C^{r}(\H^d)}=
 \sup_{|\al|\leq k}
\Bigl( \norm {{\mathcal X}{^\al} u} {L^\infty} + \sup_{w\not = w'}
\frac{| {\mathcal X}{^\al} u(w)-{\mathcal X}{^\al} u(w')|}{\wt{d}(w,
w')^ \sigma} \Bigr)< \infty \virgp $$ where~$\wt d$ denotes the
distance on the Heisenberg group defined by\refeq{defdith}. }
\end{defin}
 \begin{rem} Thanks to\refeq{heisdistleft} and the fact that the distances ~$ d$ and~$\wt d$ are equivalent, the spaces~$C^{r}({\H}^d)$ are
invariant under left translations. It will be useful to point out that
  H\"older spaces on the Heisenberg group can be also
defined using the Littlewood-Paley  theory on the Heisenberg
group,  in the same way as in the Euclidian case (see
Section~\ref{prelimins}.\ref{besovspaceslp}).
\end{rem}
Finally let us define the Schwartz space.
 \begin{defin}\label{defschwartz}
 The Schwartz space~${\mathcal S}(\H^d)$ is the set of smooth functions~$u$ on~$\H^d$ such that, for any~$k \in \N$, we have
 $$
 \|u\|_{k,{\mathcal S}} \eqdefa \supetage{|\alpha| \leq k, \,n \leq k}{(z,s) \in \H^d} \left| {\mathcal Z}^\alpha  \left((|
z|^2 -is)^{2n}
 u(z,s) \right)\right|
< \infty.$$
 \end{defin}
The  Schwartz space on the Heisenberg group~${\mathcal S}(\H^d)$
coincides with the classical Schwartz space~${\mathcal
S}(\R^{2d+1})$. This allows to define the space of tempered
distributions~${\mathcal S}'(\H^d)$.  The weight in~$(z,s)$
appearing in the definition above is linked to the  Heisenberg
distance to the origin~$\rho$   defined above.

\subsection{Irreducible representations and the Fourier transform}\label{sec:Fourier}
Let us now recall the definition of the Fourier transform. We
refer for instance to ~\cite{farautharzallah},  \cite{nach},
\cite{stein2}, \cite{taylor} or~\cite{thangavelu} for
 more details. The
Heisenberg group being non commutative,  the Fourier transform
on~$\H^d $ is defined using  irreducible unitary representations
of~$ \H^d $. As explained for instance in~\cite{taylor} Chapter 2,
all irreducible representations of~$ \H^d $ are unitarily
equivalent to one of two representations: the Bargmann
representation or the~$L^2$ representation.  The representations
on~$L^2(\R^d)$ can be deduced from Bargmann representations thanks
to
 intertwining operators. The reader can consult J. Faraut and K. Harzallah\ccite{farautharzallah} for more details.
Both representations will be used here.

\subsubsection{The Bargmann representations} They are described
by~$(u^\lam,  {\mathcal H}_{\lam} )$, with~$\lam \in \R \setminus
\{0 \}$, where~${\mathcal H}_{\lam}$ is the space defined by
\[
{\mathcal H}_{\lam} \eqdefa \{F \mbox{ holomorphic on } \C^d, \|
F\|_{{\mathcal H}_{\lam} }
 < \infty\},
 \]
 with
 \begin{equation}
 \label{normbargman}
\| F\|_{{\mathcal H}_{\lam} }^2 \eqdefa \left(\frac{2 |\lam |}{
\pi}
  \right)^d \int_{\C^d}
{\rm e}^{-2|\lam| |\xi|^2 }|F(\xi)|^2 d\xi,
\end{equation}
 while~$u^\lam $ is the  map from~$ \H^d $ into
 the group of unitary operators of~${\mathcal H}_{\lam} $  defined by
 \begin{equation}\label{def:ulamw}
 \left\{
 \begin{array}{c}
u^{\lam}_{z,  s } F(\xi) \eqdefa F(\xi - \overline z) {\rm e}^{i
\lam s + 2 \lam (\xi \cdot  z - |z|^2/2)}  \quad \mbox{for} \quad
\lam
>0, \\
u^{\lam}_{z,  s} F(\xi) \eqdefa F(\xi - z) {\rm e}^{i \lam s - 2
\lam (\xi \cdot \overline z - |z|^2/2)}  \quad \mbox{for} \quad
\lam < 0.
\end{array}
\right.\end{equation}
Let us notice that~${\mathcal H}_{\lam} $ equipped with the
norm~$\|\cdot\|_{{\mathcal H}_{\lam}}$
 defined in~(\ref{normbargman}) is a Hilbert  space. The monomials
  $$
   F_{\al,  \lam} (\xi) \eqdefa \frac{(\sqrt{2|\lam|} \: \xi)^\al}{\sqrt{\al !}} ,
\quad \al \in {\N}^d, $$ constitute an orthonormal basis
 of~${\mathcal H}_{\lam} $.

The Fourier transform of an integrable function of~${\H}^d$ is
given by  the following definition.
\begin{defin}
{  For~$f \in L^1(\H^d)$,  we define
\[
 {\mathcal F}(f)(\lam)
\eqdefa \int_{\H^d} f(w) u^{\lam}_{w} dw.
\]
The function~${\mathcal F}(f) $,  which takes values in the space
of bounded operators on~$ {\mathcal H}_{\lam}$, is by definition
the Fourier transform of~$f$.}
\end{defin}
Note that one has $$ {\mathcal F}( f \star g )( \lam ) = {\mathcal
F}(f) ( \lam ) \circ {\mathcal F}(g )( \lam ). $$

We recall that an operator $A(\lam)$ of ${\mathcal H}_\lam$ such that
$$ \sum_{\al \in \N^d} \left|(A(\lam)F_{\al,  \lam}, F_{\al,
   \lam})_{{\mathcal H}_{\lam}}\right|<+\infty$$
   is said to be of {\it trace-class}. One then sets
\begin{equation}\label{def:traceclass}
    {\rm tr}   \left(A(\lam)\right)\eqdefa
   \sum_{\al \in \N^d} (A(\lam)F_{\al,  \lam}, F_{\al,
   \lam})_{{\mathcal H}_{\lam}}.
   \end{equation}
   We recall that if besides  the operator $A(\lam)$ has a kernel, namely that if there exists a function~$k_\lam(\xi,\xi')$ such that
   \begin{equation}\label{operateuranoyau}
   \forall F\in{\mathcal H}_\lam,\;\;A(\lam)F(\xi)=\int_{\C^{d}}k_\lam(\xi,\xi')F(\xi')d\xi',
   \end{equation}
   then  its trace is given by
   \begin{equation}\label{traceopanoyau}
   {\rm tr}\left(A(\lam)\right)=\int_{\C^d}k_\lam(\xi,\xi)d\xi.
   \end{equation}
 Now if $A(\lam)^*A(\lam)$ is trace class, then $A(\lam)$ is said to be a {\it Hilbert-Schmidt operator}. The quantity
 $$\left\| A(\lam)\right\|_{HS({\mathcal H}_\lam)}\eqdefa \left(\sum_{\al\in\N^d} \|A(\lam)F_{\al,\lam}\|^2\right)^{1\over 2}$$
 is then a norm on the vector space of Hilbert-Schmidt operators. The following property on Hilbert-Schmidt norms, which can be found in~\cite{RS} (Volume 1 Chapter VI.6) will be of frequent use in what follows. Let~$ A$ and $ B$ be two bounded operators on~$
{\mathcal H}_\lam$, with~$A$ Hilbert-Schmidt. Then
\begin{equation}\label{hsproperty}
\norm {BA}{HS({\mathcal H}_\lam)}+\norm {A B}{HS({\mathcal H}_\lam)} \leq \norm {B}{{\mathcal
L}({\mathcal H}_\lam)} \norm {A}{HS({\mathcal H}_\lam)} .
 \end{equation}
 Similarly if~$A$ and~$B$ are two  Hilbert-Schmidt operators, then~$AB$ is trace-class and
 \begin{equation}\label{hstraceclassproperty}
|{\rm tr}(A B)| \leq  \norm{A(\lam )}{HS( {\mathcal H}_\lam)} \norm
{B(\lam )}{HS( {\mathcal H}_\lam) }. \end{equation}
These notions are important for stating the Plancherel theorem for the Heisenberg group. The proofs of the two following results can be found for instance in\ccite{farautharzallah}.

\begin{theo}\label{plancherelth}
   Let~$ {\mathcal A}$  denote the Hilbert
space of   one-parameter families~$ A = \{ A (\lam ) \}_{ \lam \in
\R \setminus \{0 \}}$  of operators on~$ {\mathcal H}_\lam $ which
are Hilbert-Schmidt for almost every~$\lam \in \R $,  with~$\norm
{A (\lam )}{HS ({ {\mathcal H}_\lam })}$ measurable and with norm
\[ \norm {A }{} \eqdefa \left( \frac{2^{d-1}}{\pi^{d+1}} \int_{- \infty}^{\infty}
\norm {A (\lam )}{HS ({ {\mathcal H}_\lam })}^2 |\lam |^{d} d\lam
\right)^{\frac{1}{2}} <  \infty.
\]
The Fourier
transform can be extended to an isometry from~$ L^2( \H^d) $
onto~${\mathcal A }$ and we have the Plancherel formulas:
\begin{eqnarray}
\label{Plancherelformula} \norm f { L^2( {\H}^d)}^2  &  = &
\frac{2^{d-1}}{\pi^{d + 1}}  \int_{-
\infty}^{\infty} \norm {{\mathcal  F}(f)(\lam ) } {HS(
{\mathcal H}_\lam) }^2  |\lam |^{d} d\lam \quad  \mbox{and}\\
\label{Plancherelformulaproduitscalaire}
\left(f | g\right)_ {L^2( {\H}^d) }  &  = &  \frac{2^{d-1}}{\pi^{d + 1}}
\int_{- \infty}^{\infty}   {\rm tr} \left(\left( {\mathcal F}(g)(\lam)\right)^* {\mathcal  F}(f)(\lam )  \right) |\lam |^{d} d\lam .
\end{eqnarray}
\end{theo}

\begin{rem}\label{estimatesHStretc}
If~$ A=\{ A (\lam )
\}_{ \lam \in \R \setminus \{0 \}}$ and~$ B=\{ B (\lam ) \}_{ \lam
\in \R \setminus \{0 \}}$ are two  families in~$ {\mathcal A}$,
then
$$
 \int |{\rm tr} (A(\lam )B(\lam ))| \:  |\lam|^d \: d\lam \leq \|A\| \;\| B\|.
$$
\end{rem}

Moreover, the following inversion theorem holds.
\begin{theo}
\label{inversionth}
  If a function~$f$ satisfies
 \begin{equation}
\label{injfouheis}\sum_{\al \in \N^d} \int_{- \infty}^{\infty}
\norm {{\mathcal F}(f)(\lam ) F_{\al, \lam} } {{\mathcal
H}_{\lam}}  |\lam |^{d} d\lam < \infty
\end{equation}
then we have for almost every~$w$,
 $$
 f(w)=
\frac{2^{d-1}}{\pi^{d + 1}} \int_{- \infty}^{\infty} {\rm tr}
\left( u^{\lam}_{w^{-1}}{\mathcal F}(f)(\lam )\right)
  |\lam |^{d} d\lam.
  $$
   \end{theo}

\begin{rem}The above hypothesis\refeq{injfouheis}
  is satisfied in~$ {\mathcal S} (\H^d)$ (see for example~\cite{bg2}). Therefore, if we consider for  $w_0\in\H^d$,  the Dirac distribution in $w_0$, $\delta_{w_0}(w)$, defined by
  $$\forall f\in{\mathcal S}(\H^d),\;\;<\delta_{w_0}\;,\;f>\;=f(w_0),$$
  we have an expression of $\delta_{w_0}$ as a singular integral
  \begin{equation}\label{deltaw0}
  \delta_{w_0}(w)= \frac{2^{d-1}}{\pi^{d + 1}} \int_{- \infty}^{\infty} {\rm tr}
\left( u^{\lam}_{w_0^{-1}w}\right)
  |\lam |^{d} d\lam.
  \end{equation}
  \end{rem}

Now let us study the action of the Fourier transform on
derivatives. Straightforward computations (performed in Lemma~\ref{formuleslem2} page \pageref{formuleslem2} for the convenience of the reader), show that $$ {\mathcal
F}(Z_{j}f)(\lam) = {\mathcal F}( f)(\lam) Q_{j}^\lam,$$
where~$Q_{j}^\lam$ is the operator on~${\mathcal H}_\lambda$
defined by
   \begin{eqnarray} \label{defQj}
   Q_{j} ^\lam F_{\al,\lam}& \eqdefa & -\sqrt{2|\lam|} \sqrt{\al_{j}+1} F_{\al+{\mathbf 1}_{j},\lam} \quad \mbox{if} \: \lam>0 \nonumber\\
   & \eqdefa&\sqrt{2|\lam|}\sqrt{\al_{j} }  F_{\al-{\mathbf 1}_{j},\lam} \quad \mbox{if} \:
   \lam<0
   \end{eqnarray}
and  in the same way,
   $$
    {\mathcal F}(\overline Z_{j}f)(\lam) = {\mathcal F}( f)(\lam) \overline Q_{j}^\lam,
$$
 where~$\overline Q_{j}^\lam$ is the operator on~${\mathcal
   H}_\lambda$ defined by
   \begin{eqnarray} \label{defoverQj}
  \overline Q_{j}^\lam F_{\al,\lam}& \eqdefa & \sqrt{2|\lam|}\sqrt{\al_{j} }  F_{\al-{\mathbf 1}_{j},\lam} \quad \mbox{if} \: \lam>0 \nonumber\\
   & \eqdefa & -\sqrt{2|\lam|} \sqrt{\al_{j}+1} F_{\al+{\mathbf 1}_{j},\lam}\quad \mbox{if} \:
   \lam<0,
   \end{eqnarray}
   while we have written~$\al\pm{\mathbf 1}_{j} = \beta $
where~$\beta_{k} = \al_{k} $ if~$k \neq j$ and~$\beta_{j} =
\al_{j}\pm1$.

Observe that
 $\displaystyle{
 \left({1\over
i}Q_j^\lam\right)^*={1\over i}{\overline Q}_j^\lam}$ and that
\begin{equation}\label{usefulformula}
Q_j^\lam =\left\{\begin{array}{l}
-2 |\lam |\xi_j \;\;{\rm if}\;\;\lam>0,\\
\partial_{\xi_j}\;\;{\rm if }\;\;\lam<0,\end{array}\right.
\;\; {\rm and}\;\; \overline Q_j^\lam=\left\{\begin{array}{l}
\partial_{\xi_j}\;\;{\rm if}\;\;\lam>0,\\
-2 |\lam |\xi_j\;\;{\rm if }\;\;\lam<0.\end{array}\right.
\end{equation}
We therefore can write
$$
 {\mathcal F} (-\Delta_{\H^d}  f)( \lam ) =  {\mathcal F} (f)(\lam)\circ D_{\lam} \quad \mbox{where}
 \quad
  D_{\lam} \eqdefa 2  \sum_j  (  Q_{j}\overline Q_{j} +   \overline Q_{j} Q_{j})
$$
Using~(\ref{defQj}) and~(\ref{defoverQj}) we notice that
\begin{equation}\label{defDlam}
\forall \alpha \in \N^d, \quad D_{\lam}\:  F_{\al,\lam} \eqdefa 4
|\lam| (2|\al| + d) \: F_{\al,\lam}.
\end{equation}
Powers of~$  -\Delta_{\H^d} $ can therefore be defined in the following way:  for any real number~$\r$,
\begin{equation}\label{linkdeltadlamxip}
\begin{array}{c}
 {\mathcal F} ((-\Delta_{\H^d})^\r f)( \lam ) =  {\mathcal F} (f)(\lam)\circ D^\r_{\lam} \quad \mbox{and}\\
 {\mathcal F} (({\rm Id} -\Delta_{\H^d})^\r f)( \lam ) =  {\mathcal F} (f)(\lam)\circ ({\rm Id} +D_\lam)^\r .
\end{array}
\end{equation}
Notice that~(\ref{defDlam}) shows that the quantity~$|\lam| (2|\al| + d) $ may be considered as a "frequency" on the Heisenberg group.
Finally one sees easily that
$$
 {\mathcal F} (\partial_s f)( \lam ) =i \lam  {\mathcal F} (f)(\lam).
$$
This explains why the partial derivative~$\partial_s$ is usually considered as a second-order operator, though one notices here that its "strength" is somewhat weaker than that of the Laplacian since its action, in Fourier space, corresponds to a multiplication by~$\lam$ while the Laplacian produces~$4
|\lam| (2|\al| + d)$.

Finally it will be useful later on to notice that due to formulas~(\ref{defQj}), (\ref{defoverQj}) and~(\ref{defDlam}), the operators~$D_\lam^{-m/2} \circ( Q_j^\lam )^m$ and~$D_\lam^{-m/2}  \circ (\overline Q_j^\lam)^m $ are uniformly bounded on~${\mathcal H}_\lam$ for any integer~$m$.

Note that one can also prove, in the same fashion as in the Euclidean case, relations between~${\mathcal F} \left( (is - |z|^2)f\right) (\lam)$ and~${\mathcal F}( f)(\lam)$; we refer to Proposition~\ref{is-z2} below for formulas.
\begin{rem}\label{formulaZjDeltaj}
The above computations show that for any function~$f\in{\mathcal
S}(\H^d)$,
\begin{eqnarray*}
Z_j f(w) &=& \frac{2^{d-1}}{\pi^{d + 1}} \int_{- \infty}^{\infty}
{\rm tr} \left( u^{\lam}_{w^{-1}}{\mathcal F}(f)(\lam )
Q_j^\lam\right)
  |\lam |^{d} d\lam, \\
\overline  Z_j f(w) &=& \frac{2^{d-1}}{\pi^{d + 1}} \int_{-
\infty}^{\infty} {\rm tr} \left( u^{\lam}_{w^{-1}}{\mathcal
F}(f)(\lam ) \overline Q_j^\lam\right)
  |\lam |^{d} d\lam, \quad \mbox{and} \\
  -  \Delta_{\H^d} f(w) &=& \frac{2^{d-1}}{\pi^{d + 1}} \int_{- \infty}^{\infty} {\rm tr}
\left( u^{\lam}_{w^{-1}}{\mathcal F}(f)(\lam ) D_\lam\right)
  |\lam |^{d} d\lam.
\end{eqnarray*}
In particular, if we consider the derivatives of the Dirac
distribution $\delta_{w_0}(w)$ defined as usual by duality through
$$<Z_j\delta_{w_0},f>=- <\delta_{w_0}, Z_j f>=-Z_jf(w_0) \quad
\mbox{and} $$ $$ <\overline Z_j\delta_{w_0},f>=-
<\delta_{w_0},\overline Z_j f>=-\overline Z_jf(w_0) $$ for all
$f\in{\mathcal S}(\H^d)$ and for some fixed $w_0\in\H^d$, we
obtain an expression of the derivatives of the Dirac distribution
as  singular integrals
\begin{eqnarray*}
Z_j \delta_{w_0}(w) &=& -\frac{2^{d-1}}{\pi^{d + 1}} \int_{-
\infty}^{\infty} {\rm tr} \left( u^{\lam}_{w_0^{-1}w}
Q_j^\lam\right)
  |\lam |^{d} d\lam, \\
\overline  Z_j \delta_{w_0}(w) &=&- \frac{2^{d-1}}{\pi^{d + 1}}
\int_{- \infty}^{\infty} {\rm tr} \left( u^{\lam}_{w_0^{-1}w}
\overline Q_j^\lam\right)
  |\lam |^{d} d\lam, \quad \mbox{and} \\
  -  \Delta_{\H^d} \delta_{w_0}(w) &=& \frac{2^{d-1}}{\pi^{d + 1}} \int_{- \infty}^{\infty} {\rm tr}
\left( u^{\lam}_{w_0^{-1}w}D_\lam\right)
  |\lam |^{d} d\lam.
\end{eqnarray*}
\end{rem}

It turns out that for radial functions on the Heisenberg group,
 the Fourier transform becomes simplified. Let us
 first recall the concept of radial functions on the Heisenberg
 group.\begin{defin}
 {  A function~$f$ defined on the Heisenberg group~$\H^d$ is
said to be radial if it is invariant under the action of the
unitary group~$U(d)$ of~$\C^d$, meaning that for any~$u \in U(d)$,
we have
\[ f(z, s)= f(u(z), s), \quad \forall (z,s) \in \H^d.\]
 A radial function on the Heisenberg group can  then be written
under the form $$ f(z, s) = g(|z|, s).$$ }
\end{defin}

Then it can be shown (see for instance~\cite{nach}) that the
Fourier transform of radial functions of~$L^2(\H^d),$ satisfies
the following formula:
$$
{\mathcal F}(f)(\lam) F_{\al,  \lam} = R_{|\al|} ( \lam ) F_{\al,  \lam}
$$
  where
  $$
  R_{m} (
\lam )\eqdefa \left(\begin{array}{c} m+d-1 \\ m
\end{array}\right)^{-1} \int  e^{i\lam s}f(z, s) L_m^{(d-1)}(2|\lam|
|z|^2)e^{-|\lam| |z|^2}  dz ds,
$$ and where ~$ L_m^{(p)}$ are
Laguerre polynomials  defined by
\beq\label{deflaguerre}
 L_m^{(p)} (t) \eqdefa \sum_{k=0}^{m} ( -1 )^k \left(\begin{array}{c} m+p\\ m - k
 \end{array}\right ) \frac{t^k}{k!} \virgp \quad t \geq  0, \quad (m,  p) \in \N^2.
 \eeq
Note that in that context, Plancherel and inversion formulas can
be stated as follows: $$\|f\|_{L^{2}(\H^{d})} =
\left(\frac{2^{d-1}}{\pi^{d+1}}\sum _m \left(\begin{array}{c}
m+d-1 \\ m
 \end{array}\right ) \int_{- \infty}^{\infty} |R_{m} (
\lam )|^2 |\lam|^d
 d\lam  \right)^{\frac12}
$$ and \begin{equation}\label{definvR} f( z, s) =
\frac{2^{d-1}}{\pi^{d+1}} \sum_m \int e^{-i \lam s} R_{m} ( \lam )
L_m^{(d-1)} (2|\lam| |z|^2)e^{-|\lam| |z|^2} |\lam|^d d\lam.
\end{equation}
The context of radial functions allows to compute the Fourier transform of~$(is - |z|^2)f$, as stated below (see~\cite{bgx} for a proof).
  \begin{prop}\label{is-z2}
For any radial function~$ f \in
{\mathcal S} (\H^d)$, we have for any~$m \geq 1$,
\begin{eqnarray*}
{\mathcal F} ((is-|z|^2)f)(m, \lam)& = & \frac{d}{d \lam}{\mathcal F} {f}(m,  \lam ) -
 \frac{m}{\lam}({\mathcal F} {f}(m,  \lam )-{\mathcal F} {f}(m-1,  \lam )) \quad \mbox{if} \:\: \lam >
 0,  \quad \mbox{and}
 \\   {\mathcal F} ((is-|z|^2)f)(m, \lam)& = & \frac{d}{d \lam}{\mathcal F} {f}(m,  \lam ) +
 \frac{m + d }{|\lam |}({\mathcal F} {f}(m,  \lam )- {\mathcal F}{f}(m +1,  \lam ))  \quad \quad \mbox{if} \:\:
  \lam
 < 0.
\end{eqnarray*}
\end{prop}
  \subsubsection{The $L^2$ representation}\label{Schrodinger}
  In order to define pseudodifferential operators, it will be useful to use rather the~$L^{2}$ (or Schr\"odinger) representations, denoted in the following by~$
 (v^\lambda_{z,s}f)(\xi)$, where~$\xi$ belongs to~$\R^d$ and~$f$ to~$L^2(\R^d)$.
As recalled above, the representations~$v^\lambda_{z,s}$ and
$u^\lambda_{z,s}$ are equivalent. The intertwining operator is
the Hermite-Weber transform~$ \displaystyle K_\lambda:{\mathcal
H}_\lambda\rightarrow  L^2(\R^d)
$
 given by
\beq\label{intertwining}
  (K_\lambda
\phi)(\xi)\eqdefa\frac{|\lambda|^{d/4}}{\pi^{d/4}} {\rm
e}^{|\lambda|\frac{|\xi|^2}{2}}
\phi\left(-\frac{1}{2|\lambda|}\frac{\partial}{\partial\xi}\right){\rm
e}^{-|\lambda|\,|\xi|^2}, \eeq which is unitary and intertwines
both representations: we have indeed~$\displaystyle K_\lambda
u_{z,s}^\lambda=v_{z,s}^\lambda K_\lambda $ and
\begin{equation}\label{schrorep}
v_{z,s}^\lam f(\xi)= e^{i\lam (s - 2 x \cdot y + 2 y \cdot \xi)}
f(\xi - 2x),  \quad \forall \lam\in\R^* .
\end{equation}
A short proof of this fact is given in
Appendix~\ref{representations} for the
convenience of the reader (see Proposition~\ref{formulavlam} page~\pageref{formulavlam}). We also
 recall that   the inverse of~$K_\lam$ is known as the Segal-Bargmann transform (see for instance~\cite{farautsaal}).
 Let us denote by $h_\alpha$
 the multidimensional
Hermite function defined by
 $$
 \forall\alpha=(\alpha_1,\cdots,\alpha_d)\in\N^d,\;\;\forall
t=(t_1,\cdots,t_d)\in\R^d,\;\;h_\alpha(t)\eqdefa h_{\alpha_1}(t_1)\cdots
h_{\alpha_d}(t_d) ,
 $$
 with
 $$
 h_n(t)\eqdefa \left(2^n\,n!\,\sqrt\pi\right)^{-1/2}
\,{\rm e}^{-t^2/2} H_n(t) \quad \mbox{and} \quad  H_n(t)\eqdefa {\rm
e}^{t^2}\left(-\frac{\partial}{\partial t}\right)^n\left({\rm
e}^{-t^2}\right).
 $$
Introducing
the scaling operator \beq\label{scaling} \forall f\in
L^2(\R^d),\;\;T_\lambda
f(\xi)\eqdefa|\lambda|^{-d/4}f(|\lambda|^{-1/2} \xi), \eeq
and setting $h_{\alpha,\lambda} =
T^*_\lambda h_\alpha$
we observe that
\begin{equation}\label{Klambda}
 \forall\alpha\in\N^d,\;\;K_\lambda
F_{\alpha,\lambda}=h_{\alpha,\lambda}
\end{equation}
where~$h_{\alpha,\lambda}$ is an eigenfunction of the rescaled
harmonic oscillator~$-\Delta_\xi + |\lambda||\xi|^2$.  This implies by straightforward computations   that
$$\begin{array}{c} K_\lambda Q_j^\lam
K_\lambda^*=\partial_{\xi_j}-|\lambda| \xi_j \quad \mbox{and} \quad
K_\lambda \overline Q_j^\lam
K_\lambda^*=\partial_{\xi_j}+|\lambda| \xi_j \;\;{\rm if}\;\;\lam>0,\\
K_\lambda Q_j^\lam K_\lambda^*=\partial_{\xi_j}+|\lambda| \xi_j \quad
\mbox{and} \quad K_\lambda \overline Q_j^\lam
K_\lambda^*=\partial_{\xi_j}-|\lambda|y_j \;\;{\rm if}\;\;\lam<0.
\end{array}
$$ Defining the operator \beq\label{rescaledinterlacing} J_\lam
\eqdefa T_\lambda K_\lambda,\eeq
and observing that
$$
 T_\lambda (-\Delta_ \xi + | \xi |^2 |\lam|^2)   T_\lambda^*=   |\lambda|
(- \Delta_ \xi +| \xi |^2),
$$  we infer that
\begin{equation}\label{JlamQJlam}
\begin{array}{c}
 J_\lam Q_j^\lam J_\lam^* =
\sqrt{|\lambda|}\left(\partial_{\xi_j}-\xi_j\right) \quad \mbox{and}
\quad J_\lam \overline Q_j^\lam J_\lam^* =
\sqrt{|\lambda|}\left(\partial_{\xi_j}+ \xi_j\right)\;\;{\rm
if}\;\;\lam>0\\
 J_\lam Q_j^\lam J_\lam^* =
\sqrt{|\lambda|}\left(\partial_{\xi_j}+ \xi_j\right) \quad \mbox{and}
\quad J_\lam \overline Q_j^\lam J_\lam^* =
\sqrt{|\lambda|}\left(\partial_{\xi_j}-\xi_j\right)\;\;{\rm
if}\;\;\lam<0 ,
\end{array}
\end{equation}
which finally implies that
 \beq\label{JlamDjJlam}
 J_\lam D_\lambda
J_\lam^* = 4 |\lam| (-\Delta_ \xi + | \xi |^2)  .
\eeq In view of
Remark~\ref{formulaZjDeltaj}, the Laplacian~$-\Delta_{\H^d}$ is
associated with the operator $D_\lam$ of~${\mathcal H}_\lam$ in
the Bargmann representation; by Equation~\refeq{JlamDjJlam}, it is
associated with the harmonic oscillator in the~$L^2(\R^d)$
framework.

 \medskip

 These computations indicate that symbolic calculus on~${\mathcal H}_\lam$ is, via the unitary operator~$J_\lambda$, equivalent to  symbolic calculus on the harmonic oscillator. That theory is well understood: it consists in Weyl-H\"ormander calculus associated with a harmonic oscillator metric. This is made precise in the next section.

 Before proceeding further, it is instructive to compute the Fourier
  transform for instance of the function~$Z_j \overline Z_jf$ for~$f\in{\mathcal S}(\H^d)$.
   Indeed, we notice that with the previous notations,  for~$\lam>0$,
 \begin{eqnarray*}{\mathcal F}(
 (-i Z_j) (-i \overline Z_j) f)(\lam) &=& {\mathcal F}(-i \overline Z_jf)(\lam) J_\lam^*
\sqrt{|\lam|} (-i\partial_{\xi_j} + i \xi_j) J_\lam\\
   &=& {\mathcal F}( f)(\lam) J_\lam^*
|\lam| (-i\partial_{\xi_j} - i \xi_j) (-i\partial_{\xi_j} + i \xi_j)
J_\lam
 \\
 &=&   {\mathcal F} ( f)(\lam) J_\lam^* |\lam|(\xi_j^2 -\partial_{\xi_j}^2 + 1)J_\lam
.
 \end{eqnarray*}
This implies that symbols on the Heisenberg group must not only
include harmonic oscillator type symbols, but also functions such
as  powers of~$\lam$.

\section{Weyl-H\"ormander calculus}
 \setcounter{equation}{0}

Let us recall in this section some results on the
Weyl-H\"ormander calculus of  the harmonic oscillator which we
shall be using. We shall only state the definitions that will be
needed in the following, and  for further details, we   refer for
instance to~\cite{bonychemin}, \cite{bonylerner}, \cite{ccx}, \cite{cx1}, \cite{hormander} and~\cite{lernerbook}.

\subsection{Admissible weights and metrics}
Let us denote  by~$\omega[\Theta, \Theta']$ the  standard symplectic form
on~$T^*\R^{d}$ (which we shall identify in the following
to~$\R^{2d}$) : if~$\Theta= (\xi,\eta) $ and~$\Theta' = (\xi',\eta')$, then~$
\omega[\Theta, \Theta'] \eqdefa \eta \cdot \xi' - \eta' \cdot \xi.$

For any point~$\Theta = (\xi,\eta)$ in~$\R^{2d}$, we consider a
Riemannian metric~$g_ \Theta $ (which depends measurably on~$\Theta $) to which we associate the conjugate metric~$g_ \Theta ^\omega$ by
$$
\forall T \in \R^{2d}, \quad (g_ \Theta ^\omega(T)) ^{1/2}
=
\sup_{T' \in\R^{2d}} {| \omega[T, T']| \over
 g_ \Theta(T')^{1/2} }\cdotp
$$ We also define the {\it gain factor}
\begin{equation}\label{deflambda}
\Lambda_ \Theta \eqdefa
\inf_T \frac{g_ \Theta ^\omega(T)}{g_ \Theta (T)} \cdotp
\end{equation}
 \begin{defin}\label{hormandermetric}
 We shall say that the metric~$g$ is of H\"ormander type if it is:
 \begin{enumerate}

 \item Uncertain: For all~$\Theta \in \R^{2d}$, $\Lambda_ \Theta \geq 1$.

  \item Slowly varying: There is a constant~$\overline C>0$ such that
  $$
  g_ \Theta(\Theta-\Theta') \leq \overline C^{-1} \Rightarrow \sup_{T \in \R^{2d}} \left(\frac{g_ \Theta(T)}{g_{\Theta'}(T)}\right)^{\pm 1} \leq \overline C.
  $$

 \item Temperate: There are constants~$\overline C>0$ and~$\overline N \in \N$ such that for all~$(\Theta, \Theta') \in \R^{4d}$,
 $$
  \sup_{T \in \R^{2d}} \left(\frac{g_ \Theta(T)}{g_{\Theta'}(T)}\right)^{\pm 1} \leq \overline C (1+g_ \Theta ^\omega (\Theta-\Theta'))^{\overline N}.
 $$
 \end{enumerate}
 In the following any constant depending only on~$\overline C$ and~$\overline N$ will be  called a structural constant.
 \end{defin}
 In the definition above we have used the notation
 $$
  \left(\frac{g_ \Theta(T)}{g_{\Theta'}(T)}\right)^{\pm 1}  \eqdefa \frac{g_ \Theta(T)}{g_{\Theta'}(T)} + \frac{g_{\Theta'}(T)}{g_ \Theta(T)}\cdotp
 $$
 We also define a weight as a positive function on~$\R^{2d}$ satisfying the same type of conditions as a  H\"ormander metric.
   \begin{defin}\label{hormanderweight}
   Let~$g$ be a metric in the sense of Definition~\ref{hormandermetric}.
 A positive function~$m$ on~$\R^{2d}$ is a $g$-weight if there are structural constants~$\overline C'>0$ and~$\overline N' \in \N$ such that
 \begin{enumerate}
   \item
  $\displaystyle
  g_\Theta(\Theta-\Theta') \leq \overline C'^{-1} \Rightarrow  \left(\frac{m(\Theta)}{m(\Theta')}\right)^{\pm 1} \leq \overline C'.
  $

 \item    $\displaystyle
 \left(\frac{m(\Theta)}{m(\Theta')}\right)^{\pm 1} \leq \overline C' (1+g_ \Theta ^\omega (\Theta-\Theta'))^{\overline N'}.
 $
 \end{enumerate}
 \end{defin}
 It is easy to see that the set of~$g$-weights has a group structure (for the usual product of functions).

 For such metrics and weights, one can then define the class~$S(m,g)$ of   smooth functions~$a$
on~$\R^{2d}$ such that, for any integer~$n$,
\begin{equation}\label{seminorm}
\|a\|_{n;S(m , g)} \eqdefa \renewcommand{\arraystretch}{0.5}
\begin{array}[t]{c}
\sup\\ {\scriptstyle j \leq n, \Theta \in\R^{2d}}\\ {\scriptstyle
g_\Theta(T_j) \leq 1}
\end{array}
\renewcommand{\arraystretch}{1}
\frac{|\partial_{T_1} ...\partial_{T_j} a(\Theta) |}{m(\Theta)} < \infty,
\end{equation}
 where~$\partial_{T}
 a$ denotes the map~$\langle da,T\rangle$.
Now, if~$a$ is a symbol in~$S(m  ,g)$, then its Weyl quantization
is the operator which associates to~$u \in {\mathcal S}(\R^d)$ the
 function~$op^w (a) u$ defined by
  \begin{equation}\label{weylquant}
  \forall \xi \in \R^d, \quad
\left( op^w (a)u \right)( \xi) \eqdefa (2\pi)^{-d} \int_{\R^{2d}}
{\rm e}^{i( \xi- \xi') \cdot \eta} a\left(\frac{ \xi+ \xi'}2\virgp \:
\eta\right) u( \xi') d \xi'd\eta.
 \end{equation}
 The main interest of this quantization is that $op^w(a)^*=op^w(\overline a)$.

 Observe also that if $a(\xi,\eta)=\wt a( \xi)$, the operator $op^w(a)$ is the operator of multiplication by the function $\wt a$ and if $a(\xi,\eta)=\wt a(\eta)$, the operator $op^w(a)$ is the Fourier multiplier $\wt a(D)$. In particular one has~$\displaystyle{op^w(\eta_j^k)=\left({1\over i}\partial_{ \xi_j}\right)^k}$ for any $k\in\N$.

 Besides,
 for all symbols~$a \in S(m_1
,g) $ and~$b \in S(m_2 ,g) $ where~$m_1$ and~$m_2$ are~$g$-weights, we have  the following composition
formulas: $$ op^w(a) \circ\ op^w(b) = op^w(a \# b)\;\;{\rm
with}\;\;a\# b\in S(m_1m_2,g) \quad \mbox{and} $$
 \begin{equation}\label{sharp}
  (a \# b)(\Theta) =
\pi^{-2d} \int_{\R^{2d} \times \R^{2d}} {\rm e}^{-2i \omega[\Theta - \Theta_1, \Theta -
\Theta_2]} a(\Theta_1) b(\Theta_2) d \Theta_1 d \Theta_2.
\end{equation}
The (non commutative) bilinear operator~$ \# $ is often referred to as the Moyal product.

This leads to  an asymptotic
formula
 \beq\label{asympt}
  a \# b  = ab  + \frac1{2i} \{a,b\}+\cdots + r_N,
\eeq
where~$ ab $  belongs to~$   S(m_1m_2,g)$ and~$ \displaystyle {1\over 2i} \{a,b\} $ belongs to~$   S(\Lam^{-1}m_1m_2,g)$, recalling that~$ \{a,b\}$ is the usual Poisson bracket
$$
\{a,b\}  \eqdefa  \sum_{j = 1}^d \left( \partial_{\eta_j}  a\,
\partial_{\xi_j} b - \partial_{\xi_j} a\, \partial_{\eta_j} b
\right).
$$Finally for any integer~$N$, the remainder~$r_N$ belongs to~$ S(\Lam^{-N}m_1m_2,g)$.

Let us  mention  that the operator $op^w(a)$ has a kernel $k(\xi,\xi')$
defined by
\beq \label{defkernelweyl}
k(\xi,\xi')=  (2\pi)^{-d} \int_{\R^{d}} {\rm e}^{i(\xi-\xi')
\cdot \eta} a\left(\frac{\xi+\xi'}2\virgp \: \eta\right) d\eta
\eeq
which
is linked to its symbol through
\begin{eqnarray}\label{ktoa}
a(\xi,\eta)=\int _{\R^d} {\rm e}^{-i\xi'\eta}k\left(\xi+{\xi'\over
2},\xi-{\xi'\over 2}\right)d\xi'.
\end{eqnarray}

Let us also point out that a concept of Sobolev space~$H(m, g)$
 was introduced by R. Beals
in~\cite{rbeals2}. We will use the following characterization of
those spaces.
 \begin{defin}\label{definHmg}
 {  Let~$g$ and~$m$ be respectively a H\"ormander metric and a $g$-weight, in the sense of Definitions~\ref{hormandermetric} and~\ref{hormanderweight}.
 We denote by~$H(m ,g)$ the set of all  tempered distributions~$u$
on~$\R^d$ such that, for any~$a \in S(m , g)$, we have~$op^w(a) u
\in L^2(\R^d)$. In particular~$H(1, g)$ coincides
with~$L^2(\R^d)$.}
  \end{defin}
  Note that the study of Sobolev spaces associated with a
H\"ormander  metric~$  g$ and a~$  g$-weight has been developed
in~\cite{rbeals2},~\cite{bonychemin},~\cite{bonylerner},~\cite{ccx}
and~\cite{taylor} and in particular in~\cite{bonychemin}, it was
shown that these spaces  are ``almost independent'' of the
metric~$  g$. The Weyl quantization defined by\refeq{weylquant}
can be extended to an operator on~${\mathcal S}'(\R^d)$ which acts
on the Sobolev spaces~$H(m , g)$ in the following way.
\begin{prop}\label{opsmghmg}
 {Let~$g$ be a H\"ormander metric, and let~$m$ and~$m_1$ be~$g$-weights. There  exists a
constant~$C$, depending only on the structural constants of
Definitions~\ref{hormandermetric} and~\ref{hormanderweight}, such
that the following holds. Let~$a$ be in~$S(m_1,g)$. Then, there
exist  an integer~$n$ and a constant~$C$ such that for any~$u$ in~$H(m , g)$, we have
$$ \|op^w (a)u \|_{H(m m_1^{-1} ,g )} \leq C
\|a\|_{n;S(m_1,g)}\|u\|_{H(m , g)}. $$
In particular, there exist   an integer~$n$ and a constant  $C$ such that if $a\in S(1,g)$, then for any $u\in L^2(\R^d)$ one has
\begin{equation}\label{bornitude}
\|op^w(a)u\|_{L^2(\R^d)}\leq C\,\|a\|_{n;S(1,g)}\|u\|_{L^2(\R^d)}.
\end{equation}
 }
\end{prop}

   \subsection{The  case of the harmonic oscillator}\label{harmonic}
As pointed out in Section~\ref{intro}.\ref{sec:Fourier}, it is
natural to base the quantization of symbols on the Heisenberg
group on the calculus related  to  the harmonic oscillator. In
that case one is considering the metric defined by
\begin{equation}\label{osmet}
\forall \Theta = (\xi,\eta) \in \R^{2d},  \quad  g_\Theta(d\xi,d\eta) \eqdefa \frac{d\xi^2 + d\eta^2}{1+\xi^2+\eta^2}
 \end{equation}
 while the $g$-weight is
\begin{equation}\label{oswei}
 \forall \Theta = (\xi,\eta) \in \R^{2d},  \quad  m(\Theta) \eqdefa (1+\xi^2+\eta^2)^\frac12.
 \end{equation}
 It is an exercise to check that $g$ is a H\"ormander metric in the sense of Definition~\ref{hormandermetric}, and that~$m$ is a~$g$-weight in the sense of Definition~\ref{hormanderweight}. This will in fact be performed in the proof of Proposition~\ref {symboltilde}   below in a more general setting.

We will be interested in the class of symbols belonging
to~$S(m^\mu, g)$  for some real number~$\mu$, where we notice
that~(\ref{seminorm}) can simply be written equivalently in the
following way:
 \begin{equation}\label{equivwaysymbol}
 \|a\|_{n; S(m^\mu, g)} \eqdefa \renewcommand{\arraystretch}{0.5}
\sup_{\scriptstyle |\beta|\leq n, (\xi,\eta) \in\R^{2d}}
 (1+\xi^2+\eta^2)^{\frac{|\beta|-\mu}{2}}|\partial_{(\xi,\eta)}^\beta a(\xi,\eta)|
< \infty.
\end{equation}

 It is useful, in particular in the framework of the Littlewood-Paley transformation on the Heisenberg group investigated in Chapter~\ref{prelimins}, to be able to write the Weyl symbol of functions of the harmonic oscillator on~$L^2(\R^d)$. The formula for such symbols is   derived using Mehler's formula (see \cite{feynman} for instance)
    \begin{equation}\label{Mehler}
  {\rm e}^{-t(\xi^2-\Delta_{\xi})}=({\rm ch}\,t)^{-d}\,op^w\left({\rm e}^{- (\xi^2+ \eta^2){\rm th}\, t}\right).
  \end{equation}
  More precisely, we have the following result, whose proof is postponed to Appendix~B (see page~\pageref{proofsymbR}).
     \begin{prop}\label{symbR}
Consider $R$ a smooth function satisfying symbol
estimates:
\begin{equation}\label{*de1.6}
\exists\mu\in\R,\;\;\exists  C>0,\;\;\forall n\in\N,
\;\; \left\|(1+|\cdot|)^{n-\mu}\partial^n  R \right\|_{L^\infty(\R)}\leq C^n.
\end{equation}
Then
$R (\xi^2-\Delta_\xi)$
is a pseudodifferential operator. Moreover one has formally
$$
R (\xi^2-\Delta_\xi) = op^w (r(\xi^2 + \eta^2))
$$
 with for all~$x\neq 0$,
\begin{equation}\label{formuler}
r(x)=\frac 1{2\pi}\int_{\R \times \R} ({\rm cos}\,\tau)^{-d} {\rm e}^{i (x {\rm tg} \tau
-\xi\tau)} R(\xi)d\tau\,d\xi.
\end{equation}
Besides $(\xi,\eta)\mapsto r(\xi^2+\eta^2)$ is satisfies the symbol estimates  of the class
$S(m^\mu,g)$, in the sense of~(\ref{equivwaysymbol}).
\end{prop}
Note that~$r$ is not well defined at~$x=0$ in general, which explains why the relation~$R (\xi^2-\Delta_\xi) = op^w (r(\xi^2 + \eta^2))$ is only formal.
 One also has the inverse formula
 \begin{equation}\label{invfor}
 op^w\left(r(y^2+\eta^2)\right)={1\over 2\pi}\int \widehat r(\tau) {\rm e}^{i\,(y^2-\Delta){\rm Arctg}\tau}(1+\tau^2)^{-d/2}d\tau.
 \end{equation}
 This yields that the operator $J_\lam^*op^w(r(y^2+\eta^2))J_\lam$ is diagonal in the  basis $(F_{\alpha,\lam})_{\alpha\in\N^d}$ and thus commutes with operators of the form $\chi(D_\lam)$ for all continuous bounded functions $\chi$, where~$\chi(D_\lam)$ is the operator
 \beq
\label{defchidlam}
\chi(D_\lam) F_{\alpha, \lam} = \chi (4|\lam| (2|\alpha|+d )) F_{\alpha, \lam} .
 \eeq

\medskip

 \begin{rem}\label{m-mua}
Let us note  that the  operator~$ { \rm Id }- \Delta_\xi    + \xi^2   $ has for symbol~$m^2$, while the symbol of~$ 4(- \Delta_\xi    + \xi^2)$ is
$
\widetilde m_2 (\xi,\eta) $ where $\widetilde m_2(\xi,\eta) \eqdefa 2 (
\xi^2+\eta^2)^{1\over 2}.$

 Besides, for $\mu\in\R$,
Proposition~\ref{symbR} shows that  there exists a function
$  m_\mu\in S(m^\mu,g)$ such that
$2^\mu({\rm Id} -\Delta_\xi+\xi^2)^{\mu/2}=op^w(  m_\mu)$. In
particular, for any $\mu,\mu'\in\R$, $  m_\mu \#
  m_{\mu'}=  m_{\mu+\mu'}.$

  Finally if~$\mu \geq 0$, then there exists a function
$\widetilde m_\mu\in S(m^\mu,g)$ such that
$2^\mu(-\Delta_\xi+\xi^2)^{\mu/2}=op^w(\widetilde m_\mu)$. In
particular, for any $\mu,\mu'\in\R$, $\widetilde m_\mu \#
\widetilde m_{\mu'}=\widetilde m_{\mu+\mu'}.$ Note that the restriction to~$\mu \geq 0$ is natural and holds also in the euclidean case.
   \end{rem}

  \section{Main results: pseudodifferential operators on the Heisenberg group}\label{statementresult}
  \setcounter{equation}{0}

  In this section,
  motivated by the examples studied in the previous sections of this chapter, we shall   give   a definition of symbols, and pseudodifferential operators, on the Heisenberg group. Then we will state the main results proved in this paper concerning those operators.

\subsection{Symbols}\label{defsymbolsops}
 Our approach inspired by the Euclidian strategy of  R. Coifman and Y.
Meyer~\cite{cm} allows to consider symbols with limited regularity with respect the Heisenberg variable.  Therefore, in what follows, we shall define a
 positive, noninteger real number~$\rho$, which will measure the regularity assumed on the symbols (in the Heisenberg variable).
 This number~$\rho$ is fixed from now on and we emphasize that  the definitions below depend on~$\rho$. We have chosen not to keep
 memory of this number on the notations for the sake of simplicity.

\begin{defin}\label{definsymb}
 { \it A smooth function~$a$
 defined
 on~$\H^d \times \R^{*}\times \R^{2d}$ is a symbol if there is a real number~$\mu$ such that for all $n\in\N$,  the following semi norm is finite:
$$
 \|a\|_{n;S_{\H^d}(\mu)}  \eqdefa \supetage{\lam
\neq 0}{\Theta \in \R^{2d}} \sup_{| \beta | + k \leq n}
 |\lam|^{ -\frac{|\beta |}{2}} \left( 1 + |\lam| (1+ \Theta ^2) \right)^{\frac{| \beta | - \mu}{2}}
   \| ( \lam \partial_\lam)^k  \partial_{\Theta}^ \beta a
(\cdot,\lam, \Theta) \|_{C^\rho (\H^d)} .
 $$
Besides,  one additionally requires that the function
\begin{equation}\label{def:sigma(a)}
(w,\lam,\xi,\eta) \mapsto \sigma(a)(w,\lam,\xi,\eta)\eqdefa a\left(w,\lam,{\rm sgn}(\lam){\xi\over\sqrt{|\lam|}},{\eta\over\sqrt{|\lam|}}\right)
\end{equation}
is uniformly smooth close to $\lam=0$  in the sense that there exists $C>0$ such that $$\forall(w,\xi,\eta)\in\H^d\times\R^*\times \R^{2d},\;\forall\lambda\in[-1,1],\;\;
 \left\|\partial_\lam^k \partial_{(\xi,\eta)}^\beta (\sigma(a))\right\|_{{\mathcal C}^\rho(\H^d)}\leq C_{n,k}\left(1+|\lam|+\xi^2+\eta^2\right)^{\mu-|\beta|\over 2}.$$
 In that case we shall write~$a \in S_{\H^d}(\mu)$.
 }
  \end{defin}

 \begin{rem}
 The additional assumption on~$\sigma(a)$  is necessary in order to guarantee that pseudodifferential operators associated with those symbols are continuous on~${\mathcal S}(\H^d)$ (see Proposition~\ref{contrexempleschwartz}). It is also required to obtain that the space of pseudodifferential operators is an algebra.
 \end{rem}
  In the remainder of this section, we shall discuss two points of view. The first consists in considering the symbol $a\in S_{\H^d}(\mu)$ as a symbol on~$\R^{2d}$ depending on the parameters~$(w,\lam)$ in~$\H ^d\times\R$ and
 belonging to a $\lam$-dependent H\"ormander class  (see Proposition~\ref{symboltilde}). The second point of view consists in emphasizing the function $\sigma(a)$ (see Proposition~\ref{prop:sigma(a)}). Both points of view are in fact interesting,  and both will be used in the following.

  Let us first analyze the properties of $a\in S_{\H^d}(\mu)$ for a  fixed   $\lam$.
  The following proposition is proved in Appendix~B (see page~\pageref{prooflambdadependent}).\\

    \begin{prop}\label{symboltilde}
    The ($\lam$-dependent) metric $g^{(\lam)}$ defined by
     $$
     \forall \lam \neq 0, \: \forall \,\Theta  \in \R^{2d}, \quad g_ \Theta ^{(\lam)}(  d\xi,  d\eta)\eqdefa{|\lam| (d  \xi ^2+d \eta^2)\over 1+|\lam|(1+  \Theta^2)}
     $$
     is  a H\"ormander metric in the sense of Definition~\ref{hormandermetric}, and the function
     $$
     m^{(\lam)}( \Theta)\eqdefa\left( 1+|\lam|(1+ \Theta ^2)\right)^{1/2}
     $$
     is a $g^{(\lam)}$-weight. Moreover the   constants~$\overline C$  and~$\overline N$ of Definitions~\ref{hormandermetric} and~\ref{hormanderweight} are independent of~$\lam$.

     Finally if~$a$ is a smooth function defined
 on~$\H^d \times \R^{*}\times \R^{2d}$, then~$a$ belongs to~$ S_{\H^d}(\mu)$ if and only if~(\ref{def:sigma(a)}) defines a smooth function and for any $k\in\N$, the function~$(\lam\partial_\lam)^k  a$ is a symbol of order $\mu$ in  the  Weyl-H\"ormander class defined by the metric~$g^{(\lam)}$ and the~$g^{(\lam)}$-weight~$m^{(\lam)}$, uniformly with respect to~$\lam$.
  \end{prop}
Proposition~\ref{symboltilde} has important consequences which are stated below. The first one will be used often in the sequel and states that the continuity constants of Weyl quantizations of symbols are independent of~$\lam$ and~$w$.
\begin{cor}\label{cor:continuiltL2}
Let $a$ be a symbol in~$ S_{\H^d}(\mu)$. Then for any $w\in\H^d$ and $\lam\in\R^*$, the operator~$op^w(a(w,\lam))$ is continuous from $H(m,g^{(\lam)})$ into $H\left(m(m^{(\lam)})^{-\mu},g^{(\lam)}\right)$ for any~$g^{(\lam)}$-weight~$m$, and the constant of continuity is uniform wih respect to $\lam$ and $w$.  In particular for~$\mu=0$, the operator $op^w(a(w,\lam))$ maps $L^2(\R^d)$ into itself  uniformly with respect to $w$ and~$\lam$.
\end{cor}
The second consequence concerns the stability  of our class of symbols with respect to the Moyal product (see~(\ref{sharp})):
if $a\in S_{\H^d}(\mu_1 )$ and $b\in
S_{\H^d}(\mu_2 )$, then the functions~$ab$ and~$a\,\#\,b$
 are  symbols in the class~$ S_{\H^d}(\mu_1 + \mu_2)$.
 Besides, the asymptotic formula can be written
   $$ a\,\#\,b  = ab  + \frac {|\lambda|}{2i} \sum_{j = 1}^d
\Bigl( \frac {1}{\sqrt{|\lambda|}} \partial_{\eta_j}  a\, \frac
{1}{\sqrt{|\lambda|}}\partial_{\xi_j} b - \frac
{1}{\sqrt{|\lambda|}}\partial_{\xi_j} a\, \frac
{1}{\sqrt{|\lambda|}}\partial_{\eta_j} b \Bigr)+\cdots $$ Let us
also point out that if $a$ belongs to~$S_{\H^d}(\mu )$, then for any~$j \in
\{1,\dots, d\}$ the functions~$\displaystyle \frac
{1}{\sqrt{|\lambda|}}\partial_{\xi_j} a$ and $\displaystyle\frac
{1}{\sqrt{|\lambda|}} \partial_{\eta_j}  a$ belong
to~$S_{\H^d}(\mu-1 )$.

Let us now mention some properties of the function $\sigma(a)$ defined in \aref{def:sigma(a)}. The following proposition, which is proved in  Appendix~B (see page~\pageref{prooflambdadependentsymbols}), will be   useful  in the proofs of Chapter~\ref{algebra}.

\begin{prop}\label{prop:sigma(a)}
A function $a$ belongs to~$S_{\H^d}(\mu) $ if and only if $\sigma(a)\in{\mathcal C}^\infty( \H^d\times\R^{2d+1})$ satisfies:
for all $k,n\in\N$, there exists a constant $C_{n,k}>0$ such that for any $\beta\in \N^d$ satisfying~$|\beta|\leq n$, and for
all~$ (w,\lam,y,\eta)\in \H^d\times \R^{2d+1},$
\begin{equation}\label{tilt2}
 \left\|\partial_\lam^k \partial_{(\xi ,\eta)}^\beta (\sigma(a))\right\|_{{\mathcal C}^\rho(\H^d)}\leq C_{n,k}\left(1+|\lam|+ \xi ^2+\eta^2\right)^{\mu-|\beta|\over 2}(1+|\lam|)^{-k}.
\end{equation}
\end{prop}
 \subsection{Operators}
   We   define pseudodifferential operators as follows.
\begin{defin}\label{definpseudo}
 { To a  symbol~$a$ of order~$\mu$ in the sense of
Definition~\ref{definsymb}, we associate the pseudodifferential
operator on~$\H^d$
 defined in the following way:
for any~$f\in{\mathcal S}(\H^d)$,
\begin{equation}
\label{pseudosformule}
 \forall w \in \H^d, \quad {\rm Op}(a) f (w) \eqdefa \frac{ 2^{d-1}}{\pi^{d+1}} \int_{\R} {\rm tr} \left(
 u^\lam_{w^{-1}} {\mathcal F}(f)(\lam) A_\lam(w)
 \right) \: |\lam|^d \: d\lam,
\end{equation}
 where
 \begin{equation}\label{def:Alam}
  A_\lam (w) \eqdefa    J_\lam^* \,op^w (a(w,\lam,  \xi, \eta) )\,J_\lam\;\; \mbox{if} \;\;\lam\neq 0 .
 \end{equation}
   while~$J_\lam$ is defined in~(\ref{rescaledinterlacing}), page~\pageref{rescaledinterlacing}.}
  \end{defin}

  Examples of pseudodifferential operators are provided in Section~\ref{examples} of Chapter~\ref{fundamental}.

Observe that the operator ${\rm Op}(a)$ has a kernel
 \begin{equation}\label{eq:ka}
  k_a(w,w')= \frac{2^{d-1}}{\pi^{d + 1}} \int_{-
\infty}^{\infty} {\rm tr}
\left(u^\lambda_{w^{-1}w'}A_\lambda(w)\right)
  |\lam |^{d} d\lam
  \end{equation}
since by definition of  the  Fourier transform, one can write
 \beq\label{defkernel}
 {\rm Op}(a) f (w) = \int_{\H^d} k_a(w,w') \: f(w') \:
dw'. \eeq
We shall prove in Chapter~\ref{fundamental} an integral formula giving an expression of the kernel in terms of the function $\sigma (a)$ defined in \aref{def:sigma(a)}: see Proposition~\ref{prop:kernelintegral} page~\pageref{prop:kernelintegral}.

Let us denote by~$   m^{(\lam)}_{ \mu}$ the function
 \begin{equation} \label{tilderem}
   m^{(\lam)}_{ \mu} (\xi,\eta) \eqdefa
m_{ \mu}(\sqrt{|\lambda|}\xi,\sqrt{|\lambda|}\eta) ,
\end{equation}
where~$  m_{ \mu}$ is defined in Remark\refer{m-mua}, page~\pageref{m-mua}.

Then we
note that if $a$ is a symbol of order
$\mu$, then the operators
$$A_\lam
({\rm Id}  + D_\lam)^{-\mu/2}=J_\lam^*op^w(a(w,\lam)\,\#\,
m^{(\lam)}_{-\mu})J_\lam \quad \mbox{and} $$
$$
 ({\rm Id}  + D_\lam)^{-\mu/2}A_\lam
=J_\lam^*op^w(  m^{(\lam)}_{-\mu}\,\#\,a(w,\lam))J_\lam$$ are
uniformly bounded on~${\mathcal H}_\lam$ (see Corollary~\ref{cor:continuiltL2}, page~\pageref{cor:continuiltL2}). More precisely
 we have, for some integer~$n$,
  \begin{equation}\label{alamdlam}
  \|A_\lam ({\rm Id}  + D_\lam)^{-\mu/2}\|_{{\mathcal L}({\mathcal H}_\lam)} +   \|({\rm Id}  + D_\lam)^{-\mu/2}A_\lam \|_{{\mathcal L}({\mathcal H}_\lam)}   \leq C_n \|a\|_{n;S_{\H^d}(\mu)} .
  \end{equation}

\subsection{Statement of the results}
Let us first state a result concerning the action of
pseudodifferential operators on the Schwartz class. This theorem  is proved in Chapter~\ref{fundamental}.

\begin{theo}\label{theo:actiononS}
 {  If $a$ is a symbol in~$ S_{\H^d}(\mu)$ with $\rho=+\infty$, then ${\rm Op}(a)$ maps continuously ${\mathcal S}(\H^d)$ into itself.}
\end{theo}

Notice that Theorem~\ref{theo:actiononS} allows to consider the composition of  pseudodifferential operators, as well as their adjoint operators. The following result  therefore considers the adjoint and
the composition of such operators. It is proved in
Chapter~\ref{algebra}.
\begin{theo}\label{adjcompo}
  {  Consider ${\rm Op}(a)$ and ${\rm Op}(b)$ two pseudodifferential operators on the Heisenberg group of order $\mu$ and $\nu$ respectively.
 \begin{itemize}
 \item If~$\rho >2(2d+1) + |\mu|$, then the operator~${\rm Op}(a)^*$ is a pseudodifferential operator of order~$\mu$ on the Heisenberg group. We denote by~$a^*$ its symbol, which is given by~(\ref{formulaadjoint}).
 \item  If~$\rho >2(2d+1) + |\mu|+ |\nu|$, then the operator  ${\rm Op}(a)\circ{\rm Op}(b)$ is a pseudodifferential operator of order less or equal to $\mu+\nu$. We denote by $a\,\#_{\H^d}\,b$ its symbol.
 \end{itemize}

We have the following asymptotic formulas for $\lam\in \R^{*}$,
 \begin{eqnarray}\label{adjoint}
 \:  \: \:  \: \:\:  \: a^* & = & \overline a+{1\over 2\sqrt{|\lam|}}
 \sum_{1\leq j\leq d}\left( \left\{Z_j\overline a\,,\,\eta_j+i\xi_j\right\} + \left\{\overline Z_j\overline a\,,\,\eta_j-i\xi_j\right\}\right)
 +r_1
  \\
 \label{product}
\qquad  a\,\#_{\H^d}\, b & = & b\,\#\, a \\ \nonumber & & +
 {1\over 2\sqrt{|\lam|}}\sum_{1\leq j\leq d}\left(
 Z_jb\, \# \,\left(\left\{ a\,,\,\eta_j+i\xi_j\right\}\right)
 + \overline Z_j b\, \#\, \left( \left\{ a\,,\,\eta_j-i\xi_j\right\}\right)
 \right)
 +r_2
 \end{eqnarray}
 where $r_1$ (resp. $r_2$)
 depends only on  ${\mathcal Z}^\alpha a$ (resp. ${\mathcal Z}^\alpha b$) for $|\alpha|\geq 2$.}
 \end{theo}
One can find precise formulas for~$a^*$ and~$  a\;   \#_{\H^d} b$ respectively in~(\ref{tildea*}) and~(\ref{formulasigma(d)}).

The first term appearing in the asymptotic formula for~$  a  \#_{\H^d} b$ is not~$ a   \# b$
as could be expected: this is due to the fact that in Definition~(\ref{pseudosformule})
the Fourier transform is composed on the right.

Note that the asymptotic formulas only make sense when the semi~norms~$
\|\cdot\|_{n;S_{\H^d}(\mu)}$ are finite for~$\r
>0$ large enough.
 Let us also emphasize that due to \refeq{product}, the pseudodifferential
 operator~$[ {\rm Op}(a)\,,\,{\rm Op}(b)]$ is of order~$\mu+\nu$. Actually the same phenomenon occurs
 when~${\rm Op}(a)$ and~${\rm Op}(b)$ are differential operators: there is no gain in the order of the commutators.

   It is also important to point out that the asymptotics of
\aref{adjoint} (respectively of \aref{product})  can be pushed to
higher order, as shown in  Section~\ref{asympt2} of Chapter~\ref{algebra}. We will discuss in that section in which sense the formula are asymptotic.  In fact, in the case where~${\rm Op}(a)$ is a differential operator,
 one  obtains a complete description in~\aref{adjoint}
and  in~\aref{product} since the asymptotic series are in fact
finite.

Finally, we point out that even though $a$ is real valued, $a^*$ is generally different  from  $a$.

The final  result of this paper concerns the action of pseudodifferential
operators on Sobolev spaces.
\begin{theo}\label{contHs}
{ Let~$\mu$  be a real number, and~$\rho>2(2d+1)$ be a  noninteger real number. Consider   a symbol~$a$ in~$ S_{\H^d}( \mu
)$ in the sense of Definition~\ref{definsymb}. Then the operator
${\rm Op}(a)$ is bounded from~$H^s(\H^d)$ into $H^{s-\mu}(\H^d)$,
for any real number~$s$ such that~$|s-\mu|< \rho$. More precisely
there exists $n\in\N $ such that
$$
\|{\rm Op}(a)\|_{{\mathcal
L}(H^s(\H^d), H^{s-\mu}(\H ^d))}\leq C_n\| a\|_{n;S_{\H^d}(\mu)}.
$$
If $\rho>0$, then the result holds for $0<s-\mu<\rho$.}
\end{theo}

\begin{rem}
The weaker result for small values of $\rho$ is due to the fact that the adjoint of a pseudodifferential operator is also a pseudodifferential operator is only  known to be true under the assumption that  $\rho$ is large enough. A way of overcoming this difficulty would be to have a quantification,  stable by adjonction (of the type of the Weyl quantization in the Euclidean space). Unfortunately, the non commutativity of the Heisenberg group seems to make such a quantization   difficult to define.
\end{rem}
 Theorem~\ref{contHs} is proved  in Chapter~\ref{classical}. The idea of the proof consists, as in the classical case, in decomposing the symbol into a series of reduced symbols. The new difficulty here compared to the classical case is that an additionnal  microlocalization, in the~$\lam$ direction, is   necessary in order to conclude. This requires significantly more work, as paradifferential-type techniques have to be introduced in order to ensure the convergence of the truncated series (see for instance Proposition~\ref{4.2}, page~\pageref{4.2}).


 \chapter{Fundamental properties of pseudodifferential operators}\label{fundamental}

\setcounter{equation}{0}

The main part of this chapter is devoted to the proof a number of important properties concerning pseudodifferential operators on~$\H^d$ defined in Definition~\ref{definpseudo} page~\pageref{definpseudo}, which will be crucial in the proof of
the main results of this paper.  Before stating those properties, we first present several elementary examples of pseudodifferential operators, and analyze their action on Sobolev spaces.
 Then,
we study the
action of   pseudodifferential operators on the Schwartz space, and prove
Theorem~\ref{theo:actiononS} stated in the introduction.

  \section{Examples of pseudodifferential operators}\label{examples}
 \setcounter{equation}{0}
  Let us give some examples of pseudodifferential operators
and their associate symbols. In this section and more generally in
this article we will make constant use of functional calculus.

\subsection{Multiplication operators}
It is easy to see that if~$b$ is a smooth function on~$\H^{d}$,  then~${\rm Op}(b)$ is the
multiplication  operator   by $b(w)$ and clearly maps $H^s(\H^d)$ into itself provided that there exists~$  \rho >|s|$ and a constant~$C$ such that~$\|b\|_{C^\rho}\leq C$.

\subsection{Generalized multiplication operators}\label{generalizedmultiplication}
Consider  $b(w,\lam)$  a $C^\rho(\H^d)$ real-valued function
depending smoothly on $\lam$ so that for some  $C \geq 0$,
 $$
\sup_\lam\; \| b(\cdot,\lam)\| _{C^\rho(\H^d)}\leq C .
 $$
  If $b$
is rapidly decreasing in~$\lam$ in the sense that $$ \forall k \in
\N,  \quad \sup_{\lam\in\R} \| (1+|\lam|)^k \partial_\lam^k b
(\cdot,\lam)\|_{C^\rho(\H^d)} <\infty, $$ then
 $b$ is a symbol of order $0$ and the operator $op^w(b(w,\lam))$ is the operator of multiplication
  by the constant $b(w,\lam)$, which does not depend on~$(y,\eta)$. Therefore,
$A_\lam(w)=b(w,\lam)$ is a uniformly bounded operator of
${\mathcal H}_\lam$. Moreover, if $f\in L^2(\H^d)$ then
$\{{\mathcal F}(f)(\lam) \circ A_\lam(w)\}_\lam\in{\mathcal A}$
(as defined in Theorem~\ref{plancherelth}), then
$$ \| {\mathcal
F}(f)(\lam) \circ A_\lam(w)\|_{HS({\mathcal H}_\lam)}=|b(w,\lam)|\,\| {\mathcal F}(f)(\lam)\|_{HS({\mathcal H}_\lam)}\leq C
\|{\mathcal F}(f)(\lam)\|_{HS({\mathcal H}_\lam)} $$ which implies
that
 $$
 \| {\rm Op}(b)f\|_{L^2(\H^d)}\leq \,C \, \|
f\|_{L^2(\H^d)}.
$$
Besides, one observes that for all $m\in\N$ and all~$j \in \{1,\dots, d\}$, we have by Lemma~\ref{formuleslem2},
\begin{eqnarray*}
{\mathcal F}\left(Z_j^m\left({\rm Op}(b)f\right)\right)(\lam) &= & {\mathcal F}\left({\rm Op}(b)f\right)(\lam)\circ (Q^\lam_j)^m \\
& =& b(w,\lam)\;{\mathcal
F}\left((-\Delta_{\H^d})^{m/2}f\right)(\lam)\circ
D_\lam^{-m/2}\circ (Q_j^\lam)^m
\end{eqnarray*}
with $ D_\lam^{-m/2}\circ
(Q_j^\lam)^m$ uniformly bounded on ${\mathcal H}_\lam$. A similar fact occurs for $\overline Z_j$. This computation shows
that  Theorem~\ref{contHs} is easily proved for all $s$,
 by interpolation and duality. More precisely, there exists  a constant~$C$ such that
$$
 \|{\rm Op}(b)f\|_{H^s(\H^d)}\leq C  \, \| f\|_{H^s(\H^d)}.
$$

 \subsection{Differentiation operators}
  Let us prove the following result, which provides the symbols of the family of left-invariant vector fields.

 \begin{prop}\label{propformulasdifferentiation}
 {  We have for $1\leq j\leq d$, $\mu \in \R$, $\nu \geq 0$
  $$\displaylines{{1\over i}Z_j={\rm Op}\left(\sqrt{|\lam|} (\eta_j + i \,{\rm sgn } (\lam) \, \xi_j)\right),\;\;{1\over i}\overline Z_j={\rm Op}\left(\sqrt{|\lam|} (\eta_j - i \, {\rm sgn } (\lam) \, \xi_j)\right),\cr
X_j={\rm Op}(2i\,{\rm sgn } (\lam)\,  \sqrt{|\lam|} \eta_j
),\;\;Y_j=-{\rm Op}(2i \, \sqrt{|\lam|}
\xi_j),  \cr
S={\rm Op}(i\lambda), \;\;
  -\Delta_{\H^d}=4 \,{\rm Op}\left(|\lambda| (\eta^2 + \xi^2)\right), \cr
     ({\rm Id} - \Delta_{\H^d})^\frac\mu2 = {\rm Op} (  m_\mu^{(\lam)} ( \xi,  \eta)),  \quad ( - \Delta_{\H^d})^\frac\nu2 = {\rm Op} (  \widetilde m_\nu^{(\lam)} ( \xi,  \eta)).\cr}
 $$
In particular  $Z_j$, $\overline Z_j$, $X_j$ and~$Y_j$ are
pseudodifferential operators of order 1, while~$S$ and~$
\Delta_{\H^d}$ are of order~2 and~$ ({\rm Id} - \Delta_{\H^d})^\mu $ is of
order~$2\mu$.}
\end{prop}
Observe that if $\displaystyle{{1\over i} Z_j={\rm Op}(d_j),\;\;{1\over i}\overline Z_j={\rm Op}(\overline d_j)}$, we have using the map~$\sigma$ defined in~(\ref{def:sigma(a)}) page~\pageref{def:sigma(a)},
$$\sigma(d_j)(\xi,\eta)=\eta_j+i\xi_j\;\;and\;\;\sigma(\overline d_j)(\xi,\eta)=\overline{\sigma(d_j)(\xi,\eta)}=\eta_j-i\xi_j.$$

\begin{proof}
We perform the proof for $Z_j$. For $\lam> 0$, we have
from~(\ref{JlamQJlam}) along with  Lemma~\ref{formuleslem2} stated page~\pageref{formuleslem2},
\begin{eqnarray*}
{\mathcal F}\left({1\over i}Z_jf\right) (\lambda )  & = &{1\over
i} {\mathcal F}(f)  (\lam)\circ Q_j^\lam\\ & = & {\mathcal
F}(f)(\lam)\circ J_\lam ^*\sqrt{|\lam|} \left({1\over
i}\partial_{\xi_j} -{1\over i} \xi_j\right)J_\lam\\
 &= & {\mathcal F}(f)(\lam)\circ J_\lam^*  \,op^w(\sqrt{|\lam|}(\eta_j+i\xi_j))J_\lam.
 \end{eqnarray*}
 On the other hand, for $\lam<0$,
 \begin{eqnarray*}
{\mathcal F}\left({1\over i}Z_jf\right) (\lambda ) & = & {\mathcal
F}(f)(\lam)\circ J_\lam ^* \sqrt{|\lam|} \left({1\over
i}\partial_{\xi_j} +{1\over i} \xi_j\right)J_\lam\\
 & = &{\mathcal F}(f)(\lam)\circ J_\lam^*  op^w(\sqrt{|\lam|}(\eta_j-i\xi_j))J_\lam.
 \end{eqnarray*}
The other cases are treated similarly, except for the operators~$({\rm Id} -
\Delta_{\H^d})^\mu $ and~$(-
\Delta_{\H^d})^\nu$, for which we refer to  Remark~\ref{m-mua}, page~\pageref{m-mua}. This concludes the proof of Proposition~\ref{propformulasdifferentiation}.
 \end{proof}

\subsection{Fourier multipliers}\label{fouriermultipliers}
A Fourier multiplier is an operator $K$  acting on  ${\mathcal
S}(\H^d)$ such that $${\mathcal F}(Kf)(\lam)={\mathcal
F}(f)(\lam)\circ U_K(\lam)$$ for some operator $U_K(\lam)$ on
${\mathcal H}_\lam$.

For instance, the differentiation operators
$Z_j$ and $\overline Z_j$ are Fourier multipliers, and~$U_K(\lam)$ is respectively equal to~$Q_j^\lam$ and~$\overline Q_j^\lam$ as given in formulas~(\ref{defQj}) and~(\ref{defoverQj}) page~\pageref{defoverQj}.
Similarly
the Laplacian $-\Delta_{\H^d}$ is  a Fourier multiplier, with~$U_K(\lam) = D_\lam$ according to~(\ref{linkdeltadlamxip}).

An interesting class of Fourier multipliers consist in the
operators
 obtained from the Laplacian by means of  functional calculus: for $\Psi$ bounded and smooth,  the operator $\Psi(-\Delta_{\H^d})$ is a bounded operator on $H^s(\H^d)$ for all $s\in\R$, and
$$
\forall f\in L^2(\H^d),\;\;{\mathcal
F}(\Psi(-\Delta_{\H^d})f)(\lam)={\mathcal F}(f)(\lam) \circ
\Psi(D_\lam).
$$   Such operators commute with one another,
and  so do the operators $\Psi(D_\lam)$ for different functions
$\Psi$. The Littlewood-Paley truncation operators that we will
introduce later (see Chapter~\ref{prelimins}) are of this type, and we will see that they are pseudodifferential operators (see Proposition~\ref{symbDeltap} stated page~\pageref{symbDeltap}).
Observe too that  if $\Psi\in{\mathcal C}_0^\infty(\R)$, then the
operator $\Psi(-\Delta_{\H^d})$ is a smoothing operator which maps
$H^s(\H^d)$ into $H^\infty(\H^d)$ for all $s\in\R$.

Another class of Fourier multipliers which are also
pseudodifferential operators, is built with functions~$b$ in~$S
(m^\mu,g)$ with~$\mu \geq 0$ in the following way.

 \begin{prop}  \label{prop:fourmult}
 {  If $a(w,\lam,\xi,\eta) = b\left({\rm sgn}(\lam)\sqrt{|\lam|}\xi,\sqrt{|\lam|} \eta\right)$ with $b\in S(m^\mu,g)$ and~$\mu \geq 0$, then
 $a$  belongs to~$S_{\H^d}( \mu )$, and the operator ${\rm Op}(a)$ is a Fourier multiplier. Moreover,
\beq\label{estimatesmu} \forall u\in H^s(\H^d),\;\;\|{ \rm Op }
(a)u\|_{H^{s-\mu}(\H^d)} \leq C\|b\|_{n;S
(m^\mu,g)}\|u\|_{H^s(\H^d)}. \eeq
Finally $\sigma (a) = b$ as given in Definition~\ref{definsymb}.}
 \end{prop}

 \begin{proof} The fact that $a$  belongs to~$S_{\H^d}( \mu )$ and that the operator ${\rm Op}(a)$ is a Fourier multiplier are straightforward. Now let us prove~(\ref{estimatesmu}).
We have $$ { \rm Op } (a) u(w)= \frac{ 2^{d-1}}{\pi^{d+1}}
\int_{\R} {\rm tr} \left(
 u^\lam_{w^{-1}} {\mathcal F}(u)(\lam) A_\lam  \right) \: |\lam|^d \: d\lam,
 $$
 with~$  A_\lam = J_\lam^* \,op^w ( a)\,J_\lam.
$

In view of the Plancherel formula~\refeq{Plancherelformula} recalled page~\pageref{Plancherelformula}, to
estimate the~$H^{s-\mu}$-norm of~${ \rm Op } ( a) u$, we     evaluate
the Hilbert-Schmidt norm of~${\mathcal
F}\left(({\rm Id} -\Delta_{ \H^d  })^{\frac {s-\mu} 2 }{ \rm Op } (a) u
\right)(\lam).$ We have
\begin{eqnarray*}
 {\mathcal F}\left(({\rm Id} -\Delta_{ \H^d  })^{\frac {s-\mu}  2 }{ \rm Op } (a)
u \right)(\lam) & =& {\mathcal F}(u) (\lam)A_\lam ({\rm Id} +D_\lam)^{s-\mu\over 2} \\
&=& {\mathcal F}\left(({\rm Id} -\Delta_{ \H^d  })^{\frac s 2 }u\right)
(\lam)( {\rm Id} +D_\lam)^{-{s\over 2}}A_\lam ( {\rm Id} +D_\lam)^{s-\mu\over 2}.
\end{eqnarray*}
 In
light of \refeq{alamdlam} page~\pageref{alamdlam}, the operators $( {\rm Id} +D_\lam)^{-{s\over
2}}A_\lam ( {\rm Id} +D_\lam)^{s-\mu\over 2}$ are uniformly bounded
on~${\mathcal L}({\mathcal H}_\lam)$ by~$C \|b\|_{n;S (m^\mu,g)}$
which ends the proof of the estimate thanks to
property~\refeq{hsproperty}, recalled page~\pageref{hsproperty}. This ends the proof of Proposition~\ref{prop:fourmult}.
\end{proof}

 More generally, a pseudodifferential operator which is a Fourier multiplier has a symbol which does not depend on $w$. For this reason, Theorem~\ref{adjcompo} is easy to prove in that case.

\begin{prop}
Consider $a$ and $b$ two symbols of $S_{\H^d}(\mu)$ which do not
depend on the variable $w$. Then ${\rm Op}(a)^*={\rm Op}(\overline
a)$ and ${\rm Op}(a)\circ{\rm Op}(b)={\rm Op}(b\#a).$
\end{prop}

\begin{proof} By the Plancherel formula,
$$({\rm Op}(a)f,g)= \frac{ 2^{d-1}}{\pi^{d+1}}  \int_{\R} {\rm
tr}\left(\left( {\mathcal F}(g) (\lam)\right)^* {\mathcal F}(f)
(\lam) A_\lam\right)|\lam|^d\,d\lam$$ with $A_\lam =J_\lam^*
op^w(a(\lam))J_\lam$. Therefore, $${\mathcal F}\left({\rm Op}(a)^*
g\right)(\lam)=  {\mathcal F}(g)(\lam) A_\lam^*.$$ The fact that
$A_\lam^*= J_\lam^* op^w(\overline a(\lam))J_\lam$ gives the first
part of the proposition.

 Let us now consider ${\rm Op}(a)\circ
{\rm  Op}(b)$. We have $${\mathcal F}({\rm Op}(a)\circ {\rm Op}(b)
f)(\lam)={\mathcal F}(f)(\lam)\circ B_\lam\circ A_\lam$$ with
$B_\lam=J_\lam^* op^w(b(\lam))J_\lam$. The fact that $op^w(b)\circ
op^w(a)=op^w(b\# a)$ finishes the proof.
\end{proof}


  \section[The kernel and the symbol of a pseudodifferential operator]{The link between the kernel and the symbol of a pseudodifferential operator}\label{kernel}
 \setcounter{equation}{0}

The kernel of a  pseudodifferential operator  on the Heisenberg
group is given by~\aref{eq:ka} page~\pageref{eq:ka}.
 The following proposition provides an integral formula for the kernel of a pseudodifferential operator, as well as a formula enabling one to recover the symbol of an operator, from its kernel.

\begin{prop}\label{prop:kernelintegral}
The
kernel of the pseudodifferential operator ${\rm Op}(a)$ is given by
 $$
 k(w,w')={1\over 2\pi^{2d+1}}\int {\rm e}^{2i\lam(x\cdot y'-y\cdot x')}
\sigma(a)(w,\lam, \xi,\zeta){\rm
e}^{i\lam(s'-s)+2iz\cdot(y'-y)-2i\zeta\cdot(x'-x)}d\lam\,d\xi\,d\zeta,
$$
where~$\sigma(a)$ is defined in \aref{def:sigma(a)}, page~\pageref{def:sigma(a)}.

Conversely, one recovers the symbol~$a$ through the formula
\begin{equation}\label{eq:sigma(b)}
\sigma(a)(w,\lam,\xi,\eta)= \int_{\H^d} {\rm e}^{2i\,(y'\cdot \xi-x'\cdot\eta)} {\rm e}^{i\lam s'}k(w,w(w')^{-1}) \: dw'.
\end{equation}
\end{prop}
Before proving the proposition, we notice that it allows to obtain directly the symbol of a pseudodifferential operator if one knows its kernel: the following corollary is obtained simply by using Proposition~\ref{prop:kernelintegral} and  Relation~\aref{def:sigma(a)} between~$a$ and~$\sigma (a)$.

 \begin{cor}\label{prop:symbol}
Let $Q $  be an operator on $\H^d$ of kernel $k(w,w')$ such that
for some~$\mu\in\R$, the function defined for $(w, \xi,\eta)\in\H^d
\times \R^{2d}$ by
 \begin{eqnarray}\label{symboldeWeyl1}
a(w,\lam,\xi,\eta) &\eqdefa  &  \int_{\H^d} {\rm e}^{2i\sqrt{|\lam|}\,({\rm sgn}(\lam)y'\cdot \xi-x'\cdot\eta)} {\rm e}^{i\lam s'}k(w,w(w')^{-1}) \: dw'
  \end{eqnarray}
belongs to $S_{\H^d}(\mu)$. Then $Q=Op(a)$.
  \end{cor}

  {\it Proof of Proposition~\ref{prop:kernelintegral}. } ---  Let us start by recalling~\aref{eq:ka}, which states that
$$
k(w,w') =  {2^{d-1}\over \pi^{d+1}} \int{\rm tr} \left(
u^\lam_{w^{-1} w'}J_\lam^* op^w\left( a(w,\lam)\right)
J_\lam\right) |\lam|^d d\lam.
$$
Note that everywhere in the proof, integrals are to be understood as oscillatory integrals.
The Bargmann representation~$u^\lam_w$ and the Schr\"odinger representation~$v^\lam_w$ are linked by the intertwining formula~$u^\lam_w = K_\lam^* v^\lam_w K_\lam$, so using the operator~$T_\lam = J_\lam K_\lam^*$ we have
$$
k(w,w') =   {2^{d-1}\over \pi^{d+1}}
\int{\rm tr} \left( v^\lam_{w^{-1} w'}T_\lam^* op^w\left(
a(w,\lam)\right) T_\lam\right) |\lam|^d d\lam.
$$
By rescaling it is easy to see that
\begin{equation}\label{Tlamopw}
T_\lam^* op^w\left(
a(w,\lam)\right) T_\lam = op^w\left( a\left(w,\lam,\sqrt{|\lam|}\:
\cdot,{\cdot\over\sqrt{|\lam|}}\right)\right) ,
\end{equation}
so we get
 \begin{equation}\label{eq:233'}
 k(w,w') = {2^{d-1}\over \pi^{d+1}} \int{\rm tr} \left( v^\lam_{w^{-1} w'}
op^w\left( a\left(w,\lam,\sqrt{|\lam|} \:
\cdot,{\cdot\over\sqrt{|\lam|}}\right)\right) \right) |\lam|^d d\lam.
\end{equation}
In order to compute the trace of the operator $\displaystyle v^\lam_{w^{-1} w'} op^w\left(
a\left(w,\lam,\sqrt{|\lam|}\:
\cdot,{\cdot\over\sqrt{|\lam|}}\right)\right)$, we shall  start by finding its kernel~$ \theta(\xi,\xi')$, and then use the formula~(\ref{traceopanoyau}) page~\pageref{traceopanoyau}, giving the trace of an operator in terms of its kernel.

So let us first compute~$ \theta(\xi,\xi')$, which we recall is defined by
$$
 v^\lam_{w^{-1} w'} op^w\left( a\left(w,\lam,\sqrt{|\lam|}\:
\cdot,{\cdot\over\sqrt{|\lam|}}\right)\right ) f(\xi)= \int \theta(\xi,\xi') f(\xi') d \xi'.
$$
We also recall that
$$
op^w\left(
a\left(w,\lam,\sqrt{|\lam|}\:
\cdot,{\cdot\over\sqrt{|\lam|}}\right)\right) f(\xi)= \int A(\xi,\xi') f(\xi') \: d\xi',
$$
where as stated in~(\ref{defkernelweyl}) page~\pageref{defkernelweyl},
$$
A(\xi,\xi') = (2\pi)^{-d} \int {\rm e}^{i(\xi-\xi') \cdot \Xi}  a \left(w,\lam,\sqrt {|\lam|} \left(\frac{\xi+\xi'}2\right), \frac \Xi{ \sqrt{|\lam|} } \right) \: d \Xi.
$$
Finally using Formula~\aref{schrorep} page~\pageref{schrorep} defining~$v^\lam_{w^{-1} w'}$,
we get
 $$
 \theta(\xi,\xi')=(2\pi)^{-d}{\rm e}
^{i\lam \left( \tilde s -2 \tilde x\cdot \tilde y+2\tilde y \cdot
\xi\right)}\int a\left( w,\lam,\sqrt{|\lam|} \left(\xi-2\tilde
x+\xi'\over 2\right),{\Xi\over \sqrt{|\lam|}}\right) {\rm
e}^{i\Xi\cdot(\xi-2\tilde x-\xi')}d\Xi,
$$
where~$\tilde w \eqdefa w^{-1} w'$. Using the relation~(\ref{operateuranoyau}) given page~\pageref{operateuranoyau} between the trace and the kernel of an operator and \aref{eq:233'} above, we infer that
\begin{eqnarray*}
k(w,w') & = & {1\over 2  \pi ^{2d+1}} \int {\rm e} ^{i\lam \left(
\tilde s -2 \tilde x\cdot \tilde y+2\tilde y \cdot
\xi\right)-2i\Xi\cdot\tilde x} a\left(
w,\lam,\sqrt{|\lam|}(\xi-\tilde x),{\Xi\over \sqrt{|\lam|}}\right)
|\lam|^d d\lam\,d\Xi\,d\xi\\
& = &  {1\over 2 \pi^{2d+1}} \int
{\rm e} ^{i\lam \tilde s +2i\tilde y \cdot z-2i\zeta\cdot\tilde x}
a\left( w,\lam,{z\over \sqrt{|\lam|}} {\rm sgn}(\lam),{\zeta\over
\sqrt{|\lam|}}\right)d\lam\,dz\,d\zeta
\end{eqnarray*}
where we have performed the change of variables  $\displaystyle \xi-\tilde x
={z\over |\lam|}{\rm sgn}(\lam)$, and $\Xi=\zeta$.

To end the proof of the proposition, one just needs to notice that
 $$k(w,w(w')^{-1})={1\over 2 \pi^{2d+1}}\int {\rm e}^{-i\lam s'-2iy'\cdot z+2ix'\cdot\zeta}
\sigma(a) (w,\lam,z,\zeta)dz\,d\zeta\,d\lam$$
and to apply an inverse Fourier transform (in the Euclidean space).
\qed

\section{Action on the Schwartz class}
 \setcounter{equation}{0}
 The aim of this section is to prove Theorem~\ref{theo:actiononS}, stating that if $a$ belongs to~$S_{\H^d}(\mu)$ and $\rho=+\infty$, then~${\rm Op}(a)$
 maps continuously~${\mathcal S}(\H^d)$ into~${\mathcal S}(\H^d)$.

 Before entering the proof of that result, let us point out that the smoothness condition~(\ref{def:sigma(a)}) (see page~\pageref{def:sigma(a)}) is necessary in order for~$ {\rm Op} (a) $  to act on~${\mathcal S}(\H^d)$. A counterexample is provided in the  proof of the next statement. Actually one can
define~$ {\rm Op} (a)$  without that condition, and typically the counterexample provided below provides an operator which is  continuous on all Sobolev spaces.
   \begin{prop}\label{contrexempleschwartz}
Let~$\mu$ be an odd integer. There is a function~$a$ such that~$\|a\|_{n;S_{\H^d}(\mu)}$ is finite for all integers~$n$, and such that the
operator~$ {\rm Op} (a) $ is not continuous over~${\mathcal S}(\H^d)$.
 \end{prop}
\begin{proof}
Let us define~$\mu = 2k+1$ and     the function~$ a(w,\lam,\xi,\eta) = A(\lam)$, where~$A(\lam) = |\lam|^{k+\frac12}$.
 Let~$f$ be defined by
$$ {\mathcal F}(f)(\lam) F_{0,\lam} = \phi (\lam) F_{0,\lam},
\quad {\mathcal F}(f)(\lam) F_{\alpha,\lam} = 0 \: \quad \forall
\alpha \neq 0, $$ where~$\phi$ is a nonnegative, smooth, compactly
supported function such that~$\phi (0) = 1$. An easy computation
 shows that~$f \in {\mathcal S} (\H^d)$. Indeed writing
 $$
 f(w) = \frac{2^{d-1}}{\pi^{d+1}} \int {\rm tr} \: \left(u^\lam_{w^{-1}}{\mathcal F}(f)(\lam) \right) \: |\lam|^d \: d\lam
 $$
 and using the definition of the Fourier transform of~$f$ given above, a simple computation shows that for some constant~$C$,
 $$
 f(w) = C \int e^{-i\lam s} \phi (\lam)  {\rm e}^{-|\lam| |z|^2}\left( \int {\rm e}^{-2|\lam| |\xi|^2} \: d\xi\right) |\lam|^{2d} \: d\lam
 $$
 which gives the result since~$\phi$ is smooth and compactly supported. Now let us consider~$  {\rm Op} (a) f$. A similar computation shows that  if~$N$ is any integer, then for some fixed constants~$C'$ and~$C''$ one has
 \begin{eqnarray*}
 s^N  {\rm Op} (a) f (w)& = & C' \int s^N e^{-i\lam s}  \phi(\lam) A(\lam) {\rm e}^{-|\lam| |z|^2} |\lam|^d \: d\lam \\
 & = & C'' \int e^{-i\lam s} \partial_\lam^N (\phi (\lam) {|\lam|} ^{d} A(\lam) {\rm e}^{-|\lam| |z|^2}) \: d\lam.
 \end{eqnarray*}
 For any fixed~$z$, this is the (real) Fourier
 transform at the point~$s$ of
  the function
  $$\lam \mapsto \partial_\lam^N (\phi (\lam) {|\lam|} ^{d} A(\lam) {\rm e}^{-|\lam| |z|^2}) .
  $$
  Let us evaluate this integral at the point~$z=0$. Taking~$N$ large enough, the result     is clearly not bounded in~$s$.
\end{proof}

{\it Proof of
Theorem~\ref{theo:actiononS}.} ---
Consider $f\in{\mathcal S}(\H^d)$, and let us start by proving
that~$ {\rm Op}(a) f $ belongs to~$L^\infty(\H^d).$ By definition
of~${\rm Op}(a)$,  we need to find a constant $C_0$ such  that for
all $w\in\H^d$,
\begin{equation}\label{est}
\left|  \int
 {\rm tr}\left(u^\lam_{w^{-1}}
 {\mathcal
F}(f)(\lam ) A_\lam(w) \right)
|\lam |^{d} d\lam \right|< C_0.
\end{equation}
Consider $\chi$  a frequency  cut-off  function  defined
by $\chi(r)=1$ for $|r|\leq 1$ and~$\chi(r)=0$ for~$|r|>2$.
We write
$$\int {\rm tr}\left(u^\lam_{w^{-1}}
 {\mathcal
F}(f)(\lam ) A_\lam(w) \right)
|\lam |^{d} d\lam= I_1+I_2$$
where
$$\displaystyle{I_1 \eqdefa \int {\rm tr}\left(u^\lam_{w^{-1}}
 {\mathcal
F}(f)(\lam )\chi(D_\lam) A_\lam(w) \right)
|\lam |^{d} d\lam}$$
and we deal separately with each part.

Let us first observe that for any $k\in\N$ and by Remark~\ref{estimatesHStretc} stated page~\pageref{estimatesHStretc}, we have
\begin{eqnarray}\label{garps}
|I_1| &\leq&
\left(\int \|u^\lam_{w^{-1}}{\mathcal F}(f)(\lam) ({\rm Id}  + D_\lam)^k\|_{HS({\mathcal H}_\lam)}^2|\lam|^dd\lam\right)^{1\over 2} \nonumber\\
& & \quad  \quad  \quad  \quad  \quad \times  \left(\int\|({\rm Id}  + D_\lam)^{-k}\chi(D_\lam)A_\lam(w) \|_{HS({\mathcal H}_\lam)}^2|\lam|^dd\lam\right)^{1\over 2}.
\end{eqnarray}
Besides,  using \aref{hsproperty} page~\pageref{hsproperty}, there exists a constant $C$ such that
$$
\|u^\lam_{w^{-1}}{\mathcal F}(f)(\lam) ({\rm Id}  + D_\lam)^k\|_{HS({\mathcal H}_\lam)} \leq C\, \|{\mathcal F}(f)(\lam) ({\rm Id}  + D_\lam)^k\|_{HS({\mathcal H}_\lam)},
$$
and
$$
\longformule{
\! \! \! \! \! \!\! \! \! \! \! \!  \|({\rm Id}  + D_\lam)^{-k}\chi(D_\lam)A_\lam(w)\|_{HS({\mathcal H}_\lam)}  \leq
\|({\rm Id}  + D_\lam)^{-{\mu\over 2}}A_\lam(w)\|_{{\mathcal L}({\mathcal H}_\lam)}\| ({\rm Id}  + D_\lam)^{{\mu\over 2}-k}\chi(D_\lam)\|_{HS({\mathcal H}_\lam)}}{
\! \! \! \! \! \!\! \! \! \! \! \! \leq  C\,\| ({\rm Id}  + D_\lam)^{{\mu\over 2}-k}\chi(D_\lam)\|_{HS({\mathcal H}_\lam)}}
$$where we have used  \aref{alamdlam} (see page~\pageref{alamdlam}) for the last bound.
We then observe that on the one hand
$$\displaystyle{
 {\mathcal F}(f)(\lam) ({\rm Id}  + D_\lam)^k={\mathcal F}(({\rm Id} -\Delta_{\H^d} )^kf)(\lam)}
 $$
  so that by the Plancherel formula
 $$
\frac{2^{d-1}}{\pi^{d+1}} \int \| {\mathcal F}(f)(\lam) ({\rm Id}  + D_\lam)^k\|_{HS({\mathcal H}_\lam)}^2|\lam|^d d\lam=\| ({\rm Id} -\Delta_{\H^d})^k f\|_{L^2(\H^d)}^2.$$
On the other hand
 \begin{eqnarray*}
 \int\| ({\rm Id}  + D_\lam)^{{\mu\over 2}-k}\chi(D_\lam)\|^2_{HS({\mathcal H}_\lam)} |\lam|^d d\lam
 & = & \sum_{\alpha\in\N^d}
 \int \| ({\rm Id}  + D_\lam)^{{\mu\over 2}-k}\chi(D_\lam) F_{\alpha,\lam}\|_{{\mathcal H}_\lam}^2 |\lam|^dd\lam \\
 & = & \sum_{\alpha\in\N^d} \int  \left(1+|\lam|(2|\alpha|+d)\right)^{{\mu\over 2}-k}\chi(|\lam|(2|\alpha|+d))|\lam|^dd\lam,
 \end{eqnarray*}
hence
$$
\longformule{
 \int\| ({\rm Id}  + D_\lam)^{{\mu\over 2}-k}\chi(D_\lam)\|^2_{HS} |\lam|^d d\lam}{\leq      C  \sum_{m\in  \N }   (2m+d)^{d-1}   \! \! \!\int   \! (1+|\lam|(2m+d))^{{\mu\over 2}-k}\chi(|\lam|(2m+d))|\lam|^d d\lam
}
$$
where we have used that the number of $\alpha\in\N^d$ such that $|\alpha|=m$ is controlled by $m^{d-1}$.
Then, the change of variables $\beta=(2m+d)\lambda$ gives
\begin{eqnarray*}
 \int\| ({\rm Id}  + D_\lam)^{\mu-k}\chi(D_\lam)\|^2_{HS} |\lam|^d d\lam & \leq &
 C\,\left( \sum_{m\in\N} {1\over 1+m^2}\right) \int \chi(|\beta|)(1+ |\beta|)^{{\mu\over 2}-k+d}d\beta.
 \end{eqnarray*}
Therefore, \aref{garps} becomes
$$|I_1|\leq C \,\| ({\rm Id} -\Delta_{\H^d})^k f\|_{L^2(\H^d)}\,\left( \sum_{m\in\N} {1\over  1+m^2}\right) \int \chi(|\beta|)(1+ |\beta|)^{{\mu\over 2}-k+d}d\beta\leq C_0$$
for any~$k$.

A similar  argument applies to $I_2$ and allows to get
$$ |I_2|\leq  C
\norm { ({\rm Id} -\Delta_{\H^d})^{k} f } {L^2(\H^d)}\,\left( \sum_{m\in\N} {1\over  1+m^2}\right) \int \wt{\chi}(|\beta|)(1+ |\beta|)^{{\mu\over 2}-k+d}d\beta$$
where $\wt{\chi}$ is a frequency  cut-off  function  defined by $\wt{\chi}(r)=1$ for $ \displaystyle |r|\geq \frac 3 4 \virgp $ and $\wt\chi(r)=0$ for~$\displaystyle |r|< \frac 1 2 \cdotp$  The choice~$k > 1+ d  +\frac {\mu}{2}$ achieves the estimate of the term~$I_2$.

The end of the proof of Theorem~\ref{theo:actiononS} is a direct consequence of the following
lemma. We
will emphasize later other formulas of that type which will be useful in
the following sections.
\qed

\begin{lemme} \label{usefullem} {  For any symbol~$a\in S_{\H^d}(\mu)$ and $j\in\{1,\dots, d\}$, there are symbols~$b_j^{(1)}, b_j^{(2)}$ belonging to~$S_{\H^d}(\mu+1)$ and~$c_j^{(1)}, c_j^{(2)} \in S_{\H^d}(\mu-1)$ and~$p \in  S_{\H^d}(\mu)$  such that
$$ [ Z_j\,,\,{\rm Op}(a)]={\rm Op} (b_j^{(1)}), \: \:  [\overline
Z_j\,,\,{\rm Op}(a)]={\rm Op} (b_j^{(2)}), $$ $$ [z_j\,,\,{\rm
Op}(a)] ={\rm Op} (c_j^{(1)}), \: \: [\overline z_j\,,\,{\rm
Op}(a)]={\rm Op} (c_j^{(2)}),$$
$$[ is,{\rm Op}(a)]={\rm Op}(p).$$
In particular, one has
$$\displaylines{ b_j^{(1)} = Z_ja+ \sqrt{|\lam|}\{a,\eta_j+i\,{\rm sgn}(\lam)\xi_j\}
\;\;{ and}\;\;b_j^{(2)} =\overline Z_ja
+\sqrt{|\lam|}\{a,\eta_j-i\,{\rm sgn}(\lam)\xi_j\} ,\cr c_j^{(1)}
=\frac{1}{2\sqrt{|\lam|}} \{a,i\xi_j-\,{\rm sgn}(\lam)\eta_j\} \;\;{
and}\;\;c_j^{(2)} =
\frac{1}{2\sqrt{|\lam|}}\{a,i\xi_j+\,{\rm sgn}(\lam)\eta_j\}.\cr}$$}
\end{lemme}

\begin{rem}
Notice that contrary to the classical case (see~\cite{ag} for
instance), $[ Z_j\,,\,{\rm Op}(a)]$ is an operator of
order~$\mu+1$ instead of~$\mu$, due to the additionnal Poisson
bracket appearing in the definition of~$b_j^{(1)}$ (and the same
goes for~$[ \overline Z_j\,,\,{\rm Op}(a)]$).

On the other hand,
$[z_j\,,\,{\rm Op}(a)] $ and $[\overline z_j\,,\,{\rm Op}(a)]$ are of order~$\mu-1$
 as in the classical setting, but~$[s,{\rm Op}(a)]$  is only of order~$\mu$.
\end{rem}

Let us now prove Lemma~\ref{usefullem}.

\begin{proof}
Let us consider a function~$f$ in~${\mathcal S}(\H^d)$,  and a symbol~$a$ belonging to~$S_{\H^d}(\mu)$. We have for~$1\leq j\leq d$,
$$
Z_j {\rm Op}(a)f(w)={2^{d-1}\over\pi^{d+1}}\int {\rm tr} \left(
Z_j(u^\lam_{w^{-1}}){\mathcal F}(f)(\lam) A_\lam(w)+
u^\lam_{w^{-1}}{\mathcal F}(f)(\lam)
Z_jA_\lam(w)\right)|\lam|^d\,d\lam
$$
 with $Z_jA_\lam(w)=J_\lam^*op^w(Z_ja(w,\lam))J_\lam$.

 Thanks to Lemma \refer{formuleslem2} page~\pageref{formuleslem2}, we have $Z_ju^\lam_{w^{-1}}=Q_j^\lam u^\lam_{w^{-1}}$, recalling
 that~$Q_j^\lam$
is defined in~(\ref{defQj}) page~\pageref{defQj}. Therefore, since~${\mathcal
F}(Z_jf)(\lam)={\mathcal F}(f)(\lam)Q_j^\lam$, and using the fact
that~${\rm tr} (AB) = {\rm tr} (BA) $, we obtain
 $$
 [Z_j\;,\;{\rm Op}(a)]f(w)={2^{d-1}\over\pi^{d+1}} \int{\rm tr}\left( u^\lam_{w^{-1}}{\mathcal F}(f)(\lam)\left(\left[A_\lam(w),Q^\lam_j\right] +Z_jA_\lam(w)\right)\right)|\lam|^d\,d\lam.
 $$
 We then use~(\ref{JlamQJlam}) page~\pageref{JlamQJlam} to find,  for $\lam>0$,
 $$
  \left[A_\lam(w),Q^\lam_j\right] =J^*_\lam\left[ op^w(a(w,\lam))\;,\;\sqrt{|\lambda|}(\partial_{\xi_j}-\xi_j)\right]J_\lam$$
and for $\lam<0$,
$$
\left[A_\lam(w),Q^\lam_j\right]
=J^*_\lam\left[
op^w(a(w,\lam))\;,\;\sqrt{|\lambda|}(\partial_{\xi_j}+\xi_j)\right]J_\lam.
$$
Therefore, by standard symbolic calculus, using in particular the fact that if~$b$ is a polynomial of degree one in~$(\xi,\eta)$, then
\begin{equation}\label{commutpoisson}
[op^w( a), op^w( b)] = \frac1iop^w( \{a,b\}),
\end{equation}
we get
 $$
 \displaylines{ \left(\left[A_\lam(w),Q^\lam_j\right]
+Z_jA_\lam(w)\right)=J_\lam^*op^w\left({\sqrt{\lam}}\{a(w,\lam),\eta_j+i\xi_j\} + Z_ja(w,\lam)\right)\;\;{\rm
for}\;\;\lam>0,\cr \left(\left[A_\lam(w),Q^\lam_j\right]
+Z_jA_\lam(w)\right)=J_\lam^*op^w\left(\sqrt{-\lam}\{a(w,\lam),\eta_j-i\xi_j\} +Z_ja(w,\lam)\right)\;\;{\rm
for}\;\;\lam<0,\cr}
$$
which are the expected formula. We moreover
observe that  if~$a\in S_{\H^d}(\mu)$ and~$1\leq j\leq d$,
then~$\sqrt{|\lam|}\partial_{\xi_j} a$
and~$\sqrt{|\lam|}\partial_{\eta_j} a$ are symbols of
order~$\mu+1$. Indeed since~$a$ is of order~$\mu$,  there exists a constant $C$ such that, for
$k\in\N$ and $\beta\in\N^{2d}$,
\begin{eqnarray*}
\left|(\lam\partial_\lam)^k \partial_{(\xi,\eta)}^\beta
(\sqrt{|\lam|}\partial_{\xi_j}a) \right|& \leq &  C\,
\sqrt{|\lam|}^{2+|\beta|}
\left(1+|\lam|(1+|\xi|^2+|\eta|^2) \right)^{\mu-|\beta|-1\over 2}\\ & \leq & C\,
\sqrt{|\lam|}^{|\beta| }
\left(1+|\lam|(1+|\xi|^2+|\eta|^2) \right)^{\mu+1-|\beta|\over 2}.
\end{eqnarray*}

A similar computation gives the result for $[\overline Z_j,{\rm
Op}(a)]$.

$ $

Let us now consider the other types of commutators. For
$f\in{\mathcal S}(\H^d)$ and $1\leq j\leq d$, we have
$$
[z_j,{\rm
Op} (a)]f(w)={2^{d-1}\over \pi^{d+1}}\int(z_j-z_j'){\rm tr}
\left(u^\lam_{w^{-1}w'} A_\lam(w)\right)f(w')\,|\lam|^dd\lam d
w'.
$$
By   Lemma\refer{formuleslem} page~\pageref{formuleslem}, we have
 $\displaystyle z_j u^\lam_w={1\over 2\lam}[\overline Q_j^\lam,u^\lam_w]. $
Therefore, setting
 $\tilde w=w^{-1}w'=(\tilde z,\tilde s)$, we get, using~(\ref{JlamQJlam}) page~\pageref{JlamQJlam} along with the fact that~$A_\lam (w)= J_\lam^* op^w (a(w,\lam)) J_\lam$,
\begin{eqnarray*}
 {\rm tr}\left(\tilde z_j u^\lam_{\tilde w} A_\lam(w)\right) &=&  {\sqrt {| \lam|}\over 2\lam}
 {\rm tr}\left( [J_\lam^*(\partial_{\xi_j}+\,{\rm sgn}(\lam)\xi_j)J_\lam,  u^\lam_{\tilde w}
 ]A_\lam(w)\right) \\
& = &  {{\rm sgn}(\lam)\over 2\sqrt { |\lam|}}
{\rm tr}\left(
J_\lam^*[\partial_{\xi_j}+\,{\rm sgn}(\lam)\xi_j,J_\lam u^\lam_{\tilde w}
J_\lam^*]op^w(a(w,\lam))J_\lam\right)\\
&  = &  {1\over 2\sqrt {
|\lam|}} {\rm tr}\left([op^w(a(w,\lam)),\,{\rm sgn}(\lam)\partial_{\xi_j}+\xi_j] J_\lam
u^\lam_{\tilde w} J_\lam^*\right).
\end{eqnarray*}
 By standard symbolic calculus, this implies that
 \begin{equation}\label{Z}
  {\rm tr}\left(\tilde z_j u^\lam_{\tilde w} A_\lam(w)\right)= {1\over 2\sqrt{ |\lam| }} {\rm tr} \left(u^\lam_{\tilde w} J^*_\lam op^w(\{a,\,{\rm sgn}(\lam)\eta_j-i\xi_j\})J_\lam\right)
 \end{equation}
 which gives the announced formula. Besides, the same argument as before gives  that if~$a$ is a symbol in~$ S_{\H^d}(\mu)$ and if~$1\leq j\leq d$, then~$\displaystyle  {1\over\sqrt{|\lam|}}\partial_{\xi_j}a$ and $\displaystyle{1\over\sqrt{|\lam|}}\partial_{\eta_j}a$ are symbols of $S_{\H^d}(\mu-1)$. Indeed, for   $k\in\N$ and $\beta\in\N^{2d}$
 \begin{eqnarray*}
\left| (\lam\partial_\lam)^k \partial_{(\xi,\eta)}^\beta
\left(\frac1{\sqrt{|\lam|}}\partial_{\xi_j}a\right)\right| & \leq & C\,
|\lam|^{|\beta|}  \left( 1+ |\lam|( 1+
|\xi|^2+|\eta|^2)\right)^{\frac{\mu-|\beta|-1}{2}}
 \end{eqnarray*}
 A similar argument gives the result for the  multiplication by $\overline{z_j}$. In particular, one finds for all~$\lam \in \R^*$,
 \begin{equation}
 \label{overlineZ}
 {\rm tr}\left(\overline{\tilde z_j} u^\lam_{\tilde w} A_\lam(w)\right)= -
 {1\over 2\sqrt{|\lam|}} {\rm tr} \left(u^\lam_{\tilde w} J^*_\lam op^w(\{a,{\rm sgn} (\lam) \eta_j+ i\xi_j\})J_\lam\right).
 \end{equation}

 Finally, let us consider the last commutator. We have
 \begin{eqnarray*}
 [is,{\rm Op}(a) ]f(w) & = & {2^{d-1}\over \pi^{d+1}}\int i(s-s'){\rm tr}
\left(u^\lam_{w^{-1}w'} A_\lam(w)\right)f(w')\,|\lam|^dd\lam d
w'
\end{eqnarray*}
Since with $\tilde w=w^{-1}w'$, we have $\tilde s=s'-s-2\,{\rm Im}(z\overline{z'})$ and in view of the preceding results, it is enough to observe
$$\displaylines{\qquad
{2^{d-1}\over \pi^{d+1}}\int i\tilde s{\rm tr}
\left(u^\lam_{w^{-1}w'} A_\lam(w)\right)f(w')\,|\lam|^dd\lam d
w'
\hfill\cr\hfill= {2^{d-1}\over \pi^{d+1}}\int{\rm tr}
\left(u^\lam_{w^{-1}w'} J_\lam^*op^w(g)J_\lam(w)\right)f(w')\,|\lam|^dd\lam d
w'\qquad\cr}$$
where we have used Lemma~\ref{lemme:as} stated page~\pageref{lemme:as} and where $g$ is defined by \aref{def:as}, whence the fact that~$ [is,{\rm Op}(a) ]$ is a pseudodifferential operator of order $\mu$.
   \end{proof}

We then observe that the arguments of the proof above give the
following proposition.

\begin{prop}\label{ZOp(a)prop}
For $j\in\{1,\dots, d\}$ and $a\in S_{\H^d}(\mu)$ in~$C^\rho(\H^d)$ with~$\rho >1$, we have
\begin{eqnarray*}
Z_j{\rm Op}(a) & = & {\rm Op}\left(Z_ja+ a \#\sqrt{|\lam|}( -{\rm
sgn}(\lam)\xi_j+i\eta_j)\right),\\
 {\rm Op}(a)Z_j & = & {\rm
Op}\left( \sqrt{|\lam|} (- {\rm sgn}(\lam)\xi_j+i\eta_j)\#
a\right),\\ \overline Z_j{\rm Op}(a) & = & {\rm Op}\left(\overline
Z_ja+ a \#\sqrt{|\lam|}( {\rm sgn}(\lam)\xi_j+i\eta_j)\right),\\
{\rm Op}(a) \overline Z_j & = &{\rm Op}\left( \sqrt{|\lam|} ( {\rm
sgn}(\lam)\xi_j+i\eta_j)\# a\right).
\end{eqnarray*}
Besides, for  $N\in\N$ and~$\rho >2N$,  then~$(-\Delta_{\H^d})^N {\rm
Op}(a)$ and~$ {\rm
Op}(a)(-\Delta_{\H^d})^N$
are pseudodifferential operators of order $\mu+2N$. If~$k \in \R$ and~$\rho >2k$  then~${\rm Op}(a) ({\rm Id} -\Delta_{\H^d})^k$ and~$({\rm Id} -\Delta_{\H^d})^k{\rm Op}(a) $ are
pseudodifferential operators of order~$\mu+2k$.
\end{prop}

\begin{proof}
The four first relations are by-product of the preceding proof and
they directly  imply that  $(-\Delta_{\H^d})^N {\rm Op}(a)$ and~$ {\rm
Op}(a)(-\Delta_{\H^d})^N$ are
pseudodifferential operators. Then for~$k\in\R$, we write
 $$
 {\rm
Op}(a)({\rm Id} -\Delta_{\H^d})^k f(w)={2^{d-1}\over \pi^{d+1}}\int _{\R}
{\rm tr}\left(u^\lam_{w^{-1} } {\mathcal F}(f)(\lam)
({\rm Id} +D_\lam)^kA_\lam(w)\right)|\lam|^d d\lam.
$$ Observing that
$$
({\rm Id} +D_\lam)^k A_\lam(w)= J_\lam^* op^w\left(
m^{(\lam)}_{2k}\#a(w,\lam)\right) J_\lam,
$$ where $
m^{(\lam)}_{2k}$
is the symbol defined by \refeq{tilderem} page~\pageref{tilderem},
 we obtain that ${\rm Op}(a)( {\rm Id} -\Delta_{\H^d})^k $ is a pseudo\-differential operator of order
 $\mu+2k$. We argue similarly for~${\rm Op}(a) ({\rm Id} -\Delta_{\H^d})^k$.
\end{proof}


\chapter{The algebra of pseudodifferential operators}\label{algebra}
\setcounter{equation}{0}

This chapter is devoted to the analysis of the algebra properties   of the set of pseudodifferential operators. The two first sections are devoted to the study of the adjoint of a pseudodifferential operator: we first compute what could be its symbol, and then prove that it actually is a symbol.
In order to prove that fact, the method consists in writing the formula giving the symbol as an oscillatory integral, and in writing a dyadic partition of unity centered on the stationary point  of the phase appearing in that integral.  This
creates a series of oscillatory integrals which are all individually well defined (since each integral is on a compact set). The convergence of the series is then obtained by multiple integrations by parts using a vector field adapted to the phase, as in a stationary phase method.

The approach is similar for the analysis of the composition of two pseudodifferential
operators and this is achieved in the third section. Finally, asymptotic formulas for both the adjoint and the composition are discussed in the last section.  These formulas result from a Taylor expansion in the   spirit of what is done in the Euclidian space but adapted to the case of the Heisenberg group.

\section{The adjoint of a pseudodifferential operator}
\setcounter{equation}{0}

 In this section, we prove that the adjoint of a pseudodifferential operator is a pseudodifferential operator.
 We first observe that if $a\in S_{\H^d}(\mu)$, then $A\eqdefa {\rm Op}(a)$  has a kernel  $k_A(w,w')$ as given in~(\ref{eq:ka})  page~\pageref{eq:ka}, and the kernel of $A^* = {\rm Op}(a)^*$ is  $k(w,w')=\overline{k_A(w',w)}$, whence
   \begin{eqnarray}\nonumber
 k(w,w')  & = & {2^{d-1}\over \pi^{d+1}} \int_{\R} {\rm tr} \left(( u^\lam_{(w')^{-1}w})^*J_\lam^*op^w\left( a(w',\lam)\right)^*J_\lam \right)|\lam|^d\,d\lam\\
 \label{G}
 & = & {2^{d-1}\over \pi^{d+1}} \int_{\R} {\rm tr} \left( u^\lam_{(w)^{-1}w'}\,J_\lam^*op^w\left( \overline a(w',\lam)\right)J_\lam \right) |\lam|^d\,d\lam
 \end{eqnarray}
 where we have used the fact that~${\rm tr}(AB) = {\rm tr}(BA)$, the formula for the adjoint of a Weyl symbol, and~$\overline{{\rm tr}(B)}={\rm tr}(B^*)$. Therefore,  in view of Corollary~\ref{prop:symbol} stated page~\pageref{prop:symbol}, if ${\rm Op}(a)^*$ is a pseudodifferential operator, its symbol $a^*$ will be given
 for all  $ (w,\lam,\xi,\eta)\in \H^d\times\R^*\times  \R^{2d}$ by
\begin{eqnarray}
  a^*(w,\lam, \xi,\eta)&=&{2^{d-1}\over \pi^{d+1}} \int_{\R\times\H^d}  {\rm e}^{2i\, \sqrt{|\lam|}( {\rm sgn}(\lam)y'\cdot \xi-x'\cdot\eta)+i\lam s'} \nonumber\\
  &\times &  {\rm tr} \left( u^{\lam'}_{(w')^{-1}}\,J_{\lam'}^*op^w\left( \overline a(w(w')^{-1},\lam')\right)J_{\lam'} \right) |\lam'|^d\,d\lam'\,dw'. \label{formulaadjoint}
  \end{eqnarray}
  It remains to prove that the map
  $a\mapsto a^*$ which is well defined on ${\mathcal S}(\H^d\times \R^{2d+1})$ can be extended to symbols $a\in S_{\H^d}(\mu)$ and that for such $a$, their image $a^*$   is also in $S_{\H^d}(\mu)$.
  Therefore, it is enough to prove the following proposition.

  \begin{prop}\label{prop:symbadj}
    The map $a\mapsto a^*$  extends by continuity to $S_{\H^d}(\mu)$ since for all  $ k\in\N$ there exists $n\in\N$ and $C>0$ such that
     $$\forall a\in S_{\H^d}(\mu),\;\;
 \| a^*\|_{k;S_{\H^d}(\mu)} \leq C\, \|a\|_{n;S_{\H^d}(\mu)}.$$
\end{prop}

It is not  at all obvious that the formula~(\ref{formulaadjoint}) for $a^*$ gives the
expected result for the  examples studied in
Section~\ref{examples} of Chapter~\ref{fundamental}.  To see that more clearly, it is
convenient to transform  the expression of $a^*$ into an integral
formula.

\begin{lemme}\label{otherformulation}
Let $a\in{\mathcal S}(\H^d\times \R^{2d+1})$, then
the symbol~$a^*$ of~${\rm Op}(a)^*$ given in~(\ref{formulaadjoint}) can also be written
$$\displaylines{\qquad a^*(w,\lam, \xi,\eta)= {1\over 2\pi^{2d+1}}
\int_{\R^{2d+1}\times\H^{d}} {\rm e} ^{2i \sqrt{|\lam|}({\rm
sgn}(\lam)y'\cdot
\xi-x'\cdot\eta)+is'(\lam-\lam')-2i\sqrt{|\lam'|}({\rm
sgn}(\lam')z\cdot y'-\zeta\cdot x') }\hfill\cr\hfill\times\,
\overline {
a}\left(w(w')^{-1},\lam',z,\zeta\right)|\lam'|^dd\zeta\,d
z\,d\lam'\,dw' .\qquad\cr}$$
\end{lemme}

The formula given in Lemma~\ref{otherformulation} allows to revisit the examples of Section~\ref{examples},  Chapter~\ref{fundamental}. Indeed
if $a=a(\lam, \xi,\eta)$, then integration in $s'$ gives $\lam=\lam'$,
then integration in $x'$ (resp. $y'$)) gives $\zeta=\eta$
(resp. $z=y'$); whence $a^*(w,\lam, \xi,\eta)=\overline
a(\lam, \xi,\eta)$. \\ If $a=a(w)$, then integration in $\zeta$ (resp.
$\xi$) gives~$x'=0$ (resp. $y'=0$); then integration in~$s'$
gives~$\lam=\lam'$, whence~$a^*(w)=\overline a(w)$ as expected.

   \begin{rem}\label{defatilde}
 Let $\sigma(a)$ be defined by \aref{def:sigma(a)} page~\pageref{def:sigma(a)}, then $\sigma( a^*)$ and $\sigma(
a)$ are related by
\begin{eqnarray}\label{tildea*}
\sigma(a^*)(w,\lam,  \xi,  \eta) & = & {1\over 2\pi^{2d+1}}
\int_{\R^{2d+1}\times\H^{d}} {\rm e} ^{2i y'\cdot(  \xi-
z)-2i x'\cdot ( \eta- \zeta)+is'(\lam-\lam') }\\
\nonumber& & \quad\quad\quad\quad\times\, \overline{\sigma(
a)}\left(w(w')^{-1},\lam',  z,
\zeta\right)d \zeta\,d   z\,d\lam'\,dw' .
\end{eqnarray}
\end{rem}

{\it Proof of Lemma~\ref{otherformulation}.} ---
 The first step consists in computing the trace term using the link between the trace and the kernel stated in \aref{traceopanoyau} page~\pageref{traceopanoyau}. So let us start by studying  the kernel of our operator. Using~$J_{\lam'}=T_{\lam'}K_{\lam'}$, we write
\begin{equation}\label{ping1}
{\rm tr} \left(
u^{\lam'}_{(w')^{-1}}\,J_{\lam'}^*op^w\left( \overline
a(\tilde w ,\lam')\right)J_{\lam'} \right) = {\rm tr}
\left(K_{\lam'} u^{\lam'}_{(w')^{-1}} K_{\lam'}^* T^*_{\lam'}
op^w\left(\overline a(\tilde w , \lam')\right)
T_{\lam'}\right)
\end{equation}
where~$\tilde w = w(w')^{-1}$
and we  observe that
$K_{\lam'} u^{\lam'}_{(w')^{-1}}
K_{\lam'}^*=v^{\lam'}_{(w')^{-1}}$ where~$v^{\lam'}_{(w')^{-1}}$ is the Schr\"odinger representation given by~\aref{schrorep} page~\pageref{schrorep}.
We shall use the same type of method as for the proof of Proposition~\ref{prop:kernelintegral}. We recall that if~$U$ is an operator on $L^2(\R^d)$ of kernel $k_U(\xi,\xi')$, then the kernel of the operator $$\tilde U\eqdefa v^{\lam'}_{(w')^{-1}}\circ U$$ is the function $k_{\tilde U}$ given by
$$k_{\tilde U}(\xi,\xi')={\rm e}^{-i\lam'(s'+2x'\cdot y'+2y'\cdot \xi)}\; k_U(\xi+2x',\xi').$$
This comes from the definition of the kernel in \aref{operateuranoyau}, page~\pageref{operateuranoyau}, and the definition of $v^{\lam'}_{(w')^{-1}}$ in~\aref{schrorep}, page~\pageref{schrorep}.
We take now
 $$U=T_{\lam'}^*  op^w\left(\overline a(w(w')^{-1},
\lam',\xi,\eta)\right) T_{\lam'}.$$
As in~(\ref{Tlamopw}) page~\pageref{Tlamopw}, we have
$$
T_{\lam'}^*  op^w\left(\overline a(w(w')^{-1},
\lam',\xi,\eta)\right) T_{\lam'}= op^w\left(\overline a(w(w')^{-1},
\lam',\sqrt{|\lam'|} \xi, {\eta\over \sqrt{|\lam'|}}\right)
$$
and using~(\ref{defkernelweyl}) page~\pageref{defkernelweyl} this gives
$$k_U(\xi,\xi')=(2\pi)^{-d} \int_{\R^d} \overline a\left(
w(w')^{-1}, \lam', \sqrt{|\lam'|}\left({\xi+\xi'\over
2}\right),{\Xi\over\sqrt{|\lam'|}}\right){\rm
e}^{i\Xi\cdot(\xi-\xi')}d\Xi.$$
This implies
\begin{eqnarray*}
 {\rm tr}(\tilde U)  & = &  \int_{\R^d} k_{\tilde U}(\xi,\xi)d\xi\\
& = &\int_{\R^d} {\rm e}^{-i\lam'(s'+2x'\cdot y'+2y'\cdot \xi)} k_U(\xi+2x',\xi)d\xi\\
& = &
  (2\pi)^{-d}\int_{\R^{2d}}  {\rm e}
^{-i\lam'( s'+2x'\cdot y'+2 y'\cdot \xi)+2i\Xi\cdot x'}
\overline a\left( w(w')^{-1}, \lam',
\sqrt{|\lam'|}(\xi+x'),{\Xi\over\sqrt{|\lam'|}}\right) d\Xi\,d\xi.
\end{eqnarray*}
We finally obtain via  \aref{formulaadjoint} and \aref{ping1}
$$\displaylines{
a^*(w,\lam,\xi,\eta)=\frac{1}{2\pi^{2d+1}} \int_{\R^{2d+1}\times\H^{d}} {\rm e} ^{2i
\sqrt{|\lam|}({\rm sgn}(\lam)y'\cdot
\xi-x'\cdot\eta)+is'(\lam-\lam')-2i\lam'(x'\cdot y'+y'\cdot\xi)+2i
x'\cdot\Xi}\hfill\cr\hfill\times\, \overline
a\left(w(w')^{-1},\lam',\sqrt{|\lam'|}(\xi+x'),{\Xi\over\sqrt{|\lam'|}}\right)|\lam'|^d
d\Xi\,d\xi\,d\lam'\,dw' .\cr}$$ The change of variable
$\sqrt{|\lam'|}(\xi+x')= {\rm sgn}(\lam')z$ and
$\Xi=\sqrt{|\lam'|}\zeta$ gives the formula of the lemma.
\qed

\section{Proof of Proposition \ref{prop:symbadj}}
\setcounter{equation}{0}
To prove Proposition \ref{prop:symbadj}, we shall use Remark~\ref{defatilde} and Proposition~\ref{prop:sigma(a)}.
Our aim is to analyze the symbol properties of the  oscillatory integral of \aref{tildea*} in order to prove that what should be the symbol of the adjoint actually is a symbol. More precisely, we want
   to prove that for all~$k\in\N$,  there exists a constant $C>0$ and an integer $n$ such that  for any multi-index~$\beta\in\N^{2d}$ and  for all $m\in\N$, if $m+|\beta|\leq k$, then
$$ \forall Y\in\R^{2d},\;\forall
\lam\not=0,\:  \left(1+|\lam|(1+Y^2)\right)^{|\beta|-\mu\over 2}
\left\| (\lam\partial_\lam)^m\partial_{(y,\eta)}^\beta
\sigma(a^*)(\cdot,\lam,Y)\right\|_{C^\rho(\H^d)}\leq
C\|a\|_{n;S_{\H^d}(\mu)}.$$

  The first step consists in proving this inequality when $k=0$, then, in a second step, we will suppose $k\geq 1$ and consider derivatives of the symbol $\sigma(a^*)$.

We follow the classical method of stationary phase, as developed for instance in~\cite{ag}.
 Noticing that
  the phase in \aref{tildea*} is stationary at the point~$(0,0,0,\xi,\eta,\lam)$ in~$\R^d\times\R^d\times\R\times\R^d\times \R^d \times \R $, we introduce a partition of unity centered at zero:
$$1=\tilde \psi(u)+\sum_{p\in\N} \psi(2^{-p}u),\;\;\forall u\in
\R^{4d+2}$$ where $\psi$ is compactly supported in a ring and
$\tilde\psi$ in a ball. Then decomposing the integral~(\ref{tildea*}) using that partition of unity, we notice that each integral
$$\displaylines{
b_p(w,\lam,\xi,\eta)\eqdefa {1\over 2\pi^{2d+1}} \int_{\R^{2d+1}\times\H^{d}}
\psi\left( 2^{-p}x',2^{-p}y',2^{-p}s' ,2^{-p}(z-\xi),
2^{-p}(\zeta-\eta),2^{-p}(\lam'-\lam)\right)\hfill\cr\hfill \times\,{\rm
e} ^{2i y'\cdot(  \xi-
z)-2i x'\cdot ( \eta- \zeta)+is'(\lam-\lam')} \overline {\sigma(a)}\left(w(w')^{-1},\lam',
z,\zeta\right)d\zeta\,d  z\,d\lam'\,dw'\cr}
$$ is well defined since it is on a compact set.
Notice that this is not the usual Heisenberg change of variables as could be expected, but for technical reasons
this change of variables seems more appropriate.
The convergence of
the series $\sum_{p\in\N} b_p$  will come from  integrations by
parts  which will produce powers of $2^{-p}$. Indeed, the
change of variables
 $$
 x'=2^pX,\;y'=2^p Y,\;s'=2^pS,\; z=\xi+2^p u,\;
\zeta= \eta+2^p v, \,\lambda'=\lambda+2^p\Lambda
$$
 gives with
$w(p)\eqdefa w \cdot (2^pX,2^pY,2^pS)^{-1}$
 $$\displaylines{ b_p(w,\lam,\xi,\eta)=
{2^{(4d+2)p}\over 2\pi^{2d+1}} \int_{\R^{2d+1}\times\H^{d}} \psi\left(
X,Y,S ,u, v,\Lambda\right){\rm e} ^{- i 2^{2p} (2Y\cdot u-2X\cdot
v+ S\Lambda) } \hfill\cr\hfill \times\, \overline {\sigma(
a)}\left(w(p),\lam+2^p\Lambda, \xi+2^p u,\eta+2^pv\right)du\,d
v\,dX\,dY\,d\Lambda\,dS.\cr}$$

Let us define
 the differential operator
$$
L\eqdefa{1\over i} (X^2+Y^2+S^2+u^2+v^2+
\Lambda^2)\displaystyle
^{-1}\left(\frac12 X\partial_v+\frac12 v\partial_X-\frac12 Y\partial_u-\frac12 u\partial_Y-S\partial_\Lambda-\Lambda
\partial_S\right),$$
which satisfies
$$
L {\rm e} ^{- i 2^{2p} (2Y\cdot u-2X\cdot
v+  S\Lambda) } = 2^{2p}{\rm e} ^{- i 2^{2p} (2Y\cdot u-2X\cdot
v+  S\Lambda) }.
$$
We remark that the coefficients of~$(L^*)^N$ are uniformly bounded on the support of~$\psi$.
Performing $N$ integration by parts (here we assume that~$\rho >N$)
we obtain
$$\displaylines{ b_p(w,\lam,\xi,\eta)=
{2^{-p(2N-4d-2)}\over 2\pi^{2d+1}} \int_{\R^{2d+1}\times\H^{d}} {\rm e} ^{- i 2^{2p} (2Y\cdot u-2X\cdot
v+ S\Lambda) }\hfill\cr\hfill\times\,
(L^*)^N\left( \psi\left( X,Y,S ,u, v,\Lambda\right) \overline
{\sigma(a)}\left(w(p),\lam+2^p\Lambda, \xi+2^p
u,\eta+2^pv\right)\right)du\,d  v\,dX\,dY\,d\Lambda\,dS.\cr}$$ We
then use that~$\sigma(a)$ satisfies symbol estimates, so
$$\displaylines{\left|
(L^*)^N \overline{\sigma(a)}
\left(w(p),\lam+2^p\Lambda, \xi+2^p
u,\eta+2^pv\right)\right|\hfill\cr\hfill \leq
C 2^{pN}\,\|a\|_{N,S_{\H^d}(\mu)}\left(1+|\lam+2^p\Lambda|+|\xi+2^pu|^2+|\eta+2^pv|^2\right)^{\mu/2}.\cr}$$

Peetre's inequality
 $$
 \displaylines{\qquad
\left(1+|\lam+2^p\Lambda|+|\xi+2^pu|^2+|\eta+2^pv|^2\right)^{\mu/2}\hfill\cr\hfill\leq
\left(1+|\lam|+\xi^2+\eta^2\right)^{\mu/2}\left(1+|2^p\Lambda|+|2^pu|^2+|2^pv|^2\right)^{|\mu|/2}\cr}$$
yields
$$
\displaylines{\left(1+|\lam|+\xi^2+\eta^2\right)^{-\mu/2}
 \left| (L^*)^N\overline{\sigma(
a)}\left(w(p),\lam+2^p\Lambda, \xi+2^p
u,\eta+2^pv\right)\right|\hfill\cr\hfill \leq
C\,\|a\|_{N,S_{\H^d}(\mu)}\left(1+|2^p\Lambda|+|2^pu|^2+|2^pv|^2\right)^{|\mu|/2}.\cr}
$$
Therefore,
 $$\left(1+|\lam|+\xi^2+\eta^2\right)^{-\mu/2}\left|
b_p(w,\lam,\xi,\eta)\right| \leq C \,\|a\|_{N,S_{\H^d}(\mu)}\,
2^{p(4d+2+|\mu|-N)}, $$ which gives the expected inequality for
$k=0$ choosing $N>4d+2+ |\mu|$.

 $ $

Let us now consider   derivatives of $\sigma(a^*)$. We observe
that by integration by parts,
$$\displaylines{ \partial_\lambda
\sigma(a^*)(w,\lam,\xi,\eta) \hfill\cr\hfill = {i\over 2\pi^{2d+1}}
\int_{\R^{2d+1}\times\H^{d}} {\rm e} ^{2i y'\cdot(\xi- z)-2i x'\cdot
(\eta-\zeta)+is'(\lam-\lam') }s'\, \overline {\sigma(
a)}\left(w(w')^{-1},\lam', z, \zeta\right)d\zeta\,d
z\,d\lam'\,dw'\cr\hfill = {1\over 2\pi^{2d+1}}
\int_{\R^{2d+1}\times\H^{d}} {\rm e} ^{2i y'\cdot(\xi- z)-2i x'\cdot
(\eta-\zeta)+is'(\lam-\lam') } \partial_{\lam'}\left(
\overline {\sigma( a)}\left(w(w')^{-1},\lam', z,
\zeta\right)\right)d\zeta\,d z\,d\lam'\,dw'.\cr} $$ Since for
$m\in\N$, $ \partial_\lambda^m \sigma(a)$ satisfies the same
symbol estimates as $ \sigma(a)$, the arguments developed just above
allow to deal with the derivatives in $\lam$. Similarly,  integrating by parts
$$\displaylines{ 2\pi^{2d+1} \xi_j\partial_{ \xi_k} \sigma(a^*)(w,\lam, \xi,\eta)
\hfill\cr\hfill = 2i \int_{\R^{2d+1}\times\H^{d}} {\rm
e} ^{2i y'\cdot(\xi- z)-2i x'\cdot (\eta-\zeta)+is'(\lam-\lam') }y_k'\xi_j\,
\overline {\sigma(a)}\left(\tilde w,\lam', z,
\zeta\right)d\zeta\,d z\,d\lam'\,dw'\cr\hfill = -
\int_{\R^{2d+1}\times\H^{d}} {\rm e} ^{2i y'\cdot(\xi- z)-2i x'\cdot
(\eta-\zeta)+is'(\lam-\lam') }y'_k
\,(\partial_{y_j'}-2iz_j)\left( \overline
{\sigma(a)}\left(\tilde w,\lam', z, \zeta\right)\right)d\zeta\,d
z\,d\lam'\,dw'\cr \hfill =  {i\over 2}
\int_{\R^{2d+1}\times\H^{d}} {\rm e} ^{2i y'\cdot(\xi- z)-2i x'\cdot
(\eta-\zeta)+is'(\lam-\lam') }\partial_{z_k}(\partial_{y_j'}-2iz_j)\left(
\overline {\sigma(a)}\left(\tilde w,\lam', z,
\zeta\right)\right)d\zeta\,d z\,d\lam'\,dw,\cr} $$
with~$\tilde w = w(w')^{-1}$. So,
for $m\in\N$ and $\alpha\in\N^{2d}$, $(\xi_j\partial_{\xi_k})^m  \sigma(a)$ satisfies the same symbol estimates as $ \sigma(a)$, thus we
can treat these derivatives as above with exactly the same arguments.  Besides, it is also the case for
derivatives in $\eta$. This concludes the proof of
Proposition~\ref{prop:symbadj}. \qed

\section{Study of the composition of two pseudodifferential operators}
\setcounter{equation}{0}
We consider now  two pseudodifferential operators
 ${\rm Op}(a)$ and ${\rm Op}(b)$ and study their composition.
 We shall follow the classical method (see for instance~\cite{ag}) consisting in studying rather~${\rm Op}(a)\circ {\rm Op}(c)^*$, where~$c$ is such that~$ {\rm Op}(c)^* =  {\rm Op}(b)$.

   We recall that if $A$ (resp. $B$) is an operator of kernel $k_A(w,w')$ (resp. $k_B(w,w')$), then the kernel of  $A\circ B$ is
  $$k_{A\circ B}(w,w')=\int k_A(w,W)k_B(W,w')dW.$$
 If moreover $B=C^*$ with $C$ of kernel $k_C(w,w')$, then
 $$k_B(w,w')=\overline{ k_C(w',w)}.$$
Those (well-known) results applied to $A={\rm Op}(a)$ and $C={\rm Op}(c)$, imply that the operator~${\rm Op}(a)\circ{\rm Op}(c)^*$ has a kernel $k(w,w')$ given by
 \begin{equation}\label{k(w,w')}
 k(w,w')=\int_{\H^d} k_A(w,W) \overline{k_C(w',W)}\,dW.
 \end{equation}
 If ${\rm  Op}(a)\circ{\rm Op}(c)^*$ is a pseudodifferential operator of symbol $d$, then, by Proposition~\ref{prop:kernelintegral} page~\pageref{prop:kernelintegral}, the symbol~$d$ is given by its associated function $\sigma(d)$ which satisfies,
\begin{equation}\label{sigma(d)new}
\sigma (d)(w,\lam, \xi, \eta)= \int_{\H^d} {\rm e}^{2i(y'\cdot  \xi-x'\cdot \eta)+i\lam s'} k(w,w(w')^{-1})\,dw'.
\end{equation}
 We shall now study the map $(a,c)\mapsto d$ which is well defined for $a,c\in{\mathcal S}(\H^d)$.

 \begin{prop}
 \label{prop:symbcompo}
    The map $(a,c)\mapsto d$  extends by continuity to $S_{\H^d}(\mu)\times S_{\H^d}(\mu')$ since for all~$ k\in\N$ there exist  $n\in\N$ and $C>0$ such that
     $$
 \| d\|_{k;S_{\H^d}(\mu+\mu')} \leq C\, \|a\|_{n;S_{\H^d}(\mu)}\,\|c\|_{n;S_{\H^d}(\mu')}.$$
 \end{prop}

Note that the Proposition implies that
the symbol $d$ of $A\circ B$ satisfies
$$
 \| d\|_{k;S_{\H^d}(\mu+\mu')} \leq C\, \|a\|_{n;S_{\H^d}(\mu)}\,\|b\|_{n;S_{\H^d}(\mu')}$$
 since $c$ is the symbol of $B^*$ and $\|c\|_{n;S_{\H^d}(\mu')}\leq C\,\|b\|_{n;S_{\H^d}(\mu')}$
for all $n\in\N$ by Proposition~\ref{prop:symbadj}.

\begin{proof}  The proof is very similar to the one for the adjoint written in the previous section: one writes the function $\sigma( d)$ as an oscillatory integral that we study with standard techniques. We first obtain, thanks to Proposition~\ref{prop:kernelintegral} page~\pageref{prop:kernelintegral}, \aref{k(w,w')}
 and \aref{lawH},  that the kernel of~${\rm Op}(a)\circ {\rm Op}(c)^*$ is
$$\displaylines{ k(w,\tilde w)={1\over (2\pi^{2d+1})^2} \int \sigma(a)
(w,\lam_1,z_1,\zeta_1) \overline {\sigma(
c)}(\tilde w,\lam_2,z_2,\zeta_2)\hfill\cr\hfill{\times \rm e
}^{i\lam_1s_1+2iy_1\cdot z_1-2i x_1 \cdot \zeta_1 -i
\lam_2s_2-2iy_2\cdot z_2+2i\zeta_2\cdot x_2}
d\lam_1\,d\lam_2\,dz_1\,dz_2\,d\zeta_1\,d\zeta_2\,dW\cr}
$$ where
$w^{-1} W=(x_1,y_1,s_1)$ and $\tilde w^{-1}W=(x_2,y_2,s_2)$. Therefore, recalling that
$$
\sigma (d)(w,\lam, \xi, \eta)= \int_{\H^d} {\rm e}^{2i(y'\cdot  \xi-x'\cdot \eta)+i\lam s'} k(w,w(w')^{-1})\,dw'
$$
where~$k$ is the kernel given above, we get
\begin{eqnarray}\label{formulasigma(d)}
\sigma(
d)(w,\lam,\xi,\eta)&=& {1\over (2\pi^{2d+1})^2} \int \sigma( a)
(w,\lam_1,z_1,\zeta_1) \overline {\sigma
(c)}(w(w')^{-1},\lam_2,z_2,\zeta_2)\\
&& {}\times{\rm
e}^{i\Phi(W,w',\lam_1,\lam_2,z_1,z_2,\zeta_1,\zeta_2)}
d\lam_1\,d\lam_2\,dz_1\,dz_2\,d\zeta_1\,d\zeta_2\,dW\,dw', \nonumber
\end{eqnarray}
where the phase function $\Phi$ (depending on~$w$, $\lam$, $\xi$ and $\eta$) is given by
\beq\label{defphasePhi}
\quad \Phi=\lam
s'+\lam_1s_1-\lam_2s_2+2(y'\cdot \xi +y_1\cdot z_1-y_2\cdot z_2)-2(
x'\cdot\eta +x_1 \cdot \zeta_1 -x_2\cdot\zeta_2)
\eeq
with
$w_1=(x_1,y_1,s_1)=w^{-1} W$ and $w_2=(x_2,y_2,s_2)=w'w^{-1}W$; in
particular $w_2=w'w_1$ so writing~$W = (X,Y,S)$ and using the group law on~$\H^d$, we have
\begin{eqnarray}\label{formulasxX}
&&x_1=X-x,\;x_2=X-x+x',\; y_1=
Y-y,\;y_2=Y-y+y',s_1=S-s-2Xy+2xY, \nonumber\\
&& \quad \quad s_2=S-s+s'-2(x'-x)\cdot Y+2(y'-y)\cdot X+2x'\cdot y-2y'\cdot x.
\end{eqnarray}

 The function $\Phi$ is polynomial of degree $3$ in its variables and
straightforward computations give
\begin{eqnarray*}
&\partial_{\lam_1}\Phi=s_1, \;\;\partial_{\lam_2}\Phi=-s_2,\;\;
\partial_{z_1}\Phi=2y_1,\;\;\partial_{z_2}\Phi=-2y_2&\\ &
\partial_{\zeta_1}\Phi=-2x_1,\;\;\partial_{\zeta_2}\Phi=2x_2,\;\;\partial_{s'}\Phi=\lam-\lam_2,\;\;\partial_{S}\Phi=\lam_1-\lam_2&\\
&\partial_{x'}\Phi=-2(\eta-\zeta_2)+2\lam_2(Y-y) ,\;\;
\partial_{y'}\Phi=2(\xi-z_2)-2\lam_2(X-x)& \\
&\partial_X\Phi=-2(\zeta_1-\zeta_2)-2\lam_2y'+2y(\lam_2-\lam_1),\;\;\partial_Y\Phi=2\lam_2x'-2x(\lam_2-\lam_1)+2(z_1-z_2).&
\end{eqnarray*}

Therefore, one can check easily that the phase $\Phi$ satisfies
$d\Phi=0$ if and only if
$$w=W,\;w'=0,\;\;\lam=\lam_1=\lam_2,\;
z_1=z_2=\xi,\;\zeta_1=\zeta_2=\eta.
$$
In the following we shall denote by~$U_0 \in \R^D$ that critical point, with~$D =4(2d+1) $:
$$
 U_0\eqdefa (x,y,s,0,\lam,\lam,\xi, \xi,\eta,\eta).
 $$
 By a tedious but straightforward computation, we check that~$\Phi(U_0)=0$, $d\Phi(U_0)=0$ and that~$d^2\Phi(U_0)$ is invertible for all~$(w,\lam, \xi, \eta)$:
 computing the Hessian matrix~$d^2\Phi(U_0)$ one notices easily that each lign of the matrix has at least one constant term (and the others are either zero or linear in~$ \lam, x, y$).

We then argue as in the proof
for the adjoint by use of a partition of unity centered in the
point~$U_0$ where~$\Phi$ degenerates.  For simplicity we denote the new set of variables by
$$
V=(X,Y,X,x',y',s',\lam_1,\lam_2,z_1,z_2,\zeta_1,\zeta_2) \in \R^D.
$$In the phase $\Phi$ there are     terms of order~$3$ and we observe that the only derivatives of order $3$ which are non zero are
$$\partial^3_{X,\lam_2,y'} \Phi   =- 2\;\;{\rm and}\;\; \partial^3_{Y,\lam_2,x'}\Phi = 2.$$
 We write, for any point~$U \in \R^D$,   $\Phi(U)=\Phi_0(U-U_0)+G(U-U_0)$ where  by a direct application of Taylor's formula, one has
$$\forall V \in \R^D, \quad \Phi_0(V)\eqdefa \frac12 D^2\Phi(U_0)V\cdot V\quad  {\rm and} \quad G(V)\eqdefa (\lam_2-\lam) \left(
(Y-y)\cdot x' - (X-x)\cdot y'
\right).$$
  We are therefore reduced to the study of
an integral under the form
$$
 I = \int_{\R^D} f(U){\rm e} ^{i\Phi(U)}dU , $$
 where we have defined
 \begin{equation}\label{deffac}
\forall U \in \R^D, \quad f(U) =  \sigma(a)
(w,\lam_1,z_1,\zeta_1) \overline {\sigma(
c)}( w(w')^{-1},\lam_2,z_2,\zeta_2) .
\end{equation}
We shall decompose this integral into a series of integrals by a partition of unity:
\begin{eqnarray*}
I   &=&  \int_{\R^D}  f(U){\rm
e}^{i\Phi(U)}\tilde \zeta(  U-U_0 )dU + \sum_{q\in \N} \int_{\R^D}  f(U){\rm
e}^{i\Phi(U)}\zeta(2^{-q} (U-U_0))dU\\
 & = &
  \int_{\R^D}  f(U){\rm
e}^{i\Phi(U)}\tilde \zeta( U-U_0 )dU +
 \sum_{q\in \N}2^{qD} \int_{\R^D}
f(U_0+2^q V) \zeta(V){\rm e}^{i 2^{2q}\Phi_0(V)+i2^{3q}G(V)} dV,
\end{eqnarray*}
where~$ \tilde \zeta$ and~$\zeta$ are  functions defining  a partition of unity, in the sense that they are nonnegative, smooth compactly supported functions ($ \tilde \zeta$ in a ball and~$\zeta$ in a ring) such that
$$
\forall  U \in \R^D, \quad \tilde \zeta(U-U_0) +  \sum_{q\in \N} \zeta(2^{-q} (U-U_0)) = 1.
$$
Each integral is now well defined, and  the main problem consists in proving the convergence of the series in~$q \in \N$, as well as in proving symbol estimates. We shall concentrate on the second integral and leave the (easier) computation in the case of~$\tilde \zeta$ to the reader.

Consider
$$I_q\eqdefa  2^{qD} \int
f(U_0+2^q V) \zeta(V){\rm e}^{i2^{2q}\Phi_0(V)+i2^{3q}G(V)} dV.$$
 We shall use a stationary phase method, which will be implemented differently according to whether in the phase~$ 2^{2q}\Phi_0(V)+ 2^{3q}G(V)$, the dominant term is the first or the second of the two terms.
 More precisely, let~$\displaystyle \delta \in ]0,\frac12[$ be any real number and let us cut the integral~$I_q$ into two parts depending on whether $|\nabla G(V)|<2^{-q(1+\delta)}$ or not. For this, we introduce a smooth cut-off function $\chi\in{\mathcal C}_0^\infty(\R)$ compactly supported on $[-1,1]$ and write~$I_q=I_q^{1}+I_q^2,$ where
\begin{eqnarray*}
  I_q^{1} & \eqdefa & 2^{qD} \int \chi\left(2^{2q(1+\delta)}|\nabla G(V)|^2\right)
f(U_0+2^q V) \zeta(V){\rm e}^{i 2^{2q}\Phi_0(V)+i2^{3q}G(V)} dV \quad  \mbox{and}\\
 I_q^{2} & \eqdefa & 2^{qD} \int(1- \chi )\left(2^{2q(1+\delta)}|\nabla G(V)|^2\right)
f(U_0+2^q V) \zeta(V){\rm e}^{i 2^{2q}\Phi_0(V)+i2^{3q}G(V)} dV .
\end{eqnarray*}
Let us first analyze $I_q^1$. We introduce the differential operator
$$
L\eqdefa {1\over i} { \nabla \Phi_0(V)\over |\nabla\Phi_0(V)|^2} \cdot  \nabla
$$
which satisfies
$$
L^N\left[ {\rm e}^{i2^{2q}\Phi_0(V)}\right]= 2^{2Nq} {\rm e}^{i2^{2q}\Phi_0(V)}.
$$
Note that the computation of the Hessian mentioned above allows easily to obtain a bound of the following type for~$\nabla \Phi_0$:
\beq\label{boundnablaphi0}
\forall V \in \mbox{Supp} \: \zeta, \quad \|\nabla \Phi_0(V)\|^{-1} \leq \frac C{1+|\lam| + |x| + |y|}
\eeq
where~¬†$C$ is a constant.
It follows that~$L$ is well defined, and its coefficients are at most linear in~$\lam$, $x$ and~$y$. One therefore checks easily that on the support of~$\zeta$ the operator~$(L^*)^N$ has uniformly bounded coefficients (the bound is uniform in~$V$ as well as in~$w$, $\lam, \xi$ and~¬†$ \eta$).
Therefore one can write
$$
I_q^1= 2^{qD} 2^{-2Nq} \int {\rm e}^{i2^{2q}\Phi_0(V)} (L^*)^N \left[ \zeta(V)
\chi\left(2^{2q(1+\delta)}|\nabla G(V)|^2\right) {\rm e}^{i2^{3q}G(V)}f(U_0+2^qV)\right]dV.
$$
Using the Leibniz formula, we have
\begin{eqnarray}\label{estimateLstarN}
&& \quad \quad \quad \quad  \left| (L^*)^N \left[ \zeta(V)
\chi\left(2^{2q(1+\delta)}|\nabla G(V)|^2\right) {\rm e}^{i2^{3q}G(V)}f(U_0+2^qV) \right] \right|\\
&&\leq C \sum_{|\ell|+|m|+|n| \leq N}
\left |\partial^\ell (f(U_0+2^qV)) \right | \:  \left |\partial^n ( {\rm e}^{i2^{3q}G(V)}) \right | \: \left |\partial^m \left(\chi\left(2^{2q(1+\delta)}|\nabla G(V)|^2\right) \right) \right |
|\overline  \zeta(V)|\nonumber
\end{eqnarray}
where~$\ell, m, n$ are multi-indexes in~$\N^D$ and where~$\overline \zeta$ is a function, compactly supported on a ring, defined by
$$
\overline  \zeta(V) = \sup_{|j| \leq N} |\partial^j \zeta (V)| .
$$

Now the difficulty consists in estimating each of the three terms containing derivatives on the right-hand side of the above inequality. Recalling that~$f$ is defined by~(\ref{deffac}), $f$ satisfies the following symbol-type estimate:
\begin{eqnarray}\label{estimatefassymbol}
 |\partial^\ell (f(U_0+2^qV))|  \!  \! & \leq & \!  \!  \! C 2^{|\ell| q} \sup_{\{j_1,\dots,j_6\} \in \{1,\dots D\}^d}
\left(
1+|\lam + 2^q V_{j_1}|    +|\xi + 2^q V_{j_2}|^2  +|\eta + 2^q V_{j_3}|^2
\right)^{\frac\mu2} \nonumber
\\
&& \quad \quad  \times \left(
1+|\lam + 2^q V_{j_4}|    +|\xi + 2^q V_{j_5}|^2  +|\eta + 2^q V_{j_6}|^2
\right)^{\frac{\mu'}2}.
\end{eqnarray}
Now let us prove an estimate for the second term. We use Faa-di-Bruno's formula, which in general can be stated as follows:
 $$
 \longformule{
D^N (e^{F(V)})[h_1,\dots,h_N ]= \sum_{\sigma \in {\mathfrak\sigma}_N}¬†\sum_{p=1}^N \sum_{
 r_1+\dots + r_p = N}\frac{1}{r_1!\dots r_p!p!}
}{{}\times{} e^{F(V)} [D^{r_1} F(V)(h_{\sigma(1)},\dots , h_{\sigma(r_1)}) ,\dots,
 D^{r_p} F(V)(h_{\sigma(N-r_p+1)},\dots , h_{\sigma(N)}) ]  .}
$$
But on the support of $\zeta$, the function~$G$ is bounded as well as its derivatives,
so this implies that on the support of~$ \chi$,
 $$
  |\partial^n ( {\rm e}^{i2^{3q}G(V)})| \leq C \sum_{p=1}^{ |n|} \sum_{
 r_1+\dots + r_p = |n|}\frac{1}{r_1!\dots r_p!p!}2^{3qp}  \left( 2^{-q(1+\delta)}\right)^{ K}
$$
where $K\eqdefa{\rm card} \{j, \: r_j = 1\}$ is the number of integers $j$ in $\{1,\cdots ,p\}$ such that $r_j=1$.
We notice that the worst situation corresponds to the case when~$   \{j, \: r_j = 1\} = \emptyset$, which means in particular that~$r_j \geq 2$ for all~$j$ (in the above summation it is implicitly assumed that the~$r_j$ are not zero).
The largest possible~$p$ for which such a situation may occur is~$p = |n|/2$ (or~$(|n|-1)/2$ if~$|n|$ is odd). But one notices that since~$\delta < 1/2$,
$$
2^{\frac{3p|n|}2} \leq 2^{2 |n| p - \delta |n| p}
$$
so using the fact that for any~$p \leq |n|$ one has clearly~$2^{2pq-pq\delta} \leq 2^{2 |n|q- |n|q\delta} $
we infer that
\begin{equation}\label{estimateexp}
  |\partial^n ( {\rm e}^{i2^{3q}G(V)})| \leq C  2^{2 |n|q- |n|q\delta}.
\end{equation}

Finally let us consider the last term, namely~$\displaystyle
\partial^m \left(\chi\left(2^{2q(1+\delta)}|\nabla G(V)|^2\right) \right).
$ Taking~$|m |= 1$ and writing~$\partial_j$ for any derivative in~$\R^D$ we have
$$
\frac12\partial_j \left(\chi\left(2^{2q(1+\delta)}|\nabla G(V)|^2\right) \right) =  2^{2q(1+\delta)}
\chi' \left(2^{2q(1+\delta)}|\nabla G(V)|^2\right) \sum_{i=1}^D   \partial^2_{ij} G(V)  \partial_i G(V)
$$
which can be written
$$
\frac12\partial_j \left(\chi\left(2^{2q(1+\delta)}|\nabla G(V)|^2\right) \right) = 2^{ q(1+\delta)} \sum_{i=1}^D h_i \left(2^{ q(1+\delta)}  \nabla G(V) \right) \partial^2_{ij} G(V),
$$
where~$h_i$ is the smooth, compactly supported function defined by
$$
\forall U \in \R^D, \quad h_i (U) \eqdefa U_i \chi'(|U|^2).
$$
So, using that the derivatives of  $G$ are bounded and by Leibniz formula,   one gets
$$
\left |\partial_j \left(\chi\left(2^{2q(1+\delta)}|\nabla G(V)|^2\right) \right) \right | \leq C2^{ q(1+\delta)},
$$
and arguing in the same way for higher order derivatives one finds finally
\begin{equation}\label{estimatechi}
\left |\partial^m \left(\chi\left(2^{2q(1+\delta)}|\nabla G(V)|^2\right) \right) \right | \leq C2^{|m| q(1+\delta)}.
\end{equation}
Plugging~(\ref{estimatefassymbol}), (\ref{estimateexp}) and~(\ref{estimatechi}) into~(\ref{estimateLstarN}), we get
\begin{eqnarray*}
&& 2^{-2qN+qD} \left| (L^*)^N \left[ \zeta(V)
\chi\left(2^{2q(1+\delta)}|\nabla G(V)|^2\right) {\rm e}^{i2^{3q}G(V)}f(U_0+2^qV) \right] \right|\\
&&\leq C \sup_{\{j_1,\dots,j_6\} \in \{1,\dots D\}^d}
 \sum_{|\ell|+|m|+|n| \leq N} 2^{|\ell| q}  2^{2 |n|q- |n|q\delta} 2^{|m| q(1+\delta)} \\
&& \quad \quad \quad  \times \left(
1+|\lam + 2^q V_{j_1}|    +|\xi + 2^q V_{j_2}|^2  +|\eta + 2^q V_{j_3}|^2
\right)^{\frac\mu2} \\
&& \quad \quad \quad \quad \times  \left(
1+|\lam + 2^q V_{j_4}|    +|\xi + 2^q V_{j_5}|^2  +|\eta + 2^q V_{j_6}|^2
\right)^{\frac{\mu'}2}.
 \end{eqnarray*}
 Noticing that
 $$
  2^{-2qN+qD} \sum_{|\ell|+|m|+|n| \leq N} 2^{|\ell| q}  2^{2 |n|q- |n|q\delta} 2^{|m| q(1+\delta)} \leq C   2^{qD} (2^{-Nq\delta}+ 2^{Nq(\delta-1)})
 $$
 it suffices to choose~$N$ large enough and to use Peetre's inequality as in the case of the adjoint to conclude on the summability of the series, and on the symbol estimate on~$\sum_q I_q^1$.

   \smallskip

Let us now focus on $I^2_q$. In that case~$\Phi_0$ is no longer predominant, so we shall use the full operator
 $$
 L_q(V)\eqdefa{1\over i}{\nabla\Phi_0 (V)+2^q\nabla G(V)\over |  \nabla\Phi_0(V)+2^q\nabla G(V)|^2
} \cdot \nabla
$$
which is well defined   on the support of $\zeta$  and  satisfies
$$
L_q(V)\left[ {\rm e}^{i 2^{2q}\Phi_0(V)+i2^{3q}G(V)} \right]=2^{2q}\,{\rm e}^{i 2^{2q}\Phi_0(V)+i2^{3q}G(V)} .
$$
This implies that~$I_q^2 $ is equal to
$$
2^{qD-2Nq}   \int  (L_q(V)^*)^N\left[(1-\chi) \left(2^{2(1+\delta) q}|\nabla G(V)|^2\right)
f(U_0+2^q V) \zeta(V) \right]{\rm e}^{i 2^{2q}\Phi_0(V)+i2^{3q}G(V)} dV ,
$$
and it is not difficult to prove by induction  that   for $N\in \N$, the operator $(L_q^*)^N$ is of the form
$$
(L_q^*)^NF(V) = \sum_{k=0}^N \sum_{|\alpha| \leq N-k} \frac{f_0(V)+2^qf_1(V)+ \cdots 2^{kq}f_k(V)}{ |\Phi_0(V)+2^q\nabla G(V)|^{2k}} \partial^\alpha F(V),
$$
where the~$f_i$ are uniformly bounded functions on the support of~$\zeta$.
As in the case of~$I_q^1$, we apply the Leibniz formula to write
$$\longformule{\left|
\partial^\alpha \left[(1-\chi) \left(2^{2(1+\delta) q}|\nabla G(V)|^2\right)
f(U_0+2^q V) \zeta(V) \right] \right|
}{
 \leq C \sum_{|\ell|+|m|  \leq |\alpha|}
|\partial^\ell (f(U_0+2^qV))| \:  \left |\partial^m \left((1-\chi)\left(2^{2q(1+\delta)}|\nabla G(V)|^2\right) \right) \right |
|\overline  \zeta(V)|,
}
$$
where~$\ell$ and $ m $ are multi-indexes in~$\N^D$ and where~$\overline \zeta$ is a function, compactly supported on a ring.
The first term of the right-hand side was estimated in~(\ref{estimatefassymbol}), and the second one may be estimated similarly to~(\ref{estimatechi}) since as soon as~$|m| \geq 1$, the support of~$\displaystyle \partial^m (1-\chi)(V) $ is  in a ring
far from zero.
It follows that
\begin{eqnarray*}
&& \left|
\partial^\alpha \left[(1-\chi) \left(2^{2(1+\delta) q}|\nabla G(V)|^2\right)
f(U_0+2^q V) \zeta(V) \right] \right| \leq C \sum_{|\ell|+|m|  \leq |\alpha|}
 2^{|\ell| q}  2^{|m| q(1+\delta)}
 \\
 && \quad \quad \times \sup_{\{j_1,\dots,j_6\} \in \{1,\dots D\}^d}
\left(
1+|\lam + 2^q V_{j_1}|    +|\xi + 2^q V_{j_2}|^2  +|\eta + 2^q V_{j_3}|^2
\right)^{\frac\mu2} \nonumber
\\
&& \quad \quad\quad \quad \quad \quad   \times \left(
1+|\lam + 2^q V_{j_4}|    +|\xi + 2^q V_{j_5}|^2  +|\eta + 2^q V_{j_6}|^2
\right)^{\frac{\mu'}2}.
\end{eqnarray*}
Since on the other hand, on the support of~$(1-\chi) \left(2^{2(1+\delta) q}|\nabla G(V)|^2\right)$ and on the support of~$\zeta$,
$$
\left|\frac{f_0(V)+2^qf_1(V)+ \cdots 2^{kq}f_k(V)}{ |\Phi_0(V)+2^q\nabla G(V)|^{2k}} \right|
\leq C 2^{-kq} 2^{2kq(1+\delta)},
$$
this implies that $$\displaystyle X_q^N   \eqdefa (L_q(V)^*)^N\left[(1-\chi) \left(2^{2(1+\delta) q}|\nabla G(V)|^2\right)
f(U_0+2^q V) \zeta(V) \right]$$ may be bounded by
\begin{eqnarray*}
|X_q^N| & \leq & C \sum_{k=0}^N \sum_{|\alpha| \leq N -k} \sum_{|\ell|+|m|  \leq |\alpha|} 2^{-kq} 2^{2k(1+\delta)q}  2^{|\ell| q}  2^{|m| q(1+\delta)} \\
&& \quad \quad \times \sup_{\{j_1,\dots,j_6\} \in \{1,\dots D\}^d}
\left(
1+|\lam + 2^q V_{j_1}|    +|\xi + 2^q V_{j_2}|^2  +|\eta + 2^q V_{j_3}|^2
\right)^{\frac\mu2} \nonumber
\\
&& \quad \quad\quad \quad \quad \quad \quad \quad \quad   \times \left(
1+|\lam + 2^q V_{j_4}|    +|\xi + 2^q V_{j_5}|^2  +|\eta + 2^q V_{j_6}|^2
\right)^{\frac{\mu'}2}.
\end{eqnarray*}
Since
$$
  \sum_{k=0}^N \sum_{|\alpha| \leq N -k} \sum_{|\ell|+|m|  \leq |\alpha|} 2^{-kq} 2^{2k(1+\delta)q}  2^{|\ell| q}  2^{|m| q(1+\delta)} \leq C 2^{Nq} 2^{2N\delta q} $$
we conclude that
\begin{eqnarray*}
X_q^N  \!  \! & \leq& \!  \!  \! C 2^{-Nq+ 2N\delta q+ ND} \sup_{\{j_1,\dots,j_6\} \in \{1,\dots D\}^d}
\left(
1+|\lam + 2^q V_{j_1}|    +|\xi+ 2^q V_{j_2}|^2  +|\eta+ 2^q V_{j_3}|^2
\right)^{\frac\mu2}  \\
&& \quad \quad\quad \quad \quad\quad   \times \left(
1+|\lam + 2^q V_{j_4}|    +|\xi + 2^q V_{j_5}|^2  +|\eta + 2^q V_{j_6}|^2
\right)^{\frac{\mu'}2}.
\end{eqnarray*}
The choice of $\delta\in]0,1/2[$ allows to
conclude as in the previous proof via Peetre's inequality.

\medskip

 The analysis  of derivatives of~$\sigma(d)$
is very similar. Let us for the sake of simplicity only deal with the~$\lam$-derivative, and leave the study of the other derivatives to the reader. Taking a partial derivative of~$\sigma (d)$, defined in~(\ref{formulasigma(d)}),  in the~$\lam$ direction produces a factor~$is'$ in the integral, namely
  $$
\displaylines{\partial_\lam \sigma(
d)(w,\lam,\xi,\eta)= {1\over (2\pi^{2d+1})^2} \int is' \sigma( a)
(w,\lam_1,z_1,\zeta_1) \overline {\sigma
(c)}(w(w')^{-1},\lam_2,z_2,\zeta_2)\hfill\cr\hfill \times {\rm
e}^{i\Phi(W,w',\lam_1,\lam_2,z_1,z_2,\zeta_1,\zeta_2)}
d\lam_1\,d\lam_2\,dz_1\,dz_2\,d\zeta_1\,d\zeta_2\,dW\,dw'.\cr}
$$
But one notices that
$$
\longformule{
\partial_{\lam_2} ( {\rm
e}^{i\Phi(W,w',\lam_1,\lam_2,z_1,z_2,\zeta_1,\zeta_2)}) = - i \left(S-s + s'+2xY-2yX - 2x'(Y-y) + 2y'(X-x) \right)}{\times {}  {\rm
e}^{i\Phi(W,w',\lam_1,\lam_2,z_1,z_2,\zeta_1,\zeta_2)}}
$$
which can also be written, using~(\ref{formulasxX})
$$ is'  {\rm
e}^{i\Phi } =( -\partial_{\lam_2} -is_1)   {\rm
e}^{i\Phi }-i (-2x'y_1+2y'x_1) {\rm
e}^{i\Phi } .
 $$
On the other hand an easy computation, using the formula defining~$\Phi$ in~(\ref{defphasePhi}) above, allows to write that
$$
is_1  {\rm e}^{i\Phi }  = \partial_{\lam_1} {\rm e}^{i\Phi } , \quad 2iy_1 e^{i\Phi } =\partial_{z_1}  {\rm e}^{i\Phi } , \quad \mbox{and} \quad -2ix_1  {\rm
e}^{i\Phi } =\partial_{\zeta_1}  {\rm
e}^{i\Phi }
$$
so we find the following identity:
$$
is'  {\rm
e}^{i\Phi} = (-\partial_{\lam_2}-\partial_{\lam_1} + x'\partial_{z_1} +y'\partial_{\zeta_1}  ) {\rm
e}^{i\Phi }.
$$
Finally~$  (2\pi^{2d+1})^2 \partial_\lam \sigma(
d)(w,\lam,\xi,\eta) $ is equal to
\begin{eqnarray*}
 &&\int  {\rm e}^{i\Phi } \partial_{\lam_2} \overline {\sigma
(c)}(w(w')^{-1},\lam_2,z_2,\zeta_2)\sigma( a)
(w,\lam_1,z_1,\zeta_1) d\lam_1\,d\lam_2\,dz_1\,dz_2\,d\zeta_1\,d\zeta_2\,dW\,dw'\\
&+&   \int {\rm e}^{i\Phi } \partial_{\lam_1}\sigma( a)
(w,\lam_1,z_1,\zeta_1)  \overline {\sigma
(c)}(w(w')^{-1},\lam_2,z_2,\zeta_2)d\lam_1\,d\lam_2\,dz_1\,dz_2\,d\zeta_1\,d\zeta_2\,dW\,dw'\\
&+&   \int {\rm e}^{i\Phi } \partial_{z_1}\sigma( a)
(w,\lam_1,z_1,\zeta_1)  (x-x')\overline {\sigma
(c)}(w(w')^{-1},\lam_2,z_2,\zeta_2)d\lam_1\,d\lam_2\,dz_1\,dz_2\,d\zeta_1\,d\zeta_2\,dW\,dw'\\
&-&   \int {\rm e}^{i\Phi } x \partial_{z_1}\sigma( a)
(w,\lam_1,z_1,\zeta_1)  \overline {\sigma
(c)}(w(w')^{-1},\lam_2,z_2,\zeta_2)d\lam_1\,d\lam_2\,dz_1\,dz_2\,d\zeta_1\,d\zeta_2\,dW\,dw'\\
&-&   \int{\rm e}^{i\Phi } \partial_{\zeta_1}\sigma( a)
(w,\lam_1,z_1,\zeta_1) (y'-y)\overline {\sigma
(c)}(w(w')^{-1},\lam_2,z_2,\zeta_2)d\lam_1\,d\lam_2\,dz_1\,dz_2\,d\zeta_1\,d\zeta_2\,dW\,dw'\\
&-&   \int {\rm e}^{i\Phi } y\partial_{\zeta_1}\sigma( a)
(w,\lam_1,z_1,\zeta_1)  \overline {\sigma
(c)}(w(w')^{-1},\lam_2,z_2,\zeta_2)d\lam_1\,d\lam_2\,dz_1\,dz_2\,d\zeta_1\,d\zeta_2\,dW\,dw'.
\end{eqnarray*}

Since~$\sigma( a)$ and~$\sigma( c)$  satisfy symbol estimates, the expressions above can be dealt with exactly by the same arguments as those developed above. One proceeds similarly for all the other derivatives. Details are left to the reader.
 \end{proof}

\section{The  asymptotic  formulas}\label{asympt2}
\setcounter{equation}{0}

In this section, we  give the asymptotics for the symbol of the adjoint
and of the composition,  up to one order more than in Theorem~\ref{adjcompo}. The proof that we propose does not use the integral formula obtained for $a^*$ and $a \#_{\H^d} b$ but relies more precisely on functional calculus, which suits more to the Heisenberg properties to our opinion.

\begin{prop}\label{prop:adjcompo}
Let $a\in S_{\H^d}(\mu_1)$ and $b\in S_{\H^d}(\mu_2)$. Then the symbol of the adjoint of~${\rm Op}(a)$ is given by
 \begin{eqnarray*}
 a^* & = & \overline a+{1\over 2\sqrt{|\lam|}}
 \sum_{1\leq j\leq d}( Z_jT_j+\overline Z_j T_j^*)\overline  a  +{1\over 8|\lam|}\sum_{1\leq j,k\leq d}( Z_jT_j +\overline Z_j    T_j^*) ( Z_kT_k+\overline Z_k    T_k^*)  \overline a \\ & &
 +{1\over i\lam}\left(-\lam\partial_\lam  +{1\over 2}\sum_{1\leq j\leq d}(\eta_j\partial_{\eta_j}+\xi_j\partial_{\xi_j})\right)S\, \overline a+ \tilde r_1
  \end{eqnarray*}
  whereas the symbol of the composition~${\rm Op}(a) \circ {\rm Op}(b)$ is given by
  \begin{eqnarray*}
  a\;   \#_{\H^d} b & = & b\,\#\, a+
 {1\over 2\sqrt{|\lam|}}\sum_{1\leq j\leq d}\left(
 Z_jb\, \# \,T_ja+ \overline Z_j b\, \#\, T_j^* a
 \right)\\
 +&  \displaystyle  {1\over 8|\lam|} &
  \sum_{1\leq j,k\leq d}
  (Z_jZ_kb\,\#\,T_jT_k a+ \overline Z_j\overline Z_k b\, \#\, T_j^*T_k^* a+ Z_j\overline Z_kb\, \# \,T_j  T_k^*a + \overline Z_j Z_k b\, \# \, T_j^* T_k a)\\
+ &  \displaystyle {1\over i\lam}&  Sb\, \#\,\left(-\lam\partial_\lam+{1\over 2} \sum_{1\leq j\leq d}(\eta_j\partial_{\eta_j}+\xi_j\partial_{\xi_j})\right)a
+\tilde r_2
 \end{eqnarray*}
 where~$S$ denotes ~$\partial_s$,  $\tilde r_1$ (resp. $\tilde r_2$)
 depends only on  ${\mathcal Z}^\alpha a$ (resp. ${\mathcal Z}^\alpha b$) for $|\alpha|\geq 3$ and finally where
 $$T_j a\eqdefa {1\over i} \partial_{\eta_j} a -\,{\rm sgn}(\lam)\partial_{\xi_j}a.$$
 \end{prop}

Recall that formulas for~$a^*$ and~$  a\;   \#_{\H^d} b$ are provided respectively in~(\ref{tildea*}) and~(\ref{formulasigma(d)}).

 In view of the second term of the asymptotic expansion, one understands better in what sense these formula are asymptotics. Let us comment the development of $a^*$. The first term is a symbol of order $\mu-1$, it is of order strictly smaller than $a$.

 The first part of the second term is of order $\mu-2$; however, the second part of this term is the product of~$\lam^{-1}$  by a symbol of the same order $\mu$. This is a smaller term only for large values of $\lam$. In view of the proof below, it is easy to see that  one could obtain an expansion to any order and that the  term of order $k$ will be  the sum of terms of the form: $\lam^{-j}$ times a symbol of order $\mu-k+2j$ for $0\leq 2j\leq k$.
It is in this sense that this asymptotic has to be considered.

  We shall not discuss here the precise feature of the remainder and will discuss this point in further works for applications where these asymptotic expansions could be useful.

We point out that the asymptotic formula for $a^*$ and  $a \#_{\H^d} b$ have their counterpart for $\sigma(a^*)$ and $\sigma\left(a \#_{\H^d}b\right)$. By the definition of the function $\sigma(a)$ associated with a symbol $a$ (see \aref{def:sigma(a)}), the following corollary comes from Proposition~\ref{prop:adjcompo}.
While the asymptotics of Proposition~\ref{prop:adjcompo} appear as especially useful for large $\lam$, the asymptotics on $\sigma(a)$ seems more pertinent for $\lam$ close to $0$.

\begin{cor}
Let $a\in S_{\H^d}(\mu_1)$ and $b\in S_{\H^d}(\mu_1)$ then
 \begin{eqnarray*}
 \sigma(a^*) & = & \overline{\sigma( a)}+{1\over 2}
 \sum_{1\leq j\leq d}( Z_j{\mathcal T}_j+\overline Z_j{\mathcal T}_j^*)\overline {\sigma( a)} \\
& &
 +{1\over 8|\lam|}\sum_{1\leq j,k\leq d}( Z_jT_j +\overline Z_j    T_j^*) ( Z_kT_k+\overline Z_k    T_k^*)  \overline {\sigma( a) }\\
 & & -{1\over i}   S\, \overline{\partial_\lam\sigma( a)} + \sigma(\tilde r_1)
  \end{eqnarray*}
and similarly
  \begin{eqnarray*}
\sigma\left( a\;   \#_{\H^d} b\right) & = & \sigma ( b) \,\#_\lam\, \sigma ( a ) +
 {1\over 2}\sum_{1\leq j\leq d}\left(
 Z_j\sigma(b)\, \#_\lam \,{\mathcal T}_j\sigma(a)+ \overline Z_j \sigma(b)\, \#_\lam\, {\mathcal T}_j ^*\sigma(a)
 \right)\\
  &+ & {1\over 8}
    \sum_{1\leq j,k\leq d}
 \Bigl(Z_jZ_k\sigma(b)\,\#_\lam\,{\mathcal T}_j{\mathcal T}_k \sigma(a)
  + \overline Z_j\overline Z_k \sigma(b)\, \#_\lam\,{\mathcal T}_j^*{\mathcal T}_k^* \sigma(a)   \\ & + &
     Z_j\overline Z_k\sigma(b)\, \#_\lam \,{\mathcal T}_j {\mathcal T}_k^*\sigma(a)+ \overline Z_j Z_k\sigma( b)\, \#_\lam \,  {\mathcal T}_j^* {\mathcal T}_k a\Bigr) -{1\over i} S\sigma(b)\, \#_\lam\,\partial_\lam\sigma(a)
+\sigma(\tilde r_2)
 \end{eqnarray*}
 where $\tilde r_1$ (resp. $\tilde r_2$)
 depends only on  ${\mathcal Z}^\alpha a$ (resp. ${\mathcal Z}^\alpha b$) for $|\alpha|\geq 3$ and  where for all functions~$f=f(\xi,\eta)$ and~$g=g(\xi,\eta)$
  $$\displaylines{
 \forall \Theta \in \R^{2d}, \quad  f\,\#_\lam\, g (\Theta) \eqdefa
  (\pi\lam)^{-2d}
  \int_{\R^{2d}} {\rm e}^{-{2i\over\lam}\omega[\Theta-\Theta_1, \Theta-\Theta_2] }
 f(\Theta_1)g(\Theta_2)d \Theta_1\,d \Theta_2,\cr
  {\mathcal T}_j f\eqdefa{1\over i}\partial_{\eta_j}f -\partial_{\xi_j}f .\cr}$$

\end{cor}

The proof of the corollary is straightforward by \aref{def:sigma(a)} and \aref{sharp}.

Let us now prove Proposition~\ref{prop:adjcompo}.

\begin{proof}
It turns out that the proof of the asymptotic formula for the
composition and the adjoint are identical, so let us concentrate
on the product from now on.\\

In view of \refeq{eq:ka} and \refeq{defkernel} page~\pageref{defkernel}, we can write
$$\displaylines{ \bigl({\rm Op}(a)\circ {\rm
Op}(b)\bigr)\,f(w)=\left(\frac{2^{d-1}}{\pi^{d + 1}} \right)^2
\int\,{\rm tr}\left(u^\lam_{w^{-1}w'}\circ A_\lam(w)\right) \,{\rm tr}\left(u^{\lam'}_{(w')^{-1}w''}\circ
B_{\lam'}(w')\right)\hfill\cr\hfill \times
f(w'')|\lam|^d\,|\lam'|^dd\lam \,d\lam'\,d w'\,d w''\cr} $$ with
$$A_\lam(w)=J_\lam \,op^w(a(w,\lam))\,J_\lam^*\;\;{\rm
and}\;\;B_\lam(w)=J_\lam\,op^w(b(w,\lam))\,J_\lam^*.$$ Now, we
shall take into account the framework of the Heisenberg group and
use the dilation~$\delta_t(w^{-1}w')$,~$t\in[0,1]$ (see~\aref{def:dilation} page~\pageref{def:dilation})  to
transform~$b(w',\cdot )$ by a Taylor expansion:
$$\displaylines{\qquad
 b(w',\lam,y,\eta)=b\left(w \delta_1(w^{-1}w'),\lam,y,\eta\right)=b(w,\lam,y,\eta )+\left({d\over dt}\, b\left(w \delta_t(w^{-1}w'),\lam,y,\eta\right)\right)_{|t=0}\hfill\cr\hfill
 +{1\over 2}\left(\frac{d^2}{dt^2}\,
b\left(w \delta_t(w^{-1}w'),\lam,y,\eta\right)\right)_{|t=0}
 +{1\over 2}\int_0^1(1-t)^2\frac{d^3}{dt^3}\,
b\left(w \delta_t(w^{-1}w'),\lam,y,\eta\right)\, dt. \qquad\cr} $$
Setting~$\tilde w=(\tilde z,\tilde s)=w^{-1}w'$, we get by the group rule\refeq{lawH},
$$\displaylines{\qquad{d\over dt} b(w\delta_t(\tilde w))= 2t\tilde sSb(w\delta_t(\tilde w))+\sum_{1\leq j\leq d}\Bigl[
\tilde x_j\left(\partial_{x_j}b(w\delta_t(\tilde w))+2y_j\partial_sb(w\delta_t(\tilde w))\right) \hfill\cr\hfill+ \tilde y_j \left(\partial_{y_j} b(w\delta_t(\tilde w))-2x_j\partial_sb(w\delta_t(\tilde w))\right)\Bigr].\qquad\cr}$$
This leads by
straightforward computations to
\begin{eqnarray*}
\left({d\over dt}
b\left(w\delta_t(w^{-1}w'),\lam,y,\eta\right)\right)_{|t=0} & = & \sum_{1\leq
j\leq d} (\tilde z_jZ_j+\overline{\tilde
z_j}\,\overline Z_j) b(w,\lam,y,\eta)\\
\left(\frac{d^2}{dt^2}
b\left(w\delta_t(w^{-1}w'),\lam,y,\eta\right)\right)_{|t=0}& = &
 \sum_{1\leq
j,k\leq d} \left[(\tilde z_j Z_j+\overline{\tilde z}_j\overline Z_j)\circ(\tilde z_k Z_k+\overline{\tilde z}_k\overline Z_k)\right]
b(w,\lam,y,\eta) \\ & &
\qquad\qquad + 2 \tilde s\,Sb(w,\lam,y,\eta).
\end{eqnarray*}
 Therefore, we deduce that
$$B_\lam(w')=C_\lam(w,w')+ R_\lam(w,w')$$ where $R_\lam$ depends
only on derivatives of order~$3$ of~$b$ and $C_\lam(w,w')$ depends
polynomially on~$\tilde
w$:
\begin{equation}\label{expansion}C_\lam(w,w')\eqdefa B_\lam (w)+
C_\lam^{(1)}(w)\cdot (\tilde z,\overline{\tilde z})+
C_\lam^{(2)}(w)(\tilde z,\overline{\tilde z}) \cdot (\tilde
z,\overline{\tilde z}) +  \tilde s\,
C_\lam^{(3)}(w),\end{equation} where $C_\lam^{(1)}(w)$ is the $2d$
dimensional vector-valued operator
$$C_\lam^{(1)} \eqdefa\left(ZB_\lam(w),\overline Z B_\lam(w)\right),$$
while~$C_\lam^{(2)}(w)$ is the $2d \times 2d$ matrix-valued operator
$$C_\lam^{(2)} \eqdefa \frac 1 2 \left[(Z,\overline Z)\otimes
(Z,\overline Z) \right]B_\lam(w)$$ and~$C_\lam^{(3)}(w) \eqdefa
SB_\lam(w)$.\\

To summarize  $\bigl({\rm Op}(a)\circ{\rm Op}(b)\bigr)f(w)$ is the
sum of two terms: $$\displaylines{\bigl({\rm Op}(a)\circ{\rm
Op}(b)\bigr)f(w)=(I)+(J) \cr}$$ with
 $$\displaylines{
(I)=\left(\frac{2^{d-1}}{\pi^{d + 1}} \right)^2 \int\,{\rm
tr}\left(u^\lam_{w^{-1}w'} A_\lam(w)\right) \,{\rm
tr}\left(u^{\lam'}_{(w')^{-1}w''} C_{\lam'}(w,w')\right)
f(w'')|\lam|^d\,|\lam'|^dd\lam d\lam'd w'd w''.\cr }$$ Let us now
focus on the  term~$(I)$  which  will give the  terms of the
asymptotics in which we are interested. \\

Let us begin by the study of  the contribution $(I)_0$ of the  term  of
degree $0$ of the polynomial function~$C_\lam(w,w')$. By\refeq{expansion}, we get
$$\displaylines{(I)_0\eqdefa \left(\frac{2^{d-1}}{\pi^{d + 1}}
\right)^2 \int\,{\rm tr}\left(u^\lam_{w^{-1}w'} A_\lam(w)\right)
\,{\rm tr}\left(u^{\lam'}_{(w')^{-1}w''}  B_{\lam'}(w)\right)
f(w'')|\lam|^d\,|\lam'|^dd\lam d\lam'd w'd w''\cr
=\left(\frac{2^{d-1}}{\pi^{d + 1}} \right)^2\int\,{\rm
tr}\left(u^\lam_{w^{-1}w''}  u^\lam_{(w'')^{-1}w'} \,{\rm
tr}\left(u^{\lam'}_{(w')^{-1}w''}  B_{\lam'}(w)\right)\,
 A_\lam(w)\right)
\hfill\cr\hfill \times f(w'')|\lam|^d\,|\lam'|^dd\lam d\lam'd w'd
w''.\cr}$$ The change of variables $w'\mapsto w''w'$ turns the
integral~$(I)_0$ into
$$
\left(\frac{2^{d-1}}{\pi^{d + 1}}
\right)^2\int_{w'',\lambda}\,{\rm tr}\Bigl(u^\lam_{w^{-1}w''}
   \left[ \,\int  u^\lam_{w'}
\,{\rm tr}\left(u^{\lam'}_{(w')^{-1}}
B_{\lam'}(w)\right)\,|\lam'|^dd\lam'd w'\right]
 A_\lam(w)\Bigr)f(w'')|\lam|^d\,d\lam d w''. $$
 By the inverse Fourier formula,
we obtain that the term between   brackets is $$\,\int
u^\lam_{w'} \,{\rm tr}\left(u^{\lam'}_{(w')^{-1}}
B_{\lam'}(w)\right)\,|\lam'|^dd\lam'd w'=\left(
\frac{2^{d-1}}{\pi^{d + 1}} \right)^{-1}B_\lam(w),$$ which gives
$$(I)_0= \frac{2^{d-1}}{\pi^{d + 1}} \int \,{\rm
tr}\left(u^\lam_{w^{-1}w''}B_\lam(w) A_\lam(w)\right)
f(w'')|\lam|^d\,d\lam d w''.$$
 We then use
classical Weyl symbolic calculus to write $$op^w(b(w,\lam))\circ
op^w(a(w,\lam))=op^w((b\# a)(w,\lam)).$$
 Thus we have
$$B_\lam(w)\circ A_\lam(w)=J_\lam ^*op^w((b\# a)(w,\lam))
J_\lam,$$ whence $$(I)_0=\frac{2^{d-1}}{\pi^{d + 1}}
\int \,{\rm tr}\left(u^\lam_{w^{-1}w''} J_\lam^*
 op^w((b\# a)(w,\lambda))J_\lam\right) f(w'')|\lam|^d\,d\lam d
 w'',$$
which gives thanks to \refeq{eq:ka} and \refeq{defkernel} the
first term in the asymptotic formula for the composition.

 Let us now consider the second term of the asymptotic expansion which comes from the term of order $1$ of
 the polynomial function $C_\lam(w,w')$. To treat this term,
 we shall  use the following relations for $1\leq j\leq d$,
 \begin{eqnarray}\nonumber
\tilde z_j {\rm tr} \left(u^\lam_{\tilde w}J_\lam ^*
op^w(a(w,\lam))J_\lam\right) & = & {1\over 2\sqrt{|\lam|}} {\rm
tr} \left(u^\lam_{\tilde w} J^*_\lam op^w(\{a,-i\xi_j+\,{\rm
sgn}(\lam)\eta_j\})J_\lam\right)
  \\ \label{Tja}
  & =& {1\over 2\sqrt{|\lam|} }{\rm tr} \left(u^\lam_{\tilde w}J_\lam ^*
op^w(T_j a(w,\lam))\right) \\
\nonumber
\overline{\tilde z_j }{\rm tr} \left(u^\lam_{\tilde w}J_\lam ^*
op^w(a(w,\lam))J_\lam\right) &  = & - {1\over 2\sqrt{|\lam|}} {\rm
tr} \left(u^\lam_{\tilde w} J^*_\lam op^w(\{a,i\xi_j+\,{\rm
sgn}(\lam)\eta_j\})J_\lam\right)\\
\label{Tj*a}
& = & {1\over 2\sqrt{|\lam|}} {\rm tr} \left(u^\lam_{\tilde w}J_\lam ^*
op^w(T_j^* a(w,\lam))\right)
\end{eqnarray}
that come respectively from \aref{Z} and \aref{overlineZ} page~\pageref{overlineZ}.

This allows to write the second term under the following form
$$\displaylines{ (I)_1\eqdefa\left(\frac{2^{d-1}}{\pi^{d + 1}}
\right)^2 \int\,{\rm tr}\left(u^\lam_{w^{-1}w'} A_\lam(w)\right)
\,{\rm tr}\left(u^{\lam'}_{(w')^{-1}w''}(\tilde z,\overline{\tilde
z})\cdot C_{\lam'}^{(1)}(w)\right) \hfill\cr\hfill\times
f(w'')|\lam|^d\,|\lam'|^dd\lam d\lam'd w'd w''\cr\hfill
={1\over2\sqrt{|\lam|}}
\left(\frac{2^{d-1}}{\pi^{d + 1}} \right)^2\sum_{1\leq j\leq d} \int\,{\rm
tr}\left(u^\lam_{w^{-1}w'} J_\lam^* op^w(T_ja(w,\lam))J_\lam\right)
\,{\rm tr}\left(u^{\lam'}_{(w')^{-1}w''}
Z_jB(w,\lam')\right)\qquad\cr\hfill\times
f(w'')|\lam|^d\,|\lam'|^dd\lam d\lam'd w'd w''\cr\hfill
+{1\over2\sqrt{|\lam|}}
\left(\frac{2^{d-1}}{\pi^{d + 1}} \right)^2\sum_{1\leq j\leq d} \int\,{\rm
tr}\left(u^\lam_{w^{-1}w'} J_\lam^* op^w(T_j^*a(w,\lam))J_\lam\right)
\,{\rm tr}\left(u^{\lam'}_{(w')^{-1}w''}
\overline Z_jB(w,\lam')\right)\qquad\cr\hfill\times
f(w'')|\lam|^d\,|\lam'|^dd\lam d\lam'd w'd w''.\cr
 }$$
 Therefore, arguing as for the first term, we get
$$\displaylines{\qquad (I)_{1}= {1\over 2\sqrt{|\lam|}}\frac{2^{d-1}}{\pi^{d + 1}}\sum_{1\leq j\leq d}
\int \,{\rm tr}\Biggl(u^\lam_{w^{-1}w''}J_\lam^*
op^w\Bigl( Z_jb(w,\lam)\#  T_ja(w,\lam) \hfill\cr\hfill + \overline Z_j
b(w,\lam)\#  T_j^*a(w,\lam)\Bigr)J_\lam\Biggr)
f(w'')|\lam|^d\,d\lam d w'',\qquad\cr}$$ which leads by
\refeq{eq:ka} and \refeq{defkernel} to the second term in the
asymptotic formula for the composition.

In order to compute the third term of the expansion, we shall
consider the terms of order $2$ of the polynomial $C_\lam(w,w')$
and use Lemma~\ref{lemme:as} stated page~\pageref{lemme:as}. First, let us recall that due to\refeq {expansion}, we have $$\displaylines{
(I)_2\eqdefa\left(\frac{2^{d-1}}{\pi^{d + 1}} \right)^2 \int\,{\rm
tr}\left(u^\lam_{w^{-1}w'} A_\lam(w)\right) \,{\rm
tr}\left(u^{\lam'}_{(w')^{-1}w''} \left(\tilde s
C_{\lam'}^{(3)}(w)+C_{\lam'}^{(2)}(w)(\tilde z,\overline{\tilde
z})
 \cdot(\tilde z,\overline{\tilde z}) \right)\right)
\hfill\cr\hfill \times f(w'')|\lam|^d\,|\lam'|^dd\lam d\lam'd w'd
w''\cr\hfill \qquad\cr}$$
  where~$C_{\lam'}^{(3)}(w) =
SB_\lam(w) $ and $C_\lam^{(2)}= \frac 1 2 \left[(Z,\overline
Z)\otimes (Z,\overline Z) \right]B_\lam(w)$.

We first focus on the term in $C^{(2)}_{\lam'}$. Let us call $(I)_{2,1}$ its contribution, we have
$$\displaylines{
(I)_{2,1} \eqdefa  \left(\frac{2^{d-1}}{\pi^{d + 1}} \right)^2\sum_{1\leq j,k\leq d} \int\,{\rm
tr}\left(u^\lam_{w^{-1}w'} A_\lam(w)\right) \hfill\cr\hfill \times\,{\rm
tr}\left(u^{\lam'}_{(w')^{-1}w''}
\left( (\tilde z_jZ_j+\overline{\tilde z_j}\overline Z_j)(\tilde z_kZ_k+\overline{\tilde z_k}\overline Z_k)B_{\lam'}(w) \right)\right)
 f(w'')|\lam|^d\,|\lam'|^dd\lam d\lam'd w'd
w'.
\cr}$$
We treat those terms as those of $(I)_{1}$. We shall explain the argument for one of those terms and leave the analysis of the other terms to the reader. Set
$$\displaylines{\qquad (I)_{2,j,k}\eqdefa
 \left(\frac{2^{d-1}}{\pi^{d + 1}} \right)^2\int\,{\rm
tr}\left(u^\lam_{w^{-1}w'} A_\lam(w)\right) \,{\rm
tr}\left(u^{\lam'}_{(w')^{-1}w''} \left( \tilde z_j\overline{\tilde z_k}Z_j\overline Z_k)B_{\lam'} \right)\right)\hfill\cr\hfill\times\,
 f(w'')|\lam|^d\,|\lam'|^dd\lam d\lam'd w'd
w'.\cr}$$
Using \aref{Tja} and \aref{Tj*a}, we obtain
\begin{eqnarray*}
\tilde z_j\,\overline{\tilde z_k} \,{\rm tr} \left(u^\lam_{\tilde w}J_\lam ^*
op^w(a(w,\lam))J_\lam\right) & = &
{\rm tr} \left(u^\lam_{\tilde w}J_\lam ^*
op^w(T_jT^*_ka(w,\lam))J_\lam\right)
\end{eqnarray*}
whence, arguing as for $(I)_{1}$
$$\displaylines{\qquad
(I)_{2,j,k} =
{1\over2\sqrt{|\lam|}}
\left(\frac{2^{d-1}}{\pi^{d + 1}} \right)^2\int\,{\rm
tr}\left(u^\lam_{w^{-1}w'} J_\lam^* op^w(T_jT_k^*a(w,\lam))J_\lam\right)
\hfill\cr\hfill
{\rm tr}\left(u^{\lam'}_{(w')^{-1}w''}
Z_j\overline Z_kB(w,\lam')\right)
f(w'')|\lam|^d\,|\lam'|^dd\lam d\lam'd w'd w''\qquad\cr
 =  {1\over 2\sqrt{|\lam|}}\frac{2^{d-1}}{\pi^{d + 1}}
\int\,{\rm tr}\Biggl(u^\lam_{w^{-1}w''}J_\lam^*
op^w\Bigl( Z_j\overline Z_kb(w,\lam)\#  T_jT_k^*a(w,\lam)\Bigr)J_\lam\Biggr)
f(w'')|\lam|^d\,d\lam d w''.\cr}$$

To deal with the last term $$\displaylines{ \left(\frac{2^{d-1}}{\pi^{d
+ 1}} \right)^2 \int\,{\rm tr}\left(u^\lam_{w^{-1}w'}
A_\lam(w)\right) \,{\rm tr}\left(u^{\lam'}_{(w')^{-1}w''} \tilde s\;
C_{\lam'}^{(3)}(w)\right) f(w'')|\lam|^d\,|\lam'|^dd\lam d\lam'd
w'd w''\cr\hfill \qquad\cr}$$ let us apply Lemma~\ref{lemme:as} (see page~\pageref{lemme:as})
writing $$\displaylines{ \left(\frac{2^{d-1}}{\pi^{d + 1}}
\right)^2 \int\,{\rm tr}\left(u^\lam_{w^{-1}w'} A_\lam(w)\right)
\,{\rm tr}\left(u^{\lam'}_{(w')^{-1}w''} \tilde s
C_{\lam'}^{(3)}(w)\right) f(w'')|\lam|^d\,|\lam'|^dd\lam d\lam'd
w'd w''\cr\hfill = \frac 1 i \left(\frac{2^{d-1}}{\pi^{d + 1}}
\right)^2 \int\,{\rm tr}\left(u^\lam_{w^{-1}w'} J_\lam^*
op^w\left( g(w,\lam)\right)J_\lam\right) \,{\rm
tr}\left(u^{\lam'}_{(w')^{-1}w''}
C_{\lam'}^{(3)}(w)\right)f(w'')|\lam|^d\,d\lam d w''.\qquad\cr}$$
where $g$ is the symbol of  $ S_{\H ^d}(\mu_1 )$  given by
\aref{formulepourg} (in particular we have $\sigma(g)=-\partial_\lam
\left(\sigma(a)\right)$.

Finally, arguing as before we get

$$\displaylines{ \left(\frac{2^{d-1}}{\pi^{d + 1}} \right)^2
\int\,{\rm tr}\left(u^\lam_{w^{-1}w'} A_\lam(w)\right) \,{\rm
tr}\left(u^{\lam'}_{(w')^{-1}w''} \tilde s \,
C_{\lam'}^{(3)}(w)\right) f(w'')|\lam|^d\,|\lam'|^dd\lam d\lam'd
w'd w''\cr\hfill = \frac 1 i \frac{2^{d-1}}{\pi^{d + 1}}
\int\,{\rm tr}\left(u^\lam_{w^{-1}w''} J_\lam^* op^w\left( S\, b
(w,\lam)\#  g(w,\lam)\right)J_\lam\right) f(w'')|\lam|^d\,d\lam
d w''.\qquad\cr}$$

This ends the proof of the asymptotic formula for the composition.
 \end{proof}


 \chapter{Littlewood-Paley    theory}\label{prelimins}
\setcounter{equation}{0}


In this chapter, we shall study various properties related to
Littlewood-Paley operators, and their link with various types of
pseudodifferential operators.

In the first  section,  we  focus on the Littlewood-Paley theory
available on the Heisenberg group. Similarly to the~$\R^d$ case,
this theory enable us to split tempered distributions into a
countable sum of smooth functions frequency localized in a ball or
a ring (see Definition \ref{definlocfreqheis} for more details).
In the second  section, we recall some basic facts about Besov
spaces and introduce paradifferential calculus. Like in the~$\R^d$
case, it turns out that Sobolev and H\"older spaces come up as
special cases of  Besov spaces.  The paraproduct algorithm on the
Heisenberg group is similar to the paraproduct algorithm on
~$\R^d$  built by J.-M. Bony~\cite{bony} and allows to transpose
to the Heisenberg group a number of classical results (see for
instance ~\cite{bcg},~\cite{bg} ~\cite{bg2} and~\cite{bgx}). As
already mentioned in Section~\ref{examples} of
Chapter~\ref{fundamental}, the Littlewood-Paley truncation operators are
Fourier multipliers defined using operators which are functions of
the harmonic oscillator.  Therefore, it is important for our
theory to be able to  analyze the Weyl symbol of  such operators;
this is achieved thanks to Mehler's formula in the third  section
where we compare Littlewood-Paley operators with
pseudodifferential operators; this will be  of crucial use for
the next chapter. Finally in the last paragraph we introduce
another dyadic decomposition, in the variable~$\lam$ only, which
will also turn out to be a necessary ingredient in the proof of
Theorem~\ref{contHs}.

\section{Littlewood-Paley operators}\label{littlewoodpaley}
 \setcounter{equation}{0}
 In~\cite{bgx} and~\cite{bg} a dyadic partition of
unity   is built on the Heisenberg group~$\H^d $, similar to the
one defined in the classical~$\R^d$ case. A significant
application of this decomposition is the definition of Sobolev
spaces (and more generally Besov spaces) on the Heisenberg group
in the same way as in the classical case.

 Let us first define the concept of
localization procedure in   frequency space, in the framework of
the Heisenberg group. We start by giving the definition in the
case of smooth functions. The general case follows classically
(see~\cite{bgx} or~\cite{bg}) by   regularizing by convolution, as
shown in the remark following the definition.  We have defined, for any set~$B$, the operator~$ {\bf 1}_{D_\lam^{-1}  B} $ on~${\mathcal H}_\lam$   by
$$
\forall f \in  {\mathcal S} (\H^d), \: \forall \alpha \in \N^d,  \quad  {\mathcal F}(f)(\lam) {\bf 1}_{D_\lam^{-1}  B}  F_{\alpha,\lam}\eqdefa
  {\bf 1}_{(2|\alpha| + d) ^{-1}  B} (\lam) \, {\mathcal F}(f)(\lam)  F_{\alpha,\lam}.
$$

\begin{defin}
\label{definlocfreqheis} {  Let~${\mathcal C}_{(r_1,r_2)}=
{\mathcal C}(0,r_1,r_2) $ be a ring   and~${\mathcal B}_r=
{\mathcal B}(0,r) $ a ball of~$\R$ centered at the origin. A
function~$f$ in~$ {\mathcal S}(\H^d)$ is said to be
\begin{itemize}
\item   frequency
localized in the ball~$2^{p}{\mathcal B}_{\sqrt{r}}$, if
 \[ {\mathcal F} (f)(\lam) =  {\mathcal F}
(f)(\lam) {\bf 1}_{D_\lam^{-1} 2^{2p}{\mathcal
B}_{\sqrt{r}}}(\lam) ;  \]

\item   frequency
localized in the ring~$2^{p}{\mathcal C}_{(\sqrt{ r_1},
\sqrt{r_2})}$, if
 \[ {\mathcal F}
(f)(\lam) =  {\mathcal F}(f)(\lam){\bf 1}_{D_\lam^{-1}
2^{2p}{\mathcal C}_{( r_1, r_2)}}(\lam).
\]
\end{itemize}}
\end{defin}
 In the case of a tempered distribution~$u$, we shall say  that~$u$
is frequency localized in the ball~$2^{p}{\mathcal B}_{\sqrt{r}}$
(respectively in the ring~$2^{p}{\mathcal C}_{(\sqrt{
r_1},\sqrt{r_2})}$), if
 $$ u\star f = 0$$
 for any radial function~$f\in {\mathcal S}(\H^d)$ satisfying~${\mathcal F} (f)(\lam)
 {\bf 1}_{D_\lam^{-1}2^{2p}{\mathcal B}_{\sqrt{r}}}=0$
(respectively for any~$f$ in~${\mathcal S}(\H^d)$ satisfying~${\mathcal F} (f)(\lam)
 {\bf 1}_{D_\lam^{-1}2^{2p}{\mathcal C}_{(\sqrt{r_1},\sqrt{r_2}})}=0$).
 In other words~$u$ is frequency
localized in the ball~$2^{p}{\mathcal B}_{\sqrt{r}}$ (respectively
in the ring~$2^{p}{\mathcal C}_{(\sqrt{ r_1},\sqrt{r_2})}$), if
and only if, $$ u = u \star \phi_p, $$ where~$\phi_p =
2^{Np}\phi(\delta_{2^{p}}\cdot)$, and~$\phi $ is a radial function
in~${\mathcal S}(\H^d)$ such that
\[
{\mathcal F}(\phi)(\lam) = {\mathcal F} (\phi)(\lam)R(D_\lam) ,
\]
with~$R$ compactly supported in  a ball (respectively an ring)
of~$\R$ centered at zero.

Let us now recall the dyadic decomposition and paradifferential
techniques introduced in~\cite{bgx} and~\cite{bg},  which we refer
to for all details and proofs.

\begin{prop}
\label{dyaheisenberg} { Let us denote by~${\mathcal B}_0$ and by~$
{\mathcal C}_0$ respectively the ball~$\left\{\tau \in \R, \:
|\tau| \leq \frac{4}{3}\right\} $ and the ring~$\left\{\tau \in
\R, \: \frac{3}{4} \leq |\tau| \leq \frac{8}{3}\right\}$. Then
there exist two radial functions~$\widetilde R^*$ and~$R^* $ the
values of which are in the interval~$[0,1]$, belonging
respectively to~${\mathcal D}({\mathcal B}_0)$ and to~${\mathcal
D}({\mathcal C}_0)$  such that
\begin{equation}\label{funddylph1} \forall \tau \in \R, \quad
\widetilde R^*(\tau) + \sum_{p \geq 0} R^*(2^{-2p}\tau) = 1
\end{equation}
and satisfying the support properties $$ \displaylines{ |p-p'|\geq
1 \Rightarrow \mbox{supp}~R^*(2^{-2p}\cdot)\cap \mbox{supp}~
R^*(2^{-2p'}\cdot)=\emptyset\cr
 p\geq 1 \Rightarrow \mbox{supp}~
\widetilde R^*\cap \mbox{supp}~ R^*(2^{-2p}\cdot) = \emptyset.
\cr} $$
 Besides, we have
\begin{equation}
\label{lpfond5} \forall  \tau \in \R\,,\ \frac1 2 \leq \widetilde
R^*(\tau)^2 + \sum_{p\geq 0} R^*(2^{-2p}\tau)^2  \leq 1.
\end{equation} }\end{prop}
 The dyadic blocks~$\D_p$ and the low frequency cut-off operators~$S_p$ are defined as follows similarly to the~$\R^d$ case.
\begin{defin}\label{defLP}
 {  We define   the Littlewood-Paley operators
 associated with the functions~$\widetilde R^*$ and~$R^*$,   for~$ p \in \ZZZ$,  by the following definitions in Fourier variables:
\begin{eqnarray*}
\forall p\in \N , \quad {\mathcal F}(S_pf)(\lam)& = &{\mathcal
F}(f)(\lam) \widetilde R^*\left( 2^{-2p} D_\lambda \right),  \\
\forall p\in \N , \quad{\mathcal F}(\D_pf)(\lam) &=& {\mathcal
F}(f)(\lam) R^*\left(2^{-2p} D_\lambda \right), \\
 {\mathcal F}(\D_{-1}f)(\lam)&=& {\mathcal F}(S_0 f)(\lam),\\
 \forall p\leq -2, \quad {\mathcal F}(\D_pf)(\lam) &=&0 .
\end{eqnarray*}}
\end{defin}

 The operator $S_pf$ may be alternately defined by $$ S_pf =
\sum_{q\leq p-1}\Delta_q f. $$

 Since ${\mathcal
F}(\D_pf)(\lam) = {\mathcal F}(f)(\lam) R^*(2^{-2p} D_\lambda)$,
it is clear that the function~$\D_pf$ is frequency localized in a
ring of size~$2^p$. Along the same lines, one can notice that the
function~$S_pf$ is frequency localized in a ball of size~$2^p$.

 Moreover, according to the fact that the Fourier transform
exchanges   convolution   and   composition, the operators~$\D_p$
and~$S_p$ commute with one another and with the Laplacian-Kohn
operator~$\Delta_{ \H^d  }$.

\begin{rem}\label{nophiinnotation}
 For simplicity of notation, we
do not indicate that~$S_p$ depends on~$\widetilde R^*$ and
that~$\Delta_p$ depends on~$  R^*$. That is due to the fact that
according to Lemma~\ref{indheish} below, one can change the basis
functions (hence the Littlewood-Paley operators), keeping only the
fact that one is supported near zero and the other is supported
away from zero and satisfying~(\ref{funddylph1}),  while
conserving  equivalent norms for the   function spaces based on
those operators.

\end{rem}

 It was proved in~\cite{hulanicki}, in the more general context of nilpotent Lie groups, that
 there are radial functions
of~$ {\mathcal
  S}(\H^d)$, denoted~$ \psi $ and~$\varphi$
such that $$ {\mathcal F}(\psi)(\lam) = \widetilde
R^*(D_\lam) \quad \mbox{and} \quad {\mathcal
F}(\varphi)(\lam)= R^*(D_\lam).
$$
We also refer to~\cite{bgx} and~\cite{bg} for a different proof in the case of the Heisenberg group, the ideas of which will be used below to prove Lemma~\ref{lem.est}. Using the scaling of the Heisenberg group, it is easy to see that $$ \Delta_pu=
u\star2^{Np}\varphi(\delta_{2^{p}}\cdot)
 \quad \mbox{and} \quad
S_pu= u\star2^{Np}\psi(\delta_{2^{p}}\cdot) $$ which implies by Young's inequalities that
 those operators map~$L^q$ into~$L^q$ for all~$ q \in
[1,\infty]$ with norms which do not depend on~$p$.

Let us also notice that due to \refeq{zjdelta} (see page~\pageref{zjdelta}), if P is a left
invariant vector fields then $$P(\Delta_pu) =
2^p(u\star2^{Np}P(\varphi)(\delta_{2^{p}}\cdot)).$$ This property
is the heart of the matter in the estimate of the action of left
invariant vector fields on frequency localized functions (see
Lemma \refer{lem:lech} below).\\

In view of Mehler's formula (see
\cite{feynman}) and Lemma 4.5 in \cite{fermanian}, one can prove
that  the Littlewood-Paley operators on the Heisenberg group are
pseudodifferential operators in the sense of
Definition\refer{definpseudo}. This is discussed in
Section~\ref{LPsymbolappendix} below.

\section{Besov spaces}\label{besovspaceslp}
 \setcounter{equation}{0}

  Along the same lines as in  the~$ \R^d$ case,
 we can define Besov spaces on
the Heisenberg group~(see~\cite{bgx}).

\begin{defin}
\label{besovheis} {  Let  $s\in\R$ and $(q,r)\in[1,\infty]^2.$ The
Besov space~$B^s_{q, r}(\H^d)$ is the space of tempered
distributions~$u$ such that $$ \|u\|_{B^s_{q, r}(\H^d)} \eqdefa
\left\|2^{p s}\| \Delta_{p} u\|_{L^{q}({\H}^d)}\right\|_{\ell^r }<
\infty.$$ }
\end{defin}

\begin{rem}\label{besrem} It is also possible to characterize these
spaces using only the operator $S_p$ : for $s>0$, we have
\begin{equation}\label{bes1}
\|f\|_{B^s_{q,r}(\H^d)} \sim \left\|
2^{sp}\|({\rm Id}-S_p)f\|_{L^q(\H^d)} \right\|_{\ell^r},
\end{equation}
and for $s<0$,
\begin{equation}\label{bes2}
\|f\|_{B^s_{q,r}(\H^d)} \sim \left\| 2^{sp}\|S_pf\|_{L^q(\H^d)}
\right\|_{\ell^r},
\end{equation}
where~$\sim$ stands for equivalent norms.
\end{rem}
 It is easy to see that for any real number~$\rho$, the operators~$(-\Delta_{\H^d})^\rho$ and~$({\rm Id}-\Delta_{\H^d})^\rho$ are continuous from~$B^s_{q, r}(\H^d)$ to~$B^{s-2\rho}_{q, r}(\H^d)$.
Note that Besov spaces on the Heisenberg group contain Sobolev and
H\"older spaces. Indeed, by\refeq{lpfond5} and  the
Fourier-Plancherel equality \refeq{Plancherelformula},  the Besov
space~$B^s_{2,2}(\H^d)$ coincides with  the Sobolev
space~$H^s(\H^d)$.
  When~$s\in\R^+\setminus\N,$ one
can show that~$B^s_{\infty,\infty}(\H^d)$ coincides with the
 H\"older space~$C^{s}(\H^d)$ introduced in Definition\refer{definholderheis}.

 Let us point out that a distribution~$f$ belongs to~$B^s_{q, r}(\H^d)$ if and only if there exists some constant~$C$ and some nonnegative sequence~$(c_p)_{p \in \N}$ of the unit sphere of~$\ell^r(\N)$ such that
 \begin{equation}\label{beseq}\forall p \in \N, \quad 2^{ps}\|\Delta_p f\|_{L^q(\H^d)} \leq C c_p.\end{equation}
 This fact will be  useful  in what follows.

Arguing as in the classical case, one can prove using this theory
many results, such as Sobolev embeddings, refined Sobolev and Hardy inequalities (see~\cite{bg},\cite{bcg}).
This is due to the fact that the dyadic unity decomposition on the
Heisenberg group behaves as the classical Littlewood-Paley
decomposition. The key argument lies on the following estimates
called Bernstein inequalities, proved in~\cite{bg}.
 \begin{lemme}
\label{lem:lech} { Let~$r $ be a positive real number.  For any nonnegative integer~$k$,
there exists a positive constant~$C_k$   so that, for any couple
of real numbers~$(a,b)$ such that~$1 \leq a\leq b\geq \infty$ and any
function~$u$ of~$L^{a}(\H^d)$ frequency localized in the
ball~$2^{p}{{\mathcal B}}_{\sqrt{r}}$, one has
\begin{equation}
\label{eq:lech1} \sup_{|\beta |  = k} \norm{{\mathcal X}^\beta
u}{L^b(\H^d)}\leq C_k 2^{p N (\frac{1}{a}-\frac{1}{b})+ pk }
\norm u{L^{a}(\H^d)},
\end{equation}
where~${\mathcal X}^\beta $ denotes a product of~$|\beta |$
vectors fields of type\refeq{prodreal}, page~\pageref{prodreal}.
 }
\end{lemme}

 Let us also point out that the definition of~$B^s_{p, r}(\H^d)$
 is independent of  the dyadic  partition of unity chosen to
define this space. This is due to the following lemma proved
in~\cite{bgx}.
\begin{lemme}
\label{indheish}
 {  Let  $s\in\R$ and $(p,r)\in[1,\infty]^2.$
 Let~$ (u_p)_{p \in \N} $ be a sequence
of~$L^q(\H^d)$  frequency localized in a ring of size~$2^{p}$
  satisfying
\[ \bigl\|2^{ps} \norm {u_p}{L^q(\H^d)}\bigr\|_{\ell^r(\N)} <
 \infty, \]
 then~$\displaystyle u\eqdefa\sum_{p \in \N} u_p $ belongs to~$  B^s_{q, r}(\H^d) $ and we have
\[
  \norm u { B^s_{q, r}(\H^d)} \leq C_s \bigl\|2^{ps} \norm
{u_p}{L^q(\H^d)}\bigr\|_{\ell^r(\N)}.
\]}
\end{lemme}
Contrary to the~$\R^d$ case, there is no simple formula for the
Fourier transform of the product of two functions. The following
proposition (proved in~\cite{bg}) ensures that spectral
localization properties of the classical case are nevertheless preserved on the
Heisenberg group after the product has been taken.
\begin{prop}
\label{pro:couronnes} { Let~$r_2>r_1>0$ be two real numbers, let~$p$ and~$p'$ be two
integers, and let~$f$ and~$g$ be two functions of~${\mathcal S}'(\H^d)$
respectively frequency localized in the ring~$2^p {\mathcal
C}_{(\sqrt{r_1}, \sqrt{ r_2})}$ and~$2^{p'} {\mathcal
C}_{(\sqrt{r_1}, \sqrt{ r_2})}$. Then
\begin{itemize}
\item there exists a ring~${\mathcal C}'$ such that if~$p'-p
> 1$ then~$ f g $  is frequency localized
 in the ring~$2^{p'}{\mathcal   C}'$.
\item there  exists a
ball~${\mathcal B}'$ such that if~$|p'-p| \leq 1$, then ~$f g$ is
frequency localized
 in the ball~$2^{p'}{\mathcal B}'$.
 \end{itemize}}
\end{prop}

\begin{rem} The proof of this proposition is based on a careful use of the
link between the Fourier transform on the Heisenberg group and the
standard Fourier transform on~$\R^{2d+1}$. For a detailed proof,
see \cite{bg}.

 Proposition \refer{pro:couronnes} implies that
if two functions are spectrally localized on two rings
sufficiently far away one from the other, then their product stays
spectrally localized on a ring.\\
\end{rem}
Taking advantage of this result, one can transpose to the
Heisenberg group the paraproduct theory constructed by J.-M.
Bony~\cite{bony} in the classical case.  Let us consider two
tempered distributions~$u$ and~$v$ on~$\H^{d}$. We write
 $$ u= \sum_p \D_pu \quad \mbox{and}\quad v= \sum_q \D_qv.$$
 Formally, the product can be written as
 $$uv= \sum_{p,q} \D_pu \,\D_qv$$
 Paradifferential calculus is a mathematical tool for splitting the above sum into three parts: the first part concerns the indices~$(p,q)$ for which the size of the spectrum of~$\D_pu$ is small compared to the size of the one of~$\D_qv$. The second part is the symmetric of the first part and in the last part, we keep the indices~$(p,q)$ for which the spectrum of~$\D_pu$ and~$\D_qv$ have comparable sizes. This leads to the following definition.
 \begin{defin}{ We shall call paraproduct of~$v$ by~$u$ and shall denote by~$T_u v$ the following bilinear operator:
 \begin{equation}\label{para} T_u v \eqdefa  \sum_{q} S_{q-1}u \,\D_qv\end{equation}
 We shall call remainder of~$u$ and~$v$ and shall denote by~$R(u,v)$ the following bilinear operator:
 \begin{equation}\label{remaind} R(u,v) \eqdefa  \sum_{|p-q|\leq 1} \D_pu \,\D_qv\end{equation}
  }
 \end{defin}
 \begin{rem}
 Just by looking at the definition, it is clear that
\begin{equation}\label{decpara}uv = T_u v + T_v u + R(u,v). \end{equation}

According to Proposition \refer{pro:couronnes}, ~$S_{q-1}u \,\D_qv$ is frequency localized in a ring of size~$2^q$. But, for terms of the  kind~$\D_pu \,\D_qv$ with~$|p-q|\leq 1$, we have an accumulation of frequencies at the origin. Such terms are  frequency localized in a ball of size~$2^q$.

The way how the paraproduct and remainder act on Besov spaces is similar to the classical case. We refer to \cite{bg} for more details.
 \end{rem}
Taking advantage of this theory, one can prove the following
useful estimates.
\begin{lemme}\label{estfgHs}{Let~$\s$ be a positive,   noninteger real number and consider a real number~$s$ such that~$|s| < \s$.
 Then, there exists a positive constant~$C$ such that for all functions~$f$ and~$g$,
\begin{equation}\label{equ:1}
\|fg\|_{H^s(\H^d)} \leq C \|f\|_{C^\s(\H^d)} \|g\|_{H^s(\H^d)}.
\end{equation}
  Moreover, for any integer~$M$  there exists a positive constant~$C$   such that for any function~$f$,
\begin{equation}\label{equ:2}
 \|S_{M}  f\|_{C^\rho(\H^d)} \leq C  \|f\|_{C^\rho(\H^d)},
 \end{equation}
\begin{equation}\label{equ:3}
  \|({\rm Id} - S_{M})  f\|_{L^\infty(\H^d)} \leq C 2^{-M\rho} \|f\|_{C^\rho(\H^d)}
  \end{equation}
 and more generally, for~$0 < \s < \r$,
\begin{equation}\label{equ:4}
  \|({\rm Id} - S_{M}) f\|_{C^\s(\H^d)} \leq C 2^{-M(\rho-\s)} \|f\|_{C^\rho(\H^d)}.
\end{equation}}
 \end{lemme}
Note that Inequality~(\ref{equ:1}) is not sharp, but is sufficient for our purposes. The sharper result (proved by the same type of method) would be
$$
\|fg\|_{H^s(\H^d)} \leq C( \|f\|_{L^\infty(\H^d)} \|g\|_{H^s(\H^d)} + \|f\|_{C^\s(\H^d)} \|g\|_{L^2(\H^d)} ) .
$$
 The proof of this lemma is classical: it is the same proof as in $\R^d$ for the classical Littlewood-Paley theory and has no  specific  feature to the Heisenberg group. We provide it here for the sake of completeness, as it will be used often in the rest of this paper.

 \begin{proof}

 The first ingredient of the proof of  Estimate \refeq{equ:1} is Decomposition \refeq{decpara} which consists in writing
 $$ fg= T_f g + T_g f + R(f,g).$$
 Let us begin with the study of~$T_f g$. By definition of the paraproduct and thanks to Proposition \refer{pro:couronnes}, one has
 $$ \Delta_q (T_f g) = \sum_{|p-q |\leq N_0} \Delta_q (S_{p-1}f \,\Delta_p g),$$
 where~$N_0$ is a fixed integer, chosen large enough. We deduce thanks to the continuity of Littlewood-Paley operators on Lebesgue spaces, that
  \begin{eqnarray*}
 2^{q s}\|\Delta_q (T_f g)\|_{L^2(\H^d)}&\leq & \sum_{|p-q |\leq N_0} 2^{q s}\| \Delta_q (S_{p-1}f \,\Delta_p g)\|_{L^2(\H^d)}\\
 &\leq C & \sum_{|p-q |\leq N_0} 2^{q s}\|S_{p-1}f\|_{L^\infty(\H^d)}\|\Delta_{p}g\|_{L^2(\H^d)}
 \\
 &\leq C & \|f\|_{L^\infty(\H^d)}\sum_{|p-q |\leq N_0} 2^{q s}\|\Delta_{p}g\|_{L^2(\H^d)}.
 \end{eqnarray*}
 Using Littlewood-Paley characterization of Sobolev spaces, we infer that
  \begin{eqnarray*}
  2^{q s}\|\Delta_q (T_f g)\|_{L^2(\H^d)} &\leq C & \|f\|_{L^\infty(\H^d)}\sum_{|p-q |\leq N_0} 2^{(q-p) s}2^{p s}\|\Delta_{p}g\|_{L^2(\H^d)}\\
 &\leq C & \|f\|_{L^\infty(\H^d)}\,\|g\|_{H^s(\H^d)}\sum_{|p-q |\leq N_0} 2^{(q-p) s}c_p,
   \end{eqnarray*}
 where, as in all what follows, ~$(c_p)$ denotes a generic element of the unit sphere of~$\ell^2(\N)$. Taking advantage of Young inequalities on series, we obtain
 $$ 2^{q s}\|\Delta_q (T_f g)\|_{L^2(\H^d)} \leq C \|f\|_{L^\infty(\H^d)}\,\|g\|_{H^s(\H^d)} c_q $$
 which ensures the desired estimate for~$T_f g $ namely
 $$ \|T_f g\|_{H^s(\H^d)} \leq C \|f\|_{C^\s(\H^d)}\,\|g\|_{H^s(\H^d)}.$$

 Let us now consider the second term of the above decomposition of the product~$fg$. Again using spectral localization properties, one can write that
 $$ \Delta_q (T_g f) = \sum_{|p-q |\leq N_0} \Delta_q (S_{p-1} g\,\Delta_p f).$$
 Therefore
 \begin{eqnarray}\nonumber
 2^{q s}\|\Delta_q (T_g f)\|_{L^2(\H^d)}&\leq & 2^{q s} \sum_{|p-q |\leq N_0} \| \Delta_q (S_{p-1} g\,\Delta_p f)\|_{L^2(\H^d)}\\ \nonumber
 &\leq   &C \, 2^{q s}\, \sum_{|p-q |\leq N_0} \|S_{p-1} g\|_{L^2(\H^d)}\|\Delta_{p}f\|_{L^\infty(\H^d)}
 \\ \label{blitz}
 &\leq   &C \,  \|f\|_{C^\s(\H^d)}\,2^{q s}\,\sum_{|p-q |\leq N_0} \|S_{p-1} g\|_{L^2(\H^d)} 2^{-p\s}.
 \end{eqnarray}
 By \refeq{bes2}, we have in the case where~$s<0$,
 $$ \|S_{p-1} g\|_{L^2(\H^d)} \leq C \|g\|_{H^s(\H^d)}2^{-ps}c_p,$$
 where~$(c_p)$ still denotes an element of the unit sphere of~$\ell^2(\N)$. We deduce in that case that
 \begin{eqnarray*}
 2^{q s}\|\Delta_q (T_g f)\|_{L^2(\H^d)}
 &\leq C & \|f\|_{C^\s(\H^d)}\,\|g\|_{H^s(\H^d)}\,2^{q s}\,\sum_{|p-q |\leq N_0} 2^{-ps}c_p 2^{-p\s}\\ &\leq C & \|f\|_{C^\s(\H^d)}\,\|g\|_{H^s(\H^d)}\,2^{-q \s}\,\sum_{|p-q |\leq N_0} 2^{-(p-q)(\s- |s|)}c_p \\ &\leq C & \|f\|_{C^\s(\H^d)}\,\|g\|_{H^s(\H^d)}\,c_q.
 \end{eqnarray*}
 This leads in that case to
 $$\|T_g f\|_{H^s(\H^d)} \leq C \|f\|_{C^\s(\H^d)}\,\|g\|_{H^s(\H^d)}. $$
 Let us now estimate~$T_g f$ in the case where~$s\geq 0$.  We have
 \begin{eqnarray*}
 \|S_{p-1} g\|_{L^2(\H^d)}
 &\leq C & \sum_{p'\leq p-2} \|\Delta_{p'}g\|_{L^2(\H^d)}\\ &\leq C & \|g\|_{H^s(\H^d)} \sum_{p'\leq p-2}2^{-p' s}c_{p'}.
 \end{eqnarray*}
 Thus \aref{blitz} becomes
  \begin{eqnarray*}
 2^{q s}\|\Delta_q (T_g f)\|_{L^2(\H^d)}
 &\leq C & \|f\|_{C^\s(\H^d)}\,\|g\|_{H^s(\H^d)}\,2^{q s}\,\sum_{|p-q |\leq N_0} \sum_{p'\leq p-2}2^{-p\s}2^{-p' s}c_{p'}\\&\leq C & \|f\|_{C^\s(\H^d)}\,\|g\|_{H^s(\H^d)}\,2^{q s}\,\sum_{|p-q |\leq N_0} 2^{-p\s}\\&\leq C & \|f\|_{C^\s(\H^d)}\,\|g\|_{H^s(\H^d)}\,2^{-q (\s-s)}\\&\leq C & \|f\|_{C^\s(\H^d)}\,\|g\|_{H^s(\H^d)}\,c_q.
 \end{eqnarray*}
 This obviously ends the estimate of~$\|T_g f\|_{H^s(\H^d)}$ for any~$s$ satisfying~$|s|< \s$.

 Finally, let us consider the remainder term~$R(f,g)$. Taking into account the accumulation of frequencies at the origin, we can write
 $$ \Delta_q (R(f,g)) = \sum_{q\leq p +N_0} \,\sum_{|p-p' |\leq 1 }\Delta_q (\Delta_{p}f\,\Delta_{p'} g).$$
 Thus
 \begin{eqnarray*}
 2^{q s}\|\Delta_q (R(f,g))\|_{L^2(\H^d)}&\leq & C 2^{q s}\, \sum_{q\leq p +N_0} \,\sum_{|p-p' |\leq 1 }\| \Delta_{p}f\|_{L^\infty(\H^d)}\|\Delta_{p'} g\|_{L^2(\H^d)}\\
 &\leq  & C \, \|f\|_{C^\s(\H^d)}\,\|g\|_{H^s(\H^d)}\,2^{q s}\,\sum_{q\leq p +N_0} \,\sum_{|p-p' |\leq 1 }2^{-p \s} 2^{{-p'} s} c_{p'}
 \\
 &\leq   &C \, \|f\|_{C^\s(\H^d)}\,\|g\|_{H^s(\H^d)}\,2^{q s}\sum_{q\leq p +N_0}2^{-p \s} 2^{-{p} s} c_{p}.
 \end{eqnarray*}
 In the case where~$s \geq 0$, we infer that
 $$ 2^{q s}\|\Delta_q (R(f,g))\|_{L^2(\H^d)}\leq C \|f\|_{C^\s(\H^d)}\,\|g\|_{H^s(\H^d)}\sum_{q\leq p +N_0}2^{-({p}-q) s}c_{p}. $$
 Then, thanks to Young inequalities, we get
 $$ 2^{q s}\|\Delta_q (R(f,g))\|_{L^2(\H^d)}\leq C \|f\|_{C^\s(\H^d)}\,\|g\|_{H^s(\H^d)}c_{q} $$
 which implies that
 $$ \|R(f,g)\|_{H^s(\H^d)}\leq C \|f\|_{C^\s(\H^d)}\,\|g\|_{H^s(\H^d)}.$$
 Now, in the case where~$s < 0$, we have
  $$ 2^{q s}\|\Delta_q (R(f,g))\|_{L^2(\H^d)}\leq C \|f\|_{C^\s(\H^d)}\,\|g\|_{H^s(\H^d)}2^{-q \s}\sum_{q\leq p +N_0}2^{-(p-q) (\s -|s|)}  c_{p}. $$
  Again, Young inequalities allow to conclude. This achieves the proof of the estimate
  $$ \|R(f,g)\|_{H^s(\H^d)}\leq C \|f\|_{C^\s(\H^d)}\,\|g\|_{H^s(\H^d)},$$ for any~$|s|< \s$.

 Let us now turn to the proof of Inequality \refeq{equ:2}. By definition of the~$C^\rho$-norm, we recall that
 $$ \|S_{M}  f\|_{C^\rho(\H^d)} = \sup_{q} 2^{q\rho}\|\Delta_q S_{M}  f\|_{L^\infty(\H^d)}.$$
 Using commutation properties of~$\Delta_q$ and~$ S_{M}$, we obtain
 \begin{eqnarray*}
 \|S_{M}  f\|_{C^\rho(\H^d)}&=& \sup_{q  } 2^{q\rho}\| S_{M} \Delta_q f\|_{L^\infty(\H^d)}\\
 &\leq   &C \, \sup_{q  } 2^{q\rho}\|\Delta_q   f\|_{L^\infty(\H^d)}
 \\
 &\leq   &C \, \|f\|_{C^\rho(\H^d)}
 \end{eqnarray*}
 thanks to the continuity of Littlewood-Paley operators on Lebesgue spaces, which ends the proof of Estimate \refeq{equ:2}. Moreover, it is obvious that
 $$ \|({\rm Id} - S_{M})  f\|_{L^\infty(\H^d)} \leq \sum_{q \geq M -N_1} \|\Delta_q   f\|_{L^\infty(\H^d)},$$
 where~$N_1$ is a fixed integer, chosen large enough.
 Therefore, according to definition of the~$C^\rho$-norm, we get
 \begin{eqnarray*}
 \|({\rm Id} - S_{M})  f\|_{L^\infty(\H^d)}&\leq C & \sum_{q \geq M -N_1} 2^{-q\rho}\|f\|_{C^\rho(\H^d)}\\
 &\leq C& \|f\|_{C^\rho(\H^d)}\sum_{q \geq M -N_1} 2^{-q\rho}\\
&\leq &C \|f\|_{C^\rho(\H^d)} 2^{-M\rho}.
 \end{eqnarray*}
 This achieves the proof of Inequality \refeq{equ:3}.
 Along the same lines, for~$0 < \s < \r$, one has
 $$ \|({\rm Id} - S_{M})  f\|_{C^\s(\H^d)} \leq \sum_{q \geq M -N_1} 2^{q \s} \|\Delta_q ({\rm Id} - S_{M})  f\|_{L^\infty(\H^d)}.$$
 Using again the continuity of Littlewood-Paley operators on Lebesgue spaces, it comes
 \begin{eqnarray*}
 \|({\rm Id} - S_{M})  f\|_{C^\s(\H^d)}
 &\leq C& \sum_{q \geq M -N_1} 2^{q \s} \|\Delta_q  f\|_{L^\infty(\H^d)}\\
&\leq &C \|f\|_{C^\rho(\H^d)} \sum_{q \geq M -N_1}2^{q
(\s-\rho)}\\ &\leq &C \|f\|_{C^\rho(\H^d)} 2^{-M (\rho-\s)},
 \end{eqnarray*}
 thus the desired estimate. This ends the proof of Lemma~\ref{estfgHs}.
 \end{proof}


\section{Truncation pseudodifferential operators}\label{Preliminary}
 \setcounter{equation}{0}
In this section we shall  compare Littlewood-Paley   operators
with the pseudodifferential operators~${\rm Op}\left(\Phi (2^{-2p}
|\lam|(\xi^2+\eta^2)\right)$, for~$\Phi$ compactly supported in a
unit ring.

We shall see that~${\rm Op}\left(\Phi (2^{-2p}
|\lam|(\xi^2+\eta^2)\right)$ is ``close'' to~$\Delta_p$ in the sense
that the operator~$\Delta_q {\rm Op}\left(\Phi (2^{-2p}
|\lam|(\xi^2+\eta^2)\right)$ is small in ${\mathcal L}(H^s(\H^d))$
norm if~$|p-q|$ is large. This is made precise in the next
proposition.

\begin{prop}\label{propphipdeltap}  {  Let $\delta_0\in\left(0,1\right)$ and~$\Phi$ be a  smooth function, compactly supported in~$]0,\infty[$. There is a constant~$C $ such that the following
result holds. For any~$p \geq 0$, define the symbol $$
a_p(w,\lambda, \xi,\eta) = \Phi_p ( |\lam|(\xi^2+ \eta^2)), \quad
\mbox{where} \quad \Phi_p (r) = \Phi(2^{-2p}r), \: \forall r>0. $$
Then for any integer $q \geq -1$ and any real number~$s$, $$
 \| \Delta_q { \rm Op }
(a_p) \|_{{\mathcal L}(H^s (\H^d))}  \leq C  2^{-{ \delta_0}
|p-q|}, $$ where~$  \Delta_q$ is a Littlewood-Paley truncation, as
defined in Definition~\ref{defLP}. }
 \end{prop}
%

 \begin{proof}
 We shall start by reducing the problem to the case~$s=0$.
 Let~$u$ belong to~${\mathcal S}(\H^d)$ and let~$q \geq 0$ be given (the case~$q=-1$ is obvious).  The norm
$\|\Delta_q{\rm Op}(a_p)u\|_{H^s} $  is controlled by the quantity
$$
  2^{qs}\|\Delta_q{\rm Op}(a_p)u\|_{L^2}=2^{qs}\left(\int\|{\mathcal F}(u)(\lam)A_\lam R^*(2^{-2q}D_\lam)\|^2_{HS({\mathcal H}_\lam)}|\lam|^d\,d\lam\right)^{1/2}
$$ where $A_\lam=J_\lam^*op^w(a_p)J_\lam$. Defining a smooth,
compactly supported (away from zero) function~$ \mathcal R$ such
that~$ \mathcal R R^* = R^*$, one   has $$
  \|{\mathcal F}(u)(\lam)A_\lam R^*(2^{-2q}D_\lam) \| _{HS({\mathcal H}_\lam)} = \|{\mathcal F}(u)(\lam)A_\lam R^*(2^{-2q}D_\lam) {\mathcal R}(2^{-2q}D_\lam)\| _{HS({\mathcal H}_\lam)} .
$$ But~$A_\lam$ is a diagonal operator in the diagonalisation
basis of $D_\lam$, thus it commutes with the
operator~$R^*(2^{-2q}D_\lam)$. So
$$
 \|{\mathcal F}(u)(\lam)A_\lam R^*(2^{-2q}D_\lam) {\mathcal R}(2^{-2q}D_\lam)\| _{HS({\mathcal H}_\lam)}  = \|{\mathcal F}(\widetilde \Delta_q u)(\lam)A_\lam R^*(2^{-2q}D_\lam)\| _{HS({\mathcal H}_\lam)} ,
$$ where~$\widetilde \Delta_q$ is the Littlewood-Paley operator
associated with~$ {\mathcal R}(2^{-2q} \cdot)$.
Using~(\ref{hsproperty}) stated page~\pageref{hsproperty}, we get
 $$
 \|{\mathcal F}(u)(\lam)A_\lam R^*(2^{-2q}D_\lam) {\mathcal R}(2^{-2q}D_\lam)\| _{HS({\mathcal H}_\lam)} \leq \|{\mathcal F}(\widetilde \Delta_q u)(\lam)\| _{HS({\mathcal H}_\lam)} \|A_\lam R^*(2^{-2q}D_\lam)\|_{{\mathcal L}({\mathcal H}_\lam)},
$$
and Remark~\ref{nophiinnotation} gives the expected result: we
have reduced the problem to the~$L^2(\H^d)$ case, and by  the
Plancherel formula~\refeq{Plancherelformula} and
Inequality\refeq{hsproperty}, it is enough to study the norm as a
bounded operator of $L^2(\R^d)$ of the operators
 $$
R^*(2^{-2q}|\lam|(\xi^2-\Delta_\xi))\,op^w(a_p) \quad \mbox{and} \quad R^*(2^{-2q}|\lam|(\xi^2-\Delta_\xi))\,op^w(a_p).
 $$
 For this, we use Mehler's formula to turn $op^w(a_p) $ into an operator given by a  function of the harmonic oscillator in order to be able to use functional calculus.
 From now on we suppose to simplify    that~$\lam>0$.

 We will denote, as in Definition~\ref{defLP}, by $\widetilde R^*$ and~$R^*$ the basis functions of the truncation~$\Delta_q$ (with~$\widetilde R^*$ supported in a unit ball of~$\R$ and~$  R^*$ supported in a unit ring of~$\R$).

In view of \refeq{invfor} (see page~\pageref{invfor}), one has
\begin{eqnarray*}
op^w\left(\Phi_p(\lam(\xi^2+\eta^2))\right) & =  &    \frac{1} {2\pi}\int_{\R} \widehat \Phi (\tau) \frac{{\rm
e}^{i ( \xi^2-\Delta) {\rm Arctg} (2^{-2p} \lam
\tau)}}{\left(1+(2^{-2p}\lam \tau)^2\right)^{\frac d 2}} \: d\tau
.
\end{eqnarray*}
But
    $$
    \left\| R^*\left( 2^{-2q}|\lam|(\xi^2-\Delta_\xi)\right) op^w(a_p)\right\|_{{\mathcal L}(L^2(\R^d))} =\sup_{\alpha,\lam}^{} \left| I_p(\alpha,\lam)\right| R^*(2^{-2q}|\lam|(2|\alpha|+d))
    $$
    and a similar relation holds for $\widetilde {R^*}$, so
    we are reduced to estimating, for~$\alpha\in\N^d$  and $\lambda2^{-2q} (2|\alpha|+d )$ in a unit ring (or ball if~$q = -1$)
  $$
  I_{p}(\alpha,\lambda) \eqdefa \frac{1} {2\pi} \int_{\R}
\widehat \Phi (\tau) \frac{{\rm e}^{i (2 |\al| + d) {\rm Arctg}
(2^{-2p}\lam\tau )}}{\left(1+(2^{-2p}\lam \tau)^2\right)^{\frac d
2}} \: d\tau,
  $$
and we shall argue differently whether $q<p$ or $q>p$.

 $ $

 $\bullet$ {\bf The case when  $q>p$}. We argue differently depending on whether $2^{-2p}|\lam|\leq 2^{q-p}$ or~$2^{-2p}|\lam|\geq 2^{q-p}$.
 Let us first suppose that~$2^{-2p}|\lam|\leq 2^{ q-p}$.
Noticing that $$ \frac{d}{d\tau}{\rm e}^{i (2 |\al| + d) {\rm
Arctg} (2^{-2p}\lam\tau
 )} = \frac{i 2^{-2p}\lam (2 |\alpha|+d )}{1 + (2^{-2p}\lam\tau)^2 } {\rm e}^{i (2 |\al| + d) {\rm Arctg}
(2^{-2p}\lam\tau
 )}
$$ we have
 $$     I_p(\alpha,\lambda)   =  \frac{i}{ (2 |\alpha|+d
) 2^{-2p}\lam} \int  {\rm e}^{i (2 |\al| + d) {\rm Arctg}
(2^{-2p}\lam\tau
 )}
\frac{d}{d\tau} \left( \frac{\widehat\Phi(\tau)}{(1 +
(2^{-2p}\lam\tau)^2)^{\frac{d}{2}-1}} \right) \, d\tau
$$
so using the fact that~$ 2 |\alpha| + d \geq 1 $,
 $$\displaylines{\qquad
R^* \left((2 |\alpha|+d) \lam 2^{-2q} \right)    |
I_p(\alpha,\lambda)|   \leq   C2^{2(p-q) } \Bigl( \int
|\widehat\Phi'(\tau)|(1+(2^{-2p}\lam\tau)^2)^{1-{d\over
2}}d\tau\hfill\cr\hfill
 + \int |\widehat\Phi (\tau)| \frac{ 2^{-4p}\lam^2\tau}{(1 + (2^{-2p}\lam\tau)^2)^{\frac{d}{2}}}d\tau
\Bigr).\qquad\cr}
 $$
 Let us consider the first integral. If $d\geq 2$, it is bounded by $\|\widehat \Phi'\|_{L^1}$. On the other hand, if~$d=1$, we observe that
 $$|\widehat\Phi'(\tau)|(1+(2^{-2p}\lam\tau)^2)^{1-{d\over 2}}\leq C |\widehat\Phi'(\tau)|\,(1+2^{-2p}\lam|\tau|).$$
 Therefore, since $(1+|\tau|)\,|\widehat \Phi'(\tau)|\in L^1$, there exists a constant $C$ such that
 $$
 2^{2(p-q) }
\int |\widehat\Phi'(\tau)|(1+(2^{-2p}\lam\tau)^2)^{1-{d\over
2}}d\tau \leq C\,2^{2(p-q)}\left(1+2^{q-p}\right)
 \leq C\, 2^{-(q-p)}.
 $$

Let us  now concentrate on the last integral.  We have clearly $$
2^{2(p-q) } \int |\widehat\Phi (\tau)| \frac { 2^{-4p}\lam^2|\tau|
} { (1 + (2^{-2p}\lam\tau)^2)^{\frac{d}{2}} }   d\tau \leq
2^{2(p-q)}\, 2^{-2p}|\lam|\, \int |\widehat\Phi (\tau)|\,d\tau,$$
whence a constant $C$ such that $$2^{2(p-q) } \int |\widehat\Phi
(\tau)| \frac { 2^{-4p}\lam^2|\tau| } { (1 +
(2^{-2p}\lam\tau)^2)^{\frac{d}{2}} }   d\tau \leq C\,
2^{-(q-p)}.$$

We now suppose that $|\lam| 2^{-2p}\geq 2^{q-p}$ and we
perform the change of variables $u=\lam 2^{-2p}\tau$ in the
integral expression of $I_p(\alpha,\lam)$.
We obtain
$$I_p(\alpha,\lam)={2^{2p}\lam^{-1}\over 2\pi} \int
\widehat\Phi\left(2^{2p}\lam^{-1}u\right)(1+u^2)^{-d/2}{\rm
e}^{i(2|\alpha|+d){\rm Arctg}u}\,du.$$
Using that
$|\widehat\Phi(\tau)|\leq C\,|\tau|^{-1+\delta_0}$, we get
$$\left|
\widehat\Phi\left(2^{2p}\lam^{-1}u\right)\right|\leq
C(2^{-2p}|\lam|)^{1-\delta_0}|u|^{-1+\delta_0}.$$
This
yields that there exists a constant $C$ such that
$$\left|
I_p(\alpha,\lam)\right| \leq C\, (2^{2p}|\lam|^{-1})^{\delta_0} \int
|u|^{-1+\delta_0}(1+u^2)^{-d/2} du\leq C' 2^{-\delta_0(q-p)}.$$

 As a conclusion,  we have proved that in that case, for all~$\alpha\in\ZZZ^d$,
$$ R^* \left((2 |\alpha|+d) \lam 2^{-2q} \right)    |
I_p(\alpha,\lambda)|   \leq   C\, 2^{\delta_0(p-q)}. $$

$ $

  $\bullet$ {\bf The case when $q\leq p$}. The idea is to compare $I_p(\alpha,\lambda)$ to $\Phi(\lam 2^{-2p}(2|\alpha|+d))$.
    Taking the inverse (classical) Fourier transform we can write
  $$
    I_p(\alpha,\lambda)-\Phi(\lam 2^{-2p}( 2|\alpha|+d) )
     = \frac{1} {2\pi} \int_{\R}
\widehat \Phi (\tau) \left(\frac{{\rm e}^{i (2 |\al| + d) {\rm
Arctg} (2^{-2p} \tau\lam )}}{\left(1+(2^{-2p}
\tau\lam)^2\right)^{\frac d 2}} - {\rm e}^{i 2^{-2p}\lam \tau(2
|\al| + d)}\right) \: d\tau
  $$
  or again
  $$ I_p(\alpha,\lambda)-\Phi(\lam 2^{-2p}(2| \alpha|  +d))
= J_{p } (\alpha,\lambda)+ R_{p } (\alpha,\lambda),$$
 with
 $$ J_{p }(\alpha,\lambda) \eqdefa
 \frac{1} {2\pi}  \int_{\R}
\widehat \Phi (\tau)
 \left({\rm e}^{i (2 |\al| + d) {\rm Arctg}
(2^{-2p}\lam\tau )} - {\rm e}^{i 2^{-2p}\lam \tau(2 |\al| +
d)}\right) \: d\tau.
  $$
  It is easy to see that
 $$
   | R_{p } (\alpha,\lambda)  |  \leq   C 2^{-2p} \lam\int_{\R} |\tau\widehat \Phi (\tau) | \: d\tau $$
   so since~$\widehat \Phi$ belongs to~${\mathcal S}(\R)$, we have
   \begin{eqnarray*}
R^* \left((2 |\alpha|+d) \lam 2^{-2q} \right)       | R_{p }
(\alpha,\lambda)  | & \leq  & CR^* \left((2 |\alpha|+d) \lam
2^{-2q} \right)   \,2^{-2p}\lam \\ & \leq & C\,2^{-2(p-q)},
\end{eqnarray*}
 using the fact that~$2 |\alpha|+d \geq 1$. Similarly
$$ \widetilde R^* \left((2 |\alpha|+d) \lam  \right)      | R_{p }
(\alpha,\lambda)  |  \leq    C\,2^{-2p}. $$
  So now we are left with the estimate of~$J_{p }$, which we shall decompose into two parts:
  $$
   J_{p} =   J_{p}^1 +   J_{p}^2, \quad \mbox{with} $$
   $$ J_{p}^1(\alpha,\lambda)  \eqdefa
  \frac{1} {2\pi} \int_{|\tau 2^{-2p} \lam| \leq 1/2}
\widehat \Phi (\tau) \left({\rm e}^{i (2 |\al| + d) {\rm Arctg}
(2^{-2p}\lam\tau
 )} - {\rm e}^{i 2^{-2p}\lam \tau(2 |\al| + d)}\right) \: d\tau .
  $$
  The estimate of~$  J_{p}^2$ is very easy, since clearly as above
   \begin{eqnarray*}
| J_p^2 (\alpha,\lambda) | &\leq & C 2^{-2p}\lam \int_{\R}
|\tau\widehat \Phi (\tau) | \: d\tau \\
  &\leq &C\,2^{-2p}\lam,
  \end{eqnarray*}
  so
  $$
 R^* \left((2 |\alpha|+d) \lam 2^{-2q} \right)      | J_p^2 (\alpha,\lambda)  |     \leq C\,2^{-2(p-q) } \quad \mbox{and}
  \quad \widetilde R^* \left((2 |\alpha|+d) \lam   \right)    | J_p^2 (\alpha,\lambda)  |  \leq   C\,2^{-2p}.
  $$
  Now let us concentrate on~$  J_{p}^1$. We can write
  $$
  J_{p}^1  (\alpha,\lambda) = \frac{1} {2\pi}
   \int_{|\tau 2^{-2p}\lam| \leq 1/2}
\widehat \Phi (\tau) {\rm e}^{i 2^{-2p}\lam \tau (2 |\alpha| + d)}
 \left({\rm e}^{
 i (2 |\al| + d)  2^{-2p} \lam h(\tau)} - 1 \right) \: d\tau,
  $$
with
  $$
  h(\tau) = \tau  \sum_{n \geq 1} \frac{
 (-1)^n (2^{-2p} \lam \tau)^{2n }
 }{2n+1}  \virgp
  $$
  which is well defined, and analytic, for~$|\tau 2^{-2p} \lam | \leq 1/2$.  Observe that the function $h$ depends on the integer $p$ and on~$\lam $, and that one has to control this dependance.
  In particular, we notice that~$h'(\tau)$ can easily be bounded,  by~$1/3$, on the domain~$|\tau 2^{-2p} \lam | \leq 1/2$.
   But
  $$   {\rm e}^{
 i (2 |\al| + d)  2^{-2p}\lam h(\tau)} - 1
 =  i (2 |\al| + d)
 2^{-2p}\lam   h(\tau) \int_0^1 {\rm e}^{
 it (2 |\al| + d)   2^{-2p}\lam h(\tau) }  \: dt
  $$
  so
  $$
      J_{p}^1  (\alpha,\lambda) =  \frac{i} {2\pi}\int_0^1  \int_{|\tau 2^{-2p}\lam | \leq 1/2}
\widehat \Phi (\tau) {\rm e}^{i 2^{-2p}\lam (2 |\alpha| + d) (\tau
+ t h(\tau ))}(2 |\al| + d)   2^{-2p}\lam h (\tau) dt d\tau.
  $$
  Integrating by parts, we get
  $$
  \longformule{
    J_{p}^1 (\alpha,\lambda)  = - \frac{1} {2\pi}\int_0^1  \int_{|\tau 2^{-2p} \lam | \leq 1/2}
{\rm e}^{i 2^{-2p}\lam (2 |\alpha| + d) (\tau + t h(\tau ))}
\partial_\tau \left( \frac{\widehat \Phi (\tau)}{1 + t h'(\tau)}
h(\tau) \right) dt d\tau  }{ + \frac{1} {2\pi}\int_0^1 \left[ {\rm
e}^{i 2^{-2p}\lam (2 |\alpha| + d) (\tau + t h(\tau ))}
\frac{\widehat \Phi (\tau)}{1 + t h'(\tau)} h(\tau) \right]_{|\tau
2^{-2p}\lam | = 1/2}\,dt.}
  $$
  Writing the above formula as~$  J_{p}^1 = K_{p}^1 + K_{p}^2$, with
$$ K_{p}^2  (\alpha,\lambda) = \frac{1} {2\pi} \int_0^1 \left[
{\rm e}^{i 2^{-2p} \lam(2 |\alpha| + d) (\tau + t h(\tau ))}
\frac{\widehat \Phi (\tau)}{1 + t h'(\tau)} h(\tau) \right]_{|\tau
2^{-2p} \lam | = 1/2}\,dt, $$ it is obvious that
 $$
 | K_{p}^2 (\alpha,\lambda)  |  \leq C  \left| \widehat \Phi (\frac12
2^{2p}\lam^{-1}) h(\frac12 2^{2p}\lam^{-1}) \right|.
 $$
Writing $$
  h(\tau) = \tau  \sum_{n \geq 1} \frac{
 (-1)^n (2^{-2p} \lam\tau)^{2n }
 }{2n+1}  = 2^{-2p} \lam \tau^2 \sum_{n \geq 1} \frac{
 (-1)^n (2^{-2p} \lam\tau)^{2n-1 }
 }{2n+1}\virgp
  $$
we deduce that
\begin{eqnarray*}
 | K_{p}^2 (\alpha,\lambda)  |  & \leq & C 2^{-2p} \lam |\widehat \Phi (\frac12 2^{2p}\lam^{-1}) 2^{4p}\lam^{-2}| \\
 & \leq &C2^{-2p}\lam,
\end{eqnarray*}
where the second estimate comes from the fact that~$\widehat \Phi$ is a rapidly decreasing function. To bound~$ K_{p}^1$ we   just need to notice that
  $$
  \frac{\widehat \Phi (\tau)}{1 + t h'(\tau)} h(\tau)
= \frac{\widehat \Phi (\tau) \tau^2}{1 + t h'(\tau)} 2^{-2p}\lam
g(\tau), \quad \mbox{with} \quad  g(\tau) =    \sum_{n \geq 1}
\frac{
 (-1)^n (2^{-2p}\lam \tau)^{2n -1}
 }{2n+1}  \virgp
  $$
  so
  $$
   | K_{p}^1 (\alpha,\lambda) | \leq C 2^{-2p}\lam \int_0^1  \int_{|\tau 2^{-2p} \lam| \leq 1/2}
   \Bigl| \partial_\tau \Bigl(\frac{\widehat \Phi (\tau) \tau^2}{1 + t h'(\tau)}  g(\tau)\Bigr) \Bigr| \: d\tau\,dt \leq C2^{-2p}\lam.
  $$
  We conclude as previously that
   $$
R^* \left((2 |\alpha|+d) \lam 2^{-2q} \right)     | J_p^1
(\alpha,\lambda)  |     \leq C\,2^{-2(p-q) } \quad \mbox{and}
  \quad \widetilde R^* \left((2 |\alpha|+d) \lam \right)    | J_p^1 (\alpha,\lambda)  |  \leq   C\,2^{-2p}.
  $$
  Combining those results, we conclude that if~$p>q$, then
  $$
  R^* \left((2 |\alpha|+d) \lam 2^{-2q} \right) \left|  \Phi \left(( 2|\alpha|+d)\lam 2^{-2p} \right) -    I_p(\alpha,\lambda)    \right |  \leq C\,2^{-2(p-q) }.
  $$
  But clearly~$ R^* \left((2 |\alpha|+d) \lam 2^{-2q} \right)    \Phi \left(( 2|\alpha|+d)\lam 2^{-2p} \right ) $ is equal to zero if~$|p-q| $ is large enough, so we have proved the expected result if~$p>q$.

   That concludes the proof of the proposition.
\end{proof}


 \section{$\lam$-truncation operators}\label{sec:Lambda}
 \setcounter{equation}{0}
We shall use, in the proof of Theorem~\ref{contHs}, truncation
operators in the variable $\lam$.

Let us consider~$\psi$
and~$\phi$, two smooth radial functions, the values of which are
in the interval~$[0,1]$, belonging respectively to~${\mathcal
D}({\mathcal B})$ and~${\mathcal D}({\mathcal C})$,
where~${\mathcal B}$ is the unit ball  of~$\R $ and~${\mathcal C}$
a unit ring of~$\R $, and such that for $D=1$
\begin{equation}\label{partident}
\forall \zeta\in\R^D,\;\;1= \psi(\zeta)+\sum_{p\geq
0}\phi(2^{-2p}\zeta).
\end{equation}
We set $$\Lambda_p={\rm Op}(\phi(2^{-2p}\lam))\;\;{\rm
and}\;\;\Lambda_{-1}={\rm Op}(\psi(\lam)).$$

We notice
that~$\Lam_p$ commutes with all operators of the form~${\rm Op}(a(
\lam,y,\eta))$, and in particular with powers of~$-\Delta_{\H^d}$.

Then the operators~$\Lambda_p$ map continuously~$H^s(\H^d)$
into~$H^s(\H^d)$ independently of~$p$ and we have the following
quasi-orthogonality relation:  there exists $N_0$ such that
\begin{equation}\label{eq:quasiorth}
\Lambda_p\Lambda_q=0\;\;{\rm for}\;\;|p-q|\geq N_0,
\end{equation}
which implies that
\begin{equation}\label{eq:quasinorm}
\|\Lambda_p u\|_{L^2(\H^d)} \leq c_p \| u\|_{L^2(\H^d)},
\end{equation}
where~$c_p $ is an element of the unit sphere of $\ell^2(\ZZZ)$.
More precisely, there exist   constants~$C_1$ and~$C_2$ such that
if $f$ belongs to~$ H^s(\H^d)$, then the following inequality hold:

\begin{equation}\label{eq:quasiorthHs}
C_1\, \sum_r \|\Lambda_r f\|^2_{H^s(\H^d)} \leq \| f\|
^2_{H^s(\H^d)}\leq C_2 \, \sum_r \|\Lambda_r f\|^2_{H^s(\H^d)} .
\end{equation}

Besides, we are able to say something about the
$\Lambda_m$-localization of a product by an easy adaptation of
Lemma~4.1 and of Proposition~4.2 of~\cite{bg}. More precisely, we
have the following result which ensures that some
$\Lambda_m$-spectral localization properties are preserved after
the product has been taken.
 \begin{prop}\label{4.2}   { There is a constant $M_1\in\N$ such that the following holds. Consider~$f$ and $g$ two functions of~${\mathcal S}(\H^d)$ such that
\begin{eqnarray*}
{\mathcal F}(f)(\lam)&=& \displaystyle {\bf 1}_{2^{2m}{\mathcal C}}(\lam){\mathcal F}(f)(\lam)\quad \mbox{and}\\
{\mathcal F}(g)(\lam)&=& \displaystyle{\bf 1}_{2^{2m'}{\mathcal C}}(\lam){\mathcal F}(g)(\lam)
\end{eqnarray*}
  for some integers $m$ and $m'$. If
   $m'-m>M_1$, then there exists a ring $\widetilde{\mathcal C}$ such that
  $${\mathcal F}(fg)(\lam)={\bf 1}_{2^{2m'} \widetilde{\mathcal C}}(\lam){\mathcal F}(fg)(\lam).$$
  On the other hand, if $|m'-m|\leq M_1$, then there exists a ball $\tilde{\mathcal B}$ such that
  $${\mathcal F}(fg)(\lam)={\bf 1}_{2^{2m'} \widetilde{\mathcal B}}(\lam){\mathcal F}(fg)(\lam).$$ }
  \end{prop}
 \begin{proof} The proof of that result follows the lines of the proof of Proposition 4.2 of~\cite{bg}, and is in fact simpler. We write it here for the sake of completeness. By density, it suffices to prove
Lemma\refer{4.2} for~$f,g$ in~${\mathcal D}(\R^{2d+1})$.

For
simplicity, we will only deal with the case where~$\lam
>0$.

 By definition of~${\mathcal F}(f)(\lam)$, we have
\begin{eqnarray*}
{\mathcal F}(f)(\lam)F_{\al,  \lam}(\xi ) & = & \int_{\H^d}f(z, s)
u^{\lam}_{z, s} F_{\al,  \lam} (\xi) \: dz \,ds
\\ & = & \int_{\H^d} f(z, s) F_{\al, \lam}(\xi-\overline z){\rm e}^{i \lam s + 2\lam (\xi \cdot z -
|z|^2/2)}\: dz \,ds.
\end{eqnarray*}
Let us write~$\xi = \xi_a + i \xi_b$ and~$z = z_a + i z_b$,
where~$\xi_a $, $z_a $,  $\xi_b$ and~$z_b $ are real numbers.

Straightforward computations show that $${ \rm e }^{i \lam s +
2\lam (\xi \cdot z - |z|^2/2)}  = {\rm e}^{-i \,(-2 \lam \xi_b
\cdot z_a -2 \lam \xi_a \cdot z_b -\lam s )} {\rm
e}^{{-\lam\,(|\xi - \overline z|^2 - |\xi|^2)}}. $$ Then we can
observe that
\begin{equation}\label{heisusfour} {\mathcal F}(f)(\lam)F_{\al, \lam}(\xi)
= \left(A^\al_{\lam, \xi}f\right)^{\widehat{}} (-2 \lam \xi_b, -2
\lam \xi_a, -\lam),
 \end{equation}where~$h^{\widehat{}}$
denotes the usual Fourier transform of~$h$ on $\R^{2d+1}$ and
where \begin{equation}\label{heisusfourfunct} A^\al_{\lam,
\xi}f(z, s) = F_{\al, \lam}(\xi-\overline z)
 {\rm e}^{-\lam\,(|\xi - \overline z|^2 - |\xi|^2)} f(z, s).
\end{equation} Therefore, one can write $$ {\mathcal F}(fg)(\lam) F_{\al, \lam
}(\xi) = \left(A^\al_{\lam, \xi}fg\right)^{\widehat{}} (-2 \lam
\xi_b, -2 \lam \xi_a, -\lam). $$ Noticing that for any
multi-index~$\beta$ of~$\N^d$ satisfying~$\beta \leq \al$, we have
$$ F_{\al, \lam}(\xi) = C_{\al, \beta}\, F_{\al-\beta, \lam}(\xi)
\cdot F_{\beta, \lam}(\xi),$$ with~$ C_{\al, \beta} =
\left(\begin{array}{c} \al
\\ \beta \end{array}\right)^{-\frac{1}{2}}$, we deduce that~$ A^\al_{\lam, \xi}fg= B^\beta_{\lam,
\xi}f \cdot A^{\al-\beta}_{\lam, \xi}g$, where
 $$ B^\beta_{\lam,
\xi}f(z, s) = F_{\beta, \lam}(\xi-\overline z) f(z, s)$$
and~$\beta \leq \al$. Using the fact that the standard Fourier
transform on $\R^{2d+1}$ exchanges product and   convolution, we
get $$ \left(A^\al_{\lam, \xi}fg\right)^{\widehat{}} (-2 \lam
\xi_b, -2 \lam \xi_a, -\lam) = C_{\al, \beta}\left(B^\beta_{\lam,
\xi}f \right)^{\widehat{}} \star \left(A^{\al-\beta}_{\lam, \xi}g
\right)^{\widehat{}}  (-2 \lam \xi_b, -2 \lam \xi_a, -\lam),  $$
where~$\star$ denotes the convolution product in~$\R^{2d+1}$ and
still for any multi-index~$\beta$ of~$\N^d$ satisfying~$\beta \leq
\al$.
The question is then reduced to the study of the supports of the
functions~$(B^\beta_{\lam, \xi}f )^{\widehat{}}$
and~$(A^{\al-\beta}_{\lam, \xi}g )^{\widehat{}} $.

 According to\refeq{heisusfour}, the support in~$\lam$ of the function~$\left(A^{\al-\beta}_{\lam, \xi}g(z, s) \right)^{\widehat{}}  (-2
\lam \xi_b, -2 \lam \xi_a, -\lam)$ is included in the
ring~$2^{2m'}{{\mathcal C} }$. Now, Lemma \refer{4.2} readily
follows from the properties of the standard convolution product
in~$\R^{2d+1}$ for the supports, and from the following lemma, whose proof is given below.

 This ends the proof of Lemma \refer{4.2}.
\end{proof}
\begin{lemme}
\label{lem:suppb} {Under the hypothesis of Lemma \refer{4.2}, we
have $$\left(B^\beta_{\lam, \xi}f\right)^{\widehat{}} (-2 \lam
\xi_b, -2 \lam \xi_a, -\lam)= {\bf 1 }_{2^{2m}{{\mathcal C}
}}(\lam)\left(B^\beta_{\lam, \xi}f\right)^{\widehat{}} (-2 \lam
\xi_b, -2 \lam \xi_a, -\lam). $$ }
\end{lemme}

\begin{proof}
By definition of the standard Fourier transform on~$\R^{2d+1}$, we
have
\begin{eqnarray*} \left(B^\beta_{\lam,
\xi}f\right)^{\widehat{}} (-2 \lam \xi_b, -2 \lam \xi_a, -\lam)&=&
\int {\rm e}^{-i\,(-2 \lam \xi_b \cdot z_a -2 \lam \xi_a \cdot z_b
-\lam s)}B^\beta_{\lam, \xi}f(z,s)\,dz\,ds \\ &=& \int {\rm
e}^{i\,(2 \lam \xi_b \cdot z_a +2 \lam \xi_a \cdot z_b +\lam
s)}F_{\beta, \lam}(\xi-\overline z) f(z, s)\,dz\,ds\end{eqnarray*}
Denoting~$ 2 \lam(\xi_b \cdot z_a + \xi_a \cdot z_b) + \lam s $
by~$J_\lam (s, z, \xi)$, it follows that
 \[
 \left(B^\beta_{\lam, \xi}f \right)^{\widehat{}}(-2 \lam \xi_b, -2 \lam \xi_a,
-\lam) = \int {\rm e}^{iJ_\lam (s, z, \xi)}{\rm e}^{-\lam
(|\xi-\overline z|^2-|\xi|^2)} F_{\beta, \lam}(\xi-\overline
z){\rm e}^{\lam (|\xi-\overline z|^2-|\xi|^2)} f(z, s) \,dz\,ds.\]
Using that
 $$ {\rm e}^{\lam |\xi - \overline z|^2} = \sum_{\al \in
\N^d} (\overline \xi - z)^{\al} \frac{\lam^{|\al|} (\xi -
\overline z)^\al}{\al !}, $$ and observing  that the above series
is normally convergent on any compact, we deduce that
 $$ {\rm e}^{\lam |\xi - \overline z|^2} F_{\beta, \lam}(\xi-\overline
z) = \sum_{\al \in \N^d} (\overline \xi - z)^{\al}
\Bigl(\frac{\lam}{2}\Bigr)^{\frac{|\al|}{2}}\frac{1}{\al!}\sqrt{\frac{(\beta
 +\al)!}{\beta
 ! } }F_{\al+\beta, \lam}(\xi-\overline
z). $$ This leads, since~$f \in {\mathcal D}(\R^{2d+1})$, to
\begin{eqnarray*} &\displaystyle  \left(B^\beta_{\lam, \xi}f \right)^{\widehat{}}(-2 \lam \xi_b, -2 \lam \xi_a,
-\lam)   =  \sum_{\al \in \N^d} {\rm e}^{-\lam |\xi|^2}
\Bigl(\frac{\lam}{2}\Bigr)^{\frac{|\al|}{2}}\frac{1}{\al!}\sqrt{\frac{(\beta
 +\al)!}{\beta
 ! } } \\ & \displaystyle \quad \quad \quad \quad \quad \quad \quad \quad \quad \quad    \times \int {\rm e}^{iJ_\lam (s, z, \xi)}{\rm e}^{-\lam
(|\xi-\overline z|^2-|\xi|^2)} F_{\al+\beta, \lam}(\xi-\overline
z)(\overline \xi - z)^{\al}f(z, s) \,dz\,ds.\end{eqnarray*}
Recalling that $$A^\al_{\lam, \xi}f(z, s) = F_{\al,
\lam}(\xi-\overline z)
 {\rm e}^{-\lam\,(|\xi - \overline z|^2 - |\xi|^2)} f(z, s),$$
we get
\begin{eqnarray*}
\left(B^\beta_{\lam, \xi} f \right)^{\widehat{}}(-2 \lam \xi_b, -2
\lam \xi_a, -\lam) &= & \sum_{\al \in \N^d} {\rm e}^{-\lam
|\xi|^2}
\Bigl(\frac{\lam}{2}\Bigr)^{\frac{|\al|}{2}}\frac{1}{\al!}\sqrt{\frac{(\beta
 +\al)!}{\beta
 ! } }\\ & & \times
\left(A^{\beta+\al}_{\lam, \xi} (\overline \xi - z)^{\al} f
\right)^{\widehat{}}(-2 \lam \xi_b, -2 \lam \xi_a,
-\lam).\end{eqnarray*} Let us study separately each term of the
above series.  By  Lemma \refer{formuleslem} and using the fact
for~$\lam > 0$, ~$\overline Q_j^\lam= \partial_{\xi_j}$, we obtain
$${\mathcal F}(z_jf)(\lam) F = \frac {1}{2\lam} [
\partial_{\xi_j}, {\mathcal F}( f)(\lam)]F.$$ In particular,  for
any~$\g \in \N^d$, $${\mathcal F}(z_j f)(\lam) F_{\g, \lam}(\xi)=
\frac {1}{2\lam} \Bigl(
\partial_{\xi_j} {\mathcal F}( f)(\lam)F_{\g, \lam}(\xi) - {\mathcal F}(
f)(\lam)\partial_{\xi_j} F_{\g, \lam}(\xi)\Bigr).$$ The frequency
localization of  the function~$f$  in the ring~$2^{2m} {{\mathcal
C}}(\lam)$ implies then that the support in~$\lam$ of ${\mathcal
F}((\overline \xi_i -z_i)f)(\lam)F_{\g, \lam }(\xi)$ is included
in the same ring~$2^{2m} {{\mathcal C}}(\lam)$. An immediate
induction implies that for any multi-index~$\al$ the support
in~$\lam$ of~$\displaystyle {\mathcal F}((\overline \xi
-z)^{\al}f)(\lam)F_{\g, \lam }(\xi)$ is still included in the same
ring~$2^{2m} {{\mathcal C}}(\lam)$. Therefore, the support
in~$\lam$ of$$ \left(A^{\beta+\al}_{\lam, \xi} (\overline \xi
-z)^{\al} f\right)^{\widehat{}}(-2 \lam \xi_b, -2 \lam \xi_a,
-\lam)$$ is included in the ring~$2^{2m} {{\mathcal C}}(\lam)$.

As
each term of the series is supported in a fixed ring, the same
holds for the function $$\left(B^\beta_{\lam, \xi} f
\right)^{\widehat{}}(-2 \lam \xi_b, -2 \lam \xi_a, -\lam),$$
which ends the proof of the lemma. \end{proof}

The following results   will  also be useful in
Chapter~\ref{classical}.

  \begin{lemme}\label{lem.est}
  {  There  exists   a constant~$C$ such that for any function~$f$,
   \begin{equation}\label{eq:unifinfini}
   \| \Lambda_m \Delta_q f \|_{L^\infty(\H^d)} \leq C \|  \Delta_q f \|_{L^\infty(\H^d)}
   \end{equation}
   for any integers~$m$ and~$q$.\\
  Moreover if~$\r $ is a nonnegative real number, then there exists a constant $C$ such that for any function~$f$
  \begin{equation}\label{eq:unifest}
  \|\Lambda_m f \|_{L^\infty(\H^d)} \leq C 2^{-m\r}\| f \|_{C^\r(\H^d)} .
\end{equation}}
   \end{lemme}

\begin{proof}
 Let us first prove~\aref{eq:unifinfini}. We shall only give the general
 idea of the proof, as the method  follows closely a strategy initiated in~\cite{bgx} for the study of Littlewood-Paley operators, and followed also in~\cite{bg2} in the analysis of the heat operator.

   Recall that
   \[ {\mathcal F}(\Lambda_m \Delta_q f)(\lam)= \phi(2^{-2m}\lam){\mathcal F}(f)(\lam)(f)R^*(2^{-2q}D_\lam).\]
where~$\phi$ and~$R^*$ are smooth radial functions with values  in
the interval~$[0,1]$  supported in a unit ring of~$\R$.  This can
be also written
\[ {\mathcal F}(\Lambda_m \Delta_q f)(\lam)= \phi(2^{-2m}\lam){\mathcal F}(f)(\lam)\widetilde R^*(2^{-2q}D_\lam) R^*(2^{-2q}D_\lam)\]
where~$\widetilde R^*$ is a smooth radial
function  compactly supported in a unit ring so that~$\widetilde
R^*  R^* = R^*$.

 According to the fact that the Fourier transform exchanges   convolution   and   composition, we have
 $$ \Lambda_m \Delta_q f = \Delta_q f \star h_{m,q},$$
 where the
function~$h_{m,q}$ is defined by
 \[ {\mathcal F}(h_{m,q}) (\lam)= \phi(2^{-2m}\lam)\widetilde R^*(2^{-2q}D_\lam) .\]
Taking advantage of Young's inequalities, it therefore suffices   to prove that the function~$h_{m,q}$ belongs to~$L^1(\H^d)$ uniformly in~$m$ and~$q$.

By rescaling, we are reduced to investigating the function~$h_j$ defined by
 \[ {\mathcal F}(h_{j}) (\lam) \eqdefa \phi(2^{-2j}\lam)\widetilde R^*(D_\lam) .\]
 By the inversion formula\refeq{definvR}, we get
 \begin{equation}\label{formulagamma} h_{j}( z, s) = \frac{2^{d-1}}{\pi^{d+1}} \sum_m \int e^{-i \lam s}
\phi(2^{-2j}\lam)\widetilde R^*((2m+d)\lam) L_m^{(d-1)} (2|\lam|
|z|^2)e^{-|\lam| |z|^2} |\lam|^d d\lam.
\end{equation}
In order to prove that~$h_j$ belongs to~$L^1(\H^d)$ (uniformly in~$j$), the idea (as in~\cite{bgx} and~\cite{bg2}) consists in proving  that the function~$(z,s) \mapsto (is-|z|^2)^k h_j(z,s)$ belongs to~$L^\infty(\H^d)$ with uniform bounds in~$j$.

Let us start by considering the case~$k = 0$.  It is easy to see that the Laguerre polynomials defined in~(\ref{deflaguerre}) page~\pageref{deflaguerre} satisfy for all~$y \geq 0$
$$
|L_m^{(d-1)}(y) {\rm e}^{-\frac y2} | \leq C_d (m+1)^{d-1}
$$
Since~$\phi$ is bounded, this gives easily after the change of variables~$\beta = (2m+d)\lam$
\beq\label{esthjbeta}
| h_{j}( z, s) | \leq C \sum_m \frac1{m^2} \int |\widetilde R^*(\beta)| d \beta.
\eeq
To deal with the case~$k \neq 0$, we use the result proved in~\cite{bgx} (see also Proposition~\ref{is-z2} recalled in the introduction) stating that for any radial function~$g$, one has
$$
{\mathcal F} \left( (is-|z|^2) g(z,s)\right) (\lam)F_{\al,\lam} = Q_{|\alpha|}^*(\lam)F_{\al,\lam} ,
$$
where for all~$m \geq 1$,
\begin{eqnarray*}
Q^{\star}_{m}(\lam) &= &\frac{d}{d\lam}Q_m ( \lam) -
\frac{m}{\lam}(Q_m(\lam) -Q_{m-1}(\lam)) \quad \mbox{if} \: \lam >0,  \\
 Q^{\star}_{m}(\lam)& = &\frac{d}{d\lam}Q_m ( \lam) +
\frac{m + d}{|\lam |}(Q_{m}(\lam) -Q_{m+1}(\lam)) \quad \mbox{if} \: \lam < 0
\end{eqnarray*}
while~$Q_m$ is given by
$$
{\mathcal F} \left( g(z,s)\right) (\lam)F_{\al,\lam} = Q_{|\alpha|}(\lam)F_{\al,\lam} .
$$

The  proof then consists in applying Taylor formulas in the above expressions in order to reduce the problem to an estimate of the same type as~(\ref{esthjbeta}).
The only difference with the case treated in~\cite{bgx} and~\cite{bg2}   lies in the dependence on~$j$. However it can be noticed that due to the support assumptions on~$\phi$ and~$\widetilde R^*$, there are two positive constants~$c_1$ and~$c_2$ such that
$$
h_{j}( z, s) = \frac{2^{d-1}}{\pi^{d+1}} \sum_{m \in C_j}\int e^{-i \lam s}
\phi(2^{-2j}\lam)\widetilde R^*((2m+d)\lam) L_m^{(d-1)} (2|\lam|
|z|^2)e^{-|\lam| |z|^2} |\lam|^d d\lam
$$
with~$\displaystyle
  C_j \eqdefa \{m \in \N, \: c_1 2^{-2j} \leq 2m+d \leq c_2 2^{-2j}\}.
$
 Now let us decompose~$h_j$ into two parts:
$$
 h_{j}( z, s) =  h^1_{j}( z, s) + h_{j}^2( z, s), \quad \mbox{where}
$$
$$
 h_j^1( z, s)\eqdefa \frac{2^{d-1}}{\pi^{d+1}} \sum_{m \in C_j} \int e^{-i \lam s}
\phi((2m+d)\lam)\widetilde R^*((2m+d)\lam) L_m^{(d-1)} (2|\lam|
|z|^2)e^{-|\lam| |z|^2} |\lam|^d d\lam .
$$
The term~$h_j^1$ is dealt with exactly in the same way as in~\cite{bgx} and~\cite{bg2}.

For~$h_j^2$
we shall use the Taylor formula
$$
\phi(2^{-2j}\lam) - \phi((2m+d)\lam) = (2^{-2j} - (2m+d)) \lam\int_0^1 \phi' (t 2^{-2j}\lam + (1-t)(2m+d)\lam ) \: dt.
$$
But for any~$m \in C_j$, one can find~$\alpha_m \in [c_2^{-1},c_1^{-1}]$ such that
$$
2^{-2j} = \alpha_m (2m+d).
$$
It follows that one can write
$$
\phi(2^{-2j}\lam) - \phi((2m+d)\lam) = ( \alpha_m - 1) (2m+d) \lam\, \int_0^1 \phi' \left([t \alpha_m  + (1-t)] (2m+d)\lam \right) \: dt
$$
and the change of variables~$u = t \alpha_m  + (1-t)$ gives
$$
\longformule{
\widetilde R^*((2m+d)\lam) \left( \phi(2^{-2j}\lam) - \phi((2m+d)\lam) \right)= (2m+d) \lam \widetilde R^*((2m+d)\lam)}{\times \int_{\R} \phi'\left(u(2m+d) \lam\right) {\mathbf 1}_{[1,\alpha_m]}
du.}$$
 This form is of the same kind that considered in~\cite{bgx}, and allows to end the proof of~(\ref{eq:unifinfini}) exactly in the same way.

Let us prove now~\aref{eq:unifest}.
 On the support of the Fourier transform of $\Delta_p\Lam_m$, we have
 $D_\lam\sim2^{2p}$ and $|\lam|\sim 2^{2m}$. Therefore, $2^{2(p-m)}$ has to be greater than or equal to $1$. This implies
that the only indexes $(p,m)$ that we have to consider are those
such that $0<m\leq p$. So $$ \Lambda_m f = \Lambda_m ( {\rm Id} -
S_{m-1})f.$$  Therefore using  \aref{eq:unifinfini}, we have
\begin{eqnarray*}
 \|  \Lambda_m f \|_{L^\infty(\H^d)} & \leq & C \sum_{q \geq m -1 }\| \Lambda_m \Delta_q f \|_{L^\infty(\H^d)} \\
& \leq & C \sum_{q \geq m -1 }\|  \Delta_q f \|_{L^\infty(\H^d)}
\\& \leq & C \sum_{q \geq m -1 } 2^{-q \r}\| f \|_{C^\r(\H^d)} ,
\end{eqnarray*}
so finally
$$
 \|  \Lambda_m f \|_{L^\infty(\H^d)} \leq C 2^{-m\r}\| f \|_{C^\r(\H^d)}.
$$

 That proves the lemma.
\end{proof}

  \section{The symbol of Littlewood-Paley operators}\label{LPsymbolappendix}
 \setcounter{equation}{0}

 Applying Proposition~\ref{symbR} of Chapter~\ref{intro} (see its statement page~\pageref{symbR}) to $\lambda$-dependent functions of the harmonic oscillator, we obtain the symbol of our Littlewood-Paley operators, as stated in the next proposition. The proof of the proposition relies heavily on that of Proposition~\ref{symbR}  which is itself proved in Appendix B. Therefore we postpone the proof also to Appendix B, page~\pageref{symbollpproofappendix}.

\begin{prop} \label{symbDeltap}
 {  The operators $\Delta_p$ (resp. $S_p$) are pseudodifferental operators of order~$0$. Besides, if we denote by  $\Phi_p(\lam,\xi,\eta)$ (resp. $\Psi_p(\lambda,\xi,\eta)$) their symbols, there exist two functions~$\phi$ and~$\psi$ in~${\mathcal C}^\infty(\R^2)$ such that for $\lam\not=0$,
$$\Phi_p(\lam,\xi,\eta)= \phi(2^{-2p}|\lam|,
2^{-2p}|\lam|(\xi^2+\eta^2))\;\;{ and}\;\;
 \Psi_p(\lam,\xi,\eta)= \psi(2^{-2p}|\lam|, 2^{-2p}|\lam|(\xi^2+\eta^2)).
$$ More precisely one has
\begin{equation}\label{eq:statphas}
\forall\lam\not=0,\;\; \quad
\phi(\lambda,\rho)={\mbox{sgn} \: \lam \over  \lambda  }\int (\cos \tau)^{-d} {\rm
e}^{{i\over\lambda}(-r\tau + \rho{\rm tg}\tau)} R^*(4r)d\tau dr,
  \end{equation}
   and a similar formula for~$\psi$. }
 \end{prop}

$ $
\begin{rem}\label{rem:statphas}
The stationary phase theorem  (see Appendix B) implies that the function~$\phi(\lam,\rho)$ of \aref{eq:statphas} has an asymptotic expansion in powers of $\lam$  as $\lam$ goes to $0$, the first term of which is $R^*(\rho)$. Besides, the change of variables $\tau\mapsto -\tau$ gives that $\phi(-\lam,\rho)=\phi(\lam,\rho)$.
Therefore, the function
$$(y,\eta)\mapsto \Phi_p\left(\lam,{\rm sgn}(\lam){\xi\over \sqrt{|\lam|}},{\eta\over\sqrt{|\lam]}}\right)$$
 is equal to $\phi(2^{-2p}|\lam|,2^{-2p}(\xi^2+\eta^2))$ and is smooth
 close to $\lam=0$.
\end{rem}


 \chapter{The action of pseudodifferential operators on Sobolev spaces}\label{classical}
\setcounter{equation}{0}


In this chapter we shall be giving the proof of
Theorem~\ref{contHs}.
  In the first paragraph we reduce the study  to the case of operators of order zero, and in the second paragraph we show that it is possible to restrict our attention to a fixed regularity index in a certain range.  We  then follow the
strategy of the proof of continuity of pseudodifferential
operators in the~$\R^d$ case due to R. Coifman and
Y.~Meyer\ccite{cm}. The proof is based on the two following ideas:
  we introduce the notion of   {reduced symbols} (see Section~\ref{redandred}) of which we prove
the  continuity. Then,  we obtain in
Section~\ref{decompositionreduced} that any symbol~$a$  of
order~$0$ on the Heisenberg group is a sum of a convergent series
of reduced symbols,  and finally deduce the continuity for the
operator~${\rm Op}(a)$.

Let us mention that the proof below would be much easier if the
symbols were only functions of~$(w,y,\eta)$, and not also
of~$\lam$ :  in that case, one would not need to use an additional
cutoff in~$\lam$ via the operators $\Lambda_p$ (see
Section~\ref{Proofprop}), which will induce some
technicalities.

\section{Reduction to the case of operators of order zero}\label{k=0}
 \setcounter{equation}{0}
In this paragraph we shall reduce the study to the case of zero-order operators. Suppose therefore that the result has been proved for any zero-order operator, meaning that for any operator~$b \in S_{\H^d}(0)$ of regularity~$C^\rho (\H^d)$ and for any~$|s| \leq \rho$ if~$\rho >2(2d+1)$ (resp.~$0<s< \rho$ if~$\rho >0$), the operator~${\rm Op}(b)$ maps continuously~$ H^s ({\H}^d)$ into itself.

 Let~$a$ be a symbol of order~$\mu \in \R$.
Then for any $f\in H^s ({\H}^d)$, $$ {\rm Op}(a) f(w)  =
{2^{d-1}\over\pi^{d+1}} \int{\rm tr}\left( u^\lam_{w^{-1}}
{\mathcal F}(f)(\lam) A_\lam(w)\right)|\lam|^d\,d\lam$$ with
\begin{eqnarray*}
{\mathcal F}(f)(\lam)A_\lam (w) & = & {\mathcal
F}(f)(\lam)J^*_\lam op^w(a(w,\lam))J_\lam\\ & = & {\mathcal
F}(({\rm Id}-\Delta_{\H^d})^{\mu\over 2} f)(\lam) J_\lam^* op^w(
m_{-\mu}^{(\lam)}\# a ) J_\lam.
\end{eqnarray*}
This can be written
 $$ {\rm Op}(a)f(w)={\rm Op}(b)
({\rm Id}-\Delta_{\H^d})^{\frac{\mu}{ 2}} f(w), $$ where~$b\eqdefa
m_{-\mu}^{(\lam)}\# a$ is a symbol of order 0.
 The boundedness of ${\rm Op}(b)$ from~$H^{s-\mu}$ to $H^{s-\mu}$ for~$|s-\mu|<\rho$ (resp.~$0<s< \rho$ if~$\rho >0$) then yields the existence of  constants
$C$ and $C'$ such that $$\|{\rm Op}(a)f\|_{H^{s-\mu}}\leq
\,C\,\|({\rm Id}-\Delta_{\H^d})^{\frac{\mu}{ 2}} f\|_{H^{s-\mu}}\leq
C'\,\|f\|_{H^s}.$$

Therefore it suffices to prove the
theorem for symbols of order 0, which we will assume from now on.

\section{Reduction to the case of a fixed regularity index}\label{reds}
 \setcounter{equation}{0}

In this paragraph, we shall reduce the study of the continuity of
pseudodifferential operators of order~$0$ on Sobolev spaces from
arbitrary Sobolev spaces~$H^t ({\H}^d)$, to one Sobolev space~$H^s
({\H}^d)$ with a  regularity index~$s$ such that~$ 0< s <
\delta_0$, where~$\delta_0$ (chosen equal to $\rho-[\rho]$) will be the index entering the assumptions of
Proposition\refer{propphipdeltap}, page~\pageref{propphipdeltap}.

In order to do so, let us suppose that the continuity in $H^s(\H^d)$ is proved for any symbol of order $0$ with $0<s<\delta_0$ (note that $\delta_0\leq \rho$). Consider a symbol~$a(w,\lambda, \xi,\eta)$ of order~$0$.
Let~$\alpha $ be a multi-index in~$\N^d$ with $|\alpha|\leq [\rho]$ and, using Proposition~\ref{ZOp(a)prop}, define the ${\mathcal C}^{\delta_0}$ symbol $b_\alpha$ by
$${\rm Op}(b_\alpha)={\mathcal Z}^\alpha {\rm Op}(a) ({\rm Id}-\Delta_{\H^d})^{-{|\alpha|\over 2}}.$$
Then ${\rm Op}(b_\alpha)$ maps $H^t(\H^d)$ into itself for $0<t<\delta_0$. Therefore, there exists a constant $C$ such that for any $f\in H^{t+[\rho]}(\H^d)$,
\begin{eqnarray*}
\|{\rm Op}(a) f\|^2_{H^{t+[\rho]}(\H^d)}
 & = & \sum_{|\alpha|\leq[\rho]} \| {\rm Op}(b_\alpha)({\rm Id}-\Delta_{\H^d} )^{|\alpha|\over 2}f\|^2_{H^t(\H^d)}\\
 & \leq &  C\,\sum_{|\alpha|\leq[\rho]} \| ({\rm Id}-\Delta_{\H^d} )^{|\alpha|\over 2}f\|^2_{H^t(\H^d)}= C \,\|f\|_{H^{t+[\rho]}(\H^d)}^2.
\end{eqnarray*}
Therefore, ${\rm Op}(a)$ maps $H^s(\H^d)$ into itself for $s=t+[\rho]$, $t<\delta_0$, whence for $0<s<\rho$.

Assuming $\rho>2(2d+1)$ and using the fact that the adjoint of a pseudodifferential operator is a pseudodifferential operator of the same order, we get the continuity on $H^s(\H^d)$ for $0<|s|<\rho$.

Then $s=0$ is obtained by interpolation.

\section{Reduced and reduceable symbols}\label{redandred}
 \setcounter{equation}{0}
Let us start by defining the notion of reduced and reduceable
symbols.
\begin{defin}\label{reducedsymbol}
 { Let $t$ be a symbol. Then~$t$ is   reduceable if it  can be decomposed in the following way: for all~$(w,\lam,\xi     ,\eta) \in \H^d \times \R^{*}\times \R^{2d}$
\begin{eqnarray*}
t(w,\lam,\xi     ,\eta) &=& \sum_{k \in \ZZZ^{2d}} t^k (w,\lam,\xi     ,\eta) ,
\quad \mbox{where}\\ t^k (w,\lam,\xi     ,\eta)  &=&   b_{-1}^k(w,\lam)
\Psi^k(\lam,\xi     ,\eta)+ \sum_{p=0}^\infty b_{p}^k(w,\lam )   \Phi_{p
}^k (\lam,\xi     ,\eta).
\end{eqnarray*}
 with
 $$
  \Phi_{p}^k (\lam,\xi     ,\eta) \eqdefa \widetilde\Phi_{p}^k(\sqrt{|\lam|} \xi     , \sqrt{|\lam|} \eta)\;\;{\rm while}\;\; \widetilde\Phi_{p}^k ( \xi     ,\eta) \eqdefa e^{ik \cdot( 2^{-p}\xi     ,2^{-p}\eta)} \Phi ( 2^{-2p}(\xi     ^2+ \eta^2))
 $$
 and~$\Phi$ is a smooth function with values in~$[0,1]$, compactly supported in~$]0,\infty[$.

Similarly $$\Psi^k(\lam,\xi     ,\eta) \eqdefa
\widetilde\Psi^k(\sqrt{|\lam|}\xi     ,\sqrt{|\lam|}\eta)\;\;{\rm
where}\;\;\widetilde    \Psi^k ( \xi     ,\eta) \eqdefa e^{ik \cdot(
\xi     ,\eta)} \Psi ( \xi     ^2+ \eta^2)$$ and $\Psi$  a smooth function with values in~$[0,1]$, compactly supported in~$]-1,1[$.

Finally   the functions~$b_{p}^k(\cdot,\lam)$
belong to the H\"older space~$C^{\r}({\H}^d)$ with
\begin{equation}\label{def:Ak}
\displaystyle \sup_{p ,\lam} \|b_{p}^k(\cdot,\lam)
\|_{C^{\r}({\H}^d)}= A_k < \infty.
\end{equation}
 The symbols~$t^k$ are called reduced symbols. }
\end{defin}

 It follows from the analysis of the examples of Chapter~\ref{fundamental}, Section~\ref{examples} that for any $k\in\ZZZ^{2d}$ and $p\in\N$, the operator ${\rm Op}(b^k_p(w,\lam)\Phi^k_p(\lam,\xi     ,\eta))$ is bounded in $H^s(\H^d)$ since one can write by easy functional calculus
$${\rm Op}\Bigl(b^k_p(w,\lam)\Phi^k_p(\lam,\xi     ,\eta)\Bigr)={\rm
Op}\Bigl(b^k_p(w,\lam)\Bigr)\circ{\rm
Op}\Bigl(\Phi^k_p(\lam,\xi     ,\eta)\Bigr)$$ where the two operators of
the right-hand side are bounded operators on $H^s(\H^d)$ (see Chapter~\ref{intro},
Sections~\ref{generalizedmultiplication}
and~\ref{fouriermultipliers} respectively).\\ The same fact is true
for ${\rm Op}\Bigl(b^k_{-1}(w,\lam)\Psi^k(\lam,\xi     ,\eta)\Bigr)$.
Besides,  by Proposition~\ref{prop:fourmult} stated page~\pageref{prop:fourmult}, there is a constant~$C$ (independent of~$k$) such that
\begin{eqnarray}\label{bk-1}
\|{\rm Op}(b^k_{-1}(w,\lam)\Psi^k(\lam,\xi     ,\eta))\|_{{\mathcal
L}(H^s(\H^d))} & \leq & C\,A_k\, \|\widetilde
\Psi^k\|_{n;S(1,g)}\quad \mbox{and}\\ \nonumber \|{\rm
Op}(b^k_p(w,\lam)\Phi^k_p(\lam,\xi     ,\eta))\|_{{\mathcal
L}(H^s(\H^d))}  & \leq  & C \,A_k\,\|\widetilde
\Phi^k_p\|_{n;S(1,g)}
\end{eqnarray}
where we recall that $g$ is the harmonic oscillator metric of Section~\ref{harmonic} in Chapter~\ref{intro}.

The main ingredient in  the proof of  Theorem~\ref{contHs} is the
following result.
\begin{prop}\label{prop:normsymbred}{
Let~$k$ be fixed in~$\ZZZ^{2d}$ and~$t^k$ be a reduced symbol as defined in Definition~\ref{reducedsymbol}. The
operator~${\rm Op} (t^k)$  maps continuously~$H^s(\H^d)$ into
itself for~$0\leq s< \r$. Its operator norm is bounded by~$ C A_k
 (1+ |k|)^n $ for some integer~$n$, where~$C$ is a   constant (independent of~$k$).
}\end{prop}

The proof of this proposition is postponed to
Section~\ref{Proofprop}.

\begin{rem}\label{crucial}
Due to Proposition~\ref{prop:normsymbred}, a reduceable symbol $t$
is the symbol of a bounded operator on $H^s(\H^d)$ as soon as
$(A_k(1+|k|)^n)_{k \in \ZZZ^{2d}}$ belongs to~$\ell^1(\ZZZ^{2d})$.
\end{rem}

\section{Decomposition into reduced symbols and proof of the theorem}\label{decompositionreduced}
 \setcounter{equation}{0}

The aim of this section is to prove the following lemma.
\begin{lemme}\label{decsymbolred}
 {  Let~$a$ be a symbol of order~$0$. Then $a$ is reduceable and, with the notation of Definition~\ref{reducedsymbol},
for any integer~$N$, there is a constant~$C_N$ such that for
any~$k \in \ZZZ^{2d}$, \beq\label{estimateOptk} A_k \leq
\frac{C_N}{(1+|k|)^N}\cdotp \eeq }
\end{lemme}

In view of Remark~\ref{crucial}, Lemma~\ref{decsymbolred} gives
directly Theorem~\ref{contHs} (up to the proof of Proposition~\ref{prop:normsymbred}).

\begin{proof}
Let us consider~$\psi$ and~$\phi$ defining a partition of unity as in~\aref{partident} page~\pageref{partident}: one can write
\begin{equation}\label{funddylph2}
\forall (\lam,\xi     ,\eta) \in \R^{*}\times\R^{2d}, \quad \psi
\left(|\lam|(\xi     ^2+ \eta^2) \right) + \sum_{p \geq 0} \phi \left(
2^{-2p} |\lam|(\xi     ^2+ \eta^2)  \right) = 1.
\end{equation}
 Then
 \begin{eqnarray*}
a(w,\lam,\xi     ,\eta)& = &a(w,\lam,\xi     ,\eta)\psi\left( |\lam|(\xi     ^2+
\eta^2)\right) + \sum_{p \geq 0} a(w,\lam,\xi     ,\eta)\phi \left(
2^{-2p} |\lam|(\xi     ^2+ \eta^2 ) \right)\\
 & = & b_{-1}(w,\lam,\sqrt{|\lam|}\xi     ,\sqrt{|\lam|}\eta) + \sum_{p \geq 0}
b_p(w,\lam,2^{-p }\sqrt{|\lam|}\xi     ,2^{-p }\sqrt{|\lam|}\eta)
\end{eqnarray*} with
 \begin{eqnarray*}
 b_{-1}(w,\lam,\xi     ,\eta) &\eqdefa & \wt{a}(w,\lam,\xi     ,\eta)\psi ( \xi     ^2+ \eta^2) \quad  \mbox{and} \\
\\ b_p(w,\lam,\xi     ,\eta)& \eqdefa& \wt{a}(w, \lam ,2^{ p}\xi     , 2^{ p} \eta)\phi (\xi     ^2+
\eta^2)  \quad  \mbox{ for} \quad p \geq 0,
\end{eqnarray*}
where~$\wt{a}(w,\lam,\xi     ,\eta) \eqdefa \displaystyle a(w, \lam ,\frac{\xi     }{\sqrt{|\lam|}},
\frac{\eta}{\sqrt{|\lam|}})$. The functions~$b_p$ are compactly
supported in~$( \xi     , \eta )$, in the ring~${\mathcal C}$ for~$p \geq
0$ and in the   ball~${\mathcal B}$ for~$p = -1$. Moreover,
denoting by~$\partial$ a differentiation in~$ \xi     $ or~$\eta$, we
have, for all~$p \geq -1$, $$ \partial b_p (w,\lam,\xi     ,\eta) =  2^{
p}(\partial\wt{a})(w, \lam ,2^{ p}\xi     , 2^{ p} \eta)\phi(\xi     ^2+ \eta^2)
+ 2\,\xi      \phi' (\xi     ^2+ \eta^2)\wt{a}(w, \lam ,2^{ p}\xi     , 2^{ p} \eta).$$
We deduce that
\begin{eqnarray*}
\left|\partial b_p(w,\lam,\xi     ,\eta)\right| & \leq &  C\,\frac {2^{p}}{\sqrt{1+|\lam|+ (2^{p}\xi     )^2+(2^{p}\eta)^2}} |\phi(\xi     ^2+
\eta^2)|+ C |\xi     |\left | \phi' (\xi     ^2+ \eta^2)\right|,  \quad
\mbox{and}\\ \left|\lam \partial_\lam b_p(w,\lam,\xi     ,\eta)\right| &
\leq & C |\lam\partial_\lam \wt{a}(w, \lam ,2^{ p}\xi     , 2^{ p} \eta)|
\end{eqnarray*}
so using the boundedness of the symbol norm of $a$ and the fact
that $\phi $ is compactly supported,
 and arguing similarly for higher order derivatives, one gets the following uniform norm bound on ~$b_p$:
\begin{eqnarray} \label{unif3} \sup_{p,\lam,\xi     ,\eta}\|(\lam
\partial_\lam)^m\partial^\beta_{(\xi     ,\eta)}
b_p(\cdot,\lam,\xi     ,\eta)\|_{C^{\r}({\H}^d)} &\leq & C_{ \beta,m }.
\end{eqnarray}

 $ $

 Now, since for~$p \geq 0$ the functions~$b_p$
are compactly supported in~$(\xi     ,\eta )$, in a ring~${\mathcal C}$
independent of~$p$, we can write a decomposition in Fourier
series: $$
 b_p(w,\lam,\xi     ,\eta) =\sum_{k
\in \ZZZ^{2d}} {\rm e}^{ik \cdot (\xi     ,\eta)} b_{p}^k (w,\lam )
\widetilde \phi (\xi     ^2+ \eta^2), $$ where~$ \widetilde \phi$ is a
smooth, radial function, compactly supported in a unit ring, so
that~$\phi \widetilde \phi = \phi$. We have of    course
\begin{equation}\label{intformula}
 b_{p}^k (w,\lam) ={1\over (2\pi)^d}\int_{ {\mathcal C}} {\rm e}^{-ik \cdot (\xi     ,\eta)} b_p(w,\lam,\xi     ,\eta)   d\xi     d\eta.
\end{equation}
Along the same lines, we get
 $$ b_{-1}(w,\lam,\xi     ,\eta)
=\sum_{k \in \ZZZ^{2d}} {\rm e}^{ik \cdot (\xi     ,\eta)} b_{-1}^k
(w,\lam ) \widetilde \psi(\xi     ^2+ \eta^2), $$ where~$ \widetilde
\psi$ is a smooth, radial function, compactly supported in a unit
ball, so that~$\psi \widetilde \psi = \psi$.

Defining $$ \Phi^k (\xi     ,\eta) \eqdefa {\rm e}^{ik \cdot (\xi     ,\eta)}
\widetilde \phi(  \xi     ^2+ \eta^2),
 $$ it turns out that
 \begin{eqnarray*}
    a(w,\lam,\xi     ,\eta)
&=& b_{-1}(w,\lam,\sqrt{|\lam|}\xi     ,\sqrt{|\lam|}\eta) + \sum_{p,k}
b_{p}^k (w,\lam )\Phi^k (2^{-p}\sqrt{|\lam|}\xi     , 2^{-p}
\sqrt{|\lam|}\eta)\\ & =&
b_{-1}(w,\lam,\sqrt{|\lam|}\xi     ,\sqrt{|\lam|}\eta) +  \sum_k t^k
(w,\lam,\xi     ,\eta).
\end{eqnarray*}

 $ $

That concludes the fact that $a$ is reduceable. It
remains to prove~(\ref{estimateOptk}).
 From the integral
 formula \refeq{intformula}, we infer that for any multi-index~$\beta$ and, to simplify,  for $p\geq 0$
\begin{eqnarray*}
\left| k^\beta b_{p}^k (w,\lam) \right| &=  &\left|{1\over (2\pi)^d}
\int_{{\mathcal C}} k^\beta {\rm e}^{-ik \cdot (\xi     ,\eta)}
b_p(w,\lam,\xi     ,\eta) \,d\xi     d\eta\right|
\\
 &\leq  & C\,\int_{{\mathcal C}} \left| \partial_{(\xi     ,\eta)}^\beta
b_p(w,\lam,\xi     ,\eta) \right|\,d\xi     d\eta
\end{eqnarray*}
Using\refeq{unif3}, we deduce that
\begin{equation}\label{eq:bpk}
 \sup_{p ,\lam} \left\| k^\beta b_{p}^k (\cdot,\lam) \right\|_{C^\rho(\H^d)} \leq C_\beta
 \end{equation}
 and Lemma\refer{decsymbolred} is proved.
\end{proof}

 \section[Proof of Proposition 4.1]{Proof of Proposition \ref{prop:normsymbred}}\label{Proofprop}
 \setcounter{equation}{0}
Now it remains to prove Proposition \ref{prop:normsymbred}. We will first give the main steps of the proof and peform some reductions, and then prove the result.
\subsection{Reductions}
 Let us give the main steps of the proof.   An easy computation gives that there is a constant~$C$ such that for any integer~$p$ and any~$k \in \ZZZ^{2d}$,
\begin{equation}\label{phikp}
 \| \widetilde \Psi^k\|_{n;S(1,g)}+\| \widetilde \Phi^k_p\|_{n;S(1,g)}\leq C\,(1+|k|)^n.
 \end{equation}
 Therefore, in view of (\ref{bk-1}), one  has
  $$
 \left\|{\rm Op}\left(b^k_{-1}(w,\lam)\Psi^k(\lam,\xi     ,\eta)\right) \right\|_{{\mathcal L}(H^s(\H^d))}\leq C A_k (1+|k|)^n.
 $$
 It remains to consider $p\in\N$, and in particular to control the sum over~$p$. The fact that~$b_p^k(w,\lam)$ depends on $\lam$ induces a serious difficulty,
 which we shall deal with by considering a partition of unity in~$\lam$.
  Thus by the same trick as before, we use functions~$\phi$ and~$\psi$ such that~\aref{partident} holds and we write
 $$  b_p^k(w,\lam)=b_p^k(w,\lam)\psi(\lam)+ \sum_{r\in\N} b_{p}^k(w,\lam) \phi(2^{-2r}\lam) .
 $$
 Using the fact that $\phi$ is compactly supported, we decompose the function~$b_p^k(w,2^{2r}\lam)\phi(\lam)$  in Fourier series and write
 $$b_p^k(w,\lam)=\sum_{j\in\ZZZ} b_{p,-1}^{kj}(w){\rm e}^{ij\lam} \widetilde \psi(\lam)+\sum_{r\in\N, j\in\ZZZ} b_{p,r}^{kj}(w){\rm e}^{ij2^{-2r}\lam} \widetilde \phi(2^{-2r}\lam),$$
 where
 $$b_{p,-1}^{kj}(w)=\int_{{\mathcal B}} {\rm e} ^{-ij\lam}b_p^k(w,\lam)\psi(\lam)\,d\lam,\;\;
 b_{p,r}^{kj}(w)=\int_{{\mathcal C}} {\rm e} ^{-ij\lam}b_p^k(w,2^{2r}\lam)\phi(\lam)\,d\lam$$
 and  $\widetilde\phi$,~$\widetilde\psi$ are smooth and compactly supported respectively in~${\mathcal C} $ and~${\mathcal B} $,    such that $\widetilde\phi\phi=\phi$, and~$\widetilde\psi \psi=\psi$.
 We observe that   Estimate\refeq{unif3} satisfied by $b_p$ ensures that for all integers~$N$, there is a constant~$C_N$ such that for all indexes~$p,r,j,k$, we have
 \begin{equation}\label{eq:fond}
 \sup_{p ,r}(1+|j|)^N \|b_{p,r}^{kj}\|_{C^\rho(\H^d)}\leq \frac {C_N}{(1+|k|)^N}\cdotp
 \end{equation}
  Indeed, by the Leibniz formula
 \begin{eqnarray*}
 \left| j^n k^\beta b_{p,r}^{kj}(w)\right|
  &\leq& C \sum_{m\leq n }\left| \int_{{\mathcal C}} {\rm e} ^{-ij\lam}  (\lam 2^{2r})^m  k^\beta(\partial_\lam ^m
  b_p^k)(w,2^{2r}\lam) \lam^{-m}(\partial_\lam ^{n-m}\phi)(\lam)\,d\lam \right|
 \\ &\leq & C \supetage{\mu}{m\leq n}\left| k^\beta
 (\mu\partial_\mu)^m b_p^k(w,\mu) \right|
 \\ &\leq & C
  \supetage{\lam}{m\leq n}\left| (\lam \partial_\lam)^m\partial_{(\xi,\eta)}^\beta
b_p(w,\lam,y,\eta) \right|.
\end{eqnarray*}
Owing to\refeq{unif3}, we deduce that~(\ref{eq:fond}) holds. That
estimate will  ensure the convergence in~$j$ of the series.
 In the following, we therefore consider, for each~$j$ and~$k$, the quantities
 \begin{eqnarray*}
 \widetilde t^{kj}(w,\lam,\xi     ,\eta) & \eqdefa& \sum_p b_{p,-1}^{kj} (w)\psi^j(\lam)\Phi_p^k(\lam,\xi     ,\eta)\;\;{\rm and}\\
 t^{kj}(w,\lam,\xi     ,\eta) & \eqdefa& \sum_{p,r} b^{kj}_{p,r} (w) \phi^j(2^{-2r}\lam)  \Phi_p^k(\lam,\xi     ,\eta)
 \end{eqnarray*}
 where $\phi^j(\lam)={\rm e}^{ij\lam }\phi(\lam)$,
 and
 $\psi^j(\lam)={\rm e}^{ij\lam}\psi(\lam)$.
 Then we will consider the summation in $k$ and $j$ of $t^{kj}$ and $\widetilde t^{kj}$.

 The analysis of the convergence of $\widetilde t^{kj}$ follows the same lines as that of $t^{kj}$ with great simplifications since the summation is only on one index, namely~$p$. Therefore, we focus on the convergence of~$t^{kj}$ and leave to the reader the easy adaptation of the proof to the case of~$\widetilde t^{kj}$.

 Let us therefore now study $t^{kj}$.
  We
  truncate~$b_{p,r}^{kj}$ into high and low frequencies, by defining (for some integer~$M$ to be chosen large enough later, independently of all the other summation indices),
 \begin{equation}\label{deflh}
 \ell_{pr}\eqdefa S_{p-M} b_{p,r} ^{kj}\quad \mbox{and} \quad h_{pr} \eqdefa ({\rm Id}  - S_{p-M}) b_{p,r}^{kj},
 \end{equation}
 where~$S_p$ is a
Littlewood-Paley truncation operator on the Heisenberg group, as
defined in Chapter~\ref{prelimins}, Section~\ref{littlewoodpaley}.
Let us notice that by Lemma\refer{estfgHs}, one has the following
norm estimates on~$\ell_{pr}$ and~$h_{pr}$:
 \begin{eqnarray*}
 \sup_{p ,r}\|\ell_{pr}\|_{C^\rho(\H^d)}&\leq& \sup_{p ,r} \|b_{p,r}^{kj}\|_{C^\rho(\H^d)}\\
 \sup_{r}\|h_{pr}\|_{L^\infty(\H^d)}&\leq&
   2^{-p\rho}\sup_{p ,r} \|b_{p,r}^{kj}\|_{C^\rho(\H^d)} \\
  \sup_{r}\|h_{pr}\|_{C^\s(\H^d)} &\leq&  2^{-p(\rho-\s)}\sup_{p ,r} \|b_{p,r}^{kj}\|_{C^\rho(\H^d)},
 \end{eqnarray*}
 for ~$0 < \s \leq
 \r$.

 This allows us to
write~$t^{kj}= \widetilde t^ \sharp + \widetilde t^ \flat $, with
$$\displaylines{ \widetilde t^\sharp(w,\lam,\xi     ,\eta)  \eqdefa
\sum_{p,r} h_{pr} (w) \phi^j(2^{-2r}\lam)\Phi^k (
2^{-p}\sqrt{|\lam|}\,\xi     ,2^{-p}\sqrt{|\lam|}\, \eta)  \quad
\mbox{and} \cr \widetilde t^\flat(w,\lam,\xi     ,\eta)  \eqdefa
\sum_{p,r} \ell_{pr}(w)\phi^j(2^{-2r}\lam) \Phi^k (
2^{-p}\sqrt{|\lam|}\,\xi     ,2^{-p}\sqrt{|\lam|}\, \eta) .\cr}
 $$
 We have dropped the indexes $k$ and~$ j$ to avoid too heavy notations.
 Before performing the
  study of each of those operators, we begin by a remark which will happen to be crucial for our purpose.

  \subsection{Spectral localization}
  In this subsection, we take advantage of Proposition~\ref{propphipdeltap} of Chapter~\ref{prelimins} (see page~\pageref{propphipdeltap}) to use  spectral localisation.
  We first  observe that
  \begin{eqnarray*}
\Phi^k_{p} (\lam, \xi     , \eta) &=& {\rm e}^{i\sqrt{|\lam|}k
\cdot(2^{-p}\xi     ,2^{-p}\eta)} \Phi (2^{-2p}|\lam|(\xi     ^2+ \eta^2))\\&=&
{\rm e}^{i\sqrt{|\lam|}k \cdot(2^{-p}\xi     ,2^{-p}\eta)}\Phi
(2^{-2p}|\lam|(\xi     ^2+ \eta^2))\widetilde \Phi (2^{-2p}|\lam|(\xi     ^2+
\eta^2)),\end{eqnarray*} where~$\widetilde \Phi$ is a smooth
radial function  compactly supported in a unit ring so that~$\Phi
\widetilde \Phi = \Phi$.

Symbolic calculus  gives that for any $N\in \N$, there exists a
symbol $r^{(N)}_{k,p}$ such that
\begin{eqnarray*}
 op^w (\Phi^k_{p})&= & op^w ( \Phi^k_{p} \cdot a_p)\\
&= & op^w ( \Phi^k_{p} ) \circ op^w ( a_p) + op^w ( r_{k,p}^{(N)})
,\end{eqnarray*} where~$a_p (y,\eta)= \widetilde \Phi
(2^{-2p}|\lam|(y^2+ \eta^2)) $ and for any integer~$n$ one
has
$$
 \| r_{k,p}^{(N)}\|_{n;S(1,g)} \leq C (1+|k|)^{N+n}2^{-Np}.
$$
One obtains that for some integer~$n$,
  $$\|op^w(r_{k,p}^{(N)})\|_{{\mathcal L}(L^2(\R^d))}\leq  C \,
(1+ |k|)^{N+n} \, 2^{-Np},$$ and since~${\rm Op}(r^{(N)}_{k,p})$
is a Fourier multiplier we get
\begin{equation} \label{est:rnkp}
\|{\rm Op}(r^{(N)}_{k,p}) u\|_{H^s(\H^d)}\leq C 2^{-Np}
(1+|k|)^{N+n} \| u\| _{H^s(\H^d)}.
\end{equation}

Since we deal with Fourier multipliers, we have $$ { \rm Op }
(\Phi^k_{p}) u = { \rm Op } (a_p)  { \rm Op } ( \Phi^k_{p})u+{\rm
Op}(r^{(N)}_{k,p})u. $$ Finally, by
Proposition\refer{propphipdeltap} of Chapter~\ref{prelimins}, we get
\begin{eqnarray}\nonumber
 { \rm Op } (\Phi^k_{p})
u & = &  \Delta_p {\rm Op}(a_p) { \rm Op } ( \Phi^k_{p})u
+\sum_{q\not=p}\Delta_q   {\rm Op}(a_p){\rm Op}(\Phi^k_{p}) u
+{\rm Op}(r^{(N)}_{k,p})u  \\ \label{decompPhikp} & = &\Delta_p
{\rm Op}(a_p) { \rm Op } ( \Phi^k_{p})u +\sum_{q\not=p} \Delta_q
R_{p,q}{\rm Op} (\Phi^k_{p}) u+ {\rm Op}(r^{(N)}_{k,p})u ,
\end{eqnarray}
 where
 \begin{equation}\label{Prop33delta1}
 \|  R_{p,q} \|_{{\mathcal L}(H^s (\H^d))}  \leq C  2^{- \delta_0 |p-q|}.
\end{equation}
Therefore we can write $${\rm Op}(t)={\rm Op}(t^\sharp) +{\rm
Op}(t^\flat)+{\rm Op}(t^\natural)$$ with, writing~$\phi^j_r
(\lam)=\phi^j(2^{-2r}\lam)  $
\begin{eqnarray}\label{th}
{\rm Op} (t^\sharp)  & = &  \sum_{p,r} h_{pr} (w)
\Lambda_r  \Delta_p{\rm Op}(a_p){\rm Op}(\phi^j_r \Phi^k_p)\\ \nonumber &
& +\sumetage{p,r}{q\not=p} h_{pr}(w)  \Lambda_r \Delta_q R_{p,q}{\rm
Op}(\phi^j_r\Phi^k_{p})\\
\label{tl} {\rm Op} (t^\flat) & = &  \sum_{p,r} \ell_{pr}(w)
\Lambda_r \Delta_p {\rm Op}(a_p){\rm Op}(\phi^j_r\Phi^k_p)
 \quad \mbox{and}\\
\nonumber& & +\sumetage{p,r}{q\not=p} \ell_{pr}(w)\Lambda_r \Delta_q
R_{p,q}{\rm Op}(\phi^j_r\Phi^k_{p})\\ \label{treste}
 {\rm Op}(t^\natural) & = &
\sum_{p,r} b_{p,r}^{kj}(w) \Lambda_r  {\rm Op}(r^{(N)}_{k,p})
\end{eqnarray}
with $\Lambda_r={\rm Op}(\widetilde \phi(2^{-2r}\lam))$
and~$\widetilde\phi$ is a compactly supported function in~${\mathcal
C} $ such that~$\widetilde\phi\phi^j=\phi^j$.

In the following, we are going to study each of these three terms,
beginning by~${\rm Op}(t^\natural)$ which is a remainder term.
Besides, in order  to simplify the notation we shall write $$ u_{pr}^{kj}
\eqdefa {\rm Op}(\phi^j_r\Phi^k_{p})  u, $$ and we recall that due
to~\aref{phikp} and to the fact that ${\rm Op}(\phi^j_r\Phi^k_p)={\rm Op}(\phi^j_r){\rm Op}(\Phi^k_p)$ with ${\rm Op}(\phi^j_r)$ of norm $1$, there is a constant~$C$ such that for all
indexes~$p,r,k,j$,
\begin{equation}\label{upkHs}
\|u_{pr}^{kj} \|_{H^s} \leq C (1+ |k|)^n  \|u\|_{H^s}.
\end{equation}
 Moreover, by
quasi-orthogonality (see Chapter~\ref{prelimins},
Subsection~\ref{sec:Lambda}), we have
\begin{equation}\label{quasiorth}
\|\Delta_p\Lambda_r u_{pr}^{kj}\|_{L^2}\leq C\,(1+|k|)^n
\,c_p\,c_r\,2^{-ps}\,\|u\|_{H^s}
\end{equation}
where $C$ is a   constant and $c_p$, $c_r$ denote from now
on generic elements of the unit sphere of~$\ell^2(\ZZZ)$.
\medskip

\subsection{The remainder term}  We drop the $kj$-exponent in $b_{p,r}^{kj}$ for simplicity and  decompose  $b_{p,r}$ in $\lam$-frequencies:
$b_{p,r}=\sum_{m} \Lambda_m b_{p,r}$ so that ${\rm
Op}(t^\natural)$ is now a sum on three indices. We decompose this
sum  into two parts, depending on whether~$r\leq m+M_1$ or~$r\geq
m+M_1$ where~$M_1$ is the threshold of Proposition~\ref{4.2} stated page~\pageref{4.2}.

Let us consider the first case, when $r\leq m+M_1$. We choose
$\sigma$ such that $s<\sigma<\rho$ and by Lemma~\ref{lem.est} page~\pageref{lem.est}, we
find constants $C$ such that
\begin{eqnarray*}
\| \Lambda_m (b_{p,r})\,\Lambda_r{\rm Op} (r^{(N)}_{k,p}) u
\|_{H^s(\H ^d)} & \leq &
C\,\|\Lambda_m(b_{p,r})\|_{C^\sigma(\H^d)}\|{\rm
Op}(r^{(N)}_{k,p})u\|_{H^s(\H^d)}\\ & \leq & C\,
2^{-m(\rho-\sigma)} A_k \|{\rm Op}(r^{(N)}_{k,p})u\|_{H^s(\H^d)}\\
& \leq & C\, 2^{-m(\rho-\sigma)} A_k\, 2^{-Np}(1+|k|)^{N+n} \|
u\|_{H^s(\H^d)}
\end{eqnarray*}
where we have used estimates~\aref{eq:unifest}
and~\aref{est:rnkp}. We then obtain $$\displaylines{\left\|
\sum_{m,p,r\leq m+M_1} \Lambda_m (b_{p,r})\,\Lambda_r{\rm Op}
(r^{(N)}_{k,p}) u\,\right\|_{H^s(\H ^d)} \hfill\cr\hfill \leq
C\,\left(\sum_{m,p} (m+M_1)2^{-m(\rho-\sigma)}2^{-Np}\right)
(1+|k|)^{N+n} A_k\|u\|_{H^s(\H^d)}\cr}$$ which ends the first
step.

We now focus on the sum for $r\geq m+M_1$ and we use that  by
Proposition~\ref{4.2}, the function~$\Lambda_m (b_{p,r})\,\Lambda_r{\rm
Op} (r^{(N)}_{k,p}) u$ is $\lam$-localized in a ring of
size~$2^r$. Therefore, in view of \aref{eq:quasiorthHs}, it is
enough to control the $H^s(\H^d)$-norm of $\sum_{p,m}\Lambda_m
(b_{p,r})\,\Lambda_r{\rm Op} (r^{(N)}_{k,p}) u$ by $c_r$ with
$(c_r)\in\ell^2$. We observe that by Lemma~\ref{lem.est}
and~\aref{eq:quasiorthHs}, there exists a constant $C$ such that
\begin{eqnarray*}
\|\Lambda_m (b_{p,r})\,\Lambda_r{\rm Op} (r^{(N)}_{k,p})
u\|_{H^s(\H^d)} & \leq & C \|\Lambda_m
(b_{p,r})\|_{C^\sigma(\H^d)} \, c_r\,\|{\rm Op}(r^{(N)}_{k,p}u\|
_{H^s(\H^d)}\\ & \leq & C\, 2^{-m(\rho-\sigma)} A_k c_r
2^{-pN}(1+|k|)^{N+n} \|u\|_{H^s(\H^d)}
\end{eqnarray*}
where $s<\sigma<\rho$ and where we have used
again~\aref{eq:unifest} and\aref{est:rnkp}. Therefore, we obtain
$$\left\| \sum_{m,p} \Lambda_m (b_{p,r})\,\Lambda_r{\rm Op}
(r^{(N)}_{k,p}) u\right\|_{H^s(\H ^d)}\leq
c_r\,\left(\sum_{m,p}2^{-m(\rho-\sigma)}2^{-Np}\right)
(1+|k|)^{N+n} A_k\|u\|_{H^s(\H^d)}$$ which achieves the control of
the remainder term.

\medskip

\subsection{The high frequencies}\label{highfreqth}
 Let us  estimate~${\rm Op}(t^\sharp)u$ in~$H^s$ for any~$  |s | < \r$.
 For any
function~$u $ belonging to~$ H^s(\H^d) $, we have $$ {\rm Op}
(t^\sharp) u = \sum_{p,r} (u_{pr}^\sharp + w_{pr}^\sharp)
 \quad \mbox{with} $$
 $$ u_{pr}^\sharp = h_{pr} \Delta_p  \Lambda_r {\rm Op}(a_p) u_{pr}^{kj} \quad  {\rm and} \quad w_{pr}^\sharp = \sum_{q\not=p} h_{pr} \Delta_q\Lambda_r R_{p,q}u_{pr}^{kj}.
 $$
 Let us deal with $u_{pr}^\sharp $. As noticed in Chapter~\ref{prelimins} Section~\ref{sec:Lambda}, on the support of the Fourier transform of $\Delta_p{\rm Op}(\phi(2^{-2r}\lam))$ we have
 $D_\lam\sim2^{2p}$ and $|\lam|\sim 2^{2r}$. Therefore, $2^{2(p-r)}$ has to be greater than or equal to $1$. This implies
that the only indexes $(p,r)$ that we have to consider are those
such that $0<r\leq p$. We will then
 simply bound the sum of
 norms of the terms~$u^\sharp_{pr}$.

To do so, let us choose~$\s$ such that~$ |s | < \s < \r$. This
leads, by Lemma\refer{estfgHs}, to the
 following estimate
$$
 \|u_{pr}^\sharp\|_{H^s} \leq C 2^{-p(\rho-\s)} \|h_{pr}\|_{C^\rho} \|u_{pr}^{kj} \|_{H^s}.
$$ Finally, thanks to~\aref{upkHs} and to the definition
of~$h_{pr}$ recalled in~(\ref{deflh}), we obtain for some
integer~$n$ (recalling that~$0<r\leq p$)
\begin{eqnarray*}
\sum_{p,r} \|u_{pr}^\sharp\|_{H^s} &\leq &C\,(1+|k|)^n \|u\|_{H^s}
\sum_p \,p\, 2^{-p(\rho-\s)} \sup_r \|h_{pr}\|_{C^\rho}\,\\ &\leq
&C\,(1+|k|)^n \|u\|_{H^s} \sum_p p\, 2^{-p(\rho-\s)}\,A_k.
\end{eqnarray*}
Since~$\s < \r$ and~$p \geq -1$, we infer that~$\displaystyle
u\mapsto \sum_{p,r} u_{pr}^\sharp$ is bounded in the
space~${\mathcal L}(H^s(\H^d))$,  by the constant~$C(1+|k|)^n
A_k$.

Let us now study $ w_{pr}^\sharp$. Arguing as before, we restrict
the sum on the integers~$r$ such that~$r\leq q$ and we get
$$\sum_{p,r} \| w_{pr}^ \sharp \|_{H^s} \leq C \sum_{p,q\neq p}
2^{-p(\rho-\s)}  q \sup_r \|h_{pr}\|_{C^\rho} 2^{-\delta_0|p-q|}
(1+|k|)^n \|u\|_{H^s}.$$ As before, we get a control by
$C(1+|k|)^n A_k$.

So the high frequency part of~$t^{k,j}$ satisfies the required
estimate.

 \subsection{The low frequencies}\label{lowfreqtl}
  We recall that by~\aref{tl}, we have for any
function~$u$ belonging to~$  H^s(\H^d) $
 $$ {\rm Op} (t^\flat) u =
\sum_{p,r} (u_{pr}^\flat + w_{pr}^\flat)
 \quad \mbox{with} $$
 $$ u_{pr}^\flat =  \ell_{pr}\Delta_p\Lambda_r  {\rm Op}(a_p)u_{pr}^{kj} \quad \mbox{and} \quad  w_{pr} ^\flat= \sum_{q\not=p} \ell_{pr}\Delta_q\Lambda_r R_{p,q}\,u_{pr}^{kj}.
 $$

 In the following, we are going to use the frequency localization induced by $\Delta_p$ in the sense of Definition~\ref{definlocfreqheis}. In particular, using Proposition~4.1 of \cite{bg} (the statement is recalled in Proposition~\ref{pro:couronnes} page~\pageref{pro:couronnes}), we will be able to say something of the localisation of a product of localised terms.
   We want to use also the localization in~$\lambda$ induced by $\Lambda_r$. For that purpose, we truncate $\ell_{pr}$ and in doing so, we add a new index of summation.
  We set~$\displaystyle \ell_{pr}=\sum_m\Lambda_m\ell_{pr}$ and we immediately remark that since~$\ell_{pr}$ is a low frequency term, then  for~$m\geq p$ we have~$\Lambda_m\ell_{pr}=0$. Therefore, the index $m$ is controled by~$p$.

 According to\refeq{eq:unifest}, one deduce that
 \begin{equation}\label{eq:unifest2}
  \|\Lambda_m \ell_{pr} \|_{L^\infty(\H^d)} \leq C 2^{-m\r}\sup_{p ,r} \|b_{p,r}^{kj}\|_{C^\rho(\H^d)},
\end{equation}where~$C$ is a universal constant.

  We can now go into the proof of the proposition for~$u^\flat_{pr}$.
  Let us start by studying
  $$u_{prm}^{kj}\eqdefa \Lambda_m \ell_{pr}\Delta_p\Lambda_r  {\rm Op}(a_p) u_{pr}^{kj}.$$
 As soon as the threshold~$M$ is large enough, $u_{prm}^{kj}$ is frequency localized, in the sense of Definition~\ref{definlocfreqheis},
  in a ring of size~$2^p$ due to Proposition~\ref{pro:couronnes} page~\pageref{pro:couronnes}. So we can use  Lemma
\refer{indheish} to compute the~$H^s$ norm of~$\displaystyle
\sum_p u_{prm}^{kj}$.

     $ $

     Consider the threshold $M_1$ given by Proposition~\ref{4.2}. We shall argue differently depending on whether~$r\leq m-M_1$, $r\geq m+M_1$, or $|r-m|< M_1$.

 For $r\leq m-M_1$,  it is enough (due to Lemmas\refer{indheish} and~\ref{4.2})  to prove that for any $p,m\in \N$,
 \beq\label{enoughtoprove1}
 \sum_{r\leq m-M_1} \|u_{prm}^{kj}\|_{L^2}\leq C\, A_k(1+|k|)^nc_p\,c_m\|u\|_{H^s} 2^{-ps}.
 \eeq
 We observe that
 \begin{eqnarray*}
 \|  u_{prm}^{kj}\|_{L^2} & \leq &  \|\Lambda_m  \ell_{pr}\|_{L^\infty}\, \| \Delta_p \Lambda_r {\rm Op}(a_p)u_{pr}^{kj} \|_{L^2} \\
& \leq & C\, \| \Lambda_m \ell_{pr}\|_{L^\infty}\, c_p\, c_r
(1+|k|)^n
 \,2^{-ps}\,\|  u\|_{H^s}
\end{eqnarray*}
by \aref{upkHs} and \aref{quasiorth}. Therefore, for all
integers~$m$ we have
\begin{eqnarray*}
 \sum_{r\leq m-M_1}\| u_{prm}^{kj}\|_{L^2}   & \leq &    C \,
(1+ |k|)^n \, c_p\,2^{-ps} \,\| u\|_{H^s}\,\sum_{r\leq m-M_1}c_r
\| \Lambda_m\ell_{pr}\|_{L^\infty}\\ & \leq &
 C \,
(1+ |k|)^n \, c_p\,2^{-ps} \,\| u\|_{H^s} \,\sqrt m \sup_{p,r}
\|\Lambda_m \ell_{pr}\|_{L^\infty}
\end{eqnarray*}
by the Cauchy-Schwartz inequality.
 So it is enough to have
\begin{equation}\label{(1)}
\left\|\sqrt m\,\sup_{p,r} \|\Lambda_m
\ell_{pr}\|_{L^\infty}\right\|_{\ell^2(\N)} \leq C A_k
\end{equation}
to ensure that~(\ref{enoughtoprove1}) is satisfied, which is
implied by \refeq{eq:unifest2}.

$ $

Let us now consider the indexes $ r\geq m+M_1$.  This time, it is
enough to prove \beq\label{enoughtoprove2} \sum_{m\leq r-M_1}
\|u_{prm}^{kj}\|_{L^2}\leq CA_k (1+ |k|)^n\,c_p\,c_r\,2^{-ps}\|
u\|_{H^s}. \eeq
 We have, following the same computations as above,
$$
 \sum_{m\leq r-M_1} \| u_{prm}^{kj}\|_{L^2}   \leq  C
   \sum_{m\leq r-M_1 } \| \Lambda_m \ell_{pr}\|_{L^\infty} c_p\,c_r\,(1+|k|)^n \|u\|_{H^s} 2^{-ps}.$$
   Therefore, if
 \begin{equation}\label{(2)}
  \sum_m\sup_{p,r} \|\Lambda_m \ell_{pr} \|_{L^\infty}\leq CA_k,
  \end{equation}
  we  obtain the expected result, namely~(\ref{enoughtoprove2}).  Condition \refeq{(2)} is obviously ensured by \refeq{eq:unifest2} which achieves the estimate of\refeq{enoughtoprove2}.

 $ $

Finally, let us consider the case $|r-m| < M_1$. We shall analyze
for $j'\in\N\cup\{-1\}$ the quantity $\Lambda_{j'}\left(\Lambda_m
\ell_{pr} \Delta_p\Lambda_r{\rm Op}(a_p) u_p^{kj}\right)$. We
claim that
\beq\label{enoughtoprove3}
\qquad\left\|\sumetage{r,m}{|r-m|\leq M_1}\Lambda_{j'}\left(\Lambda_m
\ell_{pr} \Delta_p\Lambda_r{\rm Op}(a_p)
u_p^{kj}\right)\right\|_{L^2}\leq C\, A_k(1+|k|)^n c_{j'}\,c_p
\|u\|_{H^s} 2^{-ps}, \eeq which by quasi-orthogonality will prove
the result.

We observe indeed that by Proposition~\ref{4.2}, there exists a constant
$M_2$ such that
$$\sumetage{r,\,m}{|r-m|\leq
M_1}\Lambda_{j'}\left(\Lambda_m \ell_{pr} \Delta_p\Lambda_r{\rm
Op}(a_p) u_{pr}^{kj}\right) =\sumetage{|r-m|\leq M_1}{r\geq
j'-M_2}\Lambda_{j'}\left(\Lambda_m \ell_{pr} \Delta_p\Lambda_r{\rm
Op}(a_p) u_{pr}^{kj}\right).$$ Therefore arguing as before, $$
\displaylines{\qquad \left\|\sumetage{r,m}{|r-m|\leq
M_1}\Lambda_{j'}\left(\Lambda_m \ell_{pr} \Delta_p\Lambda_r{\rm
Op}(a_p) u_{pr}^{kj}\right)\right\|_{L^2}\hfill\cr\hfill
 \leq  C\,(1+|k|)^n\, c_p  \, 2^{-ps}\| u\|_{H^s}
\sumetage{j'\leq r-M_2}{|r-m|\leq M_1} c_r \sup_{p,r}\| \Lambda_m
\ell_{pr}\|_{L^\infty}.\qquad\cr}$$

The property
\begin{equation}\label{(3)}
\exists \e_0>0,\;\;\sup_{m}(\sup_{r,p} \| \Lambda_m
\ell_{pr}\|_{L^\infty} 2^{m\e_0}) \leq C A_k
\end{equation}
induces  that the sequence $\displaystyle \sum_{m\geq j'}
2^{-m\e_0} c_m$ belongs to~$ \ell^2_{j'}$, which is enough to
prove the claim~(\ref{enoughtoprove3}).
Estimate\refeq{eq:unifest2} implies \refeq{(3)} which concludes
the proof of~(\ref{enoughtoprove3}).

$ $

 Now let us turn to~$w_{pr}^\flat$.
 We shall separate $w_{pr}^\flat$ into three parts, depending on whether~$q\gg p$ or~$q\ll p$, or~$q \sim p$.
  More precisely, let~$N_0\in\N$ be a fixed integer, to be chosen large enough at the end,
and let us define $$ v = v^\sharp +v^\flat +v^\natural =
\sum_{p,r} (v_{pr}^\sharp +v_{pr}^\flat +v_{pr}^\natural ) =
\sum_{p,r} w_{pr}^\flat \;\;{\rm with} \; \; w_{pr} ^\flat=
v_{pr}^\sharp +v_{pr}^\flat +v_{pr}^\natural \;\;{\rm while} \; \;
$$ $$ v_{pr}^ \sharp =\sum_{q\geq p+N_0} \ell_{pr}\Delta_q
\Lambda_r R_{p,q}  u_{pr}^{kj} \quad {\rm and} \quad
v_{pr}^\flat=\sum_{q+N_0\leq p} \ell_{pr} \Delta_q\Lambda_r
R_{p,q}  u_{pr}^{kj} . $$

Recall that to compute the~$H^s$ norm of~$v $, one needs to
compute the~$\ell^2$ norm in~$j$ of~$2^{js}\|\Delta_j v \|_{L^2}$.
We are going to decompose as before $\ell_{pr}=\sum_m\Lambda_m
\ell_{pr}$ and consider the cases $m\leq r-M_1$, $m\geq r+M_1$ and
$|r-m|< M_2$. For each term, we use the same strategy as the one
developed before, in the case of~$u_{pr}^\flat$. We shall only
write the proof for the indexes $m\leq r-M_1$ and leave the other
cases to the reader.

 By quasi-orthogonality, it is enough to prove
 \beq\label{enoughtopreuve1}
\|\Delta_j v_r^*\|_{L^2}\leq CA_k (1+ |k|)^n\,c_j\,c_r\,2^{-js}\|
u\|_{H^s}, \eeq where~$v_r^*= \sum_{p} w_{pr}^*$ and $*$ stands
for $\sharp,\; \flat$ or $\natural$.

 $ $

$\bullet$ {\bf The term $v^\sharp$}:
 Let~$j \geq -1 $ be fixed. We recall that~$\ell_{pr}$ is frequency localized in a ball of size~$2^{ p-M}$ and $\Delta_q \Lambda_r R_{p,q}  u_{pr}^{kj}$ in a ring of size $2^{ q}$, so by the frequency localization of the product (see  Proposition~\ref{pro:couronnes} page~\pageref{pro:couronnes}), there is a constant~$N_1$ such that
$$ \Delta_jv_r^\sharp=  \sum_{m\leq r-M_1}\sum_{|j-q| \leq
N_1}\sum_{ q\geq p+N_0}\Delta_j\left(\Lambda_m \ell_{pr}\Delta_q
\Lambda_r R_{p,q} u_{pr}^{kj}\right). $$

Therefore, we have
\begin{eqnarray*}
2^{js}‚Äö√Ñ√∂‚àö√ë‚àö‚àÇ‚Äö√†√∂‚àö√´‚Äö√Ñ√∂‚àö√ë‚Äö√Ñ‚Ä†\|\Delta_j v_r^\sharp\|_{L^2} & \leq & 2^{js}
\sum_{m\leq r-M_1}\sum_{|j-q| \leq N_1}\sum_{ q\geq p+N_0}
‚Äö√Ñ√∂‚àö√ë‚àö‚àÇ‚Äö√†√∂‚àö√´‚Äö√Ñ√∂‚àö√ë‚Äö√Ñ‚Ä†\|\Delta_j ( \Lambda_m\ell_{pr}\Delta_q\Lambda_r R_{p,q}
u_{pr}^{kj}) \|_{L^2}
\\ &\leq C & 2^{js} \sum_{m\leq r-M_1}\sum_{|j-q| \leq N_1}\sum_{
q\geq p+N_0}   \|\Lambda_m \ell_{pr}\|_{L^\infty} \| \Delta_q
\Lambda_r R_{p,q} u_{pr}^{kj}\|_{L^2} \\ &\leq C & \sum_{m\leq
r-M_1} \sum_{|j-q| \leq N_1}\sum_{ q\geq p+N_0} 2^{(j-q)s}
\|\Lambda_m\ell_{pr}\|_{L^\infty} c_r\,c_q 2^{\delta_0(p-q)} (1+
|k|)^n \| u \|_{H^s},
\end{eqnarray*}
where we have used the fact that
\begin{eqnarray*}
2^{qs} \| \Delta_q \Lambda_rR_{p,q} u_{pr}^{kj}\|_{L^2} &\leq &C
c_q \,c_r \| R_{p,q} u_{pr}^{kj}\|_{H^s} \\ &\leq & C c_q \,c_r
2^{\delta_0(p-q)}\|u_{p}^{kj}\|_{H^s}
\end{eqnarray*}
by~(\ref{Prop33delta1}),  and then~(\ref{upkHs}). Assuming
\aref{(2)}, the result follows from Young's inequality which ends
the proof of\refeq{enoughtopreuve1} for $v^\sharp$ thanks
to\refeq{eq:unifest2}.

\medskip

$\bullet$ {\bf The term $ v^\flat$}: Using again the frequency
localization of the product, one can write that for some
constant~$N_3$,

\begin{eqnarray*}
2^{js} \| \Delta_j v^{\flat}  \|_{L^2}& \leq & C 2^{js}
\sum_{m\leq r-M_1}\sum_{ j-p  < N_3}
  \sum_{q+N_0\leq p}
 \| \Lambda_m \ell_{pr}\|_{L^\infty}\| \Delta_q
\Lambda_r R_{p,q} u_{pr}^{kj}\|_{L^2}\\& \leq & C 2^{js}
\sum_{m\leq r-M_1}\sum_{j-p  < N_3}
  \sum_{q+N_0\leq p}
 \| \Lambda_m \ell_{pr}\|_{L^\infty}
 2^{-qs}c_r\, c_q\|  R_{p,q} u_{pr}^{kj}\|_{ H^s} \\
& \leq &  C 2^{js} \sum_{m\leq r-M_1}  \sum_{j-p  < N_3}
\sum_{q+N_0\leq p}
   \| \Lambda_m\ell_{pr}\|_{L^\infty} 2^{ -q s } c_r\,c_q   2^{\delta_0(q-p)}
\|u_{pr}^{kj}\|_{ H^s} \\ & \leq &  C (1+ |k|)^n c_r \|u\|_{ H^s}
\sum_{m\leq r-M_1}  \| \Lambda_m \ell_{pr}\|_{L^\infty}\sum_{j-p
< N_3} 2^{ (j-p) s }
 \sum_{q+N_0\leq p}  c_q   2^{(\delta_0-s)(q-p)}  \end{eqnarray*}
thanks to~(\ref{Prop33delta1}) and~(\ref{upkHs}).

Applying Young inequality, we thus obtain for~$0<s<\delta_0$
\begin{equation}\label{choices1}
2^{js} \| \Delta_j v^{\flat}  \|_{L^2} \leq C (1+ |k|)^n c_r
\|u\|_{ H^s}\sum_{m\leq r-M_1}  \| \Lambda_m
\ell_{pr}\|_{L^\infty} \sum_{|j-p| < N_3}2^{ (j-p) s }
 c_p.
 \end{equation}
  This ends the proof of the result  by   Estimate \refeq{eq:unifest2}.

\medskip

$\bullet$ {\bf The term $ v^\natural$}:  We recall that $$
v^\natural = \sum_{m\leq r-M_1}\sum_{|p-q| < N_0}
\Lambda_m\ell_{pr}\Delta_q\Lambda_r R_{p,q} u_{pr}^{kj}. $$ It
follows that
\begin{eqnarray}\nonumber
2^{js} \| \Delta_j v ^\natural   \|_{L^2}& \leq & C 2^{js}
\sum_{m\leq r-M_1}\sumetage{-1 \leq j \leq q+N_3 }{|p-q| \leq N_0
}   \| \Lambda_m \ell_{pr}\|_{L^\infty} \|\Delta_q\Lambda_r
R_{p,q} u_{p}^{kj}\|_{L^2}\\ \label{choices2} & \leq &  C (1+
|k|)^n c_r \|u\|_{H^s} \sum_{m\leq r-M_1} \|\Lambda_m
\ell_{pr}\|_{L^\infty}   \sumetage{j \leq q+N_3 }{|p-q| \leq N_0 }
 2^{(j-q)s}
  \,c_q 2^{\delta_0(q-p)},
\end{eqnarray}
and we conclude as in the case of~$v^{\flat} $. We point out that
it is at this very place that we crucially use that $s>0$.

The proposition is proved. \qed
 \setcounter{theo}{0}
  \setcounter{prop}{0}

\renewcommand{\newtheorem}{\thechapter.\arabic{prop}}
\renewcommand{\theequation}{\thesection.\arabic{equation}}

\chapter*{Appendix A: Some useful  results on the Heisenberg group}
\renewcommand{\thechapter}{A}
\setcounter{section}{0}
 \section{Left invariant vector fields}\label{leftinvariant}
 \setcounter{equation}{0}
 Let us recall that on a Lie group~$G$, a vector field \[ X: G
\longrightarrow TG\] is said to be left invariant whenever the following
diagram commutes for all $h\in G:$
\[
\begin{array}{ccccc}
G&\stackrel{ \tau_h}{\longrightarrow} &G \\ X \downarrow & & \downarrow X
\\
 TG &\stackrel{d \tau_h}{\longrightarrow}  & TG\\
\end{array}
\]
 %
where~$\tau_h$ is the \emph{left translate}
 on~$G$ defined by~$\tau_h(g)=
h\cdot g$. It turns out that for any~$h \in G,$
\begin{equation}\label{leftdefequi}X \circ \tau_h= d\tau_h \circ X.
\end{equation}
In particular,
\[X(h)= d\tau_h(e) X(e),\]
where~$e$ denotes the identity of~$G$. Therefore, as soon as the vector field~$X$ is known on~$e$, so is its value everywhere.

Let us mention that this infinitesimal characterization is
equivalent to saying that, for all smooth functions~$f$,
\begin{equation}\label{leftdef}(Xf_h)= (Xf)_h,
\end{equation}where~$f_h$ is the left translate of~$f$ on~$ \H^d
$, given by~$f_h= f \circ \tau_h$.

 To start with the proof of the equivalence of the two characterizations, let us perform differential calculus
 in\refeq{leftdefequi}. We infer that\refeq{leftdefequi}
 is equivalent to
 \[ (X \circ \tau_h)f= (d\tau_h \circ X) f,\]
 for any function~$f
\in {\mathcal C}^{\infty}(G)$. This can be written for
any~$h$,~$g$ belonging to~$G$
\[ (Xf)(\tau_h(g))= df(\tau_h(g))( d\tau_h(g)X(g))=d(f\circ \tau_h)(g)X(g)= X(f\circ \tau_h)(g).\]
In other words
\[(Xf)\circ \tau_h = X(f\circ \tau_h),\]
for any~$h \in G$, which leads to the result.

\section{Bargmann and Schr\"odinger representations}
 \label{representations}
  \setcounter{equation}{0}
In this paragraph we   discuss some useful results concerning Bargmann and Schr\"odinger representations, starting with the formula giving the Schr\"odinger representation, if the Bargmann representation and the intertwining operator are known.

 In a next subsection we prove some useful commutation results.

\subsection{Connexion between the representations}\label{connexion}
In section we shall give a formula for the Schr\"odinger representation, which is linked to the Bargmann representation by an intertwining operator. This formula is of course classical, but we present it here for the sake of completeness.

We recall that the Bargmann representation is defined by
 \[
 \begin{array}{c}
u^{\lam}_{z,  s } F(\xi) = F(\xi - \overline z) {\rm e}^{i \lam s
+ 2 \lam (\xi \cdot  z - |z|^2/2)}  \quad \mbox{for} \quad \lam
>0, \\
u^{\lam}_{z,  s} F(\xi) = F(\xi - z) {\rm e}^{i \lam s - 2 \lam
(\xi \cdot \overline z - |z|^2/2)}  \quad \mbox{for} \quad \lam <
0,
\end{array}
\]
and we also recall the definition of the intertwining operator, as given in~(\ref{intertwining}) page~\pageref{intertwining}:
$$
 (K_\lambda
\phi)(\xi)\eqdefa\frac{|\lambda|^{d/4}}{\pi^{d/4}} {\rm
e}^{|\lambda|\frac{|\xi|^2}{2}}
\phi\left(-\frac{1}{2|\lambda|}\frac{\partial}{\partial\xi}\right){\rm
e}^{-|\lambda|\,|\xi|^2}.
$$
\begin{prop}\label{formulavlam}
Let~$v_{w}^\lam $ be the   Schr\"odinger representation, defined by
$$
\forall F\in {\mathcal H}_\lam, \;\; K_\lam u^\lam_w F= v^\lam
_w K_\lam F.
$$
Then~$v_{z,s}^\lam $ is given by the following formula:
 $$
 v_{z,s}^\lam f(\xi)= e^{i\lam (s - 2 x \cdot
y + 2 y \cdot \xi)} f(\xi - 2x),  \quad \forall \lam\in\R^* .
$$
\end{prop}
\begin{proof}
  It turns out to be easier to split the
representation~$ u^\lam_w $ into three parts, using the simple
fact that $$w=(x+iy, s)=(0,s+2y\cdot x)\cdot (x,0)\cdot (iy,0).$$

Let us prove the following relations:  for $\lam\in\R$,
$x,y\in\R^d$ and $s\in\R$, $\forall F\in {\mathcal H}_\lam$ and
$\eta\in \R^d$:
\begin{eqnarray}
\label{eqs} \left(K_\lam u^\lam_{(0,s)}  F\right)(\eta)&  = & {\rm e}^{i\lam s}
 \left(K_\lam F\right)(\eta),\\ \label{eqx}  \left(K_\lam u^\lam_{(x,0)} F\right)(\eta)  & = &
(K_\lam F)(\eta-2x),\\ \label{eqy}  \left(K_\lam u^\lam_{(iy,0)}  F\right)(\eta)
& = &{\rm e} ^{2i\lam y\cdot  \eta}  \left(K_\lam F\right)(\eta).
\end{eqnarray}

Notice that those relations  give
\begin{eqnarray*}
 \left(K_\lam u^\lam_w F\right)(\eta) & = & \left( K_\lam u^\lam_{(0,s+2x\cdot
y)} u^\lam_{(x,0)} u^\lam _{(iy,0)} F\right)(\eta)\\
 & = & {\rm e}^{i\lam (s+2y\cdot x)}\left( K_\lam u^\lam_{(x,0)} u^\lam _{ (iy,0)} F\right) (\eta)\\
 & = & {\rm e} ^{i\lam (s+2 y\cdot x)} \left( K_\lam u^\lam_{(iy,0)} F\right)(\eta-2x)\\
 & = & {\rm e}^{i\lam s +2i\lam y \cdot \eta -2i\lam y\cdot x} (K_\lam F)(\eta-2x).
 \end{eqnarray*}
 which is precisely the expected result.

 So
 it remains to prove the basic relations \aref{eqs}--\aref{eqy}.
 The first one comes trivially from the fact that $u^\lam_{(0,s)}$ is the multiplication by the phasis ${\rm e}^{i\lam s}$.

  For the two other  ones, we write, for any function~$F$ in~${\mathcal H}_\lam$ and using Proposition~IV.2 of~\cite {farautharzallah},
  $$ \left(K_\lam F\right) (\eta)=\left({|\lam|\over\pi}\right)^{5d/4} {\rm e} ^{|\lam| |\eta|^2\over 2}\int_{\R^{2d}} {\rm e}^{-2i\lam v\cdot(\eta-\eta')-|\lam||\eta'|^2} F(iv)\,dv\,d\eta'.$$
  Therefore, for $\lam>0$, we have  on the one hand
  \begin{eqnarray*}
  \left( K_\lam u^\lam_{(iy,0)}\right) F(\eta) & = & \left({\lam\over\pi}\right)^{5d/4} {\rm e} ^{\lam |\eta|^2\over 2}
  \int_{\R^{2d}} {\rm e}^{-2i\lam v\cdot (\eta-\eta')-\lam|\eta'|^2 -\lam |y|^2 +2i\lam y\cdot (iv)} F(i(v+y))\,dv\,d\eta'\\
  & = & \! \left({\lam\over\pi}\right)^{5d/4} \! \!  {\rm e} ^{\lam {|\eta|^2\over 2}+2i\lam y\cdot\eta}\!
  \int_{\R^{2d}} {\rm e}^{-2i\lam u\cdot (\eta-\eta'-iy)-2i\lam y\cdot\eta'+\lam|y|^2-\lam|\eta'|^2}  \! F(iu) dud\eta'\\
  & = & \left({\lam\over\pi}\right)^{5d/4} {\rm e} ^{\lam {|\eta|^2\over 2}+2i\lam y\cdot\eta}
  \int_{\R^{2d}} {\rm e}^{-2i\lam u\cdot (\eta-\eta'')-\lam|\eta''|^2} F(iu)\,du\,d\eta''\\
  & = & {\rm e}^{2i\lam y\cdot \eta} (K_\lam F)(\eta).
  \end{eqnarray*}
  On the other hand, one has
  \begin{eqnarray*}
  \left( K_\lam u^\lam_{(x,0)}F\right)(\eta ) & = &
  \left({\lam\over\pi}\right)^{5d/4} {\rm e} ^{\lam |\eta|^2\over 2}\int_{\R^{2d}} {\rm e}^{-2i\lam v(\eta-\eta')-\lam|\eta'|^2+2\lam ix\cdot v -\lam|x|^2} F(iv-x)\,dv\,d\eta'\\
  & = & \left({\lam\over\pi}\right)^{5d/4} {\rm e} ^{\lam {|\eta|^2\over 2}+2\lam|x|^2-2\lam\eta\cdot x}\int_{\R^{2d}} {\rm e}^{-2i\lam u(\eta-\eta'-x)-\lam|\eta'-x|^2} F(iu)\,du\,d\eta'\\
  & = & \left({\lam\over\pi}\right)^{5d/4} {\rm e} ^{\lam {|\eta-2x|^2\over 2}}\int_{\R^{2d}} {\rm e}^{-2i\lam u(\eta-2x-\eta'')-\lam|\eta''|^2} F(iu)\,du\,d\eta''\\
  & = &( K_\lam F)(\eta-2x).
  \end{eqnarray*}
  Similarly, for $\lam<0$,
  \begin{eqnarray*}
  \left( K_\lam u^\lam_{(iy,0)}\right) F(\eta) & = & \left(-{\lam\over\pi}\right)^{5d/4} {\rm e} ^{-{\lam |\eta|^2\over 2}}
  \int_{\R^{2d}} {\rm e}^{2i\lam v\cdot (\eta-\eta')+\lam|\eta'|^2 +\lam |y|^2 +2i\lam y\cdot (iv)} F(i(v+y))\,dv\,d\eta'\\
  & = & \! \! \left(-{\lam\over\pi}\right)^{5d/4}\!  {\rm e} ^{-\lam {|\eta|^2\over 2}+2i\lam y\cdot\eta}
  \int_{\R^{2d}} {\rm e}^{2i\lam u\cdot (\eta-\eta'+iy)-2i\lam y\cdot\eta'-\lam|y|^2+\lam|\eta'|^2}\!  F(iu)du d\eta'\\
  & = & \left(-{\lam\over\pi}\right)^{5d/4} {\rm e} ^{-\lam {|\eta|^2\over 2}+2i\lam y\cdot\eta}
  \int_{\R^{2d}} {\rm e}^{2i\lam u\cdot (\eta-\eta'')+\lam|\eta''|^2} F(iu)\,du\,d\eta''\\
  & = & {\rm e}^{2i\lam y\cdot \eta} (K_\lam F)(\eta)
  \end{eqnarray*}
and
  \begin{eqnarray*}
  \left( K_\lam u^\lam_{(x,0)}F\right)(\eta ) & = &
  \left(-{\lam\over\pi}\right)^{5d/4} {\rm e} ^{-{\lam |\eta|^2\over 2}}\int_{\R^{2d}} {\rm e}^{2i\lam v(\eta-\eta')+\lam|\eta'|^2-2\lam ix\cdot v +\lam|x|^2} F(iv-x)\,dv\,d\eta'\\
  & = & \left(-{\lam\over\pi}\right)^{5d/4} {\rm e} ^{-\lam {|\eta|^2\over 2}-2\lam|x|^2+2\lam\eta\cdot x}\int_{\R^{2d}} {\rm e}^{2i\lam u(\eta-\eta'-x)+\lam|\eta'-x|^2} F(iu)\,du\,d\eta'\\
  & = & \left(-{\lam\over\pi}\right)^{5d/4} {\rm e} ^{-\lam {|\eta-2x|^2\over 2}}\int_{\R^{2d}} {\rm e}^{2i\lam u(\eta-2x-\eta'')+\lam|\eta''|^2} F(iu)\,du\,d\eta''\\
  & = &( K_\lam F)(\eta-2x).
  \end{eqnarray*}
  This proves the estimates, hence the proposition is proved. \end{proof}

\subsection{Some useful  formulas}
 \label{formules}

This section is devoted to various properties for
Bargmann representation that we collect  in the
following lemma.
\begin{lemme} \label{formuleslem}
 { The following commutation formulas hold true:
 $$ {1\over 2\lam} [Q^\lam_j,u^\lam _w] = -  \overline z_j u^\lam _w
 \quad \mbox{and} \quad {1\over 2\lam} [\overline Q^\lam_j,u^\lam _w] = z_j u^\lam_w. $$
 for any~$\lam \in \R^*$ and any~$w=(z,s) \in \H^d$. }
\end{lemme}
\begin{proof}
In order to prove Lemma\refer{formuleslem}, let us first recall formulas \aref{usefulformula} giving the expression of~$Q_j^\lam$ and~$\overline Q_j^\lam$:
 $$Q_j^\lam =\left\{\begin{array}{l}
-2 |\lam |\xi_j \;\;{\rm if}\;\;\lam>0,\\
\partial_{\xi_j}\;\;{\rm if }\;\;\lam<0,\end{array}\right.
\;\; {\rm and}\;\; \overline Q_j^\lam=\left\{\begin{array}{l}
\partial_{\xi_j}\;\;{\rm if}\;\;\lam>0,\\
-2 |\lam |\xi_j\;\;{\rm if }\;\;\lam<0.\end{array}\right. $$ Let
us now prove the first formula, in the case when~$\lam >0$. On the
one hand, it is obvious that  $$ Q^\lam_j u^\lam _w
F(\xi)= -2\lam \xi_j u^\lam _w F(\xi).$$ On the other hand, an easy computation implies that $$ u^\lam _w  Q^\lam_jF(\xi)= -2\lam (\xi_j -
\overline z_j ) {\rm e}^{i \lam s + 2 \lam (\xi \cdot  z - |z|^2/2
)}  F(\xi - \overline z).$$ which implies that~$-\overline z_j
u^\lam_w= {1\over 2\lam}[ Q_j^\lam,u^\lam_w]$, for~$\lam >0$. In
the case when~$\lam <0$ one has \begin{eqnarray*} Q^\lam_j u^\lam
_w F(\xi) &=&  \partial_{\xi_j} (u^\lam _w F(\xi)) \\ &=& u^\lam
_w \partial_{\xi_j} F(\xi) - 2\lam \overline z_j {\rm e}^{i \lam s
- 2 \lam (\xi \cdot \overline z - |z|^2/2)} F(\xi -  z)
\\ &=& u^\lam _w \partial_{\xi_j} F(\xi) - 2\lam \overline z_j u^\lam _w F(\xi) \end{eqnarray*}
which ends the proof of the commutation properties~$\displaystyle
-\overline z_j u^\lam_w= {1\over 2\lam}[ Q_j^\lam,u^\lam_w]$.

It remains to check the formula for~$[\overline
Q_j^\lam,u^\lam_w]$. Arguing as before, one gets for~$\lam > 0$
\begin{eqnarray*} \overline Q_j^\lam u^\lam _w F(\xi) &=&
\partial_{\xi_j} (u^\lam _w F(\xi))
\\ &=& u^\lam _w \partial_{\xi_j} F(\xi) + 2\lam  z_j {\rm e}^{i
\lam s + 2 \lam (\xi \cdot  z - |z|^2/2)} F(\xi - \overline z)
\\ &=& u^\lam _w \partial_{\xi_j} F(\xi) + 2\lam  z_j u^\lam _w F(\xi),\end{eqnarray*}which gives the formula in the case when~$\lam > 0$. Finally, for~$\lam < 0$
$$ \overline Q_j^\lam u^\lam _w F(\xi)= 2 \lam \xi_j u^\lam _w
F(\xi)$$ and $$  u^\lam _w  \overline Q_j^\lam F(\xi) = 2 \lam
(\xi_j -z_j ) u^\lam _w F(\xi).$$ This leads easily to the second
commutation property.
\end{proof}

Lemma~\ref{formuleslem} allows to infer the following result, which is useful in particular
   to prove Lemma\refer{usefullem}.

\begin{lemme} \label{formuleslem2}
 {  One   has the following properties:
 $$ Z_ju^\lam_{w^{-1}} = Q_j^\lam u^\lam_{w^{-1}}
  \quad \mbox{and} \quad \overline{Z_j}u^\lam_{w^{-1}} = \overline{Q_j}^\lam u^\lam_{w^{-1}}. $$
   for any~$\lam \in \R^*$ and any~$w=(z,s) \in \H^d$. }
\end{lemme}
\begin{proof}
First, let us compute~$ Z_ju^\lam_{w^{-1}} $ in the case
when~$\lam $ is positive.  By definition, one has
\begin{eqnarray*} Z_j u^\lam_{w^{-1}}F(\xi)&=& (\partial_{z_j} + i \overline z_j
\partial_{s})u^\lam_{w^{-1}}F(\xi)\\ &=& (\partial_{z_j} + i \overline z_j
\partial_{s})F(\xi + \overline z) {\rm e}^{-i \lam s
+ 2 \lam (-\xi \cdot  z - |z|^2/2 )}\\ &=& (-2\lam \xi_j - \lam
\overline z_j +i \overline z_j (-i\lam) )u^\lam_{w^{-1}}F(\xi)\\
&=&
 -2\lam \xi_j u^\lam_{w^{-1}}F(\xi).\end{eqnarray*}
Whence the first formula thanks to \refeq{usefulformula}.

Along the same lines, when~$\lam $ is negative one can write
\begin{eqnarray*} Z_j u^\lam_{w^{-1}}F(\xi)&=& (\partial_{z_j} + i \overline z_j
\partial_{s})u^\lam_{w^{-1}}F(\xi)\\ &=& (\partial_{z_j} + i \overline z_j
\partial_{s})F(\xi + z) {\rm e}^{-i \lam s - 2 \lam
(-\xi \cdot \overline z - |z|^2/2)}\\ &=& (\lam \overline z_j  +i
\overline z_j (-i\lam) )u^\lam_{w^{-1}}F(\xi) +
u^\lam_{w^{-1}}\partial_{\xi_j} F(\xi)\\ &=& 2\lam \overline z_j
u^\lam_{w^{-1}}F(\xi) + u^\lam_{w^{-1}}\partial_{\xi_j}
F(\xi).\end{eqnarray*} We deduce thanks to \aref{usefulformula}  that~$ Z_j u^\lam_{w^{-1}} =
2\lam \overline z_j u^\lam_{w^{-1}} + u^\lam_{w^{-1}}  Q_j^\lam$.
Let us remind that by   Lemma\refer{formuleslem}, $
Q^\lam_j\,u^\lam _w - u^\lam _w\,Q^\lam_j  = - 2\lam \overline z_j
u^\lam _w$ which can be also written $$Q^\lam_j\,u^\lam_{w^{-1}} -
u^\lam_{w^{-1}}\,Q^\lam_j  =  2\lam \overline z_j
u^\lam_{w^{-1}}.$$ This implies that~$Z_j u^\lam_{w^{-1}}
=Q^\lam_j\,u^\lam_{w^{-1}}$, which ends the proof of the first
assertion.

Now, let us compute~$\overline{Z_j}u^\lam_{w^{-1}}$. Again, one
can write for~$\lam > 0$
\begin{eqnarray*} \overline Z_j u^\lam_{w^{-1}}F(\xi)&=& (\partial_{\overline z_j} -  i
z_j
 \partial_{s})u^\lam_{w^{-1}}F(\xi)\\ &=& (\partial_{\overline z_j} -  i
z_j
 \partial_{s})F(\xi + \overline z) {\rm e}^{-i \lam s
+ 2 \lam (-\xi \cdot  z - |z|^2/2 )}
\\ &=&  u^\lam_{w^{-1}}\partial_{\xi_j}F(\xi) -(\lam z_j +i
z_j(-i\lam)) u^\lam_{w^{-1}}F(\xi)\\ &=&
u^\lam_{w^{-1}}\partial_{\xi_j}F(\xi) -2\lam z_j
u^\lam_{w^{-1}}F(\xi).\end{eqnarray*} We point out that, again by \aref{usefulformula},  this can
be expressed as follows $$  \overline Z_j u^\lam_{w^{-1}} =
u^\lam_{w^{-1}}\overline Q^\lam_j -2\lam z_j u^\lam_{w^{-1}}. $$
But  Lemma\refer{formuleslem} states that~$ \overline Q^\lam_j
u^\lam _w - u^\lam _w \overline Q^\lam_j = 2\lam z_j u^\lam_w $
which can be also written $$ \overline Q^\lam_j u^\lam_{w^{-1}} -
u^\lam_{w^{-1}} \overline Q^\lam_j = -2\lam z_j u^\lam_{w^{-1}}.
$$ This ensures that~$\overline Z_j u^\lam_{w^{-1}}= \overline
Q^\lam_j u^\lam_{w^{-1}}$ in the case when~$\lam > 0$.

Finally, in the case when~$\lam < 0$, one gets
\begin{eqnarray*} \overline Z_j u^\lam_{w^{-1}}F(\xi)&=& (\partial_{\overline z_j} -  i
z_j
 \partial_{s})u^\lam_{w^{-1}}F(\xi)\\ &=& (\partial_{\overline z_j} -  i
z_j
 \partial_{s})F(\xi + z) {\rm e}^{-i \lam s - 2 \lam
(-\xi \cdot \overline z - |z|^2/2)}
\\ &=& ( 2 \lam \xi_j + \lam z_j -i
z_j(-i\lam)) u^\lam_{w^{-1}}F(\xi)\\ &=&  2 \lam \xi_j
u^\lam_{w^{-1}}F(\xi) \\ &=&  \overline Q^\lam_j
u^\lam_{w^{-1}}F(\xi)
\end{eqnarray*}
where we have used one more time \aref{usefulformula}  for the last equality.
This ends the proof of the lemma.
\end{proof}

Finally let us state one last result, which
provides the symbol of the  multiplication operator by~$s$.

\begin{lemme}\label{lemme:as} Let $a\in S_{\H^d}(\mu)$, $\tilde w=(\tilde z,\tilde s)\in\H^d$ and $w\in\H^d$, then
$$\int {\rm tr} \left( i\tilde s u^\lam_{\tilde w} J_\lam^* op^w(a(w,\lam))J_\lam \right)\,|\lam|^d \, d\lam =
\int {\rm tr} \left(u^\lam_{\tilde w }J_\lam^* op^w\left(g(w,\lam)\right)J_\lam \right)\,|\lam|^d \, d\lam$$
with  $g\in S_{\H ^d}(\mu )$ and
\begin{equation}\label{def:as}
\sigma(g)=-\partial_\lam \left(\sigma(a)\right)
\end{equation}
or equivalently
\begin{equation}\label{formulepourg}
g= -\partial_\lam a + {1\over 2\lam} \sum_{1\leq j\leq d}(\eta_j \partial_{\eta_j} +\xi_j \partial_{\xi_j})a
\end{equation}
\end{lemme}

\begin{proof} Let us first observe that by Proposition~\ref{prop:sigma(a)} page~\pageref{prop:sigma(a)}, the function $g$ defined by \aref{def:as}  is a symbol of order $\mu$ since
$$(1+|\lam|+y^2+\eta^2)^{{\mu-|\beta|\over 2}}(1+|\lam|)^{-k-1}
\leq (1+|\lam|+y^2+\eta^2)^{{\mu-|\beta|\over 2}}(1+|\lam|)^{-k}.$$
 Besides, by the definition of $u^\lam_{w}$ (see \aref{def:ulamw})
we have
\begin{eqnarray*}
\partial_\lam u^\lam_{w} & = & \left(is +2\xi\cdot z-|z|^2\right) u^\lam_w\;\;{\rm for}\;\; \lam>0,\\
\partial_\lam u^\lam_w & = & \left(is -2\xi\cdot \overline z+|z|^2\right) u^\lam_w\;\;{\rm for}\;\; \lam<0.
\end{eqnarray*}
Therefore, using Lemma~\ref{formuleslem} and using formulas \aref{usefulformula}, we have for~$\lam>0$
\begin{eqnarray*}
isu^\lam_w & = & \partial_\lam u^\lam_w - \sum_{1\leq j\leq d}
\left(  -{1\over 2\lam^2}Q_j^\lam [\overline Q_j^\lam,u^\lam_w] + {1\over 4\lam^2} \left[   Q_j^\lam\;,\;[\overline Q_j^\lam,u^\lam_w]\right]\right)\\
& = &   \partial_\lam u^\lam_w -{1\over 4\lam^2}\sum_{1\leq j\leq d}\left(
 [u^\lam_w,\overline  Q_j^\lam] Q_j^\lam + Q_j^\lam  [u^\lam_w, \overline Q_j^\lam]\right) .
\end{eqnarray*}
Similarly, for $\lam<0$, we have
\begin{eqnarray*}
isu^\lam_w & = & \partial_\lam u^\lam_w + \sum_{1\leq j\leq d}
\left(  -{1\over 2\lam^2}\overline Q_j^\lam [Q_j^\lam,u^\lam_w] + {1\over 4\lam^2} \left[   Q_j^\lam\;,\;[\overline Q_j^\lam,u^\lam_w]\right]\right)\\
& = &   \partial_\lam u^\lam_w +{1\over 4\lam^2}\sum_{1\leq j\leq d}\left(
 [u^\lam_w,Q_j^\lam]\overline Q_j^\lam +\overline Q_j^\lam  [u^\lam_w, Q_j^\lam]\right) .
\end{eqnarray*}

Setting $A_\lam(w)=J_\lam^* op^w(a(w,\lam)) J_\lam$ and using~${\rm tr} (AB) = {\rm tr}(BA)$ we get
$$\displaylines{
{\rm tr}\left( i\tilde su^\lam_{\tilde w} A_\lam(w)\right)  \!\! =\!  {\rm tr} \left(\partial_\lam u^\lam_{\tilde w}\,A_\lam(w)\right)\!  -\! {1\over 4 \lam^2}\sum_{1\leq j\leq d}{\rm tr} \left( u^\lam_{\tilde w}\left[ \overline Q^\lam_j \;,\; A_\lam(w) Q^\lam_j + Q^\lam_j A_\lam(w)\right]\right)\,{\rm if} \, \lam>0,\hfill\cr
{\rm tr}\left( i\tilde su^\lam_{\tilde w} A_\lam(w)\right)  \!\!=\!  {\rm tr} \left(\partial_\lam u^\lam_{\tilde w}\,A_\lam(w)\right) \! +\! {1\over 4 \lam^2}\sum_{1\leq j\leq d}{\rm tr} \left( u^\lam_{\tilde w}\left[ Q^\lam_j \;,\; A_\lam(w) \overline Q^\lam_j +\overline  Q^\lam_j A_\lam(w)\right]\right) \,{\rm if} \,\lam<0.\hfill\cr
}$$

By \aref{JlamQJlam},  using the fact that~$op^w(\eta_j) = -i \partial_{\xi_j}$ and~$op^w(\xi_j) =  {\xi_j}$, along with formula~(\ref{commutpoisson}) recalled page~\pageref{commutpoisson}, we get
for $\lam>0$,
\begin{eqnarray*}
\left[ \overline Q^\lam_j \;,\; A_\lam(w) Q^\lam_j + Q^\lam_j A_\lam(w)\right] & = & \lam\,J_\lam^* \left[ \partial_{\xi_j}+\xi_j, op^w(a(w,\lam)(\partial_{\xi_j}-\xi_j)+(\partial_{\xi_j}-\xi_j)op^w(a)\right]J_\lam\\
 &  = &2\lam\, J_\lam^* op^w\left(-2da+\sum_{1\leq j\leq d} (\eta_j+i\xi_j)(i\partial_{\xi_j} a-\partial_{\eta_j} a)\right) J_\lam.
 \end{eqnarray*}
Similarly, for $\lam<0$,
\begin{eqnarray*}
\left[ Q^\lam_j \;,\; A_\lam(w) \overline Q^\lam_j +\overline  Q^\lam_j A_\lam(w)\right]& = & - 2\lam\,J_\lam^* op^w\left(-2da+\sum_{1\leq j\leq d} (\eta_j+i\xi_j)(i\partial_{\xi_j} a-\partial_{\eta_j} a)\right) J_\lam.
\end{eqnarray*}
Set
\begin{equation}\label{def:b2}
b(w,\lam,y,\eta)= -2da+\sum_{1\leq j\leq d} (\eta_j+i\xi_j)(i\partial_{\xi_j} a-\partial_{\eta_j} a),
\end{equation}
we have obtained
\begin{equation}\label{b3}
\forall \lam \not=0,\;\;{\rm tr} \left(i\tilde su^\lam_{\tilde w}A_\lam(w)\right)={\rm tr} \left(\partial_\lam u^\lam_{\tilde w} A_\lam(w)\right)
-{1\over 2\lam} {\rm tr} \left( u^\lam_{\tilde w} J_\lam^* op^w(b) J_\lam\right).
\end{equation}

We focus now on the term $\partial_\lam u^\lam_{\tilde w} A_\lam(w)$. We have
$${\rm tr} \left( \partial_\lam u^\lam_{\tilde w}  A_\lam(w) \right)= \partial_\lam \left( {\rm tr} \left( u ^\lam_{\tilde w} A_\lam (w)\right)\right) -{\rm tr} \left( u^\lam_{\tilde w} \partial_\lam A_\lam(w)\right).$$
This implies, by integration by parts, that
$$\int {\rm tr}\left(\partial_\lam u^\lam_{\tilde w} A_ \lam(w)\right)|\lam|^d\,d\lam = -\int {d\over \lam}{\rm tr}\left( u^\lam_{\tilde w} A_ \lam(w)\right)|\lam|^d\,d\lam -\int {\rm tr}\left( u^\lam_{\tilde w} \partial_\lam A_ \lam(w)\right)|\lam|^d\,d\lam .$$

We claim that
\begin{equation}\label{partialA}
\partial _\lam A_\lam(w)= J_\lam^*op^w\left( \partial_\lam a(\lam,w) + {i\over2 \lam} \sum_{1\leq j\leq d} (\xi_j\partial_{\eta_j} a -\eta_j \partial_{\xi_j} a )\right) J_\lam.
\end{equation}
This yields, with \aref{def:b2} and  \aref{b3},
$$\displaylines{
\int {\rm tr} \left(i\tilde s u^\lam_{\tilde w} A_\lam(w)\right) |\lam|^d\,d\lam   =
\int  {\rm tr} \biggl( u^\lam_{\tilde w} J_\lam^* op^w\biggl( -{d\over \lam} a -\partial_\lam a - {i\over2\lam}\sum_{1\leq j\leq d} (\xi_j\partial_{\eta_j} a -\eta_j\partial_{\xi_j} a)\hfill\cr\hfill  +{d\over\lam } a -{1\over 2\lam} \sum_{1\leq j\leq d} (\eta_j+i\xi_j)(i\partial_{\xi_j} a -\partial_{\eta_j} a)\biggr)J_\lam \biggr)|\lam|^d \,d \lam\cr\hfill
 =  \int{\rm tr} \left( u^\lam_{\tilde w} J_\lam^*op^w\left(
-\partial_\lam a + {1\over 2\lam} \sum_{1\leq j\leq d}(\eta_j \partial_{\eta_j} +\xi_j \partial_{\xi_j} )a
\right)J_\lam \right) |\lam|^d d\lam.\cr}$$
We then set
$$g= -\partial_\lam a + {1\over 2\lam} \sum_{1\leq j\leq d}(\eta_j \partial_{\eta_j} +\xi_j \partial_{\xi_j})a $$
and observe that a simple computation implies  \aref{def:as}. Therefore, in order to finish the proof of the lemma, it only remains to prove~(\ref{partialA}).

Let us now prove \aref{partialA}. We have, recalling that  $A_\lam(w)=J_\lam^* op^w(a(w,\lam)) J_\lam$ and using the fact that~$\partial_\lam (J_\lam J_\lam^*) = 0$,
$$\partial_\lam A_\lam (w)= J_\lam^* op^w\left(\partial_\lam a(\lam,w)\right) J_\lam + J_\lam ^* \left[ op^w(a(w,\lam) ) \;,\; (\partial_\lam J_\lam) J_\lam ^*\right] J_\lam.$$
Besides, for $\alpha\in\N^d$, we have $ J_\lam F_{\alpha,\lam} =h_\alpha$ whence
 $$(\partial_\lam J_\lam)  F_{\alpha,\lam}=-J_\lam (\partial_\lam F_{\alpha,\lam}).
 $$

 Let us recall that for $\xi\in\C^d$, $\displaystyle F_{\alpha,\lam} (\xi)=(\sqrt{|\lam|})^{|\alpha|} {\xi^\alpha\over \sqrt{\alpha!}}$ so that
$\displaystyle{\partial_\lam F_{\alpha,\lam}={|\alpha|\over 2\lam} F_{\alpha,\lam}.}$
We get
$$\forall \alpha\in\N^d,\;\;(\partial_\lam J_\lam)J_\lam^* h_\alpha= (\partial_\lam J_\lam ) F_{\alpha,\lam}=- {|\alpha|\over 2\lam} h_\alpha= -{1\over 4\lam}(\xi^2-\Delta_\xi-d)h_\alpha.$$
Therefore,
$$(\partial_\lam J_\lam)    J_\lam^*=-{1\over 4\lam}(\xi^2 -\Delta_\xi)+ {d\over 4 \lam} {\rm Id}.$$
We then obtain
\begin{eqnarray*}
\left[ op^w(a),(\partial_\lam J_\lam)    J_\lam^*\right]  &  = &  - {1\over 4\lam} \left[op^w(a) \;,\; \xi^2-\Delta_\xi\right]\\
& = & {i\over 2 \lam} \sum_{1\leq j\leq d} op^w(\xi_j\partial_{\eta_j} a -\eta_j \partial_{\xi_j} a),
\end{eqnarray*}
which proves the lemma.
\end{proof}

\renewcommand{\thechapter}{B}
 \setcounter{theo}{0}
  \setcounter{prop}{0}

\renewcommand{\newtheorem}{\thechapter.\arabic{prop}}
\renewcommand{\theequation}{\thesection.\arabic{equation}}

\chapter*{Appendix B: Weyl-H\"ormander symbolic calculus on the Heisenberg group}\label{appendixwh}
\setcounter{section}{0}
In this appendix, we discuss results of Weyl-H\"ormander calculus associated to the Harmonic Oscillator, and in particular we prove Propositions~\ref{symboltilde}, \ref{prop:sigma(a)} and~\ref{symbR} and stated in the Introduction.

  \section{$\lam$-dependent metrics}\label{prooflambdadependent}
  \setcounter{equation}{0}
This section is devoted to the proof of Proposition~\ref{symboltilde} stated page~\pageref{symboltilde}. We therefore consider the~$\lam$-dependent metric and weight
$$
  \forall \lam \neq 0, \: \forall \,\Theta  \in \R^{2d}, \quad g_ \Theta ^{(\lam)}(  d\xi,  d\eta)\eqdefa{|\lam| (d  \xi ^2+d \eta^2)\over 1+|\lam|(1+  \Theta^2)}
    \quad \mbox{and} \quad      m^{(\lam)}( \Theta)\eqdefa\left( 1+|\lam|(1+ \Theta ^2)\right)^{1/2},
$$
and we aim at proving that the structural constants, in the sense of Definition~\ref{hormandermetric} page~\pageref{hormandermetric}, may be chosen uniformly of~$\lam$; the second point stated in Proposition~\ref{symboltilde} is obvious to check.

  It turns out that the proofs for the metric and for the weight are identical, so let us concentrate on the metric from now on, for which we need to prove   the uncertainty principle, as well as the fact that the metric is slow and temperate.

  The uncertainty principle is very easy to prove, since of course
  $$
   g_\Theta^{(\lam)\omega}(  d\xi,  d\eta)={ 1+|\lam|(1+  \Theta ^2) \over |\lam|} (d  \xi^2+d \eta^2)
  $$
  and
   $$
  |\lam| \leq  1+|\lam|(1+  \Theta^2).
  $$

  The slowness property is also not so difficult to obtain.  We notice indeed that, with obvious notation,
 $$
  g^{(\lam)}_ \Theta(\Theta-\Theta') = \frac{|\lam| | \Theta-\Theta'|^2 }{ 1 + |\lam| (1+ \Theta^2 )}
  $$
  and we want to prove that there is a constant~$\overline C$, independent of~$\lam$, such that if
  $$
  |\lam|  |\Theta-\Theta'|^2  \leq \overline C^{-1} ( 1 + |\lam| (1+  \Theta^2 )),
  $$
  then
  $$
  \frac{1 + |\lam| (1+  \Theta^2 )}{1 + |\lam| (1+  \Theta'^2 )} +
   \frac{1 + |\lam| (1+ \Theta'^2 )}{1 + |\lam| (1+  \Theta^2 )}  \leq \overline C.$$
  To do so, we shall decompose the phase space~$\R^{2d}$ into  regions in terms of the respective sizes  of~$\Theta^2$ and~$\Theta'^2$. In the following we shall write~$\Theta^2 \ll \Theta'^2$ if, say~$\Theta^2 \leq 10\,  \Theta'^2$, and~$|\Theta| \sim |\Theta'|$ will mean that, say~$\displaystyle \frac1{10}\Theta^2 \leq \Theta'^2 \leq 10\, \Theta^2$.

  Suppose first that
   $\Theta^2 \ll \Theta'^2$. Then of course
   $$
   1 + |\lam| (1+  \Theta^2 ) \leq 1
+ |\lam| (1+  \Theta'^2
   ),$$
   so we assume that~$\overline C \geq 1$.
   Moreover, using the obvious algebraic inequality
   $$
   \Theta'^2 \leq 2 |\Theta-\Theta'|^2 + 2 \Theta^2 ,
   $$
   we deduce that
   $$
  |\lam| \Theta'^2 \leq 2 |\lam| |\Theta-\Theta'|^2  + 2|\lam|  \Theta^2 \leq
    (2   \overline C^{-1} +2)( 1 + |\lam| (1+ \Theta^2  ))
       $$
       which
       leads immediately to the result as soon as
       $$
      2   \overline C^{-1} +2 \leq \overline C.
       $$

 Conversely if~$\Theta^2 \gg \Theta'^2$, then it is clear that
$$
1 + |\lam| (1+ \Theta'^2 ) \leq   1 + |\lam| (1+  \Theta^2   ).
$$
 Along the same lines as above   we get
\begin{eqnarray*}
   |\lam| \Theta^2 & \leq &  2  |\lam| \Theta'^2  + 2 \overline C^{-1} ( 1 + |\lam| (1+ \Theta^2 )) \\
  & \leq &    (2   \overline C^{-1} +2)( 1 + |\lam| (1+ \Theta^2  )) ,
   \end{eqnarray*}
   which choosing~$ \overline C$ large enough (independently of~$\lam$) gives the result.
   Since the estimate is obvious when~$|\Theta| \sim |\Theta'|$, the slowness property is proved, with a structural constant independent of~$\lam$.

  Finally let us prove that the metric is tempered, with uniform structural constants.
  This is again slightly more technical.  We need to find a  uniform constant~$\overline C$ such that
  $$
  \left( \frac{1 + |\lam| (1+  \Theta^2 )}{1 + |\lam| (1+  \Theta'^2  )}\right)^{\pm1} \leq \overline C \left(
  1 + \frac{1 + |\lam| (1+  \Theta^2 )}{|\lam|} |\Theta-\Theta'|^2
  \right).
  $$
  Notice that in the case when~$ |\Theta| \sim |\Theta'|  $, then the estimate is obvious because the
   left-hand side is bounded   by a  uniform  constant. Let us now deal with the two other types of cases,
   namely~$|\Theta|^2 \ll |\Theta'|^2$, and~$|\Theta'|^2 \ll |\Theta|^2$.

  Let us start with the case when the left-hand side has power +1. If~$|\Theta|^2 \ll |\Theta'|^2$,
   then the  left-hand side is uniformly bounded so the result follows with~$\overline C \geq 1$.
   Conversely if~$|\Theta'|^2 \ll |\Theta|^2$,
   then we notice that if~$0 <   |\lam | \leq 1$, then the   left-hand side is bounded by~$2+\Theta^2$ while
    the right-hand side is larger than~$\overline C(1  + c\Theta^2(1+\Theta^2))$ so the estimate is true.
    On the other hand when~$ |\lam | \geq 1$ then factorizing the  left-hand side by~$\lam$ and using
    the fact that~$ |\lam |^{-1} \leq 1$ and~$( |\lam |^{-1} + 1+\Theta'^2)^{-1} \leq (  1+\Theta'^2)^{-1} $ we get
  $$
  \frac{1 + |\lam| (1+  \Theta^2 )}{1 + |\lam| (1+  \Theta'^2  )} \leq 2 \frac{1+  \Theta^2}{1+  \Theta'^2
  }\leq 2(1+  \Theta^2 ). $$
  Again, since in that case~$|\Theta-\Theta'|^2 \geq c \Theta^2$, it comes
  $$ \displaystyle \left(
  1 + \frac{1 + |\lam| (1+  \Theta^2 )}{|\lam|} |\Theta-\Theta'|^2
  \right) \geq \displaystyle \left(
  1 + c \Theta^2(1+  \Theta^2 )
  \right)$$
 which implies easily the result.

  Now let us deal with the case when the left-hand side has power -1. The arguments are similar. Indeed if~$|\Theta'|^2 \ll |\Theta|^2$ then the
   left-hand side is uniformly bounded so the result follows. Conversely if~$|\Theta|^2 \ll |\Theta'|^2$ then when~$0 <   |\lam | \leq 1$ we use the fact that the left-hand side is bounded by~$2+\Theta'^2$ whereas the right-hand side is larger than~$c(1+\Theta'^2)$. When~$    |\lam | \geq 1$ then as above we write
  $$
    \frac{1 + |\lam| (1+ \Theta'^2 )}{1 + |\lam| (1+  \Theta^2  )} \leq 2\frac{1+  \Theta'^2}{1+ \Theta^2 }  \leq 2 (1+  \Theta'^2),
  $$
  and the result follows again from  the fact that since in that case~$|\Theta-\Theta'|^2 \geq c
  \Theta'^2$, one has
   $$ \displaystyle \left(
  1 + \frac{1 + |\lam| (1+  \Theta^2 )}{|\lam|} |\Theta-\Theta'|^2
  \right) \geq \displaystyle \left(
  1 + c \Theta'^2(1+  \Theta^2 )
  \right)\geq \displaystyle \left(
  1 + c \Theta'^2
  \right).$$
  The proposition is proved.  \qed

\section{$\lam$-dependent symbols}\label{prooflambdadependentsymbols}
  \setcounter{equation}{0}

In this subsection we shall prove Proposition~\ref{prop:sigma(a)} stated page~\pageref{prop:sigma(a)},   giving an equivalent definition of symbols in terms of the scaling function~$\sigma$.

For any multi-index~$\beta$   satisfying~$|\beta|\leq n$, we have
\begin{eqnarray}\nonumber
\left|\partial^\beta_{(y,\eta)}\left(\sigma(a)(w,\lam,\xi,\eta)\right)\right| & = & \left| \,|\lam|^{-{|\beta|\over 2}} \left(\partial^\beta_{(\xi,\eta)} a \right)  \left( w,\lam, {\rm sgn}(\lam) \frac{\xi}{\sqrt{|\lam|}},{\eta\over\sqrt{|\lam|}}\right)\right|\\
\label{tilt3}
& \leq & \| a\| _{n;S_{\H^d}(\mu)}\left(1+|\lam|+\xi^2+\eta^2 \right) ^{\mu-|\beta|\over 2}.
\end{eqnarray}
Besides, there exists a constant $C>0$ such that for $\lam\in\R$,
\begin{eqnarray*}
\left| (\lam\partial_\lam)^k \left( \sigma(a)(w,\lam,\xi,\eta)\right)\right|
& \leq & C  \left|\Bigl((\lam\partial_\lam)^k a\Bigr)\left(w,\lam,{\rm sgn}(\lam){\xi\over\sqrt{|\lam|}},{\eta\over\sqrt{|\lam|}}\right)\right|
\\
 {} +{} C& \displaystyle {\sum} _{|\beta|=k} & \displaystyle
|\lam|^{-{k\over 2}} (\xi^2+\eta^2)^{k\over 2}\left| \left(\partial^\beta_{(\xi,\eta)} a\right)\left(w,\lam,{\rm sgn}(\lam){\xi\over\sqrt{|\lam|}},{\eta\over\sqrt{|\lam|}}\right)\right|\ \\
& \leq & C \| a\|_{k,S_{\H^d}(\mu)}\left(1+|\lam|+\xi^2+\eta^2 \right) ^{\frac\mu2 }.
\end{eqnarray*}

The converse inequalities come easily: one has $a\in S_{\H^d}(\mu)$ if and only if for all $k,n\in\N$, there exists a constant $C_{n,k} $ such that for any $\beta\in \N^d$ satisfying~$|\beta|\leq n$ and for all~$ (w,\lam,y,\eta)$ belonging to~$ \H^d\times \R^{2d+1},$
\begin{equation}\label{tilt}
 \left\|(\lam\partial_\lam)^k \partial_{(\xi,\eta)}^\beta (\sigma(a))\right\|_{{\mathcal C}^\rho(\H^d)}\leq C_{n,k}\left(1+|\lam|+\xi^2+\eta^2\right)^{\mu-|\beta|\over 2}.
\end{equation}
We then remark that if $|\lam|\leq 1$,  the smoothness of $\sigma(a)$ yields that \aref{tilt3} implies on the compact $\{|\lam|\leq 1\}$,
$$(1+|\lam|)^k\left\|\partial_\lam^k \partial_{(\xi,\eta)}^\beta (\sigma(a))\right\|_{{\mathcal C}^\rho(\H^d)}\leq C_{n,k}\left(1+|\lam|+\xi^2+\eta^2\right)^{\mu-|\beta|\over 2}.$$
Besides, for $|\lam|\geq 1$, \aref{tilt} gives
$$\left\|\partial_\lam^k \partial_{(\xi,\eta)}^\beta (\sigma(a))\right\|_{{\mathcal C}^\rho(\H^d)}\leq C_{n,k}\left(1+|\lam|+\xi^2+\eta^2\right)^{\mu-|\beta|\over 2}(1+|\lam|)^{-k}.$$
Conversely, if \aref{tilt2} holds, then one gets \aref{tilt} since the function $\displaystyle{{|\lam|^p \over (1+|\lam|)^k}}$ is bounded for any integer~$p\in\{0,\cdots,k\}$. This ends the proof of the proposition.
\qed

  \section{Symbols of functions of the harmonic oscillator}
 \setcounter{equation}{0}\label{proofsymbR}
In this section  we aim at proving that an operator~$R(\xi^2-\Delta_\xi)$ given as a function of the harmonic oscillator by   functional calculus is a pseudodifferential operator, and at computing its (formal) symbol. We refer to Proposition~\ref{symbR} stated page~\pageref{symbR} for a   precise statement.
 Taking the inverse Fourier transform, we have by functional calculus
  $$
  R(\xi^2-\Delta_\xi)= \frac1{2\pi} \int_{\R} {\rm e}^{i\tau(\xi^2-\Delta_\xi)}\widehat R(\tau) \,d\tau.
  $$
   We  then use Mehler's formula as in \cite{fermanian}, which gives \aref{formuler} after an obvious change of variables.

   We therefore have formally
   \begin{equation}\label{formuler2}
   r(x)=\frac 1{2\pi}\int_{\R \times \R} ({\rm cos}\,\tau)^{-d} {\rm e}^{i (x {\rm tg} \tau
-y\tau)} R(y)d\tau\,dy,
   \end{equation}
and
   let us now prove that the function $r$ is well defined outside~$x=0$, and that the map~$(\xi,\eta)\mapsto r(\xi^2+\eta^2)$ satisfies the symbol estimates of the class
$S(m^\mu,g)$.

 If~$x\in\R^*$ is fixed, then~\aref{formuler2} defines $r(x) $ as  an oscillatory integral. Indeed the change of variables $u={\rm tg}\tau$
   performed on each interval of the form~$\left]-{\pi\over 2}+k\pi,k\pi+{\pi\over 2}\right[$ for~$k\in\ZZZ$
   turns the integral into a  series of oscillatory integrals: we have
   $\displaystyle r(x)=\sum_{k\in\ZZZ} r_k(x)$ with
\begin{eqnarray*}
r_k(x)&\eqdefa &\frac1{2\pi}(-1)^{kd}\,\int_{\R} {\rm e} ^{ixu} \widehat R\left(
k\pi+{\rm Arctg} u\right)(1+u^2)^{{d\over 2}-1}du \\
&=&\frac1{2\pi}(-1)^{kd}\,\int_{\R\times \R}  {\rm e} ^{ixu-iy{\rm Arctg}u-iyk\pi}   R\left(
y\right)(1+u^2)^{{d\over 2}-1}du\,dy .
\end{eqnarray*}
   We remark that these integrals have a non stationary phase for $|k|\geq 1$. This fact will be used below.
We also observe that for $N_0\in\N$,  by  integrations by parts,
\begin{eqnarray*}
k^{N_0}r_k(x) &=& \frac1{2\pi}k^{N_0 }(-1)^{kd}\,\int_{\R\times \R}  {\rm e} ^{ixu-iy{\rm
Arctg}u-iyk\pi}  R\left( y\right)(1+u^2)^{{d\over 2}-1}du\,dy\\ &
= &\frac1{2\pi} \,\frac{(-i)^{N_0}}{\pi^{N_0}}(-1)^{kd}\int_{\R\times \R}  {\rm e}
^{ixu-iyk\pi}(1+u^2)^{{d\over 2}-1}\partial_y^{N_0}\left(R(y){\rm
e}^{-iy{\rm Arctg}u}\right)du\,dy\\
& = & \frac1{2\pi} \,\frac{(-i)^{N_0}}{\pi^{N_0}}(-1)^{kd}\int_{\R\times \R}  {\rm e}
^{ixu-iyk\pi-iy{\rm Arctg}u}(1+u^2)^{{d\over 2}-1}f_{N_0}(y,u)du\,dy
\end{eqnarray*}
where $f_{N_0}(y,u)={\rm
e}^{iy{\rm Arctg}u}\partial_y^{N_0}\left(R(y){\rm
e}^{-iy{\rm Arctg}u}\right)$.
The fact that the integrals $r_k(x)$  are well defined away from zero and that the series in $k$ converges  then
comes from the following lemma.

\begin{lemme}\label{lem:intosc}
Let $f$ and $g$ be two smooth functions on $\R$ such that
$$\displaylines{
\forall n\in\N,\;\;\exists C>0,\; \; \forall u\in \R,\;\;\left|\partial^ng(u)\right|\leq C (1+u^2)^{\nu-n\over 2}\cr
\forall n\in\N,\;\;\exists C >0,\; \; \forall y\in \R,\;\;\left|\partial^nf(y)\right|\leq C (1+y^2)^{\mu-n\over 2},\cr
}$$
for some $\mu,\nu\in\R$.
Then for any~$a>0$, there exists a constant $C_0>0$ such that the function
$$I(f,g) (x)\eqdefa  \int_{\R\times \R}  {\rm e}^{ixu-iy{\rm Arctg} u-iyk\pi}f(y)g(u) dy\,du$$
satisfies
$$\forall |x|\geq a, \;\;\left|I(f,g) (x)\right|\leq C_0(1+x^2)^{\frac\mu2}.$$
\end{lemme}

Before proving this lemma, let us show how to use it. The function $  f_{N_0}(y,u)$ above writes as a sum of terms satisfying the  assumptions of the Lemma.
Therefore,   $ (1+|x|)^{-\mu} k^{N_0}r_k(x)$ is uniformly bounded in $k$ and $x$  whence the convergence of the series. To prove  the symbol estimate, we notice that two
integrations by parts give
\begin{eqnarray*}
xr'(x)& = &i\, x \int_{\R\times \R}   ({\rm cos}\tau)^{-d} {\rm tg}\tau {\rm
e}^{ix{\rm tg}\tau-iy\tau} R(y) dyd\tau\\ & = &  x\int_{\R\times \R}  ({\rm
cos}\tau)^{-d} {{\rm tg}\tau\over \tau}  {\rm e}^{ix{\rm
tg}\tau-iy\tau} R'(y) dyd\tau\\ & =& -i  \int_{\R\times \R}  ({\rm cos}\tau)^{-d}
{{\rm tg}\tau\over \tau} (1+({\rm tg}\tau)^2 )^{-1}\partial_\tau
\left( {\rm e}^{ix{\rm tg}\tau}\right){\rm e}^{-iy\tau} R'(y)
dyd\tau\\ & = &  i \int_{\R\times \R}  {\rm e}^{-iy\tau+ix{\rm
tg}\tau}\Biggl[-iy\left( ({\rm cos}\tau)^{-d} {{\rm tg}\tau\over
\tau} (1+({\rm tg}\tau)^2 )^{-1}\right) \\ &  & \qquad+
\partial_\tau \left( ({\rm cos}\tau)^{-d} {{\rm tg}\tau\over
\tau} (1+({\rm tg}\tau)^2 )^{-1}\right) \Biggr] R'(y) dyd\tau.
\end{eqnarray*}
This last integral is an oscillatory integral of the same kind as
the one defining $r(x)$, and can also be studied using Lemma~\ref{lem:intosc}.  This allows to obtain the symbol  bounds, by iteration of the argument to any order of derivatives.

Now let us prove  Lemma~\ref{lem:intosc}. The idea, as is often the case in this paper, is to use a stationary phase method. The variable~$x$ may be seen as a parameter in the problem, and one notices easily that~$x$ may be factorized out of the phase after having the change of variable~$y = x(1+t)$.   Moreover one notices that the phase is stationary at the point~$t=u=0$, when~$k=0$. This implies that one should use a dyadic partition of unity centered at that stationary point. One furthermore notices that   if~$|u|^2 \ll |t|$, then the~$u$-derivative of the phase is bounded from below, so it is enough to use a~$\partial_u$ vector field in the integrations by parts. As it produces naturally negative powers of~$t$, one can deduce the convergence of the dyadic series. In the case~$|t| \leq |u|^2$ however that vector field cannot   work since the~$u$-derivative of the phase may vanish. One must then use the whole vector field (in both~$u$ and~$t$ directions), and gaining negative powers of~$u$ turns out to be   more difficult.

 So let us start by  performing the change of variables $y = x(1+t)$ so that $I(f,g)$ writes
$$
I(f,g)(x)=x \:  {\rm e}^{-ixk\pi} \int_{\R\times\R} {\rm e} ^{ix\Phi_k(u,t)} f(x(1+t))g(u)dtdu,
$$
where
 $$
\Phi_k(u,t)\eqdefa \left(u-{\rm Arctg}u\right)   -t \left({\rm Arctg}u+k\pi\right).
$$
    The phase $\Phi_k $  satisfies
$$
\partial_t\Phi_k =-{\rm Arctg}u-k\pi\quad {\rm and}  \quad\partial_u\Phi_k ={ u^2-t\over 1+u^2}\cdotp
$$
When~$k \neq 0$,  $\Phi_k$ is therefore non stationary,  whereas when $k=0$, $\Phi_0$ has a non-degenerate  stationary point in $(0,0)$.
Therefore, we introduce a partition of unity on the real line:
$$
\forall z \in \R, \quad 1=\sum_{p\in \N\cup\{-1\}} \zeta_p(z)
$$
with $\zeta_{-1}$ compactly supported in a ball and for $p\in\N$, $\zeta_p(z)=\zeta(2^{-p}z)$ where $\zeta$ is compactly supported in a ring.
We get
$$
I(f,g)={\rm e}^{-i x k\pi}\sum_{p,q\in\N\cup\{-1\}} I_{p,q}(f,g)
$$
 with
 $$
I_{p,q}(f,g)(x)\eqdefa x  \int_{\R\times\R}  {\rm e} ^{ix\Phi_k(u,t) } \zeta_p( t)\zeta_q(u)f(x(1+t))g(u)dtdu.
$$
 These integrals are now well-defined because they are integrals of smooth compactly supported functions. We have to prove the convergence of the series in $p$ and $q$.
As explained above, we  shall argue differently whether~$|u|^2 \ll |t|$ or not. So let us fix a parameter~$\varepsilon<1/3$, to be chosen appropriately below,   and let us separate the study into two subcases, depending whether~$2^p> 2^{2q(1+\varepsilon)}$ (which corresponds to the case~$|u|^2 \ll |t|$) or~$2^p\leq 2^{2q(1+\varepsilon)} $.

$ $

Let us  suppose $p>2q(1+\varepsilon)$. We observe that in that case one has $ u^2-t\not=0$ on the support of~$ \zeta_p(t)\zeta_q(u)$, so as explained above one can  use integrations by parts with the vector field
\beq\label{defvectorfieldl}
\ell \eqdefa\left( i\partial_u\Phi_k \right)^{-1}\partial_u.
\eeq
Of course one has
$$
\ell \left({\rm exp}\left(i x \Phi_k \right)\right)=x \,  {\rm exp}\left(i\Phi_ k \right).
$$
Performing~$N$ integrations by parts for  $N\in\N$, we  find
\begin{eqnarray*}
I_{p,q}(f,g) (x)& = & x^{1-N} \int_{\R\times\R} {\rm e}^{i x \Phi_k }(\ell^*)^N\Bigl( f(x(1+t)) g(u) \zeta_p(t)\zeta_q(u) \Bigr)dtdu.
\end{eqnarray*}
We then write
 $$
 \ell^*=-\ell +ic
 $$
 where
$$
c  \eqdefa
-{\partial_u^2\Phi_k \over(\partial_u\Phi_k )^2}
 =  -2\, {u(1+t)\over (1+u^2)^2}\,{(1+u^2)^2\over ( u^2-t)^2}=-2\,{u(1+t)\over ( u^2-t)^2}\cdotp
$$
Let us analyze the properties of~$\ell^*$.
If~$(u,t)$ belongs to the support of $\zeta_q(u)\zeta_p(t)$, we have for $p>2q(1+\varepsilon)$
$$
c_2\, 2^p\leq c_12^p -C_1 2^{2q}\leq |t- u^2|\leq C_1\, 2^p(1+2^{2q-p})\leq C_2\, 2^p.
$$
We infer that
 $$
|\partial_u\Phi_k |^{-1} \leq C\, 2^{-p+2q}
 \;\;{\rm and}\;\;|c |\leq C\,2^{-p+q}.
 $$
  Using $\displaystyle q<{p\over 2(1+\varepsilon)}\virgp $ we have
 $$-p+2q<\left(-1+{1\over 1+\varepsilon}\right)p=-{\varepsilon\over 1+\varepsilon}p<-{\varepsilon\over 1+\varepsilon}\,{p\over 2}- \varepsilon q$$
 so that there exists some $\delta>0$ such that on the integration domain
 \begin{equation}\label{duphi-1dk}
 |\partial_u\Phi_k |^{-1} +|c |\leq C\,2^{-\delta(p+q)}.
 \end{equation}
 By induction
 one actually also can prove that
  \beq\label{derriveesded}
 \forall m \in \N, \quad |\partial^m_u c  | \leq C  2^{-m \delta(p+q)}.
 \eeq
Now we shall use the Leibniz formula in order to evaluate~$(\ell^*)^N \Bigl( f(x(1+t)) g(u) \zeta_p(t)\zeta_q(u) \Bigr)$.  This generates three typical terms:
\begin{eqnarray*}
(1) & \eqdefa & (\partial_u\Phi_k )^{-N} \partial_u^N\left(\zeta_q(u)g(u)\right)f(x(1+t))\zeta_p(t), \\
(2) & \eqdefa & c^N  \zeta_q(u)g(u)f(x(1+t))\zeta_p(t)\quad \mbox{and} \\
(3) & \eqdefa & \sumetage{n+m+p = N}{n,m,p<N} c^n    \:  \partial^m_u c  \: (\partial_u\Phi_k)^{-p} \partial_u^p\left(\zeta_q(u)g(u)\right) f(x(1+t))\zeta_p(t).
\end{eqnarray*}
Due to the estimates~(\ref{duphi-1dk}) and~(\ref{derriveesded}), it turns out that the term~$(3)$ is an
intermediate case between~$(1)$ and~$(2)$ so we shall only study the two first types of terms here.

We observe that defining~$\displaystyle \widetilde \zeta(u) = \sup_{n \leq N}  |\zeta^{(n)}(u)| $ and using the symbol estimate on~$g$,  we have
$$
\displaystyle |\partial_u^N\left(\zeta_q(u)g(u)\right)| \leq C \left(1+|u|\right)^{\nu}   2^{-qN}  \widetilde\zeta_q(u)
$$
so  by~(\ref{duphi-1dk})  and using the symbol estimate on~$f$ we obtain that
$$
| (1)| \leq C  \;2^{-qN}2^{-\delta N (p+q)}\left(1+|u|\right)^{\nu}  \left(1+|x(1+t)|\right)^\mu\zeta_p(t)\widetilde\zeta_q(u).
$$
Using Peetre's inequality
$$
(1+|x(1+t)|)^\mu\leq C\, (1+|x|)^\mu(1+|xt|)^{ |\mu|},
$$
we therefore conclude that (recalling that~$x$ is  away from zero)
\beq\label{integration1}
 x^{1-N} \int_{\R\times\R} |(1)|dtdu   \leq   C  \,  |x| ^\mu |x |^{1-N+ |\mu|}2^{-\delta N(p+q)+q\nu+p |\mu|+p+q-qN}  .
\eeq
 A similar argument allows to deal with the second term. Indeed we have
\beq\label{integration2}
| (2)| \leq C  \;   2^{-\delta N (p+q)}\left(1+|u|\right)^{\nu}  \left(1+|x(1+t)|\right)^\mu\zeta_p(z) \zeta_q(u)
\eeq
By integration we obtain
$$
 x^{1-N}\int_{\R\times\R} | (2)| \: dudz \leq C  |x| ^\mu |x |^{1-N+ |\mu|} 2^{-\delta N (p+q)+q\nu+p |\mu|+p+q}   .
 $$
Therefore, choosing $N> \delta^{-1} {\rm Max}(\nu +1,  |\mu|+1)$,  we obtain the convergence in $p$ and $q$ of the series, uniformly with respect to $k$ and $x$ in the set~$\{|x|\geq a\}$, with the expected bound~$ |x| ^\mu $.

\medskip

Let us now suppose $p\leq 2q(1+\varepsilon)$.  The objective is now to gain negative powers of $2^q$. The difficulty then comes from the fact that $\partial_u\Phi_k$ may vanish.
We observe that for this range of indexes $p$ and $q$, we have $q\geq 0$ so that the integral is supported far from $u=0$.
For this reason, if $\chi$ is a  smooth cut-off function, compactly supported in the unit ball and identically equal to one near zero, then the function
$$
(t,u)\mapsto \chi\left({t-u^2 \over u^\kappa}\right)
$$
is a smooth function for any $\kappa\in\R$.  The value of~$\kappa$ will be chosen later.

We now cut $I_{p,q}$ into two parts, writing $I_{p,q}=I^1_{p,q}+I^{2}_{p,q}$ with
$$
I^1_{p,q}(x)\eqdefa x \int_{\R\times\R} {\rm e}^{i x \Phi_k } \left(1 - \chi\left({t-u^2 \over u^ \kappa}\right)\right)f(x(1+t)) g(u) \zeta_p(t)\zeta_q(u)dt\,du.
$$
Let us study first $I^1_{p,q}$. We notice   that  on the domain of integration, one has~$|t-u^2  | \geq C |u|^\kappa$, so
on the support of~$\zeta_q$ we have~$ |t-u^2  | \geq C\, 2^{\kappa q}$. It follows that
 $$
 \left|{t-u^2   \over 1+ u^2}\right| \geq C\, 2^{(\kappa-2)q},
 $$
 which leads to
 \beq\label{estimationduphicas1}
  |\partial_u\Phi_k |^{-1} \leq C\, 2^{-(\kappa-2)q}.
 \eeq
 Therefore the~$u$-derivative of the phase does not vanish in this case,  so we may use again the vector field $\ell $ defined in~(\ref{defvectorfieldl}). The coefficients of that vector field are now of order~$2^{-(\kappa-2)q}$ and one has
\beq\label{estimationddk}
\left| c \right|= \left|-2{u(1+t)\over ( u^2-t)^2}\right| \leq C {2^q (1+2^p)\over 2^{2 \kappa q}}\leq C\, 2^{ -2 \kappa q+3q(1+\varepsilon)}.
\eeq
We therefore
  choose $\kappa $ such that $2 \kappa >3(1+\varepsilon)$.
By induction, one sees that
\beq\label{deriveesddknouveau}
\forall m \in \N, \quad \left| \partial^m c  \right| \leq C\, 2^{- m q -2 \kappa q+3q(1+\varepsilon)}.
\eeq
 We can write
   $$
   I^1_{p,q}(x) =
 x^{1-N}\int_{\R\times\R} {\rm e}^{i x \Phi_k } \left(\ell ^*\right)^N\left[\left(1-\chi\left({t-u^2 \over u^ \kappa}\right)\right) g(u)\zeta_q(u)\right]f(x(1+t)) \zeta_p(t)
 dt\,du.$$
 Compared to the case studied above, the terms generated by~$ \left(\ell ^*\right)^N$ are of the form
 \begin{eqnarray*}
(1') & \eqdefa & (\partial_u\Phi_k)^{-N} \partial_u^N\left( \left(1-\chi\left({t-u^2 \over u^ \kappa}\right)\right)  \zeta_q(u)g(u)\right)f(x(1+t) )\zeta_p(t), \\
(2') & \eqdefa & c^N \left(1-\chi\left({t-u^2 \over u^ \kappa}\right)\right)   \zeta_q(u)g(u)f(x(1+t))\zeta_p(t)\quad \mbox{and} \\
(3') & \eqdefa &\!\!\!\! \sumetage{n+m+p = N}{n,m,p<N}\!\! c^n      \partial^m_u c    (\partial_u\Phi_k)^{-p} \partial_u^p\left(\left(1-\chi\left({t-u^2 \over u^ \kappa}\right)\right) \zeta_q(u)g(u)\right) f(x(1+t))\zeta_p(t).
\end{eqnarray*}
As in the previous case and due to~(\ref{estimationddk}) and~(\ref{deriveesddknouveau}), it is enough to control the two first terms.

Thanks to~(\ref{estimationddk}), the term~$(2')$ is bounded exactly as before, assuming that~$2 \kappa >3(1+\varepsilon)$. Now let us study~$(1')$. As above we apply the Leibniz formula, which compared to the previous case generates   derivatives of~$\chi$. However they produce negative powers of~$2^q$, as one differentiation gives the term
 $$ \chi'\left({t-u^2 \over u^ \kappa}\right)\left[-   {2\over u^{\kappa-1}}-{\kappa\over u} \left({t-u^2 \over u^ \kappa}\right)\right]
$$
which may easily be bounded by
$$
\left|\chi'\left({t-u^2 \over u^ \kappa}\right)\left[-   {2 \over u^{\kappa-1}}-{\kappa\over u} \left({t-u^2 \over u^ \kappa}\right)\right] \right|
 \leq C  (2^{-q(\kappa-1)}  +2^{-q}) \leq C 2^{-q(\kappa-1)}$$
 assuming moreover that~$\kappa \leq 2$, which is possible since~$\e<1/3$. Similarly~$m$ derivatives produce~$2^{-q(\kappa-1)m}$, and it is easy to conclude that~$(1')$ may be dealt with as above, hence  can also be summed over~$q$ and~$p$ (recalling that~$p\leq 2q(1+\varepsilon)$, so that decay in~$2^q$ is enough to conclude to both summations).

Now let us study~$I^{2}_{p,q}$, which is more challenging as the~$u$-derivative of the phase can now vanish. We therefore need to use the full vector field
$$
L_k \eqdefa{1\over i} |\nabla\Phi_k |^{-2} \nabla\Phi_k \cdot \nabla
$$
which satisfies
$$
L_k \left({\rm exp}\left(i x \Phi_k \right)\right)=x \: {\rm exp}\left(i x \Phi_ k \right).
$$
Let us check that this vector field is well defined: on the one hand if~$k=0$, then the assumption~$\displaystyle q\geq {p\over 2(1+\varepsilon)}$ implies $q\geq 0$, thus $u$ is supported on a ring and $|{\rm Arctg}u|\geq c_0$ on the support of $\zeta_q(u)$.
On the other hand one notices that~$\displaystyle |\nabla\Phi_k|^2 \geq ({\rm Arctg}u+k\pi)^2 \geq c_0^2$ for~$k\geq 1$. It follows that there is a universal constant such that for any~$k \geq 0$ and on the domain of integration, the following bound holds:
$$
|\nabla\Phi_k|^{-1} \leq C.
$$
Moreover we have
  $$
  L_k ^*=-L_k + c_k
    $$
   with
 \begin{eqnarray*}
 c_k & \eqdefa &  -{1\over i} \nabla\cdot \left( |\nabla\Phi_k |^{-2} \nabla\Phi_k \right)\\
 & = & -{1\over i} \left[{\partial_u^2\Phi_k\over|\nabla\Phi_k|^2} - 2 {\partial_u\Phi_k\over|\nabla\Phi_k|^4}\left(\partial_u^2\Phi_k\partial_u\Phi_k+\partial^2_{ut}\Phi_k\partial_t\Phi_k\right)-2{\partial_t\Phi_k\over|\nabla\Phi_k|^4}\partial^2_{tu}\Phi_k\partial_u\Phi_k\right]\\
 & = & -{1\over i}\left[{\partial_u^2\Phi_k\over|\nabla\Phi_k|^2} - {2 \over|\nabla\Phi_k|^4}
 \left((\partial_u\Phi_k)^2\partial^2_u\Phi_k +2\partial_{tu}^2 \Phi_k\partial_t\Phi_k\partial_u\Phi_k\right)\right].
 \end{eqnarray*}
 In view of
 $$\partial_u^2 \Phi_k=2\,{u(1+t)\over (1+u^2)^2}\quad {\rm and} \quad \partial_{ut}^2\Phi_k=-{1\over 1+u^2}$$
 we have
 \beq\label{estimatec_k}
 |c_k|\leq C\,|\nabla\Phi_k|^{-2}\left(2^{p-3q}+2^{-2q}\right)\leq C\, 2^{ -(1-2\varepsilon)q} .
 \eeq
 An easy induction left to the reader actually shows that
 \beq\label{deriveesdeck}
 \forall \alpha \in \N^2, \quad |\partial^\alpha_{(u,t)} c_k| \leq C\, 2^{ -(|\alpha|+1)(1-2\varepsilon)q}.
 \eeq
  We then write for $N\in\N$
 $$
 I^2_{p,q}=x^{1-N}\int {\rm e}^{i x \Phi_k  } \left(L_k ^*\right)^N\left[\chi\left({t-u^2 \over u^ \kappa}\right)f(x(1+t)) g(u) \zeta_p(t)\zeta_q(u)\right]
 dt\,du.
 $$
  Now we need to understand the action of the operator~$(L_k ^*)^{N }$. The main difficulty will come from the~$t$-derivative, which does not produce directly negative powers of~$u$. However we notice that on the domain of integration, one has
$$
t = u^2 + Z u^\kappa \quad \mbox{with} \: |Z| \leq 1,
$$
so since~$\kappa $ has been chosen smaller than~$2$, there is a constant~$c>0 $ such that
$$
|t| \geq |u|^2 - |Z u^\kappa | \geq c |u|^2.
$$
This means that the domain of summation is actually essentially  restricted to
\beq\label{pequivtoq}
2q \leq p \leq 2q(1+\e)
\eeq
so it suffices to gain negative powers of~$t$ to conclude to convergence.

 The constant term~$c_k$ has already been computed and estimated in~(\ref{estimatec_k})-(\ref{deriveesdeck}). Moreover   following similar computations to above,   for any given function~$F$ one may write that
  \begin{eqnarray}
|(L_k ^*)^{N } F  |& \leq & C \sup_{|\alpha|=N }| \partial^{\alpha}_{(u,t)} F | +  | c_k^{N }  F | \nonumber\\
&+&C \sumetage{|\alpha + \beta |  + m= N } {|\alpha|, |\beta|, m <N }   |\partial^\alpha_{(u,t)} c_k | \: |c_k|^m \:| \partial^{\beta}_{(u,t)} F |
 \label{lestimationquonveut}.
 \end{eqnarray}
 The first step of the analysis therefore consists in estimating, for any~$|\beta| \leq N $, the quantity
  $$   \sum_{  m + m' = |\beta|} \partial_u^m \partial_t^{m'} \left(\chi \left({t-u^2 \over u^ \kappa}\right)
  \zeta_q(u) g(u) \zeta_p (t)f(x(1+t))
  \right) .
  $$
  Let us start by studying  the action of the~$u$-differentiations on~$\displaystyle \chi\left({t-u^2 \over u^ \kappa}\right)g(u)  \zeta_q(u)$. On the one hand one has, using the symbol estimate on~$g$,
 $$
  \left|\partial^m_u (\zeta_q(u) g(u)) \right| \leq C 2^{q(\nu-m)} \widetilde \zeta_q(u)
 $$
 where~$\displaystyle \widetilde \zeta_q(u)  \eqdefa \sup_{m \leq N } |\partial^m_u  \zeta_q(u) |$. This can in turn be written
   \beq\label{partialmu}
  \left|\partial^m_u (\zeta_q(u) g(u)) \right| \leq C 2^{q(\nu-m)} \overline \zeta(2^{-q}u)
 \eeq
 where~$\overline \zeta  $ is a nonnegative, smooth compactly supported function such  that~$\overline \zeta = 1$ on the support of~$\zeta$.

  On the other hand, as we have seen above one has the following identity:
$$  \partial_u  \left(\chi \left({t-u^2 \over u^ \kappa}\right)\right)  =  \chi'  \left({t-u^2 \over u^ \kappa}\right)\left[-  { 2\over u^{\kappa-1}}-{\kappa\over u } \:   \left({t-u^2 \over u^ \kappa}\right)\right]
$$
so since the support of~$\chi'$ does not touch zero, one has on the support of~$\zeta_q$ the following estimate:
$$
\left| \partial_u   \left(\chi \left({t-u^2 \over u^ \kappa}\right)\right)\right| \leq C ( 2^{-q   (\kappa-1)} + 2^{-q })  \leq C2^{-q   (\kappa-1)} ,
$$
as soon as~$\kappa \leq 2$.
Actually by induction one also has
 \beq\label{partialumleft}
\forall m \in \N, \quad \left| \partial_u^m  \left(\chi \left({t-u^2 \over u^ \kappa}\right)\right)\right| \leq   C2^{-q (\kappa-1)m} .
\eeq
The Leibniz formula yields for any~$m \leq N $
  $$
  \displaystyle \left |\partial_u^m
 \left(\chi \left({t-u^2 \over u^ \kappa}\right) \zeta_q(u) g(u)\right)
\right| \leq  C \sum_{ m' \leq m} \left(\begin{array}{c}m\\ m'
 \end{array}\right) \left |\partial^{m'}_u
(\zeta_q(u) g(u))
\right|    \left |\partial^{m - m'}_u  \left(
\chi \left({t-u^2 \over u^ \kappa}\right)
\right)
\right|
  $$
  whence by~(\ref{partialmu}) and~(\ref{partialumleft}) the estimate
  \beq \label{dualpha}
   \displaystyle \left |\partial_u^m
 \left(\chi \left({t-u^2 \over u^ \kappa}\right) \zeta_q(u) g(u)\right)
\right|\leq  C  2^{q(\nu-(\kappa-1)m)}  \overline \zeta(2^{-q}u).
  \eeq
   Now let us consider~$t$-derivatives. The Leibniz formula again implies that for any~$m' \leq N $
  \begin{eqnarray}\label{easyandhardtermzderivative}
  & \displaystyle \partial_t^{m'} \left( \chi\left({t-u^2 \over u^ \kappa}\right) \zeta_p (t)f(x(1+t)) \right) =
 \chi\left({t-u^2 \over u^ \kappa}\right)   \partial_t^{m'} \Bigl(\zeta_p (t)f(x(1+t)) \Bigr) \nonumber
\\
& \displaystyle \quad \quad \quad \quad\quad \quad \quad  \quad \quad  + \sum_{0<{n'}\leq {m'}} \left(\begin{array}{c}{m'}\\ {n'}
 \end{array}\right) \partial_t^{n'} \left(\chi \left({t-u^2 \over u^ \kappa}\right) \right)   \partial_t^{{m'}-{n'}} \Bigl(\zeta_p (t)f(x(1+t)) \Bigr) .
  \end{eqnarray}
  For the second term in the right-hand side of~(\ref{easyandhardtermzderivative}), one  uses the fact  that on the support of~$\zeta_q$, one has the estimate
\begin{eqnarray}\label{estimatedzchi}
   \left|\partial_t^{n'} \left(\chi \left({t-u^2 \over u^ \kappa}\right) \right)   \right|
 &=& \frac1{|u|^{{n'} \kappa}} \left|\chi^{({n'})} \left({t-u^2 \over u^ \kappa}\right)  \right| \nonumber \\
 &\leq & C  2^{-q {n'} \kappa} .
  \end{eqnarray}
In order to also control    the action of multiple differentiations in the~$t$ and~$u$ directions of~$\displaystyle \partial_u \left(\chi \left({t-u^2 \over u^ \kappa}\right) \right )$, it is useful to notice that
$$
\partial_u \left(\chi \left({t-u^2 \over u^ \kappa}\right) \right ) = -  {2\over u^{\kappa-1}} \chi'  \left({t-u^2 \over u^ \kappa}\right) +\frac{\kappa}u \widetilde \chi \left({t-u^2 \over u^ \kappa}\right)
$$
where~$\widetilde \chi$ is a smooth compactly supported function. So~$t$-derivatives of~$\displaystyle \partial_u   \left(\chi \left({t-u^2 \over u^ \kappa}\right)\right) $ are controled exactly like~$\displaystyle \partial_t   \left(\chi \left({t-u^2 \over u^ \kappa}\right)\right)$.

  Estimate~(\ref{estimatedzchi})   gives, along with the symbol estimate satisfied by~$f$, for any~${n'} \leq {m'}$,
  $$
 \left|\partial_t^{n'} \left(\chi \left({t-u^2 \over u^ \kappa}\right) \right)   \partial_t^{{m'}-{n'}} \Bigl(\zeta_p (t)f(x(1+t))\Bigr) \right|
 \leq  C  2^{-q {n'} \kappa} 2^{-p(m'-n')} (1+|x(1+t)|)^\mu \: \overline \zeta(2^{-p}t) ,
  $$
  where again~$\overline \zeta  $ is a nonnegative, smooth compactly supported function such  that~$\overline \zeta = 1$ on the support of~$\zeta$.

   Peetre's inequality allows finally to write that for any~${m'} \leq N $ and any~$0<{n'} \leq {m'}$,
   $$
   \left|\partial_t^{n'} \left(\chi \left({t-u^2 \over u^ \kappa}\right) \right)   \partial_t^{{m'}-{n'}} \Bigl(\zeta_p (t)f(x(1+t))\Bigr) \right|
 \leq  C 2^{-q   \kappa}  2^{-p(m'-n')} (1+|x |)^\mu (1+| xt|)^{ |\mu|} \overline \zeta(2^{-p}t),
   $$
   hence for any~${m'} \leq N $,  we get
   \begin{eqnarray}\label{easytermzderivative}
 &  \displaystyle  \!  \!  \!  \!  \!\left|\sum_{0<{n'}\leq {m'}} \left(\begin{array}{c}{m'}\\ {n'}
 \end{array}\right) \partial_t^{n'} \left(\chi \left({t-u^2 \over u^ \kappa}\right) \right)   \partial_t^{{m'}-{n'}} \Bigl(\zeta_p (t)f(x(1+t))\Bigr) \right| \nonumber \\
 & \displaystyle \quad \:  \quad \quad \quad \quad \quad \quad \quad \leq   C 2^{-q   \kappa} 2^{-p({m'}-{n'})} (1+|x |)^\mu (1+| xt|)^{ |\mu|} \overline \zeta(2^{-p}t) \nonumber\\
 & \displaystyle   \:  \quad     \quad \quad    \leq  C 2^{-q   \kappa +p |\mu|}  (1+|x |)^{\mu+ |\mu|}   \overline \zeta(2^{-p}t).
   \end{eqnarray}
       Finally let us deal with the first term on the right-hand side of~(\ref{easyandhardtermzderivative}). We   write, using Peetre's inequality again, that
 \beq\label{hardtermzderivative}
  \left|   \chi\left({t-u^2 \over u^ \kappa}\right)   \partial_t^{m'} \Bigl(\zeta_p (t)f(x(1+t))\Bigr) \right| \leq C  2^{-p({m'}- |\mu|)}  (1+|x |)^{\mu+ |\mu|}  \overline \zeta(2^{-p}t),
   \eeq
   and plugging~(\ref{easytermzderivative}) and~(\ref{hardtermzderivative})
  into~(\ref{easyandhardtermzderivative}) therefore gives
$$
 \left|   \partial_t^{m'} \left( \chi\left({t-u^2 \over u^ \kappa}\right) \zeta_p (t)f(x(1+t)) \right)  \right| \leq
   C      \overline \zeta(2^{-p}t)(
   2^{-q   \kappa +p |\mu|}  + 2^{-p({m'}- |\mu|)}) (1+|x |)^{\mu+ |\mu|}.
$$
   Putting the above estimate together with~(\ref{dualpha}) allows to obtain that
$$
\longformule{ \displaystyle   \sum_{  m + m' = |\beta|} \partial_u^m \partial_t^{m'} \left(\chi \left({t-u^2 \over u^ \kappa}\right)
  \zeta_q(u) g(u) \zeta_p (t)f(x(1+t))
  \right)}{ \displaystyle \leq    C  \overline \zeta(2^{-p}t)\overline \zeta(2^{-q}u)  \sum_{  m + m' = |\beta|}
 2^{q(\nu-(\kappa-1)m)} (
   2^{-q   \kappa +p |\mu|}  + 2^{-p({m'}- |\mu|)})(1+|x |)^{\mu+ |\mu|}
,}
$$
   hence, bounding~$p$ by~$2q(1+\e)$, we get
\begin{eqnarray}\label{finalouf}
 & \displaystyle
  (1+|x |)^{- \mu - |\mu|}  \sum_{  m + m' = |\beta|} \partial_u^m \partial_t^{m'} \left(\chi \left({t-u^2 \over u^ \kappa}\right)
  \zeta_q(u) g(u) \zeta_p (t)f(x(1+t))
  \right)  \nonumber \\
  & \leq   \displaystyle  C    \overline \zeta(2^{-p}t)\overline \zeta(2^{-q}u)    \sum_{  m + m' = |\beta|} 2^{q(\nu +2 |\mu|(1+\e)-(\kappa-1)m)} (
  2^{-q\kappa} + 2^{-pm'}).
\end{eqnarray}
  Finally let us go back to~(\ref{lestimationquonveut}).  Denoting~$\widetilde \mu\eqdefa 2 |\mu|(1+\e)$ and choosing
  $$
 F  \eqdefa \chi\left({t-u^2 \over u^ \kappa}\right)f(x(1+t)) g(u) \zeta_p(t)\zeta_q(u),
  $$
 one has the following estimate:
  \begin{eqnarray*}
  (1+|x|)^{-\mu - |\mu|} |(L_k ^*)^{N } F  |
 &  \leq &  C   \overline \zeta(2^{-p}t)\overline \zeta(2^{-q}u)  \displaystyle   2^{q(\nu +    \widetilde \mu)}   \\
   &    &{} \times  \displaystyle \sumetage{|\alpha| + |\beta| + n = N} {|\alpha|, |\beta|, n<N }    \sum_{m\leq|\beta| }2^{-(|\alpha|+1)(1-2\e)q -n (1-2\e)q -q (\kappa-1)m} (  2^{-q\kappa} + 2^{-p(|\beta| - m)}) \\
 &  +& C 2^{-N  (1-2\e)q}+ C \overline \zeta(2^{-p}t)\overline \zeta(2^{-q}u)  \displaystyle   2^{q(\nu + \widetilde\mu)}  \sum_{  m + m' =N  }2^{-q (\kappa-1)m} (
  2^{-q\kappa} + 2^{-pm'}).
   \end{eqnarray*}
   using the above estimate along with~(\ref{deriveesdeck}) and~(\ref{finalouf}).

  The conclusion comes from~(\ref{pequivtoq}).
   This ends the proof of the proposition.
\qed

  \section{The symbol of Littlewood-Paley operators on the Heisenberg group}\label{symbollpproofappendix}
\setcounter{equation}{0}
In this section we shall prove Proposition~\ref{symbDeltap} stated in Chapter~\ref{littlewoodpaley}, giving the symbol of the Littlewood-Paley truncation operators. The proof relies on the   arguments of the previous section, proving Proposition~\ref{symbR}.

 Recall that as defined in Definition~\ref{defLP},
  $${\mathcal F}(\Delta_p f)(\lam)={\mathcal F}(f)(\lam) R^*(2^{-2p} D_\lam)={\mathcal F}(f)(\lam) J_\lam^*R^*(2^{-2p} 4|\lam|(-\Delta_\xi+\xi^2))J_\lam.$$

If $\chi$ is a smooth cut-off function  compactly supported on $\R$ and such that~$\chi(\lam)=1$ for~$|\lam|\leq 4$ and~$\chi(\lam)= 0$ for~$|\lam|>5$, then
  $${\mathcal F}(\Delta_p f)(\lam)={\mathcal F}(f)(\lam) J_\lam^*R^*(2^{-2p} 4|\lam|(-\Delta_\xi+\xi^2)) \chi(2^{-2p}\lam)J_\lam.$$
  It will be important in the following to notice that for fixed $p$, we are only concerned with bounded frequencies $\lam$.

  We now apply Proposition~\ref{symbR} and write
    $$R^*(2^{-2p} 4|\lam|(-\Delta_\xi+\xi^2))=op^w\left(\Phi_p(\lam,\xi,\eta)\right)$$ with
    \begin{equation}\label{def:Phip}
    \Phi_p(\lam,\xi,\eta)=\frac1{2\pi}\int_{\R\times\R} \left({\rm cos}\tau\right)^{-d} {\rm e}^{i((\xi^2+\eta^2){\rm tg}\tau-r\tau)}R^*(2^{-2p+2}|\lam|r)dr\,d\tau.
    \end{equation}
   For $ \lam \not=0$, a change of variable shows that~$\Phi_p(\lam,\xi,\eta)=\phi(2^{-2p}|\lam|,2^{-2p}|\lam|(\xi^2+\eta^2)) $ as stated in   Proposition~\ref{symbDeltap}.

    Let us prove now that $\Phi_p\in S_{\H^d}(0)$. Actually due to the comment above, it is enough
    to prove that  the function~$(\lam,\xi,\eta)\mapsto \Phi_p(\lam,\xi,\eta)\chi(2^{-2p}\lam)$ is a symbol in $S_{\H^d}(0)$. It is moreover enough to prove it for $p=0$.

      We first observe that by  Proposition~\ref{symbR},    $\displaystyle  \Phi_0\left(\lam,{\rm sgn}(\lam) {\xi\over\sqrt{|\lam|}},{\eta\over\sqrt{|\lam|}}\right)=\phi(|\lam|,\xi^2+\eta^2)$  is well defined for $\lam\not=0$ and is a symbol in $S(1,g)$ for any $\lam$.
 Besides, Remark~\ref{rem:statphas} gives that~$\Phi_0$ has the required regularity close to $\lam=0$, and as noted above one can also restrict our attention to a compact set in~$\lam$.
      All those observations imply that to prove that the function $\Phi_0(\lam,\xi,\eta)$ belongs to the symbol class
$S_{\H^d}(0)$, it is enough due to   Proposition~\ref{symboltilde} to prove the following
     estimate: for any compact set~$K$ of~$\R^*$,
     \begin{equation}\label{claim18}
\forall k,n\in\N,
\;\;\exists C_{k,n}>0 ,\;\;\forall \rho\in \R,\;\;\forall \lam\in K,
\;\;\left| (1+\rho^2)^{\frac n2}(\lambda\partial_\lambda)^k  \partial_\rho ^n
\phi(\lambda,\rho)\right|\leq C_{n,k}.
  \end{equation}

  We point out that  by Proposition~\ref{symbR}, we already now that this estimate is true for  $\lam$ fixed in~$\R^*$.
 Moreover since~$\lam$ belongs to a compact set, it is enough to consider the~$\lam \partial_\lam $ derivatives and to prove that~$(\lam \partial_\lam) \phi(\lam,\rho)$ may be bounded  independently of~$\lam$.

 In fact we shall prove that~$\displaystyle  \lam \partial_\lam  \phi(\lambda,\rho)$ has the same integral form as~$\phi$, which by a direct induction will allow to conclude the proof of the proposition.
So let us compute~$\displaystyle  \lam \partial_\lam \phi(\lambda,\rho)$. We have
$$
   \lambda\partial_\lambda\phi(\lambda,\rho)  =  {1\over 2\lam \pi}\int
     ({\rm cos}\tau)^{-d}{\rm e}^{{i\over\lam}(\rho{\rm tg}\tau-r\tau)}\left(-{i\over\lam}(\rho{\rm tg}\tau-r\tau) -1\right)R^*(4r)dr\,d\tau,
$$
so integrating by parts we get
  $$
      \lambda\partial_\lambda\phi(\lambda,\rho)      =  - {1\over 2\lam \pi}\int
     ({\rm cos}\tau)^{-d}{\rm e}^{{i\over\lam}(\rho{\rm tg}\tau-r\tau)} \left[\partial_r \left( (\rho{{\rm tg}\tau\over\tau}-r )  R^* (4r)\right) +R^*(4r)\right] dr\,d\tau,
     $$
     which gives finally
     $$
          \lambda\partial_\lambda\phi(\lambda,\rho)      =  - {1\over 2\lam \pi}\int
     ({\rm cos}\tau)^{-d}{\rm e}^{{i\over\lam}(\rho{\rm tg}\tau-r\tau)} \left[4{\rho\,{\rm tg}\tau-r\tau\over\tau}\,(R^*)'(4r)\right] dr\,d\tau.
$$
  One then notices that
   $$\rho{\rm e}^{{i\over\lam}(\rho{\rm tg}\tau)}={\lam\over i}(1+({\rm tg}\tau)^2)^{-1} \partial_\tau\left({\rm e}^{{i\over\lam} \rho{\rm tg}\tau }\right),$$
   which allows to transform the integral into
   \begin{eqnarray*}
          \lambda\partial_\lambda\phi(\lambda,\rho)   &   =&
          \frac2{\lam \pi} \int     ({\rm cos}\tau)^{-d} {\rm e}^{{i\over\lam}(\rho{\rm tg}\tau-r\tau)} (R^*)'(4r) \: dr\,d\tau\\
          &-& \frac2{i\pi}  \int   ({\rm cos}\tau)^{-d} \frac{{\rm tg}\tau}{\tau (1+({\rm tg}\tau)^2)}  {\rm e}^{-ir\tau}
      \partial_\tau\left({\rm e}^{{i\over\lam} \rho{\rm tg}\tau }\right)(R^*)'(4r) \: dr\,d\tau.
     \end{eqnarray*}
The first integral on the right-hand side is exactly of the same form as~$\phi$, so to conclude we need to prove that the second integral can also be written in a similar way. Let us perform an integration by parts in the~$\tau$ variable. This produces the following identity:
   $$
   \longformule{
   \int   ({\rm cos}\tau)^{-d} \frac{{\rm tg}\tau}{\tau (1+({\rm tg}\tau)^2)}  {\rm e}^{-ir\tau}
      \partial_\tau\left({\rm e}^{{i\over\lam} \rho{\rm tg}\tau }\right)dr\,d\tau
     }{ =  \int {\rm e}^{-ir\tau + {i\over\lam} \rho{\rm tg}\tau }
      \left( ir -\partial_\tau \left(
        ({\rm cos}\tau)^{-d} \frac{{\rm tg}\tau}{\tau (1+({\rm tg}\tau)^2)}
      \right) \right)
        (R^*)'(4r) \: dr\,d\tau}
   $$
   which again is of a similar form that can be dealt with as in the proof of Proposition~\ref{symbR}.

The proof of Proposition~\ref{symbDeltap}  is complete.
 \qed


\end{document}